%% file: main.tex
\documentclass[11pt, a4paper]{amsart}

\usepackage{amsmath}
\usepackage{tikz}
\usepackage{amssymb}
\usepackage{enumitem}
\usepackage{geometry}
\usepackage{amsthm}
\usepackage{tikz-cd}
\usepackage[all]{xy}
\usepackage{amsfonts}

\usepackage{mathrsfs}
\usepackage[scr=boondox, cal=esstix]{mathalpha}
\usepackage{hyperref}
\usepackage{xcolor}
\usepackage{MnSymbol}

\usepackage{mathtools}
\usepackage{subcaption}

\input{header_labeling}
\input{header_newcommand}

\title{Eclipses on Zippers}
\author{KyeongRo Kim}
\address{\hskip-\parindent
School of
Mathematics\\
Korea Institute for Advanced Study(KIAS)\\
85 Hoegi-ro, Dongdaemun-gu, Seoul 02455, Republic of Korea.}
\email{kyeongrokim14@gmail.com}

\date{\today}

\begin{document}

\maketitle
\begin{abstract}
Calegari and Loukidou introduced \emph{zippers}, consisting of a disjoint pair of invariant real trees in the boundary of a closed hyperbolic $3$-manifold group, which ensure the existence of a universal circle.
We study the action of $\pi_1(M)$ on a minimal zipper and prove a fixed point dichotomy: every nontrivial element either fixes a unique point in each tree or acts freely on both. This answers a question of Calegari and Loukidou. As a consequence, there exists an element with exactly one fixed point in each tree.

\end{abstract}

\section{Introduction}
Motivated by Thurston’s vision of a universal circle \cite{ThurstonUniI}, Calegari--Dunfield proved the universal circle theorem \cite{CalegariDunfield03}. Since then, the theory has been extensively developed and applied to a wide range of topological and geometric structures on $3$-manifolds, including taut foliations, quasi-geodesic and pseudo-Anosov flows, and veering triangulations.
It now provides a common framework connecting these structures.

Recently, Calegari and Loukidou \cite{CalegariLoukidou26} introduced a remarkably simple topological notion, called \emph{zippers}, which gives a sufficient condition for the existence of a universal circle: for a closed hyperbolic $3$-manifold $M$, a \emph{zipper} is a pair $Z^\pm$ of disjoint, non-empty, path-connected subsets of the Gromov boundary $S_\infty^2 = \partial_\infty \pi_1(M)$ that are invariant under the action of $\pi_1(M)$.

The main result of this paper is the following:
\begin{thmA}[Zipper Fixed Point Dichotomy]\label{Thm:tameFix}
Let $Z^\pm$ be a minimal zipper for a closed hyperbolic $3$-manifold $M$.
Then for every $g \in \pi_1(M)\setminus\{1\}$,  one of the following holds:
\begin{itemize}
    \item(hyperbolic-like) $g$ acts freely on both $Z^\pm$;
    \item(pseudo-Anosov-like) $g$ fixes a unique point in each of $Z^\pm$.
\end{itemize}
\end{thmA}

This provides an affirmative answer to the following question posed by Calegari and Loukidou in the first version of their arXiv paper \cite{CalegariLoukidou24} (see Remark~2.27 in \cite{CalegariLoukidou26}):

\begin{ques}[Remark~2.21 in \cite{CalegariLoukidou24}]\label{Que:bothInOne}
It is natural to expect that fixed points of group elements play a key role in understanding the geometry and topology of zippers, and it would be worthwhile to develop a clearer picture of this relationship.
For instance:
is it possible that some element $g \in \pi_1(M)$ has both fixed points in $Z^+$?
Or can there exist an element $g \in \pi_1(M)$ with one fixed point in $Z^+$ and the other lying in neither of $Z^\pm$?
\end{ques}

\subsection{Calegari's Closed Orbit Conjecture}In \cite{Frankel13,Frankel18}, Frankel proved Calegari's closed orbit conjecture: every quasigeodesic flow on a closed hyperbolic $3$-manifold $M$ admits a closed orbit.
The key idea is to pass to the \emph{flowspace} $P$, the orbit space of the lifted flow on $\widetilde M \simeq \mathbb H^3$.
This is a topological plane equipped with continuous endpoint maps $e^\pm \colon P \to S^2_\infty$ assigning to each flowline its forward and backward limit points (see \cite{Calegari06}).
The fibers of $e^\pm$ define two decompositions $P^\pm$ of $P$, and in this setting, closed orbits correspond to points of $P$ fixed by nontrivial elements of $\pi_1(M)$.

Calegari and Loukidou \cite[Conjecture~3.10]{CalegariLoukidou26} conjectured that, given a minimal zipper $Z^\pm$ for $M$, there exists a quasigeodesic pseudo-Anosov flow on $M$ without perfect fits whose endpoint maps satisfy $e^\pm(P^\pm)=Z^\pm$.
In this context, it is natural to remark the following corollary:

\begin{thmA}[Existence of a pseudo-Anosov-like element]
Let $Z^\pm$ be a minimal zipper for $M$.
Then there exists $g \in \pi_1(M)$ that fixes a unique point in each of $Z^\pm$.
\end{thmA}
\begin{proof}
Assume that no such element exists. Then \refthm{tameFix} implies that $\pi_1(M)$ acts freely on each of $Z^\pm$. By \refthm{zipperRtree}, $Z^\pm$ are real trees.

By Levitt's version of the Rips theorem \cite[Corollary, p.~47]{Levitt98}, a finitely presented group acting freely by homeomorphisms on an real tree is free abelian or splits over a (possibly trivial) cyclic group. This yields the desired contradiction.
\end{proof}

\subsection{Eclipses and the Structure of the Proof}\label{Sec:proofOutline}
An element \( g \in \pi_1(M) \) as in \refque{bothInOne} gives rise to an \emph{eclipse} on the zipper.
For example, suppose that \( g \) has a repelling fixed point \( p \in Z^+ \) and an attracting fixed point \( q \notin Z^- \).
Roughly speaking, an eclipse is a bi-infinite nested family \( \{\cE_n\}_{n\in\mathbb{Z}} \) of closed Jordan domains in \( S^2_\infty \), defined by
\[
\cE_n = g^n(\cE_0), \qquad \cE_{n+1}\subsetneq \cE_n \qquad \text{for all } n\in\mathbb{Z},
\]
where \( \cE_0 \) is a suitably chosen closed Jordan domain.
The domain \( \cE_0 \) is chosen to contain both fixed points of \( g \) with $p\in \partial \cE_0$, and therefore so does each \( \cE_n \).
Hence the intersection
\[
\cE_\infty = \bigcap_{n\in\mathbb{Z}} \cE_n
\]
still contains both fixed points.
By contrast, any closed Jordan domain containing only the attracting fixed point shrinks to a point under forward iteration by \( g \).
The precise construction of eclipses requires a more detailed analysis of the topology of the zipper, and will be given in \refsec{eclipses}.

We now briefly explain the role of eclipses in the case where both fixed points of \(g\) lie in \(Z^+\); this is the first question in \refque{bothInOne}.
The other case requires more sophisticated variations of the same argument. This serves as a summary of the discussion up to \refsec{twoInOne}.
In this case, there exists a \( g \)-invariant Jordan curve \( \cJ(g) \), called a \emph{fence}, consisting of a closed arc \( \cJ(g)^+ \subset Z^+ \) and an open arc \( \cJ(g)^- \subset Z^- \), each connecting \( p \) to \( q \).
For each complementary disk \( \cD \) of \( \cJ(g) \), we construct an eclipse \( \{\cE_n\}_{n\in \mathbb{Z}} \) associated to \( g \) or \( g^{-1} \) such that \( \bigcup_{n\in \mathbb{Z}} \cE_n = \overline{\cD} \), \( \cJ(g)^+ \subset \partial \cE_n \) for all \( n\in\mathbb{Z} \), and \( \cJ(g)\cap \cE_\infty = \cJ(g)^+ \).

A key property is that for any \( h \in \pi_1(M) \), the curves \( h(\cJ(g)) \) and \( \cJ(g) \) are unlinked, meaning that \( h(\cJ(g)) \) does not meet both complementary components of \( \cJ(g) \).
We refer such a $\pi_1(M)$-orbit of $\cJ(g)$ an \emph{invariant fence system} for $\pi_1(M)$.
This becomes useful when one chooses a quasi-Fuchsian closed surface subgroup \( \Gamma \leq \pi_1(M) \) whose limit circle \( \Lambda(\Gamma) \) separates \( p \) from \( q \).
Starting at \( p \) and moving toward \( q \) along the arcs \( \cJ(g)^+ \) and \( \cJ(g)^- \), let \( x^+ \) and \( x^- \) denote the first intersection points with \( \Lambda(\Gamma) \), respectively.
The \( \Gamma \)-orbit of \( \{x^+,x^-\} \) determines a \( \Gamma \)-invariant circle lamination \( \cL \) on \( \Lambda(\Gamma) \), which in turn corresponds to a geodesic lamination on the closed surface associated to \( \Gamma \).

We then analyze sequences of leaves of \( \cL \) for which one endpoint converges to \( x^+ \) or \( x^- \).
At this stage, the eclipses associated to \( \cJ(g) \) play a crucial role: they constrain the possible limiting locations of the other endpoints by forcing the limit into \( \cE_\infty \).
This ultimately leads to a contradictory limiting configuration, completing the proof.

For the remaining case where $g$ has one fixed point in $Z^+$ and the other in neither $Z^+$ nor $Z^-$, which is the second question in \refque{bothInOne}, the same general strategy applies. We divide this case into three subcases; see the beginning of \refsec{oneprong}. In each case, one first constructs an appropriate eclipse and then introduces a modified fence, such as a \emph{fat fence} or a \emph{squashed fence}. Using these objects together with the existence of abundant quasi-Fuchsian closed surface subgroups, established by Kahn--Markovic and Markovic, one again constructs a circle lamination. The key input is the following theorem.
\begin{thm}[Enough quasi-convex surface subgroups \cite{KahnMarkovic12},\cite{Markovic13}]\label{Thm:separatingQF}
For any distinct points \( p,q \) in \( S_\infty^2 \), there is a quasi-Fuchsian closed surface subgroup of \( \pi_1(M) \) whose limit set, which is a Jordan curve, separates \( p \) and \( q \).
\end{thm}
Analyzing the possible limiting behavior of leaves then leads, as above, to the desired contradiction.

\section{Topology of Zippers}

\subsection{Topology on $\RR$-trees}
Let $T$ be a (topological) real tree.
A subset $\gamma$ of $T$ is an \emph{arc, ray} or \emph{line} in $T$ if  $\gamma$ is the image of a continuous injective map $f$ from $[0,1]$, $[0,1)$ or $(0,1)$, respectively.
We call $f$ an \emph{orientation} of $\gamma$.
If $\gamma$ is a ray, then there is a unique orientation up to parametrization.
Note that an orientation $f$ induces a natural linear order $\leq$ on $\gamma$ under which $f$ is increasing.

A subset $\gamma$ of $T$ is called a \emph{segment} in $T$ if $\gamma$ is arc, ray or line.
The \emph{interior} $\Int{\gamma}$ of $\gamma$ is defined as the maximal line, contained in $\gamma$.
We write $\partial \gamma= \gamma \setminus \Int{\gamma}$ and call its elements \emph{boundary points} of $\gamma$.
In particular, $\gamma$ is said to be \emph{oriented} if an linear order is fixed, and \emph{proper} if an associated orientation $f$ is proper.
We denote by $[x,y]$ the unique arc between $x$ and $y$ in $T$ if $x\neq y$ or the point $\{x\}$ if $x=y$.
We also define $(x,y]=[x,y]\setminus \{x\}$, $[x,y)=[x,y]\setminus \{y\}$, and $(x,y)=[x,y]\setminus \{x,y\}$.

If $\ell$ is a line in $T$, then any subray $r$ of $\ell$ is called an \emph{end ray} of $\ell$ if
an orientation of $r$ is proper with respect to $\ell$.
If $r$ is a ray in $T$, then for any $x\in \Int r$, the subray $r\setminus [\partial r, x)$ is called the \emph{truncated ray} of $r$ at $x$.

For a point $p$ in $T$, we define $T_p$ as the collection of components of $T\setminus \{p\}$.
Each element of $T_p$ is called a \emph{branch} at $p$.
The point $p$ is called
\begin{itemize}
    \item a \emph{hard end} if $|T_p|=1$;
    \item a \emph{cut point} if $|T_p|\geq 2$;
    \item a \emph{branched point} if $|T_p|>2$.
\end{itemize}
A subset $b$ of $T$ is called a \emph{branch} of a segment $\gamma$ if $b$ is a branch at $p$, that does not intersect $\gamma$.

Recall \cite[Theorem~3]{Levitt98}, which is frequently used:
\begin{thm}[\cite{Levitt98}]\label{Thm:invHypAxis}
    Every fixed-point-free homeomorphism $g$ of a real tree $T$ admits a unique $g$-invarinat proper line in $T$, called the axis of $g$.
\end{thm}
\subsection{Jordan segments}

A \emph{Jordan arc} in \(S^2\) is the image of an injective continuous map \([0,1]\to S^2\), and a \emph{Jordan curve} is the image of an injective continuous map \(S^1\to S^2\). A \emph{Jordan segment} (resp. \emph{Jordan ray}, resp. \emph{Jordan line}) is a segment (resp. ray, resp. line) in a Jordan arc, viewed as a real tree.

By the Jordan--Schoenflies theorem, every Jordan curve in \(S^2\) is carried by an orientation-preserving homeomorphism of \(S^2\) to the equator.
Moreover, there are no wild Jordan arcs in $S^2$: every Jordan arc is contained in a Jordan curve (see \cite[Chapter~VI, Theorem~14.5]{Newman51} for an elementary proof). It follows that every Jordan arc in \(S^2\) is carried by an orientation-preserving homeomorphism of \(S^2\) to a subarc of the equator.

Consequently, every interior point \(x\) of a Jordan segment \(\cS\) has a neighborhood \(U\) that is a Jordan domain, together with an orientation-preserving homeomorphism
\[
\varphi \colon (\mathbb{R}^2,\mathbb{R}\times\{0\}) \to (U,U\cap \cS).
\]
We call such a \(U\) a \emph{flat neighborhood} of \(x\) relative to \(\cS\). If \(\cS\), regarded as a real tree, is oriented, then \(\varphi\) may be chosen to be increasing on \(\mathbb{R}\times\{0\}\). In that case, the sets \(\varphi(\mathbb{R}\times [0,\infty))\) and \(\varphi(\mathbb{R}\times (-\infty,0])\) are called the \emph{left} and \emph{right neighborhoods} of \(x\) relative to \(\cS\), and their interiors the \emph{left} and \emph{right open neighborhoods} of $x$ relative to $\cS$.

Now let \(\cJ\) be a Jordan curve, and let \(\cD\) be a complementary Jordan domain of \(\cJ\). If \(\cS\) is a Jordan segment, its closure \(\closure{\cS}\) in \(S^2\) is a Jordan arc. Write
\[
\sfd \cS := \closure{\cS}\setminus \Int \cS.
\]
Then \(\sfd \cS\) consists of two points. We say that \(\cS\) \emph{crosses} \(\cJ\) through \(\cD\) if
\[
\closure{\cS}\cap \cJ=\sfd \cS
\quad\text{and}\quad
\Int \cS \subset \cD .
\]
When \(\cJ=\partial\cD\) is understood, we simply say that \(\cS\) \emph{crosses} \(\cD\).

If \(\cS\) is oriented and crosses \(\cD\), then \(\cS\) separates \(\cD\) into two Jordan domains. We denote them by \(\cD^L(\cS)\) and \(\cD^R(\cS)\), and call them the \emph{left} and \emph{right Jordan domains} of \(\cS\) in \(\cD\).

The following elementary fact will be used repeatedly:
\begin{prop}
If a Jordan curve \(\cJ\) separates the endpoints of a Jordan arc \(\cA\), then \(\cJ\cap \cA\neq \emptyset\).
\end{prop}

\subsection{Zippers}
Throughout this paper, we fix a closed hyperbolic \(3\)-manifold \(M\) and denote its fundamental group by \(G\). We also regard \(G\) as a discrete subgroup of \(\PSL(\CC)\). Let \(S_\infty^2\) denote the sphere at infinity of \(\HH^3\), identified with the Gromov boundary \(\partial_\infty G\).

\begin{defn}
A \emph{zipper} for \(M\) is a pair \(Z^\pm\) of subsets of \(S_\infty^2=\partial_\infty G\) satisfying the following:
\begin{enumerate}
    \item each of \(Z^\pm\) is nonempty and path-connected;
    \item each of \(Z^\pm\) is \(G\)-invariant; and
    \item \(Z^+\) and \(Z^-\) are disjoint.
\end{enumerate}
\end{defn}

In \cite{CalegariLoukidou26}, Calegari and Loukidou observed that, for each \(\alpha \in \{+,-\}\), any two points of \(Z^\alpha\) are joined by a unique embedded path. The \emph{convex hull} of a subset \(P \subset Z^\alpha\), where \(\alpha \in \{+,-\}\), is the union of all embedded paths in \(Z^\alpha\) joining pairs of points of \(P\). When \(P\) is finite, its convex hull is homeomorphic, as a subspace of \(S_\infty^2\), to a finite simplicial tree.

\begin{thm}[\cite{CalegariLoukidou26}]\label{Thm:zipperRtree}
Let \(Z^\pm\) be a zipper for \(M\). Then each of \(Z^\pm\) is a real tree with respect to the \emph{path topology}, namely the weak topology generated by the subspace topology, as a subset of \(S_\infty^2\), on the convex hulls of finite subsets. Moreover, each inclusion map \(Z^\pm \hookrightarrow S_\infty^2\) is continuous but not an embedding.
\end{thm}

A zipper \(Z^\pm\) is said to be \emph{minimal} if each of \(Z^+\) and \(Z^-\) coincides with the convex hull of the \(G\)-orbit of any one of its points. It is immediate that a minimal zipper has no hard ends. By \cite{CalegariLoukidou26}, every zipper contains a minimal subzipper.

\subsection{Landing rays}

Let \(Z^\pm\) be a zipper for \(M\). Since \(M\) is closed and hyperbolic, every element of \(G\) is loxodromic in \(\PSL(\CC)\). For \(g\in G\), write \(a(g)\) and \(r(g)\) for the attracting and repelling fixed points of \(g\), respectively, and \(\Fix(g)\) for the fixed point set of \(g\).

A ray in \(Z^\alpha\), where \(\alpha\in\{+,-\}\), is called a \emph{landing ray} if its image in \(S_\infty^2\) admits a compactification by a single point \(e\) to a closed interval. Equivalently, it is a Jordan ray in \(S_\infty^2\). The point \(e\) is called the \emph{end} of the landing ray, and we say that the ray \emph{lands} at \(e\), or simply that it \emph{lands}. Note that \(e\) need not lie in \(Z^\pm\). As observed in \cite{CalegariLoukidou26}, even for a proper ray \(r\), it may be difficult to determine whether \(r\) lands. On the other hand, if the topological end of a ray is preserved by an element of \(G\), then the ray lands at a fixed point; see \cite[Lemma~2.12, Lemma~2.24]{CalegariLoukidou26}.

\begin{lem}[Invariant ray lands, \cite{CalegariLoukidou26}]\label{Lem:invRayLand}
Let \(x\in Z^\alpha\), where \(\alpha\in\{+,-\}\), and let \(g\in G\). Suppose that \(\sigma=[x,g(x)]\) is an arc in \(Z^\alpha\). Then either \(\Int \sigma \) contains exactly one fixed point of \(g\), or
\(
\bigcup_{n\ge 0} g^n(\sigma)
\)
contains a ray landing at \(a(g)\) in \(S_\infty^2\). In particular, if \(\ell\) is a \(g\)-invariant line in \(Z^\alpha\), then each end ray of \(\ell\) lands at a fixed point of \(g\).
\end{lem}

Combining \reflem{invRayLand} with \refthm{invHypAxis}, we obtain the following.

\begin{prop}\label{Prop:invBranch}
Let \(p\) be a fixed point of \(g\in G\), and suppose that \(p\in Z^\alpha\) for some \(\alpha\in\{+,-\}\). If \(B\) is a branch at \(p\) such that \(g(B)=B\), then there exists a landing ray \(\gamma\subset B\cup\{p\}\) starting from \(p\) and landing at the other fixed point \(q\) of \(g\), whether or not \(q\in B\).
\end{prop}

\begin{proof}
If \(B\) contains \(q\), then \([p,q)\) is the required landing ray. Otherwise, \(g\) acts freely on \(B\), which is again a real tree, and the result follows from \refthm{invHypAxis} and \reflem{invRayLand}.
\end{proof}

Let \(\cS\) be an oriented Jordan segment. A ray \(r\) in \(Z^\pm\) is said to land at \(p\in\Int \cS\) on the \emph{right} side (resp. the \emph{left} side) of \(\cS\) if \(r\) lands at \(p\) and every right open neighborhood (resp. left open neighborhood) of \(p\) relative to \(\cS\) contains a truncated subray of \(r\). A Jordan segment is said to land at \(p\in\Int \cS\) on the \emph{right} side (resp. the \emph{left} side) of \(\cS\) if one of end rays of its interior line does.

We say that an oriented Jordan segment $\cS$ is :
\begin{itemize}
    \item \emph{left} (resp. \emph{right}) \emph{inaccessible} if there is no ray in $Z^\pm$ that lands on the left (resp. right) side of $\cS$;
    \item \emph{inaccessible} if $\cS$ is left and right inaccessible;
    \item \emph{one-sided inaccessible} if it is left inaccessible but not right inaccessible under some orientation;
    \item \emph{at least one-sided inaccessible} if it is left inaccessible under some orientation;
    \item \emph{two-sided accessible} if it is not at least one-sided inaccessible.
\end{itemize}

\subsection{Circular orders on Jordan curves}

A \emph{circular order} on a set \(X\) is a map
\[
c \colon X^{\times 3}\to \{-1,0,1\}\subset \ZZ
\]
satisfying the following conditions:
\begin{enumerate}
    \item \(c^{-1}(0)=\Delta_3(X)\), where
    \[
    \Delta_3(X)=\{(x_1,x_2,x_3)\in X^{\times 3}\mid x_i=x_j \text{ for some } i\neq j\};
    \]
    \item \(c\) satisfies the cocycle condition
    \[
    c(x_1,x_2,x_3)-c(x_0,x_2,x_3)+c(x_0,x_1,x_3)-c(x_0,x_1,x_2)=0
    \]
    for all \(x_0,x_1,x_2,x_3\in X\).
\end{enumerate}

For points \(x_1,\dots,x_n\in X\), we write
\(
x_1\le x_2\le \cdots \le x_n
\)
if \(c(x_1,x_i,x_{i+1})\ge 0\) for all \(i\). We define \(\opi[X]{x}{y}\) to be the set of all \(z\in X\) such that \(c(x,z,y)=1\). Similarly, we use \(\cldi[X]{x}{y}\), \(\ropi[X]{x}{y}\), and \(\lopi[X]{x}{y}\) in the usual sense.

The circle \(S^1\) carries a natural circular order \(co\), defined by setting \(co(x,y,z)=1\) if \(x,y,z\) occur in counterclockwise order, \(co(x,y,z)=-1\) if they occur in clockwise order, and \(co(x,y,z)=0\) if at least two of \(x,y,z\) coincide. A \emph{circular order} on a Jordan curve \(\cJ\) is the pushforward of \(co\) under a homeomorphism from \(S^1\) onto \(\cJ\).

Let \(\cS\) be a Jordan segment contained in a Jordan curve \(\cJ\). We say that a linear order \(<\) on \(\cS\) and a circular order \(c\) on \(\cJ\) are \emph{compatible} if
\[
x<y \quad\text{if and only if}\quad c(p,x,y)>0
\]
for some \(p\in \cJ\setminus \cS\). This condition is independent of the choice of \(p\). Given a linear order on \(\cS\), there is a unique compatible circular order on \(\cJ\), and conversely.

Once an orientation of the sphere \(S^2\) is fixed, it induces a canonical orientation on each Jordan domain \(\cD\), and hence a circular order on \(\partial \cD\). Conversely, if a Jordan curve \(\cJ\) is equipped with a circular order \(c\), then there is a unique Jordan domain \(\cD^L(\cJ)\) bounded by \(\cJ\) whose induced boundary order is \(c\). Note that \(\cD^L(\cJ)\) contains the left open neighborhoods of points of \(\cJ\) relative to subarcs of \(\cJ\). For this reason, we call \(\cD^L(\cJ)\) the \emph{left Jordan domain} of \(\cJ\) with respect to \(c\). The other Jordan domain bounded by \(\cJ\) is called the \emph{right Jordan domain} of \(\cJ\), and is denoted by \(\cD^R(\cJ)\).

Let \(\cJ\) be a circularly ordered Jordan curve. A landing ray \(r\) in \(Z^\pm\) is said to \emph{land at \(p\in \cJ\) on the right side} (resp. \emph{on the left side}) of \(\cJ\) if the end of \(r\) is \(p\) and \(r\) is contained in the right Jordan domain (resp. the left Jordan domain) of \(\cJ\).
We say simply that \(r\) \emph{lands on \(\cJ\)} if \(r\) lands at some point of \(\cJ\) on either side with respect to some circular order on \(\cJ\).

\subsection{Circular orders on branches}\label{Sec:COonB}
Since $Z^\alpha$, where $\alpha\in \{+,-\}$, is a continuous injective image of a real tree into $S^2$,
for each branched point $p\in Z^\alpha$, there is a natural circular order on the set of branches, $Z^\alpha_p$, induced from a given orientation of $S_\infty^2$.

For each branch point \(p\in Z^\alpha\), we define a circular order \(c_p^\alpha\) on \(Z_p^\alpha\) as follows. Let \(b_1,b_2,b_3\) be distinct elements of \(Z_p^\alpha\), and choose points \(z_i\in b_i\). Then the convex hull \(\cH\) of \(\{z_1,z_2,z_3\}\) is a tripod with median point \(p\). Fix a linear order on
\(
[z_1,p]\cup[p,z_2]=[z_1,z_2]
\)
for which \(z_1<p<z_2\). We set \(c_p^\alpha(b_1,b_2,b_3)=1\) if a left neighborhood of \(p\) relative to \([z_1,z_2]\) contains a subarc \([p,q]\subset [p,z_3]\) for some \(q\in (p,z_3)\), and set \(c_p^\alpha(b_1,b_2,b_3)=-1\) otherwise. This definition is independent of the choice of the points \(z_i\), and hence \(c_p^\alpha\) is well defined.

Let \(\gamma\) be an oriented segment in \(Z^\alpha\), where \(\alpha\in\{+,-\}\). For each \(p\in \Int \gamma\), a \emph{right branch} (resp. \emph{left branch}) at \(p\) relative to \(\gamma\) is an element \(b\in Z_p^\alpha\) such that
\[
c_p^\alpha(b_-,b,b_+)=1
\qquad
\text{(resp. }c_p^\alpha(b_-,b,b_+)=-1\text{)},
\]
where \(b_+\) and \(b_-\) are the two branches at \(p\) meeting \(\gamma\), ordered so that
\(
\gamma\cap b_-<\gamma\cap b_+
\)
with respect to the linear order on \(\gamma\). A right branch (resp. left branch) at a point \(p\in \Int \gamma\) is also called a \emph{right branch} (resp. \emph{left branch}) of \(\gamma\).

We say that \(\gamma\) is \emph{right simplicial} (resp. \emph{left simplicial}) if no point of \(\Int \gamma\) has a right branch (resp. left branch). We say that a segment is
\begin{itemize}
    \item \emph{simplicial} if it is both right simplicial and left simplicial under some orientation;
    \item \emph{one-sided simplicial} if, under some orientation, it is right simplicial but not left simplicial;
    \item \emph{at least one-sided simplicial} if, under some orientation, it is left simplicial;
    \item \emph{two-sided branched} if it is not at least one-sided simplicial.
\end{itemize}
Note that inaccessibility requires that neither \(Z^+\) nor \(Z^-\) contain a ray landing on a given side of \(\gamma\), whereas simpliciality only requires the absence of such a ray in the one of \(Z^\pm\) that contains \(\gamma\).

\begin{rmk}\label{Rmk:noSeg}
It is shown by Lemma~2.14 (Infinite ends) in \cite{CalegariLoukidou26} that no zipper is a segment. In particular, every zipper contains at least one branch point.
\end{rmk}

\subsection{Adapted segments and Fences}

A \emph{bridge} of a zipper \(Z^\pm\) is an embedding \(\gamma\colon [0,1]\to S_\infty^2\) such that
\begin{enumerate}
    \item either \(\gamma([0,1/2))\subset Z^-\), \(\gamma((1/2,1])\subset Z^+\), and \(\gamma(1/2)\in S_\infty^2\setminus (Z^-\cup Z^+)\); or
    \item either \(\gamma([0,1))\subset Z^-\) and \(\gamma(1)\in Z^+\), or \(\gamma([0,1))\subset Z^+\) and \(\gamma(1)\in Z^-\).
\end{enumerate}
A bridge of the first kind is said to be of \emph{type~1}, and a bridge of the second kind of \emph{type~2}. A point \(p\in S_\infty^2\) is called a \emph{synapse} of a bridge if \(p=\gamma(1/2)\) for some bridge of type~1, or \(p=\gamma(1)\) for some bridge of type~2.

\begin{prop}[Bridges exist, Proposition~2.22 in \cite{CalegariLoukidou26}]\label{Prop:bridgeExist}
A bridge exists. Moreover, the set of all synapses is dense in \(S_\infty^2\).
\end{prop}

\begin{proof}
The first assertion is Proposition~2.22 in \cite{CalegariLoukidou26}, while the second follows immediately from the minimality of the \(G\)-action on \(S_\infty^2\).
\end{proof}

Although the notion of a bridge was introduced in \cite{CalegariLoukidou26}, we use a slightly different notion in this paper. A \emph{connector} of a zipper \(Z^\pm\) is an embedding \(\gamma\colon [0,1]\to S_\infty^2\) such that
\(
\gamma([0,1/2))\subset Z^-
\text{ and }
\gamma((1/2,1])\subset Z^+.
\)
Note that \(\gamma(1/2)\) may be a synapse of either type. The \emph{type} of a connector is defined to be the type of its associated synapse \(\gamma(1/2)\).

A Jordan arc \(\cA\subset S_\infty^2\) is called a \emph{connector arc} of \(Z^\pm\) if it is the image of a connector. The type of a connector arc, and its associated synapse, are defined to be those of the corresponding connector. By an orientation of \(\cA\) we mean the corresponding connector \(\delta\). The map \(\delta\) also induces a canonical linear order on \(\cA\).

A Jordan segment \(\cS\subset S_\infty^2\) is said to be \emph{adapted} to \(Z^\pm\) if either \(\cS\) is a segment contained in one of the sets \(Z^\pm\), or \(\cS\) contains a synapse \(s\) such that \(s\) is a cut point of \(\cS\) and the two components of \(\cS\setminus\{s\}\) are segments contained in different sets \(Z^\pm\). In the latter case, \(s\) is called the \emph{synapse} of \(\cS\), and the \emph{type} of \(\cS\) is defined to be the type of \(s\).

For \(\alpha\in\{+,-\}\), the intersection \(\cS\cap Z^\alpha\) is called the \((\alpha)\)-\emph{segment} of \(\cS\) and is denoted by \(\cS^\alpha\). An adapted segment \(\cS\) is called \emph{mixed} if \(\cS^\alpha\neq\varnothing\) for both \(\alpha\in\{+,-\}\), and \emph{pure} if \(\cS=\cS^\alpha\) for some \(\alpha\in\{+,-\}\). If \(\cS\) is mixed, then the canonical linear order on \(\cS\) is the linear order \(<\) such that \(x<y\) for every \((x,y)\in \cS^-\times \cS^+\).

A \emph{fence} of \(Z^\pm\) is a Jordan curve \(\cJ\subset S_\infty^2\) for which there exist two distinct points \(u,v\in \cJ\), not necessarily contained in \(Z^+\cup Z^-\), such that the two components of \(\cJ\setminus\{u,v\}\) are segments contained in \(Z^+\) and \(Z^-\), respectively. The points \(u\) and \(v\) are called the \emph{nodes} of \(\cJ\). We write \(\cJ^+\) and \(\cJ^-\) for the components of \(\cJ\cap Z^+\) and \(\cJ\cap Z^-\), respectively, and call them the \((+)\)-\emph{side} and \((-)\)-\emph{side} of the fence. Note that each end ray of \(\Int \cJ^\alpha\) lands at a node.

\subsection{Circle laminations}\label{Sec:circleLami}

Let \(\cM\) denote the configuration space of unordered pairs of distinct points of \(S^1\); when a  curve is understood from the context, we use the same notation for the corresponding configuration space on that curve. This space is homeomorphic to the open M\"obius band.

We say that an element \(\ell\in \cM\) \emph{lies on} an open interval \(I\subset S^1\) if \(\ell\subset \closure{I}\).
Given \(\ell,m\in \cM\), we say that \(\ell\) and \(m\) are \emph{unlinked} if \(m\) lies on a component of \(S^1\setminus \ell\), and \emph{linked} otherwise.
A \emph{circle lamination} is a closed subset \(\cL\subset \cM\) such that every pair of elements of \(\cL\) is unlinked. The elements of \(\cL\) are called its \emph{leaves}, and the points of a leaf are called its \emph{endpoints}.

Let \(I\subset S^1\) be an open interval with \(\closure{I}\neq S^1\). A sequence \(\{\ell_k\}_{k\in \NN}\subset \cM\) is called \emph{\(I\)-side} if \(\ell_k\neq \partial I\) for all \(k\), each \(\ell_k\) lies on \(I\), and \(\ell_k\to \partial I\) in \(\cM\).
Given an open interval \(I\) with \(\partial I\in \cL\), we say that \(I\) is \emph{isolated} in \(\cL\) if there is no \(I\)-side sequence of leaves in \(\cL\).
A leaf \(\{x_0,x_1\}\) of \(\cL\) is called \emph{isolated} if both intervals \(\opi{x_i}{x_{i+1}}\), for \(i\in \ZZ/2\ZZ\), are isolated in \(\cL\). It is called a \emph{boundary leaf} if at least one of the intervals \(\opi{x_i}{x_{i+1}}\), for \(i\in \ZZ/2\ZZ\), is isolated in \(\cL\).

Let \(\cL\) be a circle lamination on the Gromov boundary \(\partial \HH^2\) of the hyperbolic plane \(\HH^2\). For each \(\{x,y\}\in \cL\), there is a unique bi-infinite geodesic \(\sfg(x,y)\) in \(\HH^2\) with endpoints \(x\) and \(y\). The \emph{geometric realization} of \(\cL\) is the union of these geodesics,
\[
\bigcup_{\{x,y\}\in \cL}\sfg(x,y),
\]
which is a geodesic lamination of \(\HH^2\).
In particular, if \(\cL\) is preserved by a closed surface group \(\Gamma<\PSL(\RR)\), then the geometric realization of \(\cL\) descends to a geodesic lamination on the surface \(\HH^2/\Gamma\).

We use the terminology and basic facts concerning geodesic laminations from Casson's book~\cite{Casson88}; for the reader's convenience, most of the material needed here is collected in \refapp{geodLami}. See also Calegari's book~\cite{Calegari07}.

\subsection{Unlinked adapted segments}
Later we will consider circle laminations on the limit set \(\Lambda(\Gamma)\) of a quasi-Fuchsian surface subgroup \(\Gamma<G\). Recall the strategy and notation from \refsec{proofOutline}. In that setting, \(\{x^+,x^-\}\) is the pair of endpoints of a connector arc \(\gamma\) crossing the Jordan domain of \(\Lambda(\Gamma)\) that contains \(p\). The following notions are introduced for later use in describing such arcs \(\gamma\) and their translates \(g(\gamma)\), where \(g\in \Gamma\).

Two segments \(s_1,s_2\) in \(Z^\alpha\), where \(\alpha\in\{+,-\}\), are said to be \emph{linked} if \(\Int s_1\) meets both a left branch and a right branch of \(s_2\). Otherwise they are said to be \emph{unlinked}. In particular, if \(s_1\) and \(s_2\) are linked, then \(s_1\cap s_2=[p,q]\) for some \(p,q\in Z^\alpha\), possibly with \(p=q\).

\begin{prop}[Enclosing a landing ray]\label{Prop:enclosing}
Let \(Z^\pm\) be a zipper for \(M\), and let \(\cD\) be a Jordan domain in \(S_\infty^2\). Suppose that an interval \(\cI\subset \partial \cD\) is a segment in \(Z^+\) (resp. \(Z^-\)). Let \(\sfr\) be a ray in \(Z^+\) (resp. \(Z^-\)) such that
\begin{itemize}
    \item \(\sfr\) starts at a point \(s\in \cI\);
    \item \(\sfr\cap (\partial \cD\setminus \cI)=\varnothing\);
    \item \(\sfr\) lands either at a point of \(\partial \cD\setminus \cI\) on the left side, or at a point of \(\cD\).
\end{itemize}
Then there exists \(v\in \cI\cap \sfr\) such that \(\cI\cap \sfr=[v,s]\) and the interior of the truncated ray of \(\sfr\) at \(v\) is contained in \(\cD\). In particular,
\(\sfr\subset \closure{\cD}\) and \(\sfr\cap \partial \cD=[v,s],\)
so \(\sfr\) is unlinked with \(\cI\).
\end{prop}

\begin{proof}
Write \(\Int \cI=\opi[\partial \cD]{x_-}{x_+}\).
If \(s\neq x_-\), let \(\sfs_-=\lopi[\partial\cD]{x_-}{s}\), and if \(s\neq x_+\), let \(\sfs_+=\ropi[\partial\cD]{s}{x_+}\).
The ray \(\sfr\) cannot meet both \(\Int \sfs_-\) and \(\Int \sfs_+\), since otherwise unique path-connectivity in \(Z^\alpha\) would produce a subarc of \(\Int \sfr\) containing \(s\), contradicting that \(\sfr\) is a ray. By symmetry, we may assume that \(\sfr\cap \Int \sfs_-=\varnothing\).

If also \(\sfr\cap \Int \sfs_+=\varnothing\), then \(\sfr\cap \partial\cD=\{s\}\) and \(\Int \sfr\subset \cD\), so we may take \(v=s\).
Otherwise \(\sfr\) meets \(\Int \sfs_+\). Then \(\sfr\) and \(\sfs_+\) meet the same branch at \(s\), and hence
\(
\sfr\cap \sfs_+=[s,v]
\)
for some \(v\in \cI\cap \sfr\). The interior of the truncated ray of \(\sfr\) at \(v\) is disjoint from \(\partial\cD\), and therefore lies in \(\cD\). This gives the result.
\end{proof}

\begin{figure}[ht]
    \centering
\begin{subfigure}[H]{0.3\textwidth}
\centering
\tikzset{every picture/.style={line width=0.75pt}} %set default line width to 0.75pt        

\begin{tikzpicture}[x=0.75pt,y=0.75pt,yscale=-1,xscale=1]

%Shape: Polygon Curved [id:ds9289019841212706] 
\draw   (224.86,92.79) .. controls (248.86,83.79) and (251.86,109.79) .. (259.86,131.79) .. controls (267.86,153.79) and (249.86,162.79) .. (251.86,181.79) .. controls (253.86,200.79) and (195.86,234.79) .. (181.86,211.79) .. controls (167.86,188.79) and (150.86,209.79) .. (135.86,176.79) .. controls (120.86,143.79) and (152.86,123.79) .. (160.86,104.79) .. controls (168.86,85.79) and (200.86,101.79) .. (224.86,92.79) -- cycle ;
%Shape: Circle [id:dp8997997952123005] 
\draw  [color={rgb, 255:red, 74; green, 144; blue, 226 }  ,draw opacity=1 ][fill={rgb, 255:red, 208; green, 2; blue, 27 }  ,fill opacity=1 ][line width=0.75]  (198,154) .. controls (198,152.5) and (199.21,151.29) .. (200.71,151.29) .. controls (202.2,151.29) and (203.42,152.5) .. (203.42,154) .. controls (203.42,155.5) and (202.2,156.71) .. (200.71,156.71) .. controls (199.21,156.71) and (198,155.5) .. (198,154) -- cycle ;
%Curve Lines [id:da04248200682567427] 
\draw [color={rgb, 255:red, 74; green, 144; blue, 226 }  ,draw opacity=1 ]   (135.86,176.79) .. controls (175.86,146.79) and (164,170) .. (198,156) ;
%Shape: Circle [id:dp14661380621089126] 
\draw  [color={rgb, 255:red, 74; green, 144; blue, 226 }  ,draw opacity=1 ][fill={rgb, 255:red, 255; green, 255; blue, 255 }  ,fill opacity=1 ][line width=0.75]  (133,177) .. controls (133,175.5) and (134.21,174.29) .. (135.71,174.29) .. controls (137.2,174.29) and (138.42,175.5) .. (138.42,177) .. controls (138.42,178.5) and (137.2,179.71) .. (135.71,179.71) .. controls (134.21,179.71) and (133,178.5) .. (133,177) -- cycle ;
%Curve Lines [id:da9822128484848912] 
\draw [color={rgb, 255:red, 208; green, 2; blue, 27 }  ,draw opacity=1 ]   (203,153) .. controls (226,133) and (240,141) .. (261,137) ;
%Shape: Circle [id:dp050191056496360686] 
\draw  [color={rgb, 255:red, 208; green, 2; blue, 27 }  ,draw opacity=1 ][fill={rgb, 255:red, 255; green, 255; blue, 255 }  ,fill opacity=1 ][line width=0.75]  (258,137) .. controls (258,135.5) and (259.21,134.29) .. (260.71,134.29) .. controls (262.2,134.29) and (263.42,135.5) .. (263.42,137) .. controls (263.42,138.5) and (262.2,139.71) .. (260.71,139.71) .. controls (259.21,139.71) and (258,138.5) .. (258,137) -- cycle ;
%Curve Lines [id:da8107158506148203] 
\draw [color={rgb, 255:red, 208; green, 2; blue, 27 }  ,draw opacity=1 ]   (222,213) .. controls (213.33,189.67) and (241.33,137) .. (224,141) ;
%Curve Lines [id:da9261983428085323] 
\draw [color={rgb, 255:red, 74; green, 144; blue, 226 }  ,draw opacity=1 ]   (177,95) .. controls (178.33,114) and (201.67,117) .. (200.33,152) ;
%Shape: Circle [id:dp28423238078235813] 
\draw  [color={rgb, 255:red, 208; green, 2; blue, 27 }  ,draw opacity=1 ][fill={rgb, 255:red, 255; green, 255; blue, 255 }  ,fill opacity=1 ][line width=0.75]  (219,212) .. controls (219,210.5) and (220.21,209.29) .. (221.71,209.29) .. controls (223.2,209.29) and (224.42,210.5) .. (224.42,212) .. controls (224.42,213.5) and (223.2,214.71) .. (221.71,214.71) .. controls (220.21,214.71) and (219,213.5) .. (219,212) -- cycle ;
%Shape: Circle [id:dp8294971997272432] 
\draw  [color={rgb, 255:red, 74; green, 144; blue, 226 }  ,draw opacity=1 ][fill={rgb, 255:red, 255; green, 255; blue, 255 }  ,fill opacity=1 ][line width=0.75]  (174,95) .. controls (174,93.5) and (175.21,92.29) .. (176.71,92.29) .. controls (178.2,92.29) and (179.42,93.5) .. (179.42,95) .. controls (179.42,96.5) and (178.2,97.71) .. (176.71,97.71) .. controls (175.21,97.71) and (174,96.5) .. (174,95) -- cycle ;

% Text Node
\draw (267.86,124.67) node [anchor=north west][inner sep=0.75pt]  [font=\normalsize,color={rgb, 255:red, 208; green, 2; blue, 27 }  ,opacity=1 ]  {$e_{1}^{+}$};
% Text Node
\draw (111.86,170.67) node [anchor=north west][inner sep=0.75pt]  [font=\normalsize,color={rgb, 255:red, 74; green, 144; blue, 226 }  ,opacity=1 ]  {$e_{1}^{-}$};
% Text Node
\draw (178.86,165.67) node [anchor=north west][inner sep=0.75pt] [color={rgb, 255:red, 0; green, 0; blue, 0 }  ,opacity=1 ]  {$s_{1}=s_{2}$};
% Text Node
\draw (123.86,111.67) node [anchor=north west][inner sep=0.75pt]  [font=\normalsize,color={rgb, 255:red, 0; green, 0; blue, 0 }  ,opacity=1 ]  {$\mathcal{D}$};
% Text Node
\draw (167.86,69) node [anchor=north west][inner sep=0.75pt]  [font=\normalsize,color={rgb, 255:red, 74; green, 144; blue, 226 }  ,opacity=1 ]  {$e_{2}^{-}$};
% Text Node
\draw (215.86,213) node [anchor=north west][inner sep=0.75pt]  [font=\normalsize,color={rgb, 255:red, 208; green, 2; blue, 27 }  ,opacity=1 ]  {$e_{2}^{+}$};
\end{tikzpicture}
\subcaption{Case (1)}
\end{subfigure}
\hspace{0.01\textwidth}
\begin{subfigure}[H]{0.3\textwidth}
\centering

\tikzset{every picture/.style={line width=0.75pt}} %set default line width to 0.75pt        

\begin{tikzpicture}[x=0.75pt,y=0.75pt,yscale=-1,xscale=1]

%Shape: Polygon Curved [id:ds5909257920987011] 
\draw   (244.86,112.79) .. controls (268.86,103.79) and (271.86,129.79) .. (279.86,151.79) .. controls (287.86,173.79) and (269.86,182.79) .. (271.86,201.79) .. controls (273.86,220.79) and (215.86,254.79) .. (201.86,231.79) .. controls (187.86,208.79) and (170.86,229.79) .. (155.86,196.79) .. controls (140.86,163.79) and (172.86,143.79) .. (180.86,124.79) .. controls (188.86,105.79) and (220.86,121.79) .. (244.86,112.79) -- cycle ;
%Shape: Circle [id:dp2810304472021107] 
\draw  [color={rgb, 255:red, 74; green, 144; blue, 226 }  ,draw opacity=1 ][fill={rgb, 255:red, 208; green, 2; blue, 27 }  ,fill opacity=1 ][line width=0.75]  (218,174) .. controls (218,172.5) and (219.21,171.29) .. (220.71,171.29) .. controls (222.2,171.29) and (223.42,172.5) .. (223.42,174) .. controls (223.42,175.5) and (222.2,176.71) .. (220.71,176.71) .. controls (219.21,176.71) and (218,175.5) .. (218,174) -- cycle ;
%Curve Lines [id:da8201695525420443] 
\draw [color={rgb, 255:red, 74; green, 144; blue, 226 }  ,draw opacity=1 ]   (155.86,196.79) .. controls (195.86,166.79) and (184,190) .. (218,176) ;
%Shape: Circle [id:dp07870131069015185] 
\draw  [color={rgb, 255:red, 74; green, 144; blue, 226 }  ,draw opacity=1 ][fill={rgb, 255:red, 255; green, 255; blue, 255 }  ,fill opacity=1 ][line width=0.75]  (153,197) .. controls (153,195.5) and (154.21,194.29) .. (155.71,194.29) .. controls (157.2,194.29) and (158.42,195.5) .. (158.42,197) .. controls (158.42,198.5) and (157.2,199.71) .. (155.71,199.71) .. controls (154.21,199.71) and (153,198.5) .. (153,197) -- cycle ;
%Curve Lines [id:da9782969240488663] 
\draw [color={rgb, 255:red, 208; green, 2; blue, 27 }  ,draw opacity=1 ]   (223,173) .. controls (246,153) and (260,161) .. (281,157) ;
%Shape: Circle [id:dp03718961715447655] 
\draw  [color={rgb, 255:red, 208; green, 2; blue, 27 }  ,draw opacity=1 ][fill={rgb, 255:red, 255; green, 255; blue, 255 }  ,fill opacity=1 ][line width=0.75]  (278,157) .. controls (278,155.5) and (279.21,154.29) .. (280.71,154.29) .. controls (282.2,154.29) and (283.42,155.5) .. (283.42,157) .. controls (283.42,158.5) and (282.2,159.71) .. (280.71,159.71) .. controls (279.21,159.71) and (278,158.5) .. (278,157) -- cycle ;
%Curve Lines [id:da758662140528384] 
\draw [color={rgb, 255:red, 208; green, 2; blue, 27 }  ,draw opacity=1 ]   (242,233) .. controls (233.33,209.67) and (215,199.49) .. (220.71,176.71) ;
%Curve Lines [id:da43272075724366443] 
\draw [color={rgb, 255:red, 74; green, 144; blue, 226 }  ,draw opacity=1 ]   (197,115) .. controls (198.33,134) and (245,139.49) .. (246.71,156.29) ;
%Shape: Circle [id:dp9753771775340537] 
\draw  [color={rgb, 255:red, 208; green, 2; blue, 27 }  ,draw opacity=1 ][fill={rgb, 255:red, 255; green, 255; blue, 255 }  ,fill opacity=1 ][line width=0.75]  (239,232) .. controls (239,230.5) and (240.21,229.29) .. (241.71,229.29) .. controls (243.2,229.29) and (244.42,230.5) .. (244.42,232) .. controls (244.42,233.5) and (243.2,234.71) .. (241.71,234.71) .. controls (240.21,234.71) and (239,233.5) .. (239,232) -- cycle ;
%Shape: Circle [id:dp8768508031900737] 
\draw  [color={rgb, 255:red, 74; green, 144; blue, 226 }  ,draw opacity=1 ][fill={rgb, 255:red, 255; green, 255; blue, 255 }  ,fill opacity=1 ][line width=0.75]  (194,115) .. controls (194,113.5) and (195.21,112.29) .. (196.71,112.29) .. controls (198.2,112.29) and (199.42,113.5) .. (199.42,115) .. controls (199.42,116.5) and (198.2,117.71) .. (196.71,117.71) .. controls (195.21,117.71) and (194,116.5) .. (194,115) -- cycle ;
%Shape: Circle [id:dp8729893288275562] 
\draw  [color={rgb, 255:red, 74; green, 144; blue, 226 }  ,draw opacity=1 ][fill={rgb, 255:red, 208; green, 2; blue, 27 }  ,fill opacity=1 ][line width=0.75]  (244,159) .. controls (244,157.5) and (245.21,156.29) .. (246.71,156.29) .. controls (248.2,156.29) and (249.42,157.5) .. (249.42,159) .. controls (249.42,160.5) and (248.2,161.71) .. (246.71,161.71) .. controls (245.21,161.71) and (244,160.5) .. (244,159) -- cycle ;

% Text Node
\draw (287.86,144.67) node [anchor=north west][inner sep=0.75pt]  [font=\normalsize,color={rgb, 255:red, 208; green, 2; blue, 27 }  ,opacity=1 ]  {$e_{1}^{+}$};
% Text Node
\draw (131.86,190.67) node [anchor=north west][inner sep=0.75pt]  [font=\normalsize,color={rgb, 255:red, 74; green, 144; blue, 226 }  ,opacity=1 ]  {$e_{1}^{-}$};
% Text Node
\draw (143.86,131.67) node [anchor=north west][inner sep=0.75pt]  [font=\normalsize,color={rgb, 255:red, 0; green, 0; blue, 0}  ,opacity=1 ]  {$\cD$};
% Text Node
\draw (187.86,87.67) node [anchor=north west][inner sep=0.75pt]  [font=\normalsize,color={rgb, 255:red, 74; green, 144; blue, 226 }  ,opacity=1 ]  {$e_{2}^{-}$};
% Text Node
\draw (235.86,231.67) node [anchor=north west][inner sep=0.75pt]  [font=\normalsize,color={rgb, 255:red, 208; green, 2; blue, 27 }  ,opacity=1 ]  {$e_{2}^{+}$};
% Text Node
\draw (209.86,154.67) node [anchor=north west][inner sep=0.75pt]  [color={rgb, 255:red, 0; green, 0; blue, 0 }  ,opacity=1 ]  {$s_{1}$};
% Text Node
\draw (243,163) node [anchor=north west][inner sep=0.75pt] [color={rgb, 255:red, 0; green, 0; blue, 0 }  ,opacity=1 ]  {$s_{2}$};
\end{tikzpicture}
\subcaption{Case (2)}
\end{subfigure}
\begin{subfigure}[H]{0.3\textwidth}
\centering

\tikzset{every picture/.style={line width=0.75pt}} %set default line width to 0.75pt        

\begin{tikzpicture}[x=0.75pt,y=0.75pt,yscale=-1,xscale=1]

%Shape: Polygon Curved [id:ds3666341492368743] 
\draw   (212.86,106.79) .. controls (236.86,97.79) and (239.86,123.79) .. (247.86,145.79) .. controls (255.86,167.79) and (237.86,176.79) .. (239.86,195.79) .. controls (241.86,214.79) and (183.86,248.79) .. (169.86,225.79) .. controls (155.86,202.79) and (138.86,223.79) .. (123.86,190.79) .. controls (108.86,157.79) and (140.86,137.79) .. (148.86,118.79) .. controls (156.86,99.79) and (188.86,115.79) .. (212.86,106.79) -- cycle ;
%Shape: Circle [id:dp22821117150652803] 
\draw  [color={rgb, 255:red, 74; green, 144; blue, 226 }  ,draw opacity=1 ][fill={rgb, 255:red, 208; green, 2; blue, 27 }  ,fill opacity=1 ][line width=0.75]  (186,168) .. controls (186,166.5) and (187.21,165.29) .. (188.71,165.29) .. controls (190.2,165.29) and (191.42,166.5) .. (191.42,168) .. controls (191.42,169.5) and (190.2,170.71) .. (188.71,170.71) .. controls (187.21,170.71) and (186,169.5) .. (186,168) -- cycle ;
%Curve Lines [id:da7271795744163541] 
\draw [color={rgb, 255:red, 74; green, 144; blue, 226 }  ,draw opacity=1 ]   (123.86,190.79) .. controls (163.86,160.79) and (152,184) .. (186,170) ;
%Shape: Circle [id:dp12959051209302586] 
\draw  [color={rgb, 255:red, 74; green, 144; blue, 226 }  ,draw opacity=1 ][fill={rgb, 255:red, 255; green, 255; blue, 255 }  ,fill opacity=1 ][line width=0.75]  (121,191) .. controls (121,189.5) and (122.21,188.29) .. (123.71,188.29) .. controls (125.2,188.29) and (126.42,189.5) .. (126.42,191) .. controls (126.42,192.5) and (125.2,193.71) .. (123.71,193.71) .. controls (122.21,193.71) and (121,192.5) .. (121,191) -- cycle ;
%Curve Lines [id:da6008693565282317] 
\draw [color={rgb, 255:red, 208; green, 2; blue, 27 }  ,draw opacity=1 ]   (191,167) .. controls (214,147) and (228,155) .. (249,151) ;
%Shape: Circle [id:dp9286128855526717] 
\draw  [color={rgb, 255:red, 208; green, 2; blue, 27 }  ,draw opacity=1 ][fill={rgb, 255:red, 255; green, 255; blue, 255 }  ,fill opacity=1 ][line width=0.75]  (246,151) .. controls (246,149.5) and (247.21,148.29) .. (248.71,148.29) .. controls (250.2,148.29) and (251.42,149.5) .. (251.42,151) .. controls (251.42,152.5) and (250.2,153.71) .. (248.71,153.71) .. controls (247.21,153.71) and (246,152.5) .. (246,151) -- cycle ;
%Curve Lines [id:da8750619273939797] 
\draw [color={rgb, 255:red, 208; green, 2; blue, 27 }  ,draw opacity=1 ]   (210,227) .. controls (201.33,203.67) and (188,192.13) .. (188.71,170.71) ;
%Curve Lines [id:da6647562720404535] 
\draw [color={rgb, 255:red, 74; green, 144; blue, 226 }  ,draw opacity=1 ]   (165,109) .. controls (165,132.13) and (156,157.13) .. (158,173.13) .. controls (164,152.13) and (176,126.13) .. (199,139.13) ;
%Shape: Circle [id:dp9586304185165072] 
\draw  [color={rgb, 255:red, 208; green, 2; blue, 27 }  ,draw opacity=1 ][fill={rgb, 255:red, 255; green, 255; blue, 255 }  ,fill opacity=1 ][line width=0.75]  (207,226) .. controls (207,224.5) and (208.21,223.29) .. (209.71,223.29) .. controls (211.2,223.29) and (212.42,224.5) .. (212.42,226) .. controls (212.42,227.5) and (211.2,228.71) .. (209.71,228.71) .. controls (208.21,228.71) and (207,227.5) .. (207,226) -- cycle ;
%Shape: Circle [id:dp14880851752445368] 
\draw  [color={rgb, 255:red, 74; green, 144; blue, 226 }  ,draw opacity=1 ][fill={rgb, 255:red, 255; green, 255; blue, 255 }  ,fill opacity=1 ][line width=0.75]  (162,109) .. controls (162,107.5) and (163.21,106.29) .. (164.71,106.29) .. controls (166.2,106.29) and (167.42,107.5) .. (167.42,109) .. controls (167.42,110.5) and (166.2,111.71) .. (164.71,111.71) .. controls (163.21,111.71) and (162,110.5) .. (162,109) -- cycle ;
%Shape: Circle [id:dp010968078683743165] 
\draw  [color={rgb, 255:red, 74; green, 144; blue, 226 }  ,draw opacity=1 ][fill={rgb, 255:red, 208; green, 2; blue, 27 }  ,fill opacity=1 ][line width=0.75]  (198,141) .. controls (198,139.5) and (199.21,138.29) .. (200.71,138.29) .. controls (202.2,138.29) and (203.42,139.5) .. (203.42,141) .. controls (203.42,142.5) and (202.2,143.71) .. (200.71,143.71) .. controls (199.21,143.71) and (198,142.5) .. (198,141) -- cycle ;
%Curve Lines [id:da8797845860137088] 
\draw [color={rgb, 255:red, 208; green, 2; blue, 27 }  ,draw opacity=1 ]   (214,154.13) .. controls (214,144.13) and (205,151.13) .. (201.71,143.71) ;
%Curve Lines [id:da028270064310275678] 
\draw    (168,111.13) .. controls (167,129.13) and (166,136.13) .. (162,157.13) .. controls (167.79,135.89) and (182.89,126.77) .. (198.32,135.14) ;
\draw [shift={(200,136.13)}, rotate = 212.01] [fill={rgb, 255:red, 0; green, 0; blue, 0 }  ][line width=0.08]  [draw opacity=0] (7.2,-1.8) -- (0,0) -- (7.2,1.8) -- cycle    ;
%Curve Lines [id:da6326674983174845] 
\draw    (127,192.33) .. controls (162.46,164.75) and (154.26,188.6) .. (186.5,172.75) ;
\draw [shift={(188,172)}, rotate = 152.99] [fill={rgb, 255:red, 0; green, 0; blue, 0 }  ][line width=0.08]  [draw opacity=0] (7.2,-1.8) -- (0,0) -- (7.2,1.8) -- cycle    ;

% Text Node
\draw (255.86,138.67) node [anchor=north west][inner sep=0.75pt]  [font=\normalsize,color={rgb, 255:red, 208; green, 2; blue, 27 }  ,opacity=1 ]  {$e_{1}^{+}$};
% Text Node
\draw (99.86,184.67) node [anchor=north west][inner sep=0.75pt]  [font=\normalsize,color={rgb, 255:red, 74; green, 144; blue, 226 }  ,opacity=1 ]  {$e_{1}^{-}$};
% Text Node
\draw (111.86,125.67) node [anchor=north west][inner sep=0.75pt]  [font=\normalsize,color={rgb, 255:red, 0; green, 0; blue, 0 }  ,opacity=1 ]  {$\mathcal{D}$};
% Text Node
\draw (155.86,81.67) node [anchor=north west][inner sep=0.75pt]  [font=\normalsize,color={rgb, 255:red, 74; green, 144; blue, 226 }  ,opacity=1 ]  {$e_{2}^{-}$};
% Text Node
\draw (203.86,225.67) node [anchor=north west][inner sep=0.75pt]  [font=\normalsize,color={rgb, 255:red, 208; green, 2; blue, 27 }  ,opacity=1 ]  {$e_{2}^{+}$};
% Text Node
\draw (192.86,163.67) node [anchor=north west][inner sep=0.75pt]  [font=\normalsize,color={rgb, 255:red, 0; green, 0; blue, 0 }  ,opacity=1 ]  {$s_{1}$};
% Text Node
\draw (204,125.4) node [anchor=north west][inner sep=0.75pt]  [font=\normalsize,color={rgb, 255:red, 0; green, 0; blue, 0 }  ,opacity=1 ]  {$s_{2}$};

\end{tikzpicture}

\subcaption{Case (3)}
\end{subfigure}
    \caption{Sample figures for each case with $\sfd \cS_i=\{e_i^{-},e_i^{+}\}$ and $s_i\in Z^+$.}
    \label{Fig:unlinkedCrossingLine}
\end{figure}
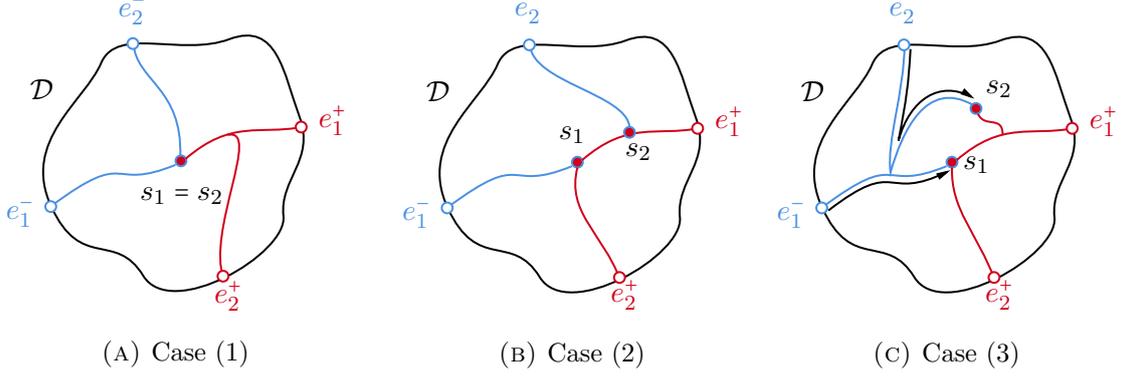

\begin{prop}[Unlinked adapted crossing lines]\label{Prop:unlinkedCrossingLine}
Let \(Z^\pm\) be a zipper for \(M\), and let \(\cD\) be a Jordan domain in \(S_\infty^2\). Suppose that \(\cS_1\) and \(\cS_2\) are adapted Jordan lines crossing \(\cD\), that \(\sfd \cS_1\) and \(\sfd \cS_2\) are linked in \(\partial \cD\), and that \(\cS_1^\alpha\) and \(\cS_2^\alpha\) are unlinked for each \(\alpha\in\{+,-\}\). Then one of the following holds:
\begin{enumerate}
    \item\label{Itm:sameSynapse} both \(\cS_1\) and \(\cS_2\) are mixed and have the same synapse;
    \item\label{Itm:oneSideLand} both \(\cS_1\) and \(\cS_2\) are mixed, and their synapses \(s_1,s_2\) are distinct points of \(Z^\alpha\) for some \(\alpha\in\{+,-\}\). Moreover, \(\cS_1^\alpha\cap \cS_2^\alpha=[s_1,s_2]\), and, writing \(\beta\neq\alpha\), there are rays in \(\cS_1^\beta\) and \(\cS_2^\beta\) landing at \(s_2\) and \(s_1\), respectively, on the same side of the line \(\cS_1^\alpha\cup \cS_2^\alpha\);
    \item\label{Itm:bothLand} there exist \(j\in \ZZ/2\ZZ\) and \(\alpha\in\{+,-\}\) such that \(\cS_j\) is mixed with synapse in \(Z^\alpha\), and \(\cS_{j+1}^\alpha\) admits two rays contained in \(\cS_j^\pm\) landing on opposite sides of \(\cS_{j+1}^\alpha\). In particular, \(\cS_{j+1}^\alpha\) is two-sided accessible.
\end{enumerate}
\end{prop}

\begin{proof}
Since \(\sfd \cS_1\) and \(\sfd \cS_2\) are linked in \(\partial \cD\), the lines \(\cS_1\) and \(\cS_2\) meet in \(\cD\). Hence they cannot both be pure: otherwise both would lie in the same \(Z^\alpha\), and then \(\cS_1^\alpha\) and \(\cS_2^\alpha\) would be linked, contrary to hypothesis. Thus at least one of \(\cS_1,\cS_2\) is mixed.

Assume first that both \(\cS_1\) and \(\cS_2\) are mixed. Let \(s_i\) be the synapse of \(\cS_i\), and write \(\sfd \cS_i=\{e_i^+,e_i^-\}\), where \(e_i^\pm\) are the ends of some end rays of \(\Int \cS_i^\pm\). If \(s_1=s_2\), then we are in Case~\ref{Itm:sameSynapse}. Thus assume \(s_1\neq s_2\).

If neither \(s_1\) lies on \(\cS_2\) nor \(s_2\) lies on \(\cS_1\), then by the disjointness of \(Z^\pm\), the segments \(\cS_1^\alpha\) and \(\cS_2^\alpha\) are linked for some \(\alpha\in\{+,-\}\), contradicting the assumption. Hence there are two cases:
\begin{enumerate}[label=(\alph*)]
    \item\label{Case:bothSynapses} \(s_1\in \cS_2\) and \(s_2\in \cS_1\);
    \item\label{Case:oneSynapse} \(s_i\in \cS_{i+1}\) and \(s_{i+1}\notin \cS_i\) for some \(i\in \ZZ/2\ZZ\).
\end{enumerate}

In Case~\ref{Case:bothSynapses}, there are again two subcases:
\begin{enumerate}[label=(\roman*)]
    \item\label{Case:differentSigns} \(s_1\) and \(s_2\) lie in different sets \(Z^\pm\);
    \item\label{Case:sameSign} \(s_1\) and \(s_2\) lie in the same \(Z^\alpha\).
\end{enumerate}

In Case~\ref{Case:bothSynapses}--\ref{Case:differentSigns}, by symmetry we may assume \(s_1\in \cS_2^+\) and \(s_2\in \cS_1^-\). Since \(\sfd \cS_1\) and \(\sfd \cS_2\) are linked in \(\partial\cD\), after relabelling we may assume \(e_1^-<e_2^+<e_1^+<e_2^-\). Applying Proposition~\ref{Prop:enclosing} to \(\cS_2^-\) relative to \(\cD^L(\cS_1)\), we obtain a ray in \(\cS_2^-\) landing on the left side of \(\cS_1^-\). If the end ray of \(\cS_2^+\) from \(s_1\) to \(s_2\) lands on the right side of \(\cS_1^-\), then we are in Case~\ref{Itm:bothLand} with \(\alpha=-\) and \(j=2\).

Otherwise, that ray lands on the left side of \(\cS_1^-\). Let \(\sft_{\mathrm{in}}\) and \(\sft_{\mathrm{out}}\) be the end rays of \(\cS_2^+\) starting from \(s_1\) and landing at \(s_2\) and \(e_2^+\), respectively. Applying Proposition~\ref{Prop:enclosing} to \(\sft_{\mathrm{in}}\) relative to \(\cD^L(\cS_1)\), and to \(\sft_{\mathrm{out}}\) relative to \(\cD^R(\cS_1)\), we obtain points \(u_{\mathrm{in}},u_{\mathrm{out}}\in \cS_1^+\) such that \(\sft_{\mathrm{in}}\cap \cS_1^+=[u_{\mathrm{in}},s_1]\) and \(\sft_{\mathrm{out}}\cap \cS_1^+=[u_{\mathrm{out}},s_1]\), and the corresponding truncated rays crossing  \(\cD^L(\cS_1)\) and \(\cD^R(\cS_1)\), respectively.
Since \(s_1\in [u_{\mathrm{in}},u_{\mathrm{out}}]\subset \cS_2^+\) and \(\cS_1^+\) and \(\cS_2^+\) are unlinked, at least one of \(u_{\mathrm{in}},u_{\mathrm{out}}\) is equal to \(s_1\).
Write \([u_{\mathrm{in}},u_{\mathrm{out}}]=[s_1,u]\).
Hence \(\cS_1^-\) contains an end ray landing at \(s_2\) on one side of \(\cS_2^+\), while \(\cS_1^+\) contains an end ray landing at $u$ on the opposite side of \(\cS_2^+\). Thus Case~\ref{Itm:bothLand} holds.

In Case~\ref{Case:bothSynapses}--\ref{Case:sameSign}, by symmetry we may assume \(\{s_1,s_2\}\subset Z^+\). Then \(\cS_1^+\cap \cS_2^+=[s_1,s_2]\), and \(\fS^+=\cS_1^+\cup \cS_2^+\) is a line in \(Z^+\) crossing \(\cD\). Orient \(\fS^+\) compatibly with \(\cS_1^+\). If \(e_1^-\in \opi[\partial \cD]{e_1^+}{e_2^+}\), then \(\cS_1^-\) crosses \(\cD^L(\fS^+)\), so an end ray of \(\cS_1^-\) lands at \(s_1\) on the left side of \(\fS^+\). Since \(\sfd \cS_1\) and \(\sfd \cS_2\) are linked, we also have \(e_2^-\in \opi[\partial \cD]{e_1^+}{e_1^-}\), and hence  \(\cS_2^-\) lands at \(s_2\) on the same side of \(\fS^+\). The other cyclic order is symmetric. Thus Case~\ref{Itm:oneSideLand} holds. 
See \reffig{unlinkedCrossingLine} (middle).

In Case~\ref{Case:oneSynapse}, by symmetry we may assume  \(s_1\in \cS_2^+\) and \(s_2\in \cD^L(\cS_1)\).
See \reffig{unlinkedCrossingLine} (right).
Then \(e_2^-\in \opi[\partial \cD]{e_1^+}{e_1^-}\), for otherwise \(\cS_2^-\) and \(\cS_1^-\) would be linked. Hence \(e_2^+\in \opi[\partial \cD]{e_1^-}{e_1^+}\). Let \(\sfq_{\mathrm{in}}\) and \(\sfq_{\mathrm{out}}\) be the end rays of \(\Int \cS_2^+\) starting from \(s_1\) and landing at \(s_2\) and \(e_2^+\), respectively. Applying Proposition~\ref{Prop:enclosing} to these rays relative to \(\cD^L(\cS_1)\) and \(\cD^R(\cS_1)\), respectively, we obtain points \(w_{\mathrm{in}},w_{\mathrm{out}}\in \cS_1^+\) such that the corresponding truncated rays lie in the indicated domains. Since \(s_1\in [w_{\mathrm{in}},w_{\mathrm{out}}]\subset \cS_1^+\), at least one of \(w_{\mathrm{in}},w_{\mathrm{out}}\) is equal to \(s_1\).
Write $[w_{\mathrm{in}},w_{\mathrm{out}}]=[s_1,w]$.
It follows that \(\cS_1^-\) and \(\cS_1^+\) contain rays landing at $s_1$ and $w$, respectively, on opposite sides of \(\cS_2^+\). Thus Case~\ref{Itm:bothLand} holds.

Finally, suppose that exactly one of \(\cS_1,\cS_2\) is mixed. By symmetry, assume \(\cS_1\) is mixed with synapse \(s\), while \(\cS_2\) is pure. Then \(s\in \cS_2\), since otherwise \(\cS_2\) would be linked with one of \(\cS_1^+\) or \(\cS_1^-\). Again by symmetry, we may assume \(\cS_2=\cS_2^+\). Then \(\cS_1^+\cap \cS_2=[s,x]\) for some \(x\in \cS_2\), and a subray of \(\cS_1^+\) lands at \(x\) on one side of \(\cS_2\), while an end ray of \(\cS_1^-\) lands at \(s\) on the other side. Hence Case~\ref{Itm:bothLand} holds.
\end{proof}

\section{Eclipses}\label{Sec:eclipses}

Let \(Z^\pm\) be a minimal zipper for \(M\). To prove \refthm{tameFix}, it suffices, by symmetry, to show that there is no element \(g\in G\) with \(r(g)\in Z^+\) and \(a(g)\notin Z^-\). As indicated in \refsec{proofOutline}, such an element would determine a certain configuration in \(Z^\pm\), which we call an \emph{eclipse}. In this section we explain how an eclipse arises and establish its basic properties. In the following sections we show that the existence of an eclipse would force an absurd configuration of points; in this sense, the eclipse is only an illusion. This completes the proof of \refthm{tameFix}.

\subsection{Prongs and Fences}
For an element $g$ in $G$,  a $g$-invariant ray in a zipper is called a \emph{$g$-prong}.
When exactly one of $Z^\pm$ contains a fixed point $p$ of an element $g$, there exists at most one $g$-prong that starts at $p$:

\begin{prop}[Unique invariant prong]\label{Prop:oneProng}
Let \(Z^\pm\) be a zipper for \(M\). Suppose that \(g\in G\) fixes a point \(p\in Z^+\) and that \(\Fix(g)\cap Z^-=\varnothing\). Then at most one \(g\)-prong in \(Z^+\) starts at \(p\) and lands at the other fixed point \(q\) of \(g\).
\end{prop}

\begin{proof}
Suppose that there are two distinct \(g\)-prongs \(r_1\) and \(r_2\) in \(Z^+\) starting at \(p\). Since both land at \(q\), the set \(r_1\cup r_2\cup\{q\}\) is a Jordan curve \(\cJ\); see \reflem{invRayLand}. Since the \(G\)-action on \(S_\infty^2\) is minimal, \(Z^-\) is dense in \(S_\infty^2\). Hence we may choose points \(x,y\in Z^-\) lying in different Jordan domains of \(\cJ\). The arc \([x,y]\subset Z^-\) must then meet \(\cJ\). By the disjointness of \(Z^+\) and \(Z^-\), the intersection point can only be \(q\). Thus \(q\in Z^-\), and since \(q\in \Fix(g)\), this gives \(q\in \Fix(g)\cap Z^-\), contrary to the assumption.
\end{proof}

A $g$-invariant prong gives rise to an $g$-invariant fence  of $Z^\pm$:

\begin{prop}[Invariant fence]\label{Prop:invFence}
Let \(Z^\pm\) be a zipper for \(M\). Assume that \(g\in G\) acts freely on \(Z^-\). If \(g\) admits a \(g\)-prong \(r\) in \(Z^+\), then the union of \(r\), the axis \(\ell\) of \(g\) in \(Z^-\), and \(\Fix(g)\) is a \(g\)-invariant fence \(\cJ\). Its nodes are precisely the points of \(\Fix(g)\); moreover, \(\ell\) is the \((-)\)-side of \(\cJ\), and the closure of \(r\) in \(Z^+\) is the \((+)\)-side of \(\cJ\). In particular, if \(B\) is a branch at a fixed point of \(g\) that is disjoint from \(r\), then \(B\) is not preserved by \(g\), and \(g(B)\cap B=\varnothing\).
\end{prop}
\begin{proof}
The description of \(\cJ\) follows immediately from \refthm{invHypAxis} and \reflem{invRayLand}.

For the final statement, let \(B\) be a branch at a fixed point \(p\in Z^+\) such that \(B\cap r=\varnothing\). If \(g(B)=B\), then \refprop{invBranch} gives a \(g\)-prong \(\gamma\subset B\cup \{p\}\) starting at \(p\). Since \(r\) is another \(g\)-prong starting at \(p\), this contradicts \refprop{oneProng}. Hence \(B\) is not \(g\)-invariant.

If \(g(B)\cap B\neq\varnothing\), then \(g(B)=B\), since distinct branches at a point are disjoint. This contradiction shows that \(g(B)\cap B=\varnothing\).
\end{proof}

\begin{rmk}\label{Rmk:denseType2}
    We denote the fence $\cJ$ in \refprop{invFence} by $\cJ(g)$ and call it the \emph{$g$-invariant fence}.
    Note that at least one node of $\cJ(g)$ is a synapse of type 2 in $Z^+$.
    Hence, the set of all synapses of type 2 in $Z^+$ is dense in $S_\infty^2$ by the minimality of the $G$-action.
\end{rmk}

\subsection{Pruning Fences}
We then show that the minimality of the zipper forces the \((+)\)-side of an invariant fence to be either inaccessible or at least one-sided inaccessible. This, in turn, implies that the invariant fence is unlinked from its translates under the \(G\)-action; see \refsec{proofOutline}.

The following lemma is the main tool for detecting this inaccessibility:
\begin{lem}[Pruning lemma]\label{Lem:pruning}
   Let $Z^\pm$ be a zipper, and let $p$ be a point of $Z^+$. Assume that $p$ is not a hard end, and let $B$ be a branch at $p$.
    If there is no connector $\gamma:[0,1]\to S_\infty^2$ such that $p\nin \gamma([0,1])$ and $\gamma(1)\in B$, then
    $Z^+\setminus \bigcup_{g\in G}g(B)$, together with $Z^-$, forms another zipper. In particular, $Z^+$ is not minimal.
\end{lem}
\begin{proof}
It suffices to show that \(g(B)\cap B=\varnothing\) for every \(g\in G\setminus\{\id\}\). Suppose otherwise, and choose \(h\in G\setminus\{\id\}\) such that \(h(B)\cap B\neq \varnothing\).

We first claim that there exists \(k\in G\) such that \(k(p)\in B\). If \(p\in h(B)\), then \(k=h^{-1}\) works. So assume \(p\notin h(B)\). Since \(h(B)\) meets \(B\), it must be contained in the branch \(B\) at \(p\). If \(h(p)\in B\), then we may take \(k=h\). Otherwise \(h(p)=p\) and \(h(B)=B\). We claim that this is impossible.

Indeed, after replacing \(h\) by \(h^{-1}\) if necessary, we may assume \(p=r(h)\), and write \(q=a(h)\). By \refprop{invBranch}, there is an \(h\)-prong in \( B\cup \{p\}\) starting from \(p\) and landing at \(q\). This produces a connector \(\gamma\) associated to \(q\) such that \(p\notin \gamma([0,1])\) and \(\gamma(1)\in B\), contrary to the assumption. To see this, note that if \(q\in Z^-\), then \(q\) is a type~2 synapse. If \(q\notin Z^-\), then \(h\) acts freely on \(Z^-\), and \refthm{invHypAxis} together with \reflem{invRayLand} gives an \(h\)-axis in \(Z^-\) with an end ray landing at \(q\); hence \(q\) is a synapse again. This proves the claim.

Choose \(k\in G\) with \(k(p)\in B\), and let \(\sigma=[p,k(p)]\subset B\cup\{p\}\). We consider the two cases given by \reflem{invRayLand}.

\smallskip
\noindent\emph{Case 1: \(\sigma\cap \Fix(k)=\varnothing\).}
Then there exists \(v\in (p,k(p)]\) such that
\[
[p,k(p)]\cap [k(p),k^2(p)]=[k(p),v],
\]
with \(v\neq k^2(p)\). Since \(p\notin [k(p),k^2(p)]\), the whole segment \([k(p),k^2(p)]\) lies in \(B\). Moreover, \(k([k(p),v])\subset k(\sigma)\), and since \(k(\sigma)\cap \Fix(k)=\varnothing\), we have \(k(v)\in (v,k^2(p)]\).

It follows that
\[
A=\bigcup_{n\in \ZZ}k^n([v,k(v)])
\]
is a \(k\)-invariant line in \(Z^+\). By construction, one end ray of \(A\) is contained in \(B\), and by \reflem{invRayLand}, each end ray of \(A\) lands at a fixed point of \(k\).

We now distinguish three possibilities.

If \(\Fix(k)\subset Z^+\) and \(A\cup \Fix(k)\) is the unique closed arc joining the two fixed points in \(Z^+\), then the fixed point \(a(k)\) lies in \(B\). On the other hand, \(k\) acts freely on \(Z^-\), so \reflem{invRayLand} yields an axis in \(Z^-\) whose end ray lands at \(a(k)\). Thus \(a(k)\) is a synapse, contradicting the assumption.

If \(A\) is proper, that is, if \(A\) is the axis of \(k\), then \(\Fix(k)\cap Z^+=\varnothing\). Since one end ray of \(A\) is contained in \(B\), we have \(\bigcup_{n\in \ZZ}k^n(B)=Z^+\). Hence there is no synapse, contradicting \refprop{bridgeExist}.

The remaining possibility is that \(|\Fix(k)\cap Z^+|=1\). Equivalently, exactly one topological end of \(A\) lands at a point \(a\in Z^+\). Then \(\closure A=A\cup a\) is a proper ray in \(Z^+\) whose other end lands at the second fixed point \(b\) of \(k\), which is not in $Z^+$.

If \(b\in Z^-\), then \(b\) is a type~2 synapse. If \(b\notin Z^-\), then \(k\) acts freely on \(Z^-\), so \reflem{invRayLand} gives an axis in \(Z^-\) whose end ray lands at \(b\), and hence \(b\) is a synapse. Either way we obtain a contradiction.

\smallskip
\noindent\emph{Case 2: \(\sigma\cap \Fix(k)\neq \varnothing\).}
Let \(q\in \sigma\cap \Fix(k)\). Then \(\sigma=[p,q]\cup [k(p),q]\) and \(k([p,q])=[k(p),q]\). Since \(B\) is the unique branch at \(p\) containing \((p,q]\), the branch \(k(B)\) is the unique branch at \(k(p)\) containing \((k(p),q]\). In particular, \(k(B)\) contains \(Z^+\setminus B\), and therefore \(B\cup k(B)=Z^+\). It follows from the assumption that there is no synapse, contradicting \refprop{bridgeExist}.

This contradiction completes the proof.
\end{proof}

Given a fence \(\cJ\) of a zipper \(Z^\pm\), we say that a connector \(\gamma\colon [0,1]\to S_\infty^2\) \emph{crosses} \(\cJ\) through a Jordan domain \(\cD\) of \(\cJ\) if \(\gamma([0,1])\) does. By the density of synapses (see \refrmk{denseType2}), each Jordan domain of an invariant fence admits a crossing connector whose synapse lies in \(Z^+\). The possible configurations of such crossing connectors are described by the following dichotomy:
\begin{lem}[Alternative for crossing connectors]\label{Lem:crossAlt}
Let \(Z^\pm\) be a zipper for \(M\), and let \(g\in G\) be such that \(g\) admits a \(g\)-prong in \(Z^+\) and acts freely on \(Z^-\). Let \(\cD\) be a Jordan domain of the fence \(\cJ(g)\) (given by \refprop{invFence}). Then exactly one of the following holds:
\begin{enumerate}
    \item\label{Itm:tail} there is a unique node \(p\) of \(\cJ(g)\) such that \(\gamma(1)=p\) for every connector \(\gamma\) crossing \(\cD\);
    \item\label{Itm:noTail} for every connector \(\gamma\) crossing \(\cD\), the point \(\gamma(1)\) lies in \(\Int \cJ(g)^+\).
\end{enumerate}
Moreover, in Case~\refitm{tail}, \(\cJ(g)^+\) is inaccessible from the \(\cD\)-side.
\end{lem}

\begin{rmk}\label{Rmk:simplicial}
We say that \(\cD\) is \emph{of type~I} in Case~\ref{Itm:tail}, and \emph{of type~II} in Case~\ref{Itm:noTail}. In the type~I case, the unique node \(p\) is called the \emph{pivot} of \(\cD\).
See \reffig{bothType}.
\end{rmk}

\begin{figure}[htbp]
\centering

\begin{subfigure}[t]{0.40\textwidth}
\centering
\tikzset{every picture/.style={line width=0.75pt}}

\begin{tikzpicture}[x=0.75pt,y=0.75pt,yscale=-1.5,xscale=1.5]

%Curve Lines [id:da6159870012263243] 
\draw [color={rgb, 255:red, 74; green, 144; blue, 226 }  ,draw opacity=1 ]   (166.43,78.26) .. controls (190.14,70.82) and (199.71,102.12) .. (212.71,106.12) .. controls (225.71,110.12) and (216.71,147.12) .. (212.71,157.12) .. controls (208.71,167.12) and (180.71,177.12) .. (156.27,178.12) ;
%Curve Lines [id:da45773744481677325] 
\draw [color={rgb, 255:red, 208; green, 2; blue, 27 }  ,draw opacity=1 ]   (161.88,79.56) .. controls (135.43,84.26) and (101.71,118.12) .. (114.71,129.12) .. controls (127.71,140.12) and (129.71,176.12) .. (152.71,178.12) ;
%Curve Lines [id:da6678563563543913] 
\draw [color={rgb, 255:red, 208; green, 2; blue, 27 }  ,draw opacity=1 ]   (166.43,80.26) .. controls (176.43,89.26) and (175.15,106.24) .. (188.15,110.24) .. controls (201.15,114.24) and (183.71,124.12) .. (203.71,128.12) ;
%Curve Lines [id:da6678345765779716] 
\draw [color={rgb, 255:red, 208; green, 2; blue, 27 }  ,draw opacity=1 ]   (164.43,81.26) .. controls (166.43,98.26) and (144.71,114.12) .. (156.71,120.12) .. controls (168.71,126.12) and (156.71,146.12) .. (176.71,150.12) ;
%Curve Lines [id:da6902714591979715] 
\draw [color={rgb, 255:red, 208; green, 2; blue, 27 }  ,draw opacity=1 ]   (161.88,81.56) .. controls (143.15,97.56) and (128.15,124.56) .. (139.15,131.56) .. controls (150.15,138.56) and (135.15,150.56) .. (154.15,159.56) ;
%Shape: Circle [id:dp04035532249361262] 
\draw  [color={rgb, 255:red, 74; green, 144; blue, 226 }  ,draw opacity=1 ][fill={rgb, 255:red, 208; green, 2; blue, 27 }  ,fill opacity=1 ][line width=0.75]  (152.86,160.27) .. controls (152.86,158.77) and (154.07,157.56) .. (155.57,157.56) .. controls (157.06,157.56) and (158.27,158.77) .. (158.27,160.27) .. controls (158.27,161.76) and (157.06,162.98) .. (155.57,162.98) .. controls (154.07,162.98) and (152.86,161.76) .. (152.86,160.27) -- cycle ;
%Curve Lines [id:da24028636405391757] 
\draw [color={rgb, 255:red, 74; green, 144; blue, 226 }  ,draw opacity=1 ]   (181.43,150.98) .. controls (190.14,152.67) and (197.71,138.83) .. (201.71,148.26) .. controls (205.71,157.69) and (205.71,150.12) .. (211.71,158.69) ;
%Curve Lines [id:da21322226693709545] 
\draw [color={rgb, 255:red, 74; green, 144; blue, 226 }  ,draw opacity=1 ]   (158.43,160.98) .. controls (167.14,162.67) and (168.71,162.69) .. (170.71,167.69) .. controls (172.71,172.69) and (166.71,170.69) .. (166.71,176.69) ;
%Curve Lines [id:da5789425138595687] 
\draw [color={rgb, 255:red, 74; green, 144; blue, 226 }  ,draw opacity=1 ]   (208.43,128.98) .. controls (217.14,130.67) and (209.71,136.69) .. (210.71,139.69) .. controls (211.71,142.69) and (210.71,141.69) .. (216.71,145.69) ;
%Straight Lines [id:da40699100043836833] 
\draw [color={rgb, 255:red, 208; green, 2; blue, 27 }  ,draw opacity=1 ] [dash pattern={on 0.84pt off 2.51pt}]  (117,123) -- (132.43,123.12) ;
%Straight Lines [id:da9221788160085597] 
\draw [color={rgb, 255:red, 208; green, 2; blue, 27 }  ,draw opacity=1 ] [dash pattern={on 0.84pt off 2.51pt}]  (178.43,93.12) -- (186.43,84.12) ;
%Straight Lines [id:da44334410980192274] 
\draw [color={rgb, 255:red, 208; green, 2; blue, 27 }  ,draw opacity=1 ] [dash pattern={on 0.84pt off 2.51pt}]  (154.43,174.12) -- (155.43,164.12) ;
%Shape: Circle [id:dp847467114035152] 
\draw  [color={rgb, 255:red, 0; green, 0; blue, 0 }  ,draw opacity=1 ][fill={rgb, 255:red, 255; green, 255; blue, 255 }  ,fill opacity=1 ][line width=0.75]  (150.86,178.27) .. controls (150.86,176.77) and (152.07,175.56) .. (153.57,175.56) .. controls (155.06,175.56) and (156.27,176.77) .. (156.27,178.27) .. controls (156.27,179.76) and (155.06,180.98) .. (153.57,180.98) .. controls (152.07,180.98) and (150.86,179.76) .. (150.86,178.27) -- cycle ;
%Shape: Circle [id:dp9413137750222575] 
\draw  [color={rgb, 255:red, 74; green, 144; blue, 226 }  ,draw opacity=1 ][fill={rgb, 255:red, 208; green, 2; blue, 27 }  ,fill opacity=1 ][line width=0.75]  (175.86,150.27) .. controls (175.86,148.77) and (177.07,147.56) .. (178.57,147.56) .. controls (180.06,147.56) and (181.27,148.77) .. (181.27,150.27) .. controls (181.27,151.76) and (180.06,152.98) .. (178.57,152.98) .. controls (177.07,152.98) and (175.86,151.76) .. (175.86,150.27) -- cycle ;
%Shape: Circle [id:dp7626337384404354] 
\draw  [color={rgb, 255:red, 74; green, 144; blue, 226 }  ,draw opacity=1 ][fill={rgb, 255:red, 208; green, 2; blue, 27 }  ,fill opacity=1 ][line width=0.75]  (202.86,128.27) .. controls (202.86,126.77) and (204.07,125.56) .. (205.57,125.56) .. controls (207.06,125.56) and (208.27,126.77) .. (208.27,128.27) .. controls (208.27,129.76) and (207.06,130.98) .. (205.57,130.98) .. controls (204.07,130.98) and (202.86,129.76) .. (202.86,128.27) -- cycle ;
%Shape: Circle [id:dp5107686317067551] 
\draw  [color={rgb, 255:red, 74; green, 144; blue, 226 }  ,draw opacity=1 ][fill={rgb, 255:red, 208; green, 2; blue, 27 }  ,fill opacity=1 ][line width=0.75]  (160.86,79.27) .. controls (160.86,77.77) and (162.07,76.56) .. (163.57,76.56) .. controls (165.06,76.56) and (166.27,77.77) .. (166.27,79.27) .. controls (166.27,80.76) and (165.06,81.98) .. (163.57,81.98) .. controls (162.07,81.98) and (160.86,80.76) .. (160.86,79.27) -- cycle ;

% Text Node
\draw (158,66) node [anchor=north west][inner sep=0.75pt]  [color={rgb, 255:red, 208; green, 2; blue, 27 }  ,opacity=1 ]  {$p$};
% Text Node
\draw (150,183) node [anchor=north west][inner sep=0.75pt]  [color={rgb, 255:red, 0; green, 0; blue, 0 }  ,opacity=1 ]  {$q$};
% Text Node
\draw (223,121) node [anchor=north west][inner sep=0.75pt]  [color={rgb, 255:red, 74; green, 144; blue, 226 }  ,opacity=1 ]  {$\ell $};
\end{tikzpicture}

\caption{type I}
\end{subfigure}
\hspace{0.01\textwidth}
\begin{subfigure}[t]{0.4\textwidth}
\centering
\tikzset{every picture/.style={line width=0.75pt}}

\begin{tikzpicture}[x=0.75pt,y=0.75pt,yscale=-1.5,xscale=1.5]

%Curve Lines [id:da2616625356684277] 
\draw [color={rgb, 255:red, 74; green, 144; blue, 226 }  ,draw opacity=1 ]   (186.27,129.11) .. controls (209.99,121.67) and (219.56,152.97) .. (232.56,156.97) .. controls (245.56,160.97) and (236.56,197.97) .. (232.56,207.97) .. controls (228.56,217.97) and (200.56,227.97) .. (176.12,228.97) ;
%Curve Lines [id:da5300851283334995] 
\draw [color={rgb, 255:red, 208; green, 2; blue, 27 }  ,draw opacity=1 ]   (181.88,129.4) .. controls (155.43,134.1) and (121.71,167.96) .. (134.71,178.96) .. controls (147.71,189.96) and (149.71,225.96) .. (172.71,227.96) ;
%Curve Lines [id:da6432369500027086] 
\draw [color={rgb, 255:red, 208; green, 2; blue, 27 }  ,draw opacity=1 ]   (154,215.27) .. controls (162,209.27) and (158,204.27) .. (165,202.27) .. controls (172,200.27) and (169,206.27) .. (178,207.27) ;
%Shape: Circle [id:dp8236093690876375] 
\draw  [color={rgb, 255:red, 74; green, 144; blue, 226 }  ,draw opacity=1 ][fill={rgb, 255:red, 208; green, 2; blue, 27 }  ,fill opacity=1 ][line width=0.75]  (177.86,207.11) .. controls (177.86,205.61) and (179.07,204.4) .. (180.57,204.4) .. controls (182.06,204.4) and (183.27,205.61) .. (183.27,207.11) .. controls (183.27,208.6) and (182.06,209.82) .. (180.57,209.82) .. controls (179.07,209.82) and (177.86,208.6) .. (177.86,207.11) -- cycle ;
%Curve Lines [id:da5826363003444915] 
\draw [color={rgb, 255:red, 74; green, 144; blue, 226 }  ,draw opacity=1 ]   (183,208.27) .. controls (191.71,209.96) and (193.56,213.82) .. (195.56,218.82) .. controls (197.56,223.82) and (191.56,221.82) .. (191.56,227.82) ;
%Straight Lines [id:da3656555418035341] 
\draw [color={rgb, 255:red, 208; green, 2; blue, 27 }  ,draw opacity=1 ] [dash pattern={on 0.84pt off 2.51pt}]  (175.43,222.96) -- (178,213.27) ;
%Shape: Circle [id:dp7772651838906471] 
\draw  [color={rgb, 255:red, 0; green, 0; blue, 0 }  ,draw opacity=1 ][fill={rgb, 255:red, 255; green, 255; blue, 255 }  ,fill opacity=1 ][line width=0.75]  (170.86,228.11) .. controls (170.86,226.61) and (172.07,225.4) .. (173.57,225.4) .. controls (175.06,225.4) and (176.27,226.61) .. (176.27,228.11) .. controls (176.27,229.6) and (175.06,230.82) .. (173.57,230.82) .. controls (172.07,230.82) and (170.86,229.6) .. (170.86,228.11) -- cycle ;
%Shape: Circle [id:dp4779585625526547] 
\draw  [color={rgb, 255:red, 74; green, 144; blue, 226 }  ,draw opacity=1 ][fill={rgb, 255:red, 208; green, 2; blue, 27 }  ,fill opacity=1 ][line width=0.75]  (180.86,129.11) .. controls (180.86,127.61) and (182.07,126.4) .. (183.57,126.4) .. controls (185.06,126.4) and (186.27,127.61) .. (186.27,129.11) .. controls (186.27,130.6) and (185.06,131.82) .. (183.57,131.82) .. controls (182.07,131.82) and (180.86,130.6) .. (180.86,129.11) -- cycle ;
%Curve Lines [id:da41894999963326507] 
\draw [color={rgb, 255:red, 208; green, 2; blue, 27 }  ,draw opacity=1 ]   (142,188.27) .. controls (154,185.27) and (166,175.27) .. (173,173.27) .. controls (180,171.27) and (177,177.27) .. (186,178.27) ;
%Shape: Circle [id:dp6336699074292964] 
\draw  [color={rgb, 255:red, 74; green, 144; blue, 226 }  ,draw opacity=1 ][fill={rgb, 255:red, 208; green, 2; blue, 27 }  ,fill opacity=1 ][line width=0.75]  (185.86,178.11) .. controls (185.86,176.61) and (187.07,175.4) .. (188.57,175.4) .. controls (190.06,175.4) and (191.27,176.61) .. (191.27,178.11) .. controls (191.27,179.6) and (190.06,180.82) .. (188.57,180.82) .. controls (187.07,180.82) and (185.86,179.6) .. (185.86,178.11) -- cycle ;
%Curve Lines [id:da7322422883084145] 
\draw [color={rgb, 255:red, 74; green, 144; blue, 226 }  ,draw opacity=1 ]   (191,179.27) .. controls (199.71,180.96) and (201.56,184.82) .. (203.56,189.82) .. controls (205.56,194.82) and (223,202.27) .. (228,213.27) ;
%Curve Lines [id:da9102641340911595] 
\draw [color={rgb, 255:red, 208; green, 2; blue, 27 }  ,draw opacity=1 ]   (144,151.27) .. controls (151,156.27) and (161,152.27) .. (168,152.27) .. controls (175,152.27) and (174,153.27) .. (183,154.27) ;
%Shape: Circle [id:dp2984870814405949] 
\draw  [color={rgb, 255:red, 74; green, 144; blue, 226 }  ,draw opacity=1 ][fill={rgb, 255:red, 208; green, 2; blue, 27 }  ,fill opacity=1 ][line width=0.75]  (182.86,155.11) .. controls (182.86,153.61) and (184.07,152.4) .. (185.57,152.4) .. controls (187.06,152.4) and (188.27,153.61) .. (188.27,155.11) .. controls (188.27,156.6) and (187.06,157.82) .. (185.57,157.82) .. controls (184.07,157.82) and (182.86,156.6) .. (182.86,155.11) -- cycle ;
%Curve Lines [id:da125815002672375] 
\draw [color={rgb, 255:red, 74; green, 144; blue, 226 }  ,draw opacity=1 ]   (188,155.27) .. controls (196.71,156.96) and (200,169.27) .. (208,166.27) .. controls (216,163.27) and (230,169.27) .. (235,158.27) ;
%Curve Lines [id:da1708778381493764] 
\draw [color={rgb, 255:red, 208; green, 2; blue, 27 }  ,draw opacity=1 ]   (161,138.27) .. controls (167,145.27) and (168,144.27) .. (173,142.27) .. controls (178,140.27) and (180,143.27) .. (185,143.27) ;
%Shape: Circle [id:dp618005722860833] 
\draw  [color={rgb, 255:red, 74; green, 144; blue, 226 }  ,draw opacity=1 ][fill={rgb, 255:red, 208; green, 2; blue, 27 }  ,fill opacity=1 ][line width=0.75]  (184.86,144.11) .. controls (184.86,142.61) and (186.07,141.4) .. (187.57,141.4) .. controls (189.06,141.4) and (190.27,142.61) .. (190.27,144.11) .. controls (190.27,145.6) and (189.06,146.82) .. (187.57,146.82) .. controls (186.07,146.82) and (184.86,145.6) .. (184.86,144.11) -- cycle ;
%Curve Lines [id:da14467910044557297] 
\draw [color={rgb, 255:red, 74; green, 144; blue, 226 }  ,draw opacity=1 ]   (190,144.27) .. controls (198.71,145.96) and (198,152.27) .. (201,146.27) .. controls (204,140.27) and (207,146.27) .. (213,138.27) ;
%Curve Lines [id:da7913947137776988] 
\draw [color={rgb, 255:red, 208; green, 2; blue, 27 }  ,draw opacity=1 ] [dash pattern={on 0.84pt off 2.51pt}]  (188,130.84) .. controls (191,134.84) and (188,138.41) .. (191,142.41) ;

% Text Node
\draw (178,115) node [anchor=north west][inner sep=0.75pt]  [color={rgb, 255:red, 208; green, 2; blue, 27 }  ,opacity=1 ]  {$p$};
% Text Node
\draw (170,233) node [anchor=north west][inner sep=0.75pt]  [color={rgb, 255:red, 0; green, 0; blue, 0 }  ,opacity=1 ]  {$q$};
% Text Node
\draw (242,168.24) node [anchor=north west][inner sep=0.75pt]  [color={rgb, 255:red, 74; green, 144; blue, 226 }  ,opacity=1 ]  {$\ell $};

\end{tikzpicture}

\caption{type II}
\end{subfigure}

\caption{Jordan domains of both types with $p\in Z^+$ whose left side is in $Z^+$.}
\label{Fig:bothType}
\end{figure}
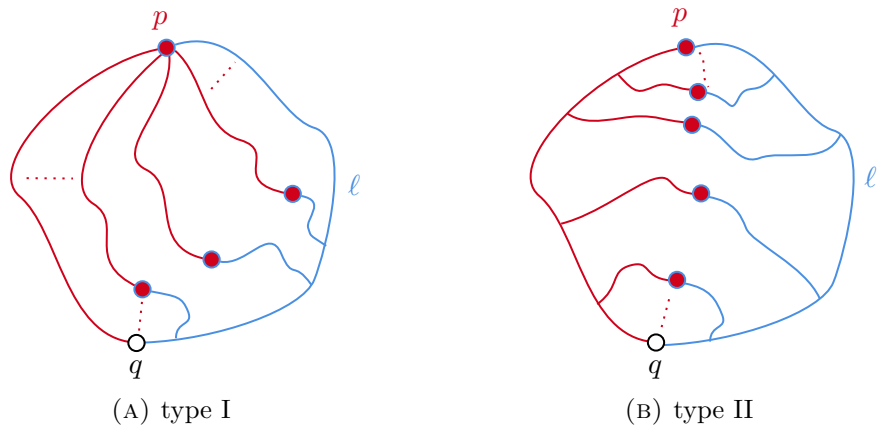
\begin{proof}[Proof of \reflem{crossAlt}]
We first claim that if \(\gamma\) and \(\delta\) are connectors crossing \(\cD\), then either \(\{\gamma(1),\delta(1)\}\subset \Int \cJ(g)^+\) or \(\{\gamma(1),\delta(1)\}\subset \Fix(g)\). Suppose not. By symmetry, after replacing \(g\) by \(g^{-1}\) if necessary, we may assume that \(\gamma(1)\in \Int \cJ(g)^+\) and \(\delta(1)=a(g)\). Since \(g\) acts as a translation on the interior of each side of \(\cJ(g)\), and since \(\gamma([0,1])\) is compact and disjoint from \(\Fix(g)\), the connector \(g^n\circ \gamma\) meets \(\delta\) away from \(\delta([0,1/2])\) for all sufficiently large \(n\in \NN\).
This contradicts the fact that \(Z^+\) and \(Z^-\) are disjoint and simply connected. The claim follows.

It remains to show that if \(\gamma\) and \(\delta\) are connectors crossing \(\cD\) with \(\{\gamma(1),\delta(1)\}\subset \Fix(g)\), then \(\gamma(1)=\delta(1)\). Suppose instead that \(\gamma(1)=r(g)\) and \(\delta(1)=a(g)\). Then, exactly as above, \(g^n\circ \gamma\) meets \(\delta\) away from  \(\delta([0,1/2])\)for all sufficiently large \(n\in \NN\), again contradicting the fact that \(Z^+\) and \(Z^-\) are disjoint and simply connected. This proves the first part.

To prove the second statement, it is enough to show that no ray in \(Z^-\) lands on the \(\cD\)-side of \(\cJ(g)^+\). Suppose otherwise, and let \(r\subset Z^-\) be such a ray, contained in \(\cD\).
There exists a ray \(r'\subset Z^-\) issuing from \(\cJ(g)^-\) and crossing \(\cD\). As above, since the nodes of \(\cJ(g)\) do not belong to \(Z^-\) and \(\cJ(g)^-\) is a line, one then finds a connector \(\gamma\) crossing \(\cD\) whose \((+)\)-segment meets \(r'\), contrary to the disjointness of \(Z^\pm\).
This completes the proof.
\end{proof}

Applying \reflem{simplicialSide} to both Jordan domains of an invariant fence \(\cJ(g)\), we conclude that \(\cJ(g)\) has at least one Jordan domain of type~I:

\begin{lem}[Existence of a  Jordan Domain of Type~I]\label{Lem:simplicialSide}
Let \(Z^\pm\) be a minimal zipper for \(M\). Assume that \(g\in G\) admits a \(g\)-prong in \(Z^+\) and acts freely on \(Z^-\). Then, for each \(p\in \Fix(g)\cap Z^+\), there is a unique Jordan domain \(\cD(p)\) of \(\cJ(g)\) that is of type~I and has pivot \(p\). In particular, every branch at \(p\) disjoint from \(\cJ(g)^+\) is contained in \(\cD(p)\).
\end{lem}

\begin{rmk}[Inaccessibility of invariant fences]\label{Rmk:inaccessibleSide}
The following are immediate consequences of \reflem{simplicialSide} and \reflem{crossAlt}.
If \(\Fix(g)\subset Z^+\), then both Jordan domains of \(\cJ(g)\) are of type~I, with distinct nodes. In particular, \(\cJ(g)^+\) is inaccessible. If \(\Fix(g)\cap Z^+=\{p\}\) and $\Fix(g) \cap Z^-=\varnothing$, then \(\cD(p)\) is the unique Jordan domain of \(\cJ(g)\) of type~I, and the other Jordan domain is of type~II. In particular, \(\cJ(g)^+\) is one-sided inaccessible.
\end{rmk}

\begin{proof}[Proof of \reflem{simplicialSide}]
For each \(p\in \Fix(g)\cap Z^+\), there is a unique \(g\)-prong \(r\subset \cJ(g)^+\) landing at the other fixed point \(q\), where \(q\) may or may not belong to \(Z^+\).

Suppose that both Jordan domains of \(\cJ(g)\) are of type~II. Let \(B\) be a branch at \(p\) disjoint from \(r\). Then \(B\) is contained in one of the Jordan domains of \(\cJ(g)\), and \reflem{crossAlt} implies that \(B\) satisfies the hypothesis of \reflem{pruning}. By \reflem{pruning}, \(Z^+\) is then not minimal, a contradiction. It follows that \(p\) is a hard end, contradicting again the minimality of \(Z^+\). Hence at least one Jordan domain of \(\cJ(g)\) is of type~I. Let \(\cD\) denote such a domain, and let \(\cD'\) denote the other.

If \(\cD'\) is also of type~I with node \(p\), then \reflem{crossAlt} implies that \(\cJ(g)^+\) is simplicial and that \(\cJ(g)^+\setminus\{p\}\) is a branch at \(p\). However, since \(p\) is a branch point of $Z^+$, there is no element \(k\in G\) such that \(k(p)\in \cJ(g)^+\setminus\{p\}\). This contradicts the minimality of \(Z^+\). Therefore, \(\cD'\) is either of type~II or of type~I with node \(q\). This proves the first assertion.

The second assertion follows from \reflem{crossAlt} and \reflem{pruning} as above.
Thus, we are done.
\end{proof}

Before turning to the other case, we note that the simpliciality of the \((+)\)-side of \(\cJ(g)\) will play an important role later in proving that \(\cJ(g)\) is unlinked from its translates under the \(G\)-action; see \refsec{proofOutline}.

\subsection{Pruning Axes}\label{Sec:pruningAxes}

Let \(Z^\pm\) be a minimal zipper for \(M\), and let \(g\in G\) act freely on \(Z^-\) while fixing a point of \(Z^+\). Suppose that \(g\) fixes a unique point \(p\in Z^+\). Then \(g\) need not admit a \(g\)-prong in \(Z^+\), and in that case there is no \(g\)-invariant fence. Nevertheless, \(g\) still admits an axis \(\ell\subset Z^-\).

We now study the one-sided simpliciality and one-sided inaccessibility of \(\ell\) in this situation. By contrast, if \(g\) admits a \(g\)-prong in \(Z^+\) and has no fixed point in \(Z^-\), then \(\ell\) is always two-sided branched by \reflem{simplicialSide}.

Assume from now on that \(g\) does not admit a \(g\)-prong in \(Z^+\). We say that a connector \(\gamma\colon [0,1]\to S^2_\infty\), or the corresponding connector arc, \emph{crosses} \(\ell\) if \(\gamma(1)=p\) and \(\gamma([0,1])\cap \ell=\{\gamma(0)\}\). We call \(p\) the \emph{pivot} of \(\ell\) if there exists such a crossing arc. Here we regard \(\ell\cup \Fix(g)\) as a \(2\)-fold Jordan curve. Note that, since \(p\) is a synapse of type~2 in \(Z^+\), there always exists a connector crossing \(\ell\) whose synapse lies in \(Z^+\).

After fixing an orientation \(f\) of \(\ell\), we classify connectors crossing \(\ell\) as follows. We say that a crossing connector \(\gamma\) is \emph{left} (resp. \emph{right}) with respect to \(f\) if \(\gamma((0,1/2))\) is contained in a left (resp. right) branch of \(\ell\). We say that \(f\) is \emph{canonical} with respect to \(g\) if \(f\) is increasing toward \(p\) along \(\ell\). Unless stated otherwise, \(\ell\) will always be endowed with this canonical orientation. Note that \(g\) sends left (reps. right) crossing connectors to left (resp. right) crossing connectors.

\begin{prop}[Alternative for branches at the pivot]\label{Prop:classifyConnector}
Let \(Z^\pm\) be a minimal zipper for \(M\). Assume that \(g\in G\) fixes a unique point \(p\in Z^+\), acts freely on \(Z^-\), and admits no \(g\)-prong in \(Z^+\). If \(\gamma_1\) and \(\gamma_2\) are connectors crossing \(\ell\) that are left and right, respectively, then \(\gamma_1((1/2,1])\cap \gamma_2((1/2,1])=\{p\}\).
\end{prop}

\begin{proof}
Suppose that \(\gamma_1((1/2,1])\cap \gamma_2((1/2,1])\) contains a nondegenerate arc \([p,q]\subset Z^+\). Then there is a unique branch \(B\) at \(p\) containing \((p,q]\).

For \(i=1,2\), let \(r_i\) be the end ray of \(\ell\) starting from \(\gamma_i(0)\) and landing at \(p\). Note that \(\cF_i:=r_i\cup \gamma_i([0,1])\) is a fence. Let \(\cD_i\) be the Jordan domain of \(\cF_i\) disjoint from the \(g\)-axis \(\ell\).

Since \(g\) admits no \(g\)-prong in \(Z^+\), \refprop{invBranch} implies that \(g(B)\) is a branch at \(p\) disjoint from \(B\). Because \(Z^+\cap Z^-=\varnothing\), it follows that \(g(B)\subset \cD_i\) for some \(i\in\{1,2\}\). Without loss of generality, assume that \(g(B)\subset \cD_2\). Then \(g\circ \gamma_1\) is a right crossing connector with respect to \(\ell\), contradicting the fact that \(g\) sends left crossing connectors to left crossing connectors.
\end{proof}

It follows that any two crossing connectors meeting the same branch at the pivot have the same type. Accordingly, we say that a branch \(B\) at the pivot \(p\) is \emph{left} (resp. \emph{right}) if some left (resp. right) crossing connector of \(\ell\) meets \(B\).

\begin{lem}[One-sided inaccessible axes]\label{Lem:sideSimplicial}
Let \(Z^\pm\) be a minimal zipper for \(M\). Assume that \(g\in G\) fixes a unique point \(p\in Z^+\), acts freely on \(Z^-\), and admits no \(g\)-prong in \(Z^+\). If there is no right (resp. left) crossing connector of the \(g\)-axis \(\ell\) with respect to the canonical orientation, then \(\ell\) is right (resp.left) inaccessible, and  one-sided inaccessible.
\end{lem}

\begin{proof}
The right (resp. left) simpliciality, and hence one-sided simpliciality, of \(\ell\) follows from \reflem{pruning} together with the fact that there exists at least one left or right branch at the pivot, since \(Z^-\) cannot be a segment by \refrmk{noSeg}.

By symmetry, it suffices to consider the case where there is no right crossing connector of \(\ell\). We must show that no ray in \(Z^+\) lands on the right side of \(\ell\). Suppose otherwise. Then there exists a ray \(r\subset Z^+\) starting from the pivot \(p\) of \(\ell\) and landing at a point \(e\in \ell\) on the right side. Let \(\sigma\) be the end ray of \(\ell\) starting from \(e\) and landing at \(p\). Then \(\sigma\cup r\) forms a fence \(\cF\). Let \(\cD\) be the Jordan domain of \(\cF\) disjoint from \(\ell\). Since, by \refrmk{denseType2}, \(\cD\) contains a type~2 synapse in \(Z^+\), there exists a connector \(\gamma\) of type~2 crossing \(\cF\) in \(\cD\). By construction, $\gamma((0,1/2))$ should land at $\gamma(0)$ on the right side of $\ell$, a contradiction. This proves the claim.
\end{proof}
\subsection{Eclipses}
We are now ready to define the notion of an \emph{eclipse} on a zipper for an element $g\in G$ with $\Fix(g)\cap Z^+\ne \varnothing$ and $\Fix(g)\cap Z^-=\varnothing$.

\begin{const}[Eclipse on a Jordan domain]\label{Const:eclipseD}
Let \(Z^\pm\) be a minimal zipper for \(M\), and let \(g\in G\). Suppose that \(g\) admits a \(g\)-prong in \(Z^+\) and acts freely on \(Z^-\). Let \(\cD\) be a type~I Jordan domain of the invariant fence \(\cJ(g)\), let \(p\) be the pivot of \(\cD\), and write \(\Fix(g)=\{p,q\}\). Choose a connector arc \(\gamma\) crossing \(\cD\) such that the synapse of \(\gamma\) lies in \(Z^+\); see \refrmk{denseType2}.

Define the \emph{eclipse} \(\{\cE_n\}_{n\in\ZZ}\) on \(\cD\) associated to \(\gamma\) by letting \(\cE_0\) be the unique element of \(\{\closure{\cD^L(\gamma)},\closure{\cD^R(\gamma)}\}\) that contains \(q\), and setting
\[
\cE_n=
\begin{cases}
\ g^n(\cE_0) & \text{if } a(g)=q,\\
\ g^{-n}(\cE_0) & \text{if } r(g)=q.
\end{cases}
\qquad\text{for all } n\in\ZZ.
\]
See \reffig{bothType} (left).
\end{const}

Even when \(g\) admits no invariant fence, one can still define an eclipse as follows.

\begin{const}[Eclipse along an axis]\label{Const:eclipseL}
Let \(Z^\pm\) be a minimal zipper for \(M\). Let \(g\in G\) fix a unique point \(p\in Z^+\), act freely on \(Z^-\), and admit no \(g\)-prong in \(Z^+\). Write \(\Fix(g)=\{p,q\}\), and note that \(p\) is the pivot of the axis \(\ell\subset Z^-\) of \(g\).

Define \(\cE_0\) as follows. If \(p\) admits both a left branch and a right branch, choose type~2 connector arcs \(\delta_1,\delta_2\) crossing \(\ell\), with \(\delta_1\) left, \(\delta_2\) right, and \(\delta_i(1/2)\in Z^+\). Let \(s=[\delta_1(0),\delta_2(0)]\subset\ell\). After replacing \(\delta_2\) by \(g^k\circ\delta_2\) for some \(k\in\ZZ\), we may assume that \(s\cap g(s)=\varnothing\). Let \(\cJ\) be the fence formed by \(\delta_1\), \(\delta_2\), and \(s\), and let \(\cE_0\) be the closure of the Jordan domain of \(\cJ\) containing \(q\). If \(p\) admits only left branches or only right branches, choose a type~2 connector arc \(\delta\) of the corresponding type crossing \(\ell\), with \(\delta(1/2)\in Z^+\). Let \(r\) be the end ray of \(\ell\) from \(\delta(0)\) to \(p\). Then \(\delta([0,1])\cup r\) is a fence \(\cJ\) whose nodes lie in \(Z^+\), and let \(\cE_0\) be the closure of the Jordan domain bounded by \(\cJ\) that contains \(q\).

The \emph{eclipse} \(\{\cE_n\}_{n\in\ZZ}\) along \(\ell\), associated to \(\delta\) or to \((\delta_1,\delta_2)\), is defined by
\[
\cE_n=
\begin{cases}
\ g^n(\cE_0) & \text{if } a(g)=q,\\
\ g^{-n}(\cE_0) & \text{if } r(g)=q.
\end{cases}
\qquad\text{for all } n\in\ZZ.
\]
See \reffig{fatFence} and \reffig{squashedFence}. 
\end{const}

For an eclipse \(\{\cE_n\}_{n\in\ZZ}\) arising from \refconst{eclipseD} or \refconst{eclipseL}, set
\[
\cE_{-\infty}=\bigcup_{n\in\ZZ}\cE_n,
\quad\text{and}\quad
\cE_{\infty}=\bigcap_{n\in\ZZ}\cE_n.
\]
The crucial properties of an eclipse are that \(\cE_{n+1}\subsetneq \cE_n\) for all \(n\in\ZZ\), and that the \((-)\)-segments of the defining connectors shrink arbitrarily small toward \(q\) as \(n\to\infty\), while the \((+)\)-side remains anchored at \(p\). In fact, as is already visible in the proof of \reflem{crossAlt}, these properties are used implicitly but crucially to establish inaccessibility.

Before concluding this section, we summarize the topological properties of eclipses on a type~I Jordan domain.

\begin{prop}\label{Prop:eclipseDProperty}
Under the assumptions of \refconst{eclipseD}, the eclipse \(\{\cE_n\}_{n\in\ZZ}\) satisfies the following properties.
\begin{enumerate}
    \item \emph{(Nested exhaustion)} One has \(\cE_{-\infty}=\closure{\cD}\). Moreover, \(\cE_{n+1}\subsetneq \cE_n\) for all \(n\in\ZZ\).

    \item \emph{(Separation property)} For every \(n\in\ZZ\) and every \(k\in\NN\), one has
    \[
    \closure{\cE_n\setminus\cE_{n+k}}\cap \cE_\infty=\{p\}.
    \]

    \item \emph{(The end of the eclipse)} The set \(\cE_\infty\) is a non-empty closed subset of \(\closure{\cD}\) such that  \(\cE_\infty\cap \cJ(g)=\closure{\cJ(g)^+}\).
    Moreover,
    \[
\cE_\infty=\closure{\cJ(g)^+}\sqcup \bigcap_{n\in\ZZ}\Int\cE_n
    \]
    \item \emph{(Thinness of the end)} The set \(\bigcap_{n\in\ZZ}\Int\cE_n\) contains neither  synapse nor  point of \(Z^\pm\). In particular, \(\Int \cE_\infty=\varnothing\).
\end{enumerate}
\end{prop}

\begin{proof}
The nestedness of the eclipse is immediate from the definition: for each \(n\in\ZZ\), one has \(\cE_{n+1}\subset \cE_n\).
To prove that \(\cE_{-\infty}=\closure{\cD}\), set \(D':=\closure{\cD\setminus \cE_0}\). Then \(D'\) is compact, \(D'\cup \cE_0=\cD\), and \(\Fix(g)\cap D'=\{p\}\).
It follows that either the forward or the backward iterates of \(D'\) under \(g\) are eventually contained in an arbitrarily small neighborhood of \(p\).
Therefore every point of \(\closure{\cD}\) belongs to \(\cE_n\) for all sufficiently negative \(n\), and hence \(\cE_{-\infty}=\closure{\cD}\).
The strict inclusion \(\cE_{n+1}\subsetneq \cE_n\) follows from the same observation.
This proves \((1)\).

Observe that \(\cE_n\cap \gamma=\{p\}\) for all \(n>0\), whereas \(\gamma\) crosses \(\Int \cE_n\) for all \(n<0\).
More generally, the same statement holds for every iterate \(g^m(\gamma)\), with the index shifted accordingly.
This yields \((2)\) and \((3)\), since the landing point of \(\gamma\) on \(\ell\) converges to \(q\) under iteration by \(g\).

It remains to prove the thinness of the end \(\cE_\infty\).
Suppose, for contradiction, that there exists a point
\(x\in \bigcap_{n\in\ZZ}\Int\cE_n\)
such that \(x\) is either a synapse or belongs to \(Z^\pm\).

First assume that \(x\) is either a synapse or belongs to \(Z^-\).
Then one can find a ray \(r\subset Z^-\) landing at \(x\) such that \(r\) meets the \((-)\)-side of \(\cJ(g)\) only at its initial point \(s\).
Let \(\gamma_k:=\partial \cE_k\cap \cD\), and let \(I\) be the component of \(\partial \cD\setminus\{q,s\}\) contained in \(\ell\).
Then \(\gamma_k\) lands at a point of \(I\) for all sufficiently large \(k\).
Fix such a \(k\).

Since \(x\in \Int \cE_k\), the set \(\partial \cE_k\) separates \(x\) from \(s\).
Hence the ray \(r\) must intersect \(\gamma_k\).
On the other hand, \(\gamma_k\) is of type~2 and its synapse lies in \(Z^+\).
Since \(Z^+\cap Z^-=\varnothing\), the intersection point must lie on the \((-)\)-segment of \(\gamma_k\).
But then \(r\), together with \(\gamma_k^-\) and \(\ell\), forms a loop in \(Z^-\), a contradiction.

Now assume that \(x\in Z^+\).
By \reflem{simplicialSide}, there is a unique branch \(B\) at \(p\) such that   \(x\in B \subset \cD\).
Since \(x\in \Int \cE_n\) for all \(n\in\ZZ\), (2) implies that \(B\subset \Int \cE_n\) for all \(n\in\ZZ\).
By the minimality of \(Z^+\), one has \(x\in (p,h(p)]\) for some \(h\in G\), so \(h(p)\in B\).
Since \(h(p)\) is a synapse, this contradicts the previous case.
Thus (4) follows from \refrmk{denseType2}.
\end{proof}

The same argument yields the corresponding properties for eclipses defined along an axis.
\begin{prop}\label{Prop:eclipseLProperty}
Assume the hypotheses of \refconst{eclipseL}.
\begin{itemize}
    \item \emph{(Nested exhaustion)} One has \(\cE_{-\infty}=S_\infty^2\). Moreover,
    \(
    \cE_{n+1}\subsetneq \cE_n
    \) for all  $n\in\ZZ$.
    \item \emph{(Separation property)}
    For every \(n\in\ZZ\) and every \(k\in\NN\), one has the following.

    \begin{enumerate}
        \item if $\ell$ is two-sided branched, then
        \[
\closure{\cE_n\setminus\cE_{n+k}}\cap \cE_\infty=\{p\}.
        \]

        \item if $\ell$ is one-sided simplicial, then
        \[        \closure{\cE_n\setminus\cE_{n+k}}\cap \cE_\infty
        =
        \{p\}\cup g^m([\delta(0),g^k\circ \delta(0)]),
        \]
        where \(m=n\) if \(q=a(g)\), and \(m=-(n+k)\) if \(q=r(g)\).
    \end{enumerate}

    \item \emph{(The end of the eclipse)} One has
\[
\cE_\infty=
\begin{dcases}
\ \{p\}\sqcup \displaystyle\bigcap_{n\in\ZZ}\Int\cE_n,
& \text{in the two-sided branched case},\\[1ex]
\ \{p\}\sqcup \ell\sqcup \displaystyle\bigcap_{n\in\ZZ}\Int\cE_n,
& \text{in the one-sided simplicial case}.
\end{dcases}
\]
In either case,
\(q\in \bigcap_{n\in\ZZ}\Int\cE_n.
\)
In particular, if $\ell$ is two-sided branched, then $\cE_\infty\cap\ell=\varnothing$.
    \item \emph{(Thinness of the end)} The set \(\bigcap_{n\in\ZZ}\Int\cE_n\) contains neither synapse or point of \(Z^\pm\). In particular, $q$ is not a synapse and
    \(
    \Int \cE_\infty=\varnothing.
    \)
\end{itemize}
\end{prop}

\section{Two fixed points in $Z^+$}\label{Sec:twoInOne}

In this section, we prove the following.
\begin{restate}{Theorem}{Thm:twoInOne}
Let $Z^\pm$ be a minimal zipper for $M$. Then no element \(g\in G\) has its fixed point set contained entirely in \(Z^+\) or entirely in \(Z^-\).
\end{restate}
Recall the strategy outlined in \refsec{proofOutline}.
Most of the lemmas proved in this section extend to the more general setting in which an element \(g\in G\) admits a \(g\)-prong on \(Z^+\).
For later use in the proof of \refthm{oneProng}, we therefore state them in this greater generality.

\subsection{Invariant Fence systems}
We say that a pair of Jordan curves, $\cI,\cJ$ in $S^2$, are \emph{unlinked} if $\cI$ does not intersect both Jordan domains of $\cJ$, equivalently, $\cI\subset \closure{\cD}$ for some Jordan domain of $\cJ$.

\begin{prop}\label{Prop:unlinking}
Let $\cI,\cJ$ be distinct fences of a zipper $Z^\pm$.
Assume that the nodes of $\cI$ lie in $Z^+$.
If $\cI^+$ and $\cJ^+$ are disjoint, then $\cI$ and $\cJ$ are unlinked.
\end{prop}
\begin{proof}
Since $\cI^+\cap \cJ^+=\varnothing$ and $Z^+\cap Z^-=\varnothing$, $\cI^+$ is disjoint from $\cJ$ and hence contained in a Jordan domain $\cD$ of $\cJ$. Choose two disjoint end rays $r_1,r_2$ of $\cI^-$ contained in $\cD$. Let $\cA\subset Z^-$ be the unique arc joining the starting points of $r_1$ and $r_2$. If $\cA\cap \cJ=\varnothing$, then $\cI\subset \cD$. Otherwise, $\cA\cap \cJ^- \neq \varnothing$, so \refprop{enclosing} yields $\cA\subset \overline{\cD}$. Thus in either case $\cI\subset \overline{\cD}$, and therefore $\cI$ and $\cJ$ are unlinked.
\end{proof}
\begin{figure}[h]
\centering

\tikzset{every picture/.style={line width=0.75pt}} %set default line width to 0.75pt        

\begin{tikzpicture}[x=0.75pt,y=0.75pt,yscale=-1.5,xscale=1.5]

%Curve Lines [id:da0038433347597566225] 
\draw [color={rgb, 255:red, 74; green, 144; blue, 226 }  ,draw opacity=1 ]   (186.43,98.26) .. controls (210.14,90.82) and (219.71,122.12) .. (232.71,126.12) .. controls (245.71,130.12) and (236.71,167.12) .. (232.71,177.12) .. controls (228.71,187.12) and (200.71,197.12) .. (176.27,198.12) ;
%Curve Lines [id:da9742120534768556] 
\draw [color={rgb, 255:red, 208; green, 2; blue, 27 }  ,draw opacity=1 ]   (181.88,99.56) .. controls (155.43,104.26) and (121.71,138.12) .. (134.71,149.12) .. controls (147.71,160.12) and (149.71,196.12) .. (172.71,198.12) ;
%Curve Lines [id:da9758218458639457] 
\draw [color={rgb, 255:red, 208; green, 2; blue, 27 }  ,draw opacity=1 ]   (186.43,100.26) .. controls (196.43,109.26) and (195.15,126.24) .. (208.15,130.24) .. controls (221.15,134.24) and (203.71,144.12) .. (223.71,148.12) ;
%Curve Lines [id:da7333613917215671] 
\draw [color={rgb, 255:red, 208; green, 2; blue, 27 }  ,draw opacity=1 ]   (184.43,101.26) .. controls (186.43,118.26) and (164.71,134.12) .. (176.71,140.12) .. controls (188.71,146.12) and (176.71,166.12) .. (196.71,170.12) ;
%Curve Lines [id:da1292432106126329] 
\draw [color={rgb, 255:red, 208; green, 2; blue, 27 }  ,draw opacity=1 ]   (181.88,101.56) .. controls (163.15,117.56) and (148.15,144.56) .. (159.15,151.56) .. controls (170.15,158.56) and (155.15,170.56) .. (174.15,179.56) ;
%Shape: Circle [id:dp5906174230219742] 
\draw  [color={rgb, 255:red, 74; green, 144; blue, 226 }  ,draw opacity=1 ][fill={rgb, 255:red, 208; green, 2; blue, 27 }  ,fill opacity=1 ][line width=0.75]  (172.86,180.27) .. controls (172.86,178.77) and (174.07,177.56) .. (175.57,177.56) .. controls (177.06,177.56) and (178.27,178.77) .. (178.27,180.27) .. controls (178.27,181.76) and (177.06,182.98) .. (175.57,182.98) .. controls (174.07,182.98) and (172.86,181.76) .. (172.86,180.27) -- cycle ;
%Curve Lines [id:da24188621678162836] 
\draw [color={rgb, 255:red, 74; green, 144; blue, 226 }  ,draw opacity=1 ]   (178.43,180.98) .. controls (187.14,182.67) and (188.71,182.69) .. (190.71,187.69) .. controls (192.71,192.69) and (186.71,190.69) .. (186.71,196.69) ;
%Curve Lines [id:da8750175090996589] 
\draw [color={rgb, 255:red, 74; green, 144; blue, 226 }  ,draw opacity=1 ]   (228.43,148.98) .. controls (237.14,150.67) and (229.71,156.69) .. (230.71,159.69) .. controls (231.71,162.69) and (230.71,161.69) .. (236.71,165.69) ;
%Straight Lines [id:da5686199741985563] 
\draw [color={rgb, 255:red, 208; green, 2; blue, 27 }  ,draw opacity=1 ] [dash pattern={on 0.84pt off 2.51pt}]  (137,143) -- (152.43,143.12) ;
%Straight Lines [id:da045917735291756356] 
\draw [color={rgb, 255:red, 208; green, 2; blue, 27 }  ,draw opacity=1 ] [dash pattern={on 0.84pt off 2.51pt}]  (198.43,113.12) -- (206.43,104.12) ;
%Straight Lines [id:da7264516054291804] 
\draw [color={rgb, 255:red, 208; green, 2; blue, 27 }  ,draw opacity=1 ] [dash pattern={on 0.84pt off 2.51pt}]  (174.43,194.12) -- (175.43,184.12) ;
%Shape: Circle [id:dp8957873923122714] 
\draw  [color={rgb, 255:red, 144; green, 19; blue, 254 }  ,draw opacity=1 ][fill={rgb, 255:red, 255; green, 255; blue, 255 }  ,fill opacity=1 ][line width=0.75]  (170.86,198.27) .. controls (170.86,196.77) and (172.07,195.56) .. (173.57,195.56) .. controls (175.06,195.56) and (176.27,196.77) .. (176.27,198.27) .. controls (176.27,199.76) and (175.06,200.98) .. (173.57,200.98) .. controls (172.07,200.98) and (170.86,199.76) .. (170.86,198.27) -- cycle ;
%Shape: Circle [id:dp6739313483128019] 
\draw  [color={rgb, 255:red, 74; green, 144; blue, 226 }  ,draw opacity=1 ][fill={rgb, 255:red, 208; green, 2; blue, 27 }  ,fill opacity=1 ][line width=0.75]  (222.86,148.27) .. controls (222.86,146.77) and (224.07,145.56) .. (225.57,145.56) .. controls (227.06,145.56) and (228.27,146.77) .. (228.27,148.27) .. controls (228.27,149.76) and (227.06,150.98) .. (225.57,150.98) .. controls (224.07,150.98) and (222.86,149.76) .. (222.86,148.27) -- cycle ;
%Shape: Circle [id:dp9368157684936457] 
\draw  [color={rgb, 255:red, 74; green, 144; blue, 226 }  ,draw opacity=1 ][fill={rgb, 255:red, 208; green, 2; blue, 27 }  ,fill opacity=1 ][line width=0.75]  (180.86,99.27) .. controls (180.86,97.77) and (182.07,96.56) .. (183.57,96.56) .. controls (185.06,96.56) and (186.27,97.77) .. (186.27,99.27) .. controls (186.27,100.76) and (185.06,101.98) .. (183.57,101.98) .. controls (182.07,101.98) and (180.86,100.76) .. (180.86,99.27) -- cycle ;
%Shape: Circle [id:dp5397747102005691] 
\draw  [color={rgb, 255:red, 74; green, 144; blue, 226 }  ,draw opacity=1 ][fill={rgb, 255:red, 208; green, 2; blue, 27 }  ,fill opacity=1 ][line width=0.75]  (207.86,160.27) .. controls (207.86,158.77) and (209.07,157.56) .. (210.57,157.56) .. controls (212.06,157.56) and (213.27,158.77) .. (213.27,160.27) .. controls (213.27,161.76) and (212.06,162.98) .. (210.57,162.98) .. controls (209.07,162.98) and (207.86,161.76) .. (207.86,160.27) -- cycle ;
%Curve Lines [id:da3449168649803357] 
\draw [color={rgb, 255:red, 208; green, 2; blue, 27 }  ,draw opacity=1 ]   (183.86,101.12) .. controls (190.86,111.12) and (179.86,124.79) .. (191.86,130.79) .. controls (203.86,136.79) and (202.86,146.79) .. (208.86,157.79) ;
%Curve Lines [id:da5318788132215212] 
\draw [color={rgb, 255:red, 74; green, 144; blue, 226 }  ,draw opacity=1 ]   (212.86,161.79) .. controls (219.86,172.79) and (230.71,175.69) .. (220.86,177.79) .. controls (211,179.88) and (215.86,181.79) .. (200.86,171.79) ;
%Shape: Circle [id:dp3526006615494177] 
\draw  [color={rgb, 255:red, 144; green, 19; blue, 254 }  ,draw opacity=1 ][fill={rgb, 255:red, 255; green, 255; blue, 255 }  ,fill opacity=1 ][line width=0.75]  (195.86,171.27) .. controls (195.86,169.77) and (197.07,168.56) .. (198.57,168.56) .. controls (200.06,168.56) and (201.27,169.77) .. (201.27,171.27) .. controls (201.27,172.76) and (200.06,173.98) .. (198.57,173.98) .. controls (197.07,173.98) and (195.86,172.76) .. (195.86,171.27) -- cycle ;

% Text Node
\draw (178,86) node [anchor=north west][inner sep=0.75pt]  [,color={rgb, 255:red, 208; green, 2; blue, 27 }  ,opacity=1 ]  {$p$};
% Text Node
\draw (168,204) node [anchor=north west][inner sep=0.75pt]  [,color={rgb, 255:red, 0; green, 0; blue, 0 }  ,opacity=1 ]  {$q$};
% Text Node
\draw (243,141.4) node [anchor=north west][inner sep=0.75pt]  [,color={rgb, 255:red, 74; green, 144; blue, 226 }  ,opacity=1 ]  {$\ell $};
% Text Node
\draw (180.86,172) node [anchor=north west][inner sep=0.75pt]  [font=\footnotesize,color={rgb, 255:red,0; green, 0; blue, 0 }  ,opacity=1 ]  {$h( q)$};
% Text Node
\draw (171.86,139.67) node[
  anchor=north west,
  inner sep=0.75pt,
  font=\footnotesize,
  color={rgb,255:red,208; green,2; blue,27},
  fill=white
] {$h(\cJ(g))^+$};

\end{tikzpicture}

\caption{$h(\cJ(g))$ in the type I Jordan domain, intersecting $\cJ(g)^+$ for $q\notin Z^+$.}
\label{Fig:plusIntersection}
\end{figure}
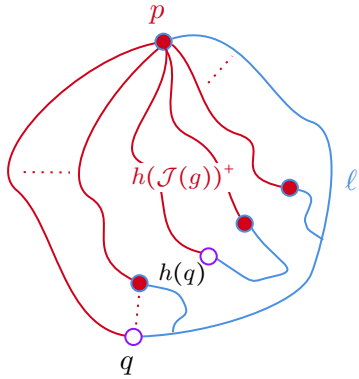
\begin{lem}[Invariant fence system]\label{Lem:fenceSystem}
Let $Z^\pm$ be a minimal zipper for $M$.
Assume that $g$ admits a $g$-prong in $Z^+$ and acts freely on $Z^-$.
Write $\Fix(g)=\{p,q\}$.
Then, for any $h\in G\setminus \{1\}$, the following hold:
\begin{itemize}
    \item $\cJ(g)$ and $h(\cJ(g))$ are unlinked;
    \item if $\Fix(g)\subset Z^+$, then:
    \begin{itemize}
        \item $h(\cJ(g)^+)\cap \cJ(g)^+\neq \varnothing$ if and only if $\cJ(g)$ is the invariant fence of $h$;
    \end{itemize}
    \item if $\Fix(g)\cap Z^+=\{p\}$, then:
    \begin{enumerate}
        \item either $h(q)\notin \cJ(g)$ or $h(q)=q$;
        \item \label{Itm:coincide} $h(p)=p$ or $h(q)=q$ if and only if $\cJ(g)$ is the invariant fence of $h$;
        \item $h(p)\notin \cD(p)$ if and only if $h(q)\notin \cD(p)$;
        \item\label{Itm:innerTouch} $h(p)\in \cD(p)$ and $h(\cJ(g)^+)\cap \cJ(g)^+\neq \varnothing$ if and only if $h(p)\neq p$, $h(\cJ(g)^+)\cap \cJ(g)^+=\{p\}$, and $h(\cJ(g)^+)\setminus\{p\}\subset \cD(p)$. 
        (see \reffig{plusIntersection})
    \end{enumerate}
\end{itemize}
\end{lem}

\begin{proof}
Fix $h\in G\setminus\{1\}$.

Assume first that $\Fix(g)\subset Z^+$.
If $\cJ(g)^+\cap h(\cJ(g)^+)=\varnothing$, then $\cJ(g)$ and $h(\cJ(g))$ are unlinked by \refprop{unlinking}.
Suppose $\cJ(g)^+\cap h(\cJ(g)^+)\neq\varnothing$.
By \refrmk{inaccessibleSide}, both $\cJ(g)^+$ and $h(\cJ(g)^+)$ are simplicial, and by \reflem{crossAlt}, all boundary points  are branched.
Hence either $\cJ(g)^+=h(\cJ(g)^+)$ or $\cJ(g)^+\cap h(\cJ(g)^+)=\{x\}$ for some $x\in \{p,q\}$.

The latter is impossible, since then $g$ and $hgh^{-1}$ would share exactly one fixed point $x$, contradicting the discreteness of $G<\PSL(\CC)$.
Thus $\cJ(g)^+=h(\cJ(g)^+)$.

If $h$ interchanged $p$ and $q$, then $h$ would have a fixed point in $\Int \cJ(g)^+$, so $h^2$ would have at least three fixed points, a contradiction.
Hence $h$ fixes both endpoints of $\cJ(g)^+$, and therefore $\Stab{G}{\cJ(g)^+}=\Stab{G}{\partial\cJ(g)^+}$.
Let $k$ generate this stabilizer.
Then $\Fix(k)=\Fix(g)=\partial\cJ(g)^+$, and $\cJ(g)^+$ is also the $(+)$-side of $\cJ(k)$.
Since $\cJ(k)=\cJ(k^m)$ for every $m\neq 0$ and $g=k^m$ for some $m\neq 0$, we get $\cJ(k)=\cJ(g)$.
Therefore $\cJ(h)=\cJ(g)$.
This proves the first two assertions in this case.

Now assume that \(\Fix(g)\cap Z^+=\{p\}\).
Before proving that \(\cJ(g)\) and \(h(\cJ(g))\) are unlinked, we first prove the third assertion.

For (1), note that \(h(q)\) is a type~1 synapse and that either \(h(q)\notin \cJ(g)\) or \(h(q)=q\). In the latter case, \(h(p)=p\) as well. Otherwise, \(g\) and \(hgh^{-1}\) have exactly one common fixed point, contradicting the discreteness of \(G<\PSL(\CC)\). This proves~(1).

The ``if'' statement of (2) is clear, so we prove the ``only if'' statement. By the discreteness of \(G<\PSL(\CC)\), we have \(h(p)=p\) if and only if \(h(q)=q\), as above. Hence, if \(h(p)=p\) or \(h(q)=q\), then \(h\in \Stab{G}{\Fix(g)}\). This is an infinite cyclic subgroup generated by an element \(k\). Moreover, \(\cJ(g)\) is the invariant fence of \(k\).
This proves (2).

For (3), assume that \(h(p)\in \cD(p)\) and \(h(q)\in \cD(p)^c\), or vice versa. Let \(n\) be the node of \(h(\cJ(g))\) in \(\cD(p)\), and let \(m\) be the other node. Write
\(
\cA=h(\Int \cJ(g)^+),
\)
so that \(\cA\) lands at both \(m\) and \(n\).

If \(m\in \{p,q\}\), then \(\{n,m\}=\{h(q),h(p)\}\) by (1) and (2), a contradiction. Hence either \(m\in \cD(p)^c\) or \(m\in \Int \cJ(g)^+\).

On the other hand, \(p\in \cA\). Indeed, suppose not. Then, by \reflem{simplicialSide}, there is a branch \(B_n\) at \(p\) containing a sufficiently small end ray of \(\cA\) landing at \(n\). If \(p\notin \cA\), then \(\cA\subset B_n\subset \cD(p)\). It follows that \(m\in \Int \cJ(g)^+\) and that \(\cA\) lands at \(m\) on the \(\cD(p)\)-side relative to \(\cJ(g)^+\). This contradicts the fact that \(\cJ(g)^+\) is inaccessible from the \(\cD(p)\)-side. Hence \(p\in \cA\).

Therefore the two line components \(\ell_n\) and \(\ell_m\) of \(\cA\setminus \{p\}\), landing at \(n\) and \(m\), respectively, lie in distinct branches \(B_n\) and \(B_m\) at \(p\). By \reflem{simplicialSide}, we have \(B_n\subset \cD(p)\) and \(B_m\subset \cD(p)^c\). By \refprop{invFence}, \(g^{-1}(\ell_n)\) and \(g(\ell_n)\) land at \(p\) on opposite sides of \(\cA\). This contradicts the one-sided inaccessibility.
This proves $(3)$.

We next prove (4). The ``if'' direction is clear, so it suffices to prove the ``only if'' direction. Since \(h(p)\in \cD(p)\), part~(3) implies that \(h(q)\in \cD(p)\) as well. As above, \reflem{simplicialSide} implies that there are two branches \(B_p\) and \(B_q\) at \(p\), containing sufficiently small end rays of \(h(\Int \cJ(g)^+)\), and contained in \(\cD(p)\). Hence \(h(\Int \cJ(g)^+)\subset B_p\cup B_q\cup \{p\}\). This proves the desired statement.

Finally, we prove that \(\cJ(g)\) and \(h(\cJ(g))\) are unlinked. By (1), either \(h(q)\notin \cJ(g)\) or \(h(q)=q\). In the latter case, the claim follows from (2). Thus assume that \(h(q)\notin \cJ(g)\).

Suppose first that \(h(q)\in \cD(p)\). Then (3) and (4) imply that \(h(p)\in \cD(p)\) and that \(h(\cJ(g)^-)\) contains disjoint end rays landing at \(h(p)\) and \(h(q)\), both contained in \(\cD(p)\). Then the arc joining the starting points of these end rays is contained in \(\closure{\cD(p)}\). This proves that \(\cJ(g)\) and \(h(\cJ(g))\) are unlinked.

Now suppose that \(h(q)\) lies in the opposite Jordan domain \(\cD'=S_\infty^2\setminus \closure{\cD(p)}\) of \(\cJ(g)\). By (2) and (3), either \(h(p)\in \cD'\) or \(h(p)\in \Int \cJ(g)^+\). In the former case, we are done by the previous argument. In the latter case, \refprop{enclosing} implies that \(h(\cJ(g)^+)\subset \closure{\cD'}\), and the \(\cD(p)\)-side inaccessibility of \(\cJ(g)^+\) implies that \(h(\cJ(g)^-)\) lands at \(h(p)\) on the \(\cD'\)-side. We may therefore apply the same argument again, taking sufficiently small distinct end rays of \(h( \cJ(g)^-)\) contained in \(\cD'\). Thus \(\cJ(g)\) and \(h(\cJ(g))\) are unlinked.
\qedhere

\end{proof}

\subsection{Eclipses as a Grid}

Eclipses play an important role as a grid, restricting the possible locations of invariant fences and axes:

\begin{lem}[Grid structure of eclipses]\label{Lem:separatingFence}
Assume the conditions and notation of \refconst{eclipseD}. Let \(h\in G\) satisfy \(h(p)\in \cD\). Then there exists \(k\in \ZZ\) such that \(h(\cJ(g))\subset \closure{\cE_k\setminus \cE_{k+3}}\).
\end{lem}
\begin{proof}

Write \(\cA=h(\Int \cJ(g)^+)\), which is a line in \(Z^+\) landing at \(h(p)\) and \(h(q)\). Since \(h(p)\in \cD\) and \(h(p)\neq p\), \reflem{fenceSystem} implies that \(h(q)\in \cD\) as well and that \(\cA\cap \cJ(g)\subset \{p\}\). By \reflem{simplicialSide}, there are branches \(B_p\) and \(B_q\) at \(p\) containing sufficiently small end rays of \(\cA\) landing at \(h(p)\) and \(h(q)\), respectively, and contained in \(\cD\). If \(B_p=B_q\), then \(\cA\subset B_p\subset \cD\). If \(B_p\neq B_q\), then \(p\in \cA\) and \(\cA\setminus \{p\}\subset \cD\). Moreover, since \(\cA\) is one-sided inaccessible and \(p\) is a branch point, there is no branch \(B\) at \(p\) contained in \(h(\cD)\); here \(h(\cD)\subset \cD\) by the unlinkedness of \(\cJ(g)\) and \(h(\cJ(g))\).

In either case, \refprop{eclipseDProperty} yields \(k\in \ZZ\) such that \(\fD=\Int(\cE_k\setminus \cE_{k+3})\) contains \(B_p\) and \(B_q\). Hence \(h(p),h(q)\in \fD\). Take sufficiently small distinct end rays \(r_1,r_2\) of \(h(\cJ(g)^-)\), landing at \(h(p)\) and \(h(q)\), respectively, and contained in \(\fD\). Since \(\partial \fD\) is a fence whose nodes lie in \(Z^+\), the arc in \(Z^-\) joining the starting points of \(r_1\) and \(r_2\) is contained in \(\closure{\fD}\). Therefore \(h(\cJ(g)^-)\subset \closure{\fD}\). This proves the statement.\qedhere

\end{proof}

\reflem{separatingFence}, together with the separation property (\refprop{eclipseDProperty}), is useful for controlling limits of convergence sequences on invariant fence systems.
\begin{lem}[Dragging limit points]\label{Lem:dragging}
Assume the conditions in \refconst{eclipseD}. Let \(\{g_n\}_{n\in \NN}\) be a sequence in \(G\), and for each \(n\in \NN\), let \(x_n,y_n\) be distinct non-node points of \(g_n(\cJ(g))\). Assume that \(\{x_n\}_{n\in \NN}\) is contained in \(\cD\) and converges to a point \(x\in \cE_\infty\setminus \Fix(g)\). If \(\{y_n\}_{n\in \NN}\) converges to a point \(y\), then \(y\in \cE_\infty\).
\end{lem}

\begin{proof}
For each \(n\in \NN\), \reflem{separatingFence} yields \(k(n)\in \ZZ\) such that \(g_n(\cJ(g))\subset \closure{\cE_{k(n)}\setminus \cE_{k(n)+3}}\). Since \(\closure{\cE_{k(n)}\setminus \cE_{k(n)+3}}\cap \cE_\infty=\{p\}\) by \refprop{eclipseDProperty}, while \(x_n\to x\in \cE_\infty\setminus \Fix(g)\), it follows that \(k(n)\to +\infty\) as \(n\to\infty\). Since \(y_n\in g_n(\cJ(g))\subset \closure{\cE_{k(n)}\setminus \cE_{k(n)+3}}\subset \cE_{k(n)}\) and \(y_n\to y\), we conclude that \(y\in \cE_\infty\).
\end{proof}

\subsection{Circle laminations on quasi-Fuchsian limit circles}
As sketched in \refsec{proofOutline}, we now construct a circle lamination on the limit circle of a given quasi-Fuchsian surface subgroup.
\begin{figure}[htbp]
\centering
\tikzset{every picture/.style={line width=0.75pt}}

\begin{tikzpicture}[x=0.75pt,y=0.75pt,yscale=-1,xscale=1]

%Shape: Polygon Curved [id:ds33540346082777717] 
\draw   (204.86,72.79) .. controls (228.86,63.79) and (231.86,89.79) .. (239.86,111.79) .. controls (247.86,133.79) and (229.86,142.79) .. (231.86,161.79) .. controls (233.86,180.79) and (175.86,214.79) .. (161.86,191.79) .. controls (147.86,168.79) and (130.86,189.79) .. (115.86,156.79) .. controls (100.86,123.79) and (132.86,103.79) .. (140.86,84.79) .. controls (148.86,65.79) and (180.86,81.79) .. (204.86,72.79) -- cycle ;
%Shape: Circle [id:dp9298659860244245] 
\draw  [color={rgb, 255:red, 74; green, 144; blue, 226 },draw opacity=1][fill={rgb, 255:red, 208; green, 2; blue, 27 },fill opacity=1][line width=0.75]  (172.86,129.27) .. controls (172.86,127.77) and (174.07,126.56) .. (175.57,126.56) .. controls (177.06,126.56) and (178.27,127.77) .. (178.27,129.27) .. controls (178.27,130.76) and (177.06,131.98) .. (175.57,131.98) .. controls (174.07,131.98) and (172.86,130.76) .. (172.86,129.27) -- cycle ;
%Shape: Circle [id:dp8052617187167587] 
\draw  [color={rgb, 255:red, 74; green, 144; blue, 226 },draw opacity=1][fill={rgb, 255:red, 208; green, 2; blue, 27 },fill opacity=1][line width=0.75]  (130.86,230.27) .. controls (130.86,228.77) and (132.07,227.56) .. (133.57,227.56) .. controls (135.06,227.56) and (136.27,228.77) .. (136.27,230.27) .. controls (136.27,231.76) and (135.06,232.98) .. (133.57,232.98) .. controls (132.07,232.98) and (130.86,231.76) .. (130.86,230.27) -- cycle ;
%Curve Lines [id:da07118106296245041] 
\draw [color={rgb, 255:red, 208; green, 2; blue, 27 },draw opacity=1]   (133.86,228.12) .. controls (135.86,208.45) and (161.86,199.79) .. (163.86,184.79) .. controls (165.86,169.79) and (132.86,203.79) .. (137.86,179.79) .. controls (142.86,155.79) and (137.86,161.79) .. (172.86,131.79) ;
%Shape: Circle [id:dp048500853001922306] 
\draw  [color={rgb, 255:red, 208; green, 2; blue, 27 },draw opacity=1][fill={rgb, 255:red, 208; green, 2; blue, 27 },fill opacity=1][line width=0.75]  (135.86,177.27) .. controls (135.86,175.77) and (137.07,174.56) .. (138.57,174.56) .. controls (140.06,174.56) and (141.27,175.77) .. (141.27,177.27) .. controls (141.27,178.76) and (140.06,179.98) .. (138.57,179.98) .. controls (137.07,179.98) and (135.86,178.76) .. (135.86,177.27) -- cycle ;
%Curve Lines [id:da30029095944271356] 
\draw [color={rgb, 255:red, 74; green, 144; blue, 226 },draw opacity=1]   (178,128) .. controls (218,98) and (197.86,149.79) .. (237.86,119.79) .. controls (277.86,89.79) and (185.86,112.79) .. (210.86,87.79) .. controls (235.86,62.79) and (239.86,109.79) .. (264.86,84.79) .. controls (289.86,59.79) and (289.86,144.12) .. (268.86,154.12) .. controls (247.86,164.12) and (214.86,182.12) .. (232.86,199.12) .. controls (250.86,216.12) and (156.86,259.12) .. (134.86,233.12) ;
%Shape: Circle [id:dp7159630722412608] 
\draw  [color={rgb, 255:red, 74; green, 144; blue, 226 },draw opacity=1][fill={rgb, 255:red, 74; green, 144; blue, 226 },fill opacity=1][line width=0.75]  (237.86,117.27) .. controls (237.86,115.77) and (239.07,114.56) .. (240.57,114.56) .. controls (242.06,114.56) and (243.27,115.77) .. (243.27,117.27) .. controls (243.27,118.76) and (242.06,119.98) .. (240.57,119.98) .. controls (239.07,119.98) and (237.86,118.76) .. (237.86,117.27) -- cycle ;

% Text Node
\draw (120.86,226.67) node [anchor=north west][inner sep=0.75pt]  [color={rgb, 255:red, 208; green, 2; blue, 27 },opacity=1]  {$q$};
% Text Node
\draw (161.86,114.67) node [anchor=north west][inner sep=0.75pt]  [color={rgb, 255:red, 208; green, 2; blue, 27 },opacity=1]  {$p$};
% Text Node
\draw (85,98.4) node [anchor=north west][inner sep=0.75pt]    {$\Lambda(\Gamma)$};
% Text Node
\draw (118,177.67) node [anchor=north west][inner sep=0.75pt]  [color={rgb, 255:red, 208; green, 2; blue, 27 },opacity=1]  {$x^{+}$};
% Text Node
\draw (275,150.4) node [anchor=north west][inner sep=0.75pt]  [color={rgb, 255:red, 74; green, 144; blue, 226 },opacity=1]  {$\ell$};
% Text Node
\draw (243.86,111.67) node [anchor=north west][inner sep=0.75pt]  [color={rgb, 255:red, 74; green, 144; blue, 226 },opacity=1]  {$x^{-}$};
% Text Node
\draw (185,130) node [anchor=north west][inner sep=0.75pt]    {$\gamma$};
\end{tikzpicture}

\caption{The connector arc $\gamma$ crossing $\cH$, joining $x^{+}\in \cJ(g)^{+}$ to $x^{-}\in \cJ(g)^{-}$.}
\label{Fig:genConnector}
\end{figure}
\begin{const}[Circle laminations on quasi-Fuchsian limit circles]\label{Const:circleLami}
Let \(Z^\pm\) be a minimal zipper for \(M\), and let \(g\in G\) satisfy \(\Fix(g)\subset Z^+\). Choose \(p\in \Fix(g)\) and write \(\Fix(g)=\{p,q\}\).

By \refthm{separatingQF}, choose a quasi-Fuchsian closed surface subgroup \(\Gamma<G\) such that \(\Lambda(\Gamma)\) separates the fixed points of \(g\). Let \(\cH\) be the Jordan domain of \(\Lambda(\Gamma)\) containing \(p\), and equip \(\Lambda(\Gamma)\) with the canonical circular order induced from \(\cH\). Then the unique connector subarc \(\gamma\) of \(\cJ(g)\) with synapse \(p\) that crosses \(\cH\) is called the \emph{generating arc} associated with \(g\), \(p\), and \(\Gamma\).
See \reffig{genConnector}.

We define \(\cL(\Gamma,p)\) to be the closure of the \(\Gamma\)-orbit of \(\sfd \gamma\) in \(\cM\).
\end{const}
\begin{prop}
In the setting of \refconst{circleLami}, $\sfd \gamma$ is unlinked with its translation under the $\Gamma$-action.  Consequently, \(\cL(\Gamma,p)\)  is a \(\Gamma\)-invariant circle lamination of \(\Lambda(\Gamma)\).
\end{prop}
\begin{proof}
    Assume that \(\sfd \gamma\) is linked with \(h(\sfd \gamma)\) for some \(h\in \Gamma\). Since \(\Fix(h)\subset \Lambda(\Gamma)\), we have \(h(p)\neq p\). Moreover, \reflem{fenceSystem} implies that \(\gamma^\alpha\) is unlinked with \(h(\gamma^\alpha)\) for each \(\alpha\in \{+,-\}\). Since \(p,h(p)\in Z^+\), \refprop{unlinkedCrossingLine} implies that at least one of \(\gamma^+\) and \(h(\gamma^+)\) is not inaccessible. This contradicts the inaccessibility of \(\cJ(g)^+\) and \(h(\cJ(g))^+\); see \refrmk{inaccessibleSide}.

    The second assertion follows immediately from the first and the construction. This completes the proof.
\end{proof}

\begin{rmk}[Geometric realizations of circle laminations]\label{Rmk:realization}
Let \(\Gamma<G\) be a quasi-Fuchsian closed surface subgroup, let \(\cH\) be a Jordan domain of \(\Lambda(\Gamma)\), and let \(\cL\) be a \(\Gamma\)-invariant circle lamination on \(\Lambda(\Gamma)\).
Since \(\cH\) is a domain of discontinuity for \(\Gamma\), the quotient \(\cH/\Gamma\) is a closed surface of genus \(>1\).
After fixing a hyperbolic metric on \(\cH/\Gamma\) and lifting it to \(\cH\), we may form the geometric realisation \(\widetilde{\cG}(\cL)\) of \(\cL\) (see \refsec{circleLami}).
By \(\Gamma\)-invariance, this descends to a geodesic lamination \(\cG(\cL)\) on \(\cH/\Gamma\).
\end{rmk}

\begin{lem}\label{Lem:isolatedLeaf}
    In the setting of \refconst{circleLami},
    $\sfd \gamma$ is an isolated leaf of $\cL(\Gamma,p)$.
\end{lem}
\begin{proof}
Assume not. Then there is a sequence \(\{g_n\}_{n\in \NN}\) in \(\Gamma\) such that \(g_n(\sfd \gamma)\to \sfd \gamma\) in \(\cL(\Gamma,p)\) and \(g_n(\sfd \gamma)\neq \sfd \gamma\) for all \(n\). Since \(\Fix(g_n)\subset \Lambda(\Gamma)\), we have \(g_n(p)\neq p\). Hence \reflem{fenceSystem} implies that \(g_n(p)\notin \cJ(g)\). Therefore there is a Jordan domain \(\cD\) such that \(g_n(p)\in \cD\) for infinitely many \(n\). Passing to a subsequence, we may assume that \(g_n(p)\in \cD\) for all \(n\). By \refrmk{inaccessibleSide}, the domain \(\cD\) is of type~I. Thus \refconst{eclipseD} yields an eclipse \(\{\cE_n\}_{n\in \ZZ}\) on \(\cD\).

Write \(\sfd \gamma=\{x^+,x^-\}\) with \(x^\alpha\in Z^\alpha\). Since \(x^+\in \Int \cJ(g)^+\subset \cE_\infty\setminus \Fix(g)\) and \(\Fix(g)\cap \Lambda(\Gamma)=\varnothing\), \reflem{dragging} implies that \(x^-\in \cE_\infty\setminus \Fix(g)\). On the other hand, \(\cE_\infty\cap \cJ(g)^-=\varnothing\) by \refprop{eclipseDProperty}, whereas \(x^-\in \cJ(g)^-\). This is a contradiction.
\end{proof}

\subsection{Two fixed points in $Z^+$}
Now, we prove the main theorem in this section:
\begin{thm}\label{Thm:twoInOne}
Let $Z^\pm$ be a minimal zipper for $M$. Then no element \(g\in G\) has its fixed point set contained entirely in \(Z^+\) or entirely in \(Z^-\).
\end{thm}

\begin{proof}
By symmetry, it suffices to show that there is no nontrivial element $g\in G$ such that $\Fix(g)\subset Z^+$.
Assume otherwise, and write $\Fix(g)=\{p,q\}$.
Recall that $p$ and $q$ are pivots of distinct Jordan domains of $\cJ(g)$ (\refrmk{inaccessibleSide}).
By \refconst{circleLami}, we construct a circle lamination $\cL(\Gamma,p)$ on $\Lambda(\Gamma)$ associated with $g$, $p$, and a quasi-Fuchsian closed surface subgroup $\Gamma$, and by \refconst{eclipseD}, we construct eclipse sequences $\{\cE_n\}_{n\in \ZZ}$ on $\cD(p)$ and $\{\cE'_n\}_{n\in \ZZ}$ on $\cD(q)$.
From now on, we use the notation from \refconst{circleLami}.

By \refrmk{realization}, \(\cL(\Gamma,p)\) gives rise to a geodesic lamination \(\cG(\Gamma,p)\) on \(\cH/\Gamma\), with lift \(\widetilde{\cG}(\Gamma,p)\subset \cH\).
Let $\widetilde{\sfg}$ be the leaf of $\widetilde{\cG}(\Gamma,p)$ associated with $\sfd\gamma$, and let $\sfg$ be its projection to $\cG(\Gamma,p)$.
By \reflem{isolatedLeaf}, $\sfd\gamma$ is isolated in $\cL(\Gamma,p)$, so $\widetilde{\sfg}$ and $\sfg$ are isolated from both sides in $\widetilde{\cG}(\Gamma,p)$ and $\cG(\Gamma,p)$, respectively.
Moreover, by construction, $\cG(\Gamma,p)$ is the closure of $\sfg$ in $\cH/\Gamma$.
Hence, if $\cG(\Gamma,p)''\ne \varnothing$, then \reflem{closureOfLeaf} implies that some end of $\sfg$ spirals toward a component of $\cG(\Gamma,p)''$.

Write $\sfd\gamma=\{x^+,x^-\}$ with $x^\alpha\in Z^\alpha$.
We claim that $x^+$ cannot be fixed by a nontrivial element of $\Gamma$.
This implies that $\cG(\Gamma,p)''\ne \varnothing$, since any half-ray of $\widetilde{\sfg}$ tending to $x^+$ projects to a ray of $\sfg$ spiraling toward a component of $\cG(\Gamma,p)''$, which is not a simple closed curve.
Indeed, if some $h\in \Gamma$ fixes $x^+$, then
\(
x^+\in h(\cJ(g)^+)\cap \cJ(g)^+\neq\varnothing.
\)
Hence \reflem{fenceSystem} gives $\cJ(g)=\cJ(h)$ and $\Fix(h)=\Fix(g)$, contrary to $x^+\notin \Fix(g)$.

By \reflem{closureOfLeaf}, the end of $\sfg$ corresponding to $x^+$ escapes toward an end of a principal region $U$ of $\cG(\Gamma,p)''$.
Hence, by \refprop{gapShape}, one of the following holds:
\begin{enumerate}
    \item\label{Itm:idealPolygon}($U$ is a finite ideal polygon) there are circularly ordered finitely many points in $\Lambda(\Gamma)$,
    \[
    x_{-m}<x_{-(m-1)}<\cdots<x_0<\cdots<x_{n-1}<x_n
    \]
    for some $m,n\ge 2$, such that the pairs $\{x_i,x_{i+1}\}\in \cL(\Gamma,p)$ are boundary leaves but are not isolated, $x_0=x^+$, and $x_{-m}=x_n=x^-$;

    \item\label{Itm:crown}($U$ admits a core) there exist an element $h\in \Gamma$, an integer $k\in \NN$, and a circularly ordered bi-infinite sequence $\{x_n\}_{n\in \ZZ}$ in $\Lambda(\Gamma)$ such that
    \[
    r(h)<\cdots<x_{-n}<\cdots<x_0<\cdots<x_n<\cdots<a(h),
    \]
    the pairs $\{x_i,x_{i+1}\}\in \cL(\Gamma,p)$ are boundary leaves but are not isolated, $x_0=x^+$, $x^-\in \opi[\Lambda(\Gamma)]{x_1}{x_{-1}}$, and $h(x_n)=x_{n+k}$ for all $n\in \ZZ$.
\end{enumerate}

First assume Case-\refitm{idealPolygon}.
By construction of $\cL(\Gamma,p)$, there is a sequence $\{g_n\}_{n\in \NN}$ in $\Gamma$ such that $g_n(\sfd\gamma)\to \{x_0,x_1\}$ in $\cL(\Gamma,p)$ as $n\to\infty$.
After passing to a subsequence, \reflem{fenceSystem} implies that either $g_n(p)\in \cD(p)$ for all $n$ or $g_n(p)\in \cD(q)$ for all $n$.
Moreover, $g_n(x^+)$ lies in the same Jordan domain as $g_n(p)$.
Since $x_0=x^+$ lies in $\Int \cJ(g)^+$, which is contained in both $\cE_\infty$ and $\cE'_\infty$ by \refprop{eclipseDProperty}, \reflem{dragging}, applied to $\{g_n(\sfd\gamma)\}$, shows that $x_1$ lies in $\cE_\infty\setminus \Fix(g)$ or in $\cE'_\infty\setminus \Fix(g)$, respectively.
Applying the same argument to $\{x_1,x_2\}$, we see that $x_2$ also lies in $\cE_\infty\setminus \Fix(g)$ or in $\cE'_\infty\setminus \Fix(g)$.
Continuing inductively, we obtain
\[
x_n\in (\cE_\infty\cup \cE'_\infty)\setminus \Fix(g).
\]
However, $x_n$ lies on the $(-)$-side of $\cJ(g)$, contradicting \refprop{eclipseDProperty}.
Thus Case-\refitm{idealPolygon} does not occur.

Now consider Case-\refitm{crown}.
There are two subcases: (a) $x^-=x_i$ for some $i\notin\{-1,0,1\}$; (b) $x^-\in \cldi{a(h)}{r(h)}$.
In Case-(a), there is a finite sequence of leaves $\{x_j,x_{j+1}\}$ connecting $x_0$ to $x_i$.
As above, repeated application of \reflem{dragging} rules this out.
Thus only Case-(b) remains.

Observe that $h(\sfd\gamma)$ is also a leaf of $\cL(\Gamma,p)$ and that $h(x^+)=h(x_0)=x_k$.
Applying the same argument to $\{x_j,x_{j+1}\}$ for $j=0,\dots,k-1$, we obtain
\[
x_k\in (\cE_\infty\cup \cE'_\infty)\setminus \Fix(g).
\]
On the other hand, by \reflem{fenceSystem} and \reflem{separatingFence},
\[
h(\cJ(g))\subset \closure{\cE_s\setminus \cE_{s+3}}
\quad\text{or}\quad
h(\cJ(g))\subset \closure{\cE'_t\setminus \cE'_{t+3}}
\]
for some $s,t\in \ZZ$.
By \refprop{eclipseDProperty}, this contradicts $x_k\in (\cE_\infty\cup \cE'_\infty)\setminus \Fix(g)$.
This completes the proof.
\end{proof}

\section{One prong in $Z^+$}\label{Sec:oneprong}
To complete the proof of \refthm{tameFix}, it remains, up to exchanging $Z^+$ and $Z^-$, to show:

\begin{thm}\label{Thm:oneInOne}
Let $Z^\pm$ be a minimal zipper for $M$. Then there is no element $g\in G$ such that $g$ fixes a unique point in $Z^+$ and acts freely on $Z^-$, with axis $\ell\subset Z^-$.
\end{thm}

As in \refrmk{inaccessibleSide}, \refconst{eclipseD}, and \refconst{eclipseL}, we divide the proof into three cases:
\begin{enumerate}
    \item $g$ admits a $g$-prong in $Z^+$;
    \item $g$ admits no $g$-prong in $Z^+$ and $\ell$ is two-sided branched;
    \item $g$ admits no $g$-prong in $Z^+$ and $\ell$ is one-sided simplicial.
\end{enumerate}
The next three sections treat these cases in order.

In each case, the argument follows the general pattern of \refsec{twoInOne}, but requires additional care.
The three cases are independent and may be read separately.

In this section, we consider the first case.
\begin{thm}\label{Thm:oneProng}
Let $Z^\pm$ be a minimal zipper for $M$. Then there is no element $g\in G$ such that $g$ admits a $g$-prong in $Z^+$ and acts freely on $Z^-$.
\end{thm}

By \refrmk{inaccessibleSide}, the $(+)$-side $\cJ(g)^+$ is one-sided inaccessible, while by \reflem{simplicialSide}, $\cJ(g)$ has a unique Jordan domain, of type~I. For this reason, the dragging argument in \reflem{dragging} must be refined to record the direction of accumulation of translates of $\cJ(g)$. This will be done in \refsec{uniformBouncing}.

To this end, we first modify \refconst{circleLami} as follows:

\begin{const}[Circle laminations on quasi-Fuchsian limit circles]\label{Const:lamiByManyLeaves}
Let \(Z^\pm\) be a minimal zipper for \(M\), and let \(g\in G\) be such that \(g\) admits a \(g\)-prong in \(Z^+\), \(\Fix(g)\cap Z^+=\{p\}\), and \(\Fix(g)\cap Z^-=\varnothing\).
Write \(\Fix(g)=\{p,q\}\).
Assume that \(\cJ(g)\) is equipped with a circular order compatible with the canonical order on the ray \(\cJ(g)^+\).

By \refthm{separatingQF}, choose a quasi-Fuchsian closed surface subgroup \(\Gamma<G\) such that \(\Lambda(\Gamma)\) separates the fixed points of \(g\). Let \(\cH\) be the Jordan domain of \(\Lambda(\Gamma)\) containing \(p\), and equip \(\Lambda(\Gamma)\) with the canonical circular order induced by \(\cH\).
Then there is a unique connector subarc \(\gamma_1\) of \(\cJ(g)\) with synapse \(p\) that crosses \(\Lambda(\Gamma)\) through \(\cH\).
Write \(\sfd\gamma_1=\{x_0,x_1\}\), where \(x_0\in Z^-\) and \(x_1\in Z^+\). We define
\[
\cL(\Gamma, \{x_0,x_1\}):=\overline{\Gamma\cdot\sfd\gamma_1}\subset\cM.
\]

More generally, let \(x_1<\cdots<x_n\) be a finite strictly increasing sequence in \(\cJ(g)^+\), where \(n\ge2\), satisfying the following \emph{bouncing condition}: for each \(k\in\{2,\dots,n\}\), the arc
\[
\gamma_k=[x_{k-1},x_{k}]
\]
in \(\cJ(g)^+\) crosses \(\cH\).
For \(0< k\le n\), write \(\ell_k=\sfd\gamma_k\), where \(\gamma_1\) is as above. We then define
\[
\cL(\Gamma,\{x_k\}_{k=0}^n)
:=
\overline{\{h(\ell_k): h\in\Gamma,\ 1\le k\le n\}}
\subset\cM.
\]

Thus \(\cL(\Gamma,\{x_k\}_{k=0}^n)\) should be viewed as an extension of the basic construction \(\cL(\Gamma,\{x_0,x_1\})\), obtained by allowing additional generating leaves arising from a bouncing sequence.
We call \(\{\ell_k\}_{k=1}^{n}\) the \emph{generating leaves} of \(\cL(\Gamma,\{x_k\}_{k=0}^n)\), and we say that \(\cL(\Gamma,\{x_k\}_{k=0}^n)\) is \emph{generated} by \(\{\ell_k\}_{k=1}^{n}\), or equivalently by \(\{\gamma_k\}_{k=1}^{n}\).
\end{const}

\begin{prop}
In the setting of \refconst{lamiByManyLeaves}, every two leaves in the \(\Gamma\)-orbit of \(\{\ell_k\}_{k=1}^{n}\) are unlinked. Consequently, \(\cL(\Gamma,\{x_k\}_{k=0}^n)\) is a \(\Gamma\)-invariant circle lamination of \(\Lambda(\Gamma)\).
\end{prop}
\begin{proof}
Assume that \(h(\ell_j)\) is linked with \(\ell_i\) for some \(h\in\Gamma\) and some \(i,j\in\{1,\dots,n\}\) with \(j\le i\).
Note that \(h(p)\neq p\), since \(\Fix(h)\subset \Lambda(\Gamma)\).

If \(1<j\) or \(j<i\), then \(\gamma_i\) is an arc in \(\Int \cJ(g)^+\subset Z^+\).
In this case, \refprop{unlinkedCrossingLine} implies that \(\gamma_i\) is two-sided accessible, a contradiction since \(\gamma_i\subset \cJ(g)^+\) is at least one-sided inaccessible by \refrmk{inaccessibleSide}.

It remains to consider the case \((i,j)=(1,1)\).
Then both \(h(\gamma_1)\) and \(\gamma_1\) are connector arcs with synapses in \(Z^+\).
Again by \refprop{unlinkedCrossingLine}, at least one of \(h(\gamma_1)^+\) and \(\gamma_1^+\) is accessible from the \(\cD(p)\)-side.
As above, this contradicts \refrmk{inaccessibleSide}.

The second assertion is immediate from the first and from the construction. This completes the proof.
\end{proof}

\subsection{Trivial stabilizers of generating leaves}
In the setting of \refconst{lamiByManyLeaves}, by \refrmk{realization}, \(\cL(\Gamma,\{x_k\}_{k=0}^n)\) induces a geodesic lamination \(\cG(\Gamma,\{x_k\}_{k=0}^n)\) on \(\cH/\Gamma\), with lift \(\widetilde{\cG}(\Gamma,\{x_k\}_{k=0}^n)\subset \cH\).
However, unlike in \reflem{isolatedLeaf}, the endpoints of the generating leaf \(\ell_1\) may be fixed by a non-trivial element of \(\Gamma\).
Thus the leaf of \(\cG(\Gamma,\{x_k\}_{k=0}^n)\) corresponding to \(\ell_1\) may be a simple closed curve.
In that case, by \reflem{closureOfLeaf}, \(\cG(\Gamma,\ell_1)''=\varnothing\), which prevents us from following the proof of \refthm{twoInOne} in the present setting.

We therefore turn to the proof that the stabiliser of \(\ell_1\) is trivial.
Note that the triviality of the stabiliser of \(\ell_k\) for \(k>1\) follows automatically from \refthm{twoInOne}.
We begin with an auxiliary lemma.

\begin{lem}[Extension of  \(\gamma_0^+\)] \label{Lem:extArc}
In the setting of \refconst{lamiByManyLeaves}, if \(\cJ(g)^+\) is left (resp. right) inaccessible, then there is an arc \(\cA\subset Z^+\) crossing \(\cH^L(\gamma_1)\) (resp. \(\cH^R(\gamma_1)\)) and joining \(p\) to a point in \(\opi[\Lambda(\Gamma)]{x_1}{x_0}\) (resp. \(\opi[\Lambda(\Gamma)]{x_0}{x_1}\)).
\end{lem}
\begin{proof}
By symmetry, it suffices to treat the left-inaccessible case.
Since \(\gamma_1^+\) is left inaccessible and \(\cD(p)\) is of type~I, we can find a ray \(r\subset Z^+\) landing at \(p\) on \(\cD(p)\) and hence on the left side of \(\gamma_1\).
Since \(q\notin \overline{\cH}\), if the initial point of \(r\) lies in \(\overline{\cH}\), then, after replacing \(r\) by a suitable iterate under \(g\), we may assume that its initial point does not lie in \(\overline{\cH}\).
It follows that \(r\) admits a truncation \(r'\) crossing \(\cH^L(\gamma_1)\), starting at a point in \(\opi[\Lambda(\Gamma)]{x_1}{x_0}\), and landing at  \(p\).
Taking \(\cA=r'\cup \{p\}\) completes the proof.
\end{proof}

\begin{lem}[Fanning arcs]\label{Lem:fanningArcs}
In the setting of \refconst{lamiByManyLeaves}, let \(h\in \Gamma\) and \(k\in\{1,\dots,n\}\).
If \(h\) fixes \(x_k\), then \(h\) does not fix \(x_{k-1}\).
Moreover, \(h(\gamma_k)\) crosses either \(\cH^R(\gamma_k)\) or \(\cH^L(\gamma_k)\).
\end{lem}

\begin{proof}
By symmetry, it suffices to consider the case where \(\gamma_k^+\) is left inaccessible.
If \(k>1\) and \(\ell_k\subset \cJ(g)^+\subset Z^+\), then \refthm{twoInOne} implies that \(h\) cannot fix \(x_{k-1}\).
Hence either \(h(x_{k-1})\in \opi[\Lambda(\Gamma)]{x_k}{x_{k-1}}\) or
\(h(x_{k-1})\in \opi[\Lambda(\Gamma)]{x_{k-1}}{x_k}\).
Replacing \(h\) by \(h^{-1}\) in the latter case, we may assume that
\(h(x_{k-1})\in \opi[\Lambda(\Gamma)]{x_k}{x_{k-1}}\).
Applying \refprop{enclosing} to \(h(\gamma_k)\) with respect to \(\cH^L(\gamma_k)\), the left-inaccessibility of \(\gamma_k^+\) implies that \(h(\gamma_k)\) crosses \(\cH^L(\gamma_k)\).

Assume now that \(k=1\).
We first claim that \(h(p)\notin \gamma_1\).
Suppose otherwise.
Since \(\Fix(h)\subset \Lambda(\Gamma)\), we have \(h(p)\neq p\), and since \(Z^+\cap Z^-=\varnothing\), it follows that \(h(p)\in \Int \gamma_1^+\).
By the unique path connectedness of \(Z^+\), this implies that \(h(\gamma_1^+)\subsetneq \gamma_1^+\).

Now let \(r_1=[x_1,h(p))\) and \(r_2=[p,h(p))\), so these are the two components of \(\gamma_1^+\setminus\{h(p)\}\).
They lie in distinct branches \(B_1\) and \(B_2\) at \(h(p)\).
Moreover, since \(h(\cD(p))\) is of type~I with respect to \(hgh^{-1}\) and \(r_1\subset h(\cJ(g))^+\subset B_1\), \reflem{simplicialSide} implies that \(B_2\subset h(\cD(p))\).
Furthermore, by \refprop{invFence}, the branch \(B_2\) is not preserved by \(hgh^{-1}\).
Hence \(hgh^{-1}(r_2)\) and \(hg^{-1}h^{-1}(r_2)\) land at \(p\) from opposite sides relative to \(\gamma_1^+\).
This contradicts the fact that \(\gamma_1^+\) is at least one-sided inaccessible.
Therefore \(h(p)\notin \gamma_1\).

Hence either $h(p)\in \cH^L(\gamma_1)$ or $h(p)\in \cH^R(\gamma_1)$.
We claim that, in the latter case, one has $p\in \cH^L(h(\gamma_1))$.
Indeed, applying the first claim to $h^{-1}$ shows that $h^{-1}(p)\notin \gamma_1$, or equivalently, $p\notin h(\gamma_1)$.
Now apply \refprop{enclosing} to $h(\gamma_1)^+$ with respect to $\cH^R(\gamma_1)$.
Since $h(\gamma_1)^+$ is left inaccessible and $p\notin h(\gamma_1)$, it follows that
\(
\Int h(\gamma_1)^+ \subset \cH^R(\gamma_1).
\)
Moreover, $\gamma_1^-$ and $h(\gamma_1)^-$ are unlinked by \reflem{fenceSystem}.
Hence $\closure{\cH^R(\gamma_1)}$ contains $h(\gamma_1)$.
Since $p\notin h(\gamma_1)$, we obtain
\(
p\in \cH^L(h(\sfa_n)).
\)
Therefore, replacing $h$ by $h^{-1}$ in the latter case, we may assume that
\(h(p)\in \cH^L(\gamma_1).\)

Applying \refprop{enclosing} to \(h(\gamma_1)^+\) with respect to \(\cH^L(\gamma_1)\), we obtain a point \(v\in \gamma_1^+\) such that the subray
\[
\sfr=[h(p),v)\subset h(\gamma_1)^+
\]
lands at \(v\) from the left side of \(\partial\cH^L(\gamma_1)\).
By the left-inaccessibility of \(\gamma_1^+\), either \(v=x_1\) or \(v=p\).
The latter case reduces to the previous one with \(h^{-1}\), since \(v=p\in h(\gamma_1^+)\), equivalently \(h^{-1}(p)\in \gamma_1^+\).
Hence only the former case can occur.

To complete the proof, it suffices to show that \(h(\gamma_1)^-\) does not meet \(\gamma_1^-\).
Assume otherwise.
Then we can take a ray \(\sft\) in \(h(\gamma_1)^-\cup \gamma_1^-\subset Z^-\) by concatenating an initial arc of \(\gamma_1^-\) with a end ray of \(h(\gamma_1)^-\).
Note that $\sft$ starts at $x_0$ and lands at $h(p)$.

On the other hand, by \reflem{extArc}, we can take an arc \(\cA\) in \(Z^+\) crossing \(\cH^L(\gamma_1)\) and joining \(p\) to a point \(z\in \opi[\Lambda(\Gamma)]{x_1}{x_0}\).
Note that \(\cA'=\cA \cup \gamma_1^+\) is the arc \([z,x_1]\) in \(Z^+\) crossing \(\cH\).
Assume that \(\cA'\) is equipped with a linear order compatible with \(\gamma_1^+\).
Since \(\sfr\) lands at \(x_1\) on \(\cH^L(\gamma_1)\) and \(Z^+\) is simply connected, \(\sfr\subset \cH^L(\cA')\), and so \(h(p)\in \cH^L(\cA')\).
However, since \(x_0\nin \closure{\cH^L(\cA')}\) and $\sft$ runs from $x_0$ to $h(p)$, this implies that \(\sft\subset Z^-\) meets \(\cA'\subset Z^+\), a contradiction.
This completes the proof.
\end{proof}

\subsection{One-sided isolatedness of generating leaves}

Now, we discuss the isolatedness of the generating leaves as in \reflem{isolatedLeaf}:

\begin{prop}[One-sided isolated]\label{Prop:sideIsolated}
In the setting of  \refconst{lamiByManyLeaves},
if \(\cJ(g)^+\) is left (resp.\ right) inaccessible, then \(\opi[\Lambda(\Gamma)]{x_1}{x_0}\) (resp.\ \(\opi[\Lambda(\Gamma)]{x_0}{x_1}\)) is isolated in \(\cL(\Gamma,\{x_k\}_{k=0}^n)\).
\end{prop}
\begin{proof}
    By symmetry, it suffices to treat the left-inaccessible case.
Assume that there is an \(\opi[\Lambda(\Gamma)]{x_1}{x_0}\)-side sequence \(\{\lambda_i\}_{i\in\NN}\) such that
\(
\lambda_i=g_i(\sfd\gamma_{k_i})
\)
for some \(g_i\in \Gamma\) and some \(k_i\in \{1,\dots,n\}\).
If infinitely many \(\lambda_i\) share an endpoint \(e\), then \(e\in \{x_0,x_1\}\).
Applying \refprop{oneEndFix} to \(\wt \cG(\Gamma,\{x_k\}_{k=0}^n)\), we conclude that \(\sfd\gamma_1\) is isolated, contradicting the assumption.
Hence, after passing to a subsequence, we may assume that
\(
\closure{I_i}\subset I_{i+1}\subset \opi[\Lambda(\Gamma)]{x_1}{x_0},
\)
where \(I_i\) is a component of \(\Lambda(\Gamma)\setminus \lambda_i\).

On the other hand, by \reflem{extArc}, we can take a ray \(\cR\subset Z^+\) crossing \(\cH^L(\gamma_1)\), starting at a point \(z\in \opi[\Lambda(\Gamma)]{x_1}{x_0}\), and landing at \(p\) from the left side of \(\gamma_1\).
Note that \(\cR\) is contained in a branch \(B\) at \(p\), and \reflem{simplicialSide} implies that \(B\subset \cD(p)\).
Now consider the arc
\[
\cA=\gamma_1^+\cup \cR,
\]
equipped with the linear order compatible with \(\gamma_1^+\).

There exists \(N\in \NN\) such that \(z\in I_i\) for all \(i>N\).
Hence \(\lambda_i\) is linked with \(\partial \cA\) in \(\Lambda(\Gamma)\) for all \(i>N\).
Fix \(i>N\).
If \(k_i>1\), then \(g_i(\gamma_{k_i})\) is an arc in \(Z^+\), and hence it is linked with \(\cA\).
This is impossible, since \(g_i(\gamma_{k_i})\) is at least one-sided inaccessible.
Therefore \(k_i=1\) for all \(i>N\).

\begin{claim*}
For all \(i>N\), we have \(g_i(p)\in \cR\).
\end{claim*}

\begin{proof}
Fix \(i>N\), and assume that \(g_i(p)\notin \cR\).
Since \(\Fix(g_i)\subset \Lambda(\Gamma)\), we have \(g_i(p)\neq p\), and hence
\(
g_i(p)\notin [z,p]=\cR\cup\{p\}.
\)
Write
\(
I_i=\opi[\Lambda(\Gamma)]{e_1}{e_2},
\)
so that \(x_1<e_1<z<e_2<x_0\) in \(\Lambda(\Gamma)\).

Consider the ray
\(
\sfr=g_i(\gamma_1)^-\subset Z^-,
\)
landing at \(g_i(p)\).
Assume first that \(\partial \sfr=e_1\).
By the disjointness of \(Z^\pm\), either \(g_i(p)\in \Int\gamma_1^+\) or \(g_i(p)\in \cH^L(\cA)\).
In the former case, \(\sfr\) lands at \(g_i(p)\) from the left side of \(\gamma_1^+\), contradicting the left-inaccessibility of \(\gamma_1^+\).
In the latter case, since \(g_i(\gamma_1)^+\) joins \(e_2\) to \(g_i(p)\), the arc \(g_i(\gamma_1)^+\) is linked with \(\cA\).
This again contradicts the fact that \(g_i(\gamma_1)^+\) is at least one-sided inaccessible.

Next assume that \(\partial \sfr=e_2\).
Then either \(g_i(p)\in \Int\gamma_1^+\) or \(g_i(p)\in \cH^R(\cA)\).
In the latter case, as above, \(g_i(\gamma_1)^+\) is linked with \(\cA\), again a contradiction.
In the former case, \(\sfr\) lands at \(g_i(p)\) from the right side of \(\gamma_1^+\).
However, \(\sfd\sfr\) and \(\sfd\gamma_1^-\) are linked in \(\partial \cH^R(\cA)\), and hence \(\sfr\) and \(\gamma_1^-\) are linked.
This contradicts their unlinkedness by \reflem{fenceSystem}.
\end{proof}

Construct an eclipse \(\{\cE_m\}_{m\in \ZZ}\) on \(\cD(p)\) by \refconst{eclipseD}.
The claim above implies that \(g_i(p)\in B\subset \cD(p)\) for all \(i>N\).
Hence, by \reflem{separatingFence} and \refprop{eclipseDProperty}, there exists \(M\in \ZZ\) such that
\(
\cE_M\cap g_i(\cJ(g))\subset \{p\}
\)
for all  \(i>N\).
However, this contradicts \refprop{eclipseDProperty}, since the endpoints of \(\lambda_i\subset g_i(\cJ(g))\) accumulate at the interior point \(x_1\) of \(\cJ(g)^+\).
This completes the proof.
\end{proof}

\subsection{Uniform bouncing direction}\label{Sec:uniformBouncing}
As explained above, although \(\cJ(g)^+\) is one-sided inaccessible, the absence of two-sided inaccessibility requires a more careful analysis of the convergence of \(\Gamma\)-translates of \(\cJ(g)^+\).
To handle this, the following three lemmas describe how to extend \(\{x_k\}_{k=1}^n\) in \refconst{lamiByManyLeaves} while preserving the bouncing condition.

\begin{lem}[One step extension at a fixed end]\label{Lem:oneStepAtFix}
In the setting of \refconst{lamiByManyLeaves}, assume that \(\cJ(g)^+\) is left (resp.\ right) inaccessible and that \(\{x_k\}_{k=1}^{n}\) satisfies the bouncing condition if $n>1$.
If \(x_n\) is fixed by a non-trivial element \(h\in \Gamma\), then there exists a point \(x_{n+1}\in \cJ(g)^+\) such that \(\{x_k\}_{k=1}^{n+1}\) satisfies the bouncing condition and \(x_{n-1}<x_n<x_{n+1}\) (resp.\ \(x_{n+1}<x_n<x_{n-1}\)) in \(\Lambda(\Gamma)\).
\end{lem}
\begin{proof}
By symmetry, it suffices to treat the case where \(\cJ(g)^+\) is left inaccessible.
By \reflem{fanningArcs}, the arc \(h(\gamma_n)\) crosses \(\cH^L(\gamma_n)\) or \(\cH^R(\gamma_n)\).
Replacing \(h\) by \(h^{-1}\) in the latter case, we may assume that \(h(\gamma_n)\) crosses \(\cH^L(\gamma_n)\), and therefore \(h(x_{n-1})\in \opi[\Lambda(\Gamma)]{x_n}{x_{n-1}}\).

Assume first that \(n>1\).
Set \(\cA:=\gamma_n\cup h(\gamma_n)\), so that \(\cA\) is the arc \([x_{n-1},h(x_{n-1})]\) in \(Z^+\).
Equip \(\cA\) with the linear order compatible with \(\gamma_n\), and let \(\cR\) be the truncated ray of \(\cJ(g)^+\) at \(x_{n-1}\).
Then \(\gamma_n\) is a common initial arc of \(\cA\) and \(\cR\).
Since \(Z^+\) is a real tree, their intersection \(\cI:=\cA\cap \cR\) is an arc containing \(\gamma_n\), and hence \(\cI=[x_{n-1},y]\) for some \(y\in h(\gamma_n)\).

If \(y=h(x_{n-1})\), then \(\cA\subset \cR\subset \cJ(g)^+\), so taking \(x_{n+1}:=y\) completes the proof.
Assume now that \(y\neq h(x_{n-1})\).
Then the terminal line \(\cR_y:=\cR\setminus \cI\) is contained in a left branch of \(\cA\) at \(y\), because \(\cR\) is left simplicial.

Consider the Jordan curve \(\cJ:=\cA\cup \opi[\Lambda(\Gamma)]{h(x_{n-1})}{x_{n-1}}\), equipped with the circular order compatible with \(\cA\).
One end of \(\cR_y\) lands at \(y\) on \(\cD^L(\cJ)\), and the other end lands at \(q\in \cD^R(\cJ)\).
Since \(\cA\cap \cR_y=\varnothing\), the line \(\cR_y\) must meet the complementary boundary arc \(\opi[\Lambda(\Gamma)]{h(x_{n-1})}{x_{n-1}}\).
Hence some end ray of \(\cR_y\), landing at \(y\), crosses \(\cD^L(\cJ)\).
Let \(x_{n+1}\) be the initial point of this end ray.
Then \(x_{n+1}\in \cJ(g)^+\), and \(x_{n-1}<x_n<x_{n+1}\) in \(\Lambda(\Gamma)\).
This proves the case \(n>1\).

Assume now that \(n=1\).
Write \(r=[p,x_1)=\gamma_1^+\setminus \{x_1\}\), so that \(h(r)\subset h(\gamma_1)^+\).
Since \(r\) and \(h(r)\) are the components of \([p,h(p)]\setminus \{x_1\}\), they lie in distinct branches \(B_1\) and \(h(B_1)\) at \(x_1\).
By \reflem{extArc}, there is an arc \(a\subset Z^+\) crossing \(\cH^L(\gamma_1)\) and joining \(p\) to a point \(z\in \opi[\Lambda(\Gamma)]{x_1}{x_0}\).
Observe that \(a\subset B_1\), and hence so is the extended ray \(\cR:=r\cup a\).
Equip \(\cR\) with the linear order compatible with \(\gamma_1^+\).
Then \(h(\cR)\subset h(B_1)\), and therefore \(\cR\cap h(\cR)=\varnothing\).
It follows that \(h(\cR)\) crosses \(\cH^L(\cR)\), and hence \(h(z)\in \opi[\Lambda(\Gamma)]{x_1}{z}\).

Set \(\fA:=[z,h(z)]=\cR\cup \{x_1\}\cup h(\cR)\), and let \(\fB\) be the subarc \([p,h(z)]\subset \fA\).
Equip \(\fA\) and \(\fB\) with the linear orders compatible with \(\cR\).
Since \(\gamma_1^+\) is an initial arc of both \(\cJ(g)^+\) and \(\fB\), and since \(Z^+\) is a real tree, the intersection \(\cI:=\cJ(g)^+\cap \fB\) is an arc containing \(\gamma_1^+\).
Hence \(\cI=[p,v]\) for some \(v\in \{x_1\}\cup h(\cR)\).

If \(v=h(z)\), then \(\fB\subset \cJ(g)^+\), so taking \(x_2:=h(z)\) completes the proof.
Assume therefore that \(v\neq h(z)\).
Then the terminal line \(\cJ(g)^+\setminus \cI\) is contained in a left branch of \(\fB\), and hence of \(\fA\), at \(v\), since \(\cJ(g)^+\) is left simplicial.

Consider the Jordan curve \(\cJ:=\fA\cup \opi[\Lambda(\Gamma)]{h(z)}{z}\), equipped with the circular order compatible with \(\fA\).
One end of \(\cJ(g)^+\setminus \cI\) lands at \(v\) on \(\cD^L(\cJ)\), and the other end lands at \(q\in \cD^R(\cJ)\).
Since \(\cJ(g)^+\setminus \cI\) is disjoint from \(\fA\), it must meet the complementary boundary arc \(\opi[\Lambda(\Gamma)]{h(z)}{z}\).
Hence some end ray of \(\cJ(g)^+\setminus \cI\), landing at \(v\), crosses \(\cD^L(\cJ)\).
Let \(x_2\) be the initial point of this end ray.
Then \(x_2\in \cJ(g)^+\), and \(x_0<x_1<x_2\) in \(\Lambda(\Gamma)\).
This completes the proof.
\end{proof}

Let \(\cL=\cL(\Gamma,\{x_k\}_{k=0}^n)\) be the lamination constructed in \refconst{lamiByManyLeaves}.
A sequence \(\{\lambda_m\}_{m\in\NN}\) of leaves of \(\cL\) is called \emph{inside} (resp.\ \emph{outside}) if
\(\lambda_m=h_m(\ell_{k_m})\) for some \(k_m\in\{1,\dots,n\}\) and \(h_m\in\Gamma\), and
\(h_m(\cJ(g))\subset \closure{\cD(p)}\) (resp.\ \(h_m(\cJ(g))\subset S^2_\infty\setminus \cD(p)\)).
Recall \reflem{fenceSystem}.
Indeed, if \(h_m(\cJ(g))\subset \closure{\cD(p)}\), then
\(h_m(\cJ(g)^+)\subset \cD(p)\cup\{p\}\) and therefore $\lambda_m\subset \cD(p)$.
Otherwise, \(h_m(\cJ(g))=\cJ(g)\), so \(h_m(p)=p\), contradicting \(\Fix(h_m)\subset \Lambda(\Gamma)\).
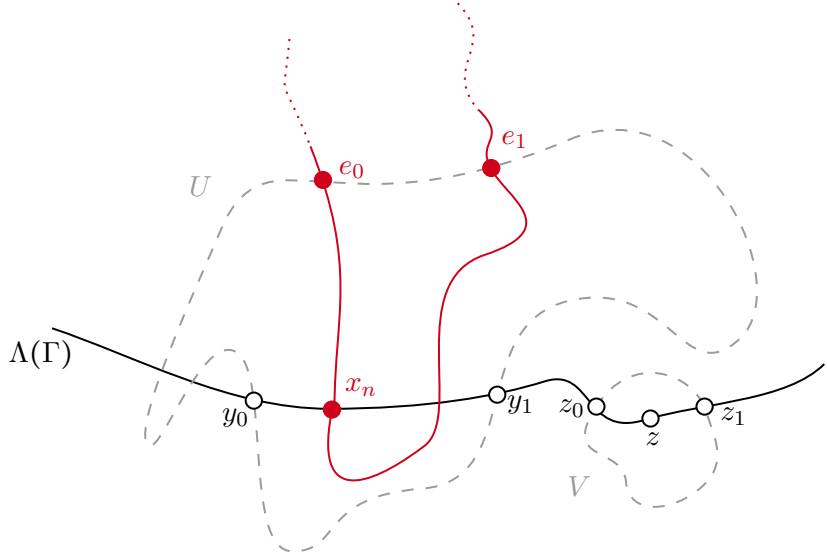
\begin{figure}[htpb]
    \centering

\tikzset{every picture/.style={line width=0.75pt}} %set default line width to 0.75pt        

\begin{tikzpicture}[x=0.75pt,y=0.75pt,yscale=-1.5,xscale=1.5]

%Curve Lines [id:da4858368031492286] 
\draw    (96.4,172) .. controls (123.4,181) and (155.4,199) .. (188.4,199) .. controls (221.4,199) and (244.4,194.8) .. (260.4,189.8) .. controls (276.4,184.8) and (273.4,208.8) .. (293.4,202.8) .. controls (313.4,196.8) and (340.9,196.5) .. (353.4,184) ;
%Curve Lines [id:da8186709388735396] 
\draw [color={rgb, 255:red, 208; green, 2; blue, 27 }  ,draw opacity=1 ]   (190.4,199) .. controls (190.4,167.8) and (198.4,152) .. (182.4,111) ;
%Curve Lines [id:da044779039123563624] 
\draw [color={rgb, 255:red, 208; green, 2; blue, 27 }  ,draw opacity=1 ] [dash pattern={on 0.84pt off 2.51pt}]  (175.4,75) .. controls (173.4,99) and (171.4,83) .. (182.4,111) ;
%Shape: Polygon Curved [id:ds5845212833610566] 
\draw  [color={rgb, 255:red, 155; green, 155; blue, 155 }  ,draw opacity=1 ][dash pattern={on 4.5pt off 4.5pt}] (268.4,108.8) .. controls (312.4,89.8) and (344.4,157.4) .. (325.4,175.4) .. controls (306.4,193.4) and (302.4,154.8) .. (268.4,162.8) .. controls (234.4,170.8) and (253.94,223.2) .. (219.4,225.8) .. controls (184.86,228.4) and (193.09,246.85) .. (176.4,246.8) .. controls (159.71,246.75) and (169.4,185.8) .. (157.4,179.8) .. controls (145.4,173.8) and (131.61,218.65) .. (127.86,210.4) .. controls (124.11,202.15) and (144.86,157.4) .. (152.86,138.4) .. controls (160.86,119.4) and (173.4,122.3) .. (185.4,122.8) .. controls (197.4,123.3) and (224.4,127.8) .. (268.4,108.8) -- cycle ;
%Shape: Circle [id:dp8756785732579062] 
\draw  [color={rgb, 255:red, 208; green, 2; blue, 27 }  ,draw opacity=1 ][fill={rgb, 255:red, 208; green, 2; blue, 27 }  ,fill opacity=1 ][line width=0.75]  (186.86,199.27) .. controls (186.86,197.77) and (188.07,196.56) .. (189.57,196.56) .. controls (191.06,196.56) and (192.27,197.77) .. (192.27,199.27) .. controls (192.27,200.76) and (191.06,201.98) .. (189.57,201.98) .. controls (188.07,201.98) and (186.86,200.76) .. (186.86,199.27) -- cycle ;
%Shape: Circle [id:dp14349243062963624] 
\draw  [color={rgb, 255:red, 0; green, 0; blue, 0 }  ,draw opacity=1 ][fill={rgb, 255:red, 255; green, 255; blue, 255 }  ,fill opacity=1 ][line width=0.75]  (160.86,196.27) .. controls (160.86,194.77) and (162.07,193.56) .. (163.57,193.56) .. controls (165.06,193.56) and (166.27,194.77) .. (166.27,196.27) .. controls (166.27,197.76) and (165.06,198.98) .. (163.57,198.98) .. controls (162.07,198.98) and (160.86,197.76) .. (160.86,196.27) -- cycle ;
%Shape: Circle [id:dp5537521053654307] 
\draw  [color={rgb, 255:red, 0; green, 0; blue, 0 }  ,draw opacity=1 ][fill={rgb, 255:red, 255; green, 255; blue, 255 }  ,fill opacity=1 ][line width=0.75]  (241.86,194.27) .. controls (241.86,192.77) and (243.07,191.56) .. (244.57,191.56) .. controls (246.06,191.56) and (247.27,192.77) .. (247.27,194.27) .. controls (247.27,195.76) and (246.06,196.98) .. (244.57,196.98) .. controls (243.07,196.98) and (241.86,195.76) .. (241.86,194.27) -- cycle ;
%Shape: Circle [id:dp2901616068857591] 
\draw  [color={rgb, 255:red, 208; green, 2; blue, 27 }  ,draw opacity=1 ][fill={rgb, 255:red, 208; green, 2; blue, 27 }  ,fill opacity=1 ][line width=0.75]  (183.86,122.27) .. controls (183.86,120.77) and (185.07,119.56) .. (186.57,119.56) .. controls (188.06,119.56) and (189.27,120.77) .. (189.27,122.27) .. controls (189.27,123.76) and (188.06,124.98) .. (186.57,124.98) .. controls (185.07,124.98) and (183.86,123.76) .. (183.86,122.27) -- cycle ;
%Curve Lines [id:da08458448534822494] 
\draw [color={rgb, 255:red, 208; green, 2; blue, 27 }  ,draw opacity=1 ]   (189.4,200.4) .. controls (182.8,236) and (207.4,221.4) .. (220.4,211.4) .. controls (233.4,201.4) and (213.4,156.4) .. (240.4,147.4) .. controls (267.4,138.4) and (247.4,128) .. (242.4,118.4) .. controls (237.4,108.8) and (248.4,108.8) .. (238.4,98.8) ;
%Shape: Circle [id:dp43597253641464717] 
\draw  [color={rgb, 255:red, 208; green, 2; blue, 27 }  ,draw opacity=1 ][fill={rgb, 255:red, 208; green, 2; blue, 27 }  ,fill opacity=1 ][line width=0.75]  (239.86,118.27) .. controls (239.86,116.77) and (241.07,115.56) .. (242.57,115.56) .. controls (244.06,115.56) and (245.27,116.77) .. (245.27,118.27) .. controls (245.27,119.76) and (244.06,120.98) .. (242.57,120.98) .. controls (241.07,120.98) and (239.86,119.76) .. (239.86,118.27) -- cycle ;
%Curve Lines [id:da028731597049740376] 
\draw [color={rgb, 255:red, 208; green, 2; blue, 27 }  ,draw opacity=1 ] [dash pattern={on 0.84pt off 2.51pt}]  (231.4,63) .. controls (245.4,77.4) and (225.4,83.8) .. (238.4,99) ;
%Shape: Circle [id:dp9312225463423296] 
\draw  [color={rgb, 255:red, 0; green, 0; blue, 0 }  ,draw opacity=1 ][fill={rgb, 255:red, 255; green, 255; blue, 255 }  ,fill opacity=1 ][line width=0.75]  (292.86,202.27) .. controls (292.86,200.77) and (294.07,199.56) .. (295.57,199.56) .. controls (297.06,199.56) and (298.27,200.77) .. (298.27,202.27) .. controls (298.27,203.76) and (297.06,204.98) .. (295.57,204.98) .. controls (294.07,204.98) and (292.86,203.76) .. (292.86,202.27) -- cycle ;
%Shape: Polygon Curved [id:ds7073187425015286] 
\draw  [color={rgb, 255:red, 155; green, 155; blue, 155 }  ,draw opacity=1 ][dash pattern={on 4.5pt off 4.5pt}] (277.4,198.4) .. controls (289.4,179.4) and (309.4,187.4) .. (313.4,198.4) .. controls (317.4,209.4) and (321.4,217.4) .. (312.4,225.4) .. controls (303.4,233.4) and (288.4,235.4) .. (287.4,223.4) .. controls (286.4,211.4) and (265.4,217.4) .. (277.4,198.4) -- cycle ;
%Shape: Circle [id:dp40793792230063775] 
\draw  [color={rgb, 255:red, 0; green, 0; blue, 0 }  ,draw opacity=1 ][fill={rgb, 255:red, 255; green, 255; blue, 255 }  ,fill opacity=1 ][line width=0.75]  (274.86,198.27) .. controls (274.86,196.77) and (276.07,195.56) .. (277.57,195.56) .. controls (279.06,195.56) and (280.27,196.77) .. (280.27,198.27) .. controls (280.27,199.76) and (279.06,200.98) .. (277.57,200.98) .. controls (276.07,200.98) and (274.86,199.76) .. (274.86,198.27) -- cycle ;
%Shape: Circle [id:dp9122289158358222] 
\draw  [color={rgb, 255:red, 0; green, 0; blue, 0 }  ,draw opacity=1 ][fill={rgb, 255:red, 255; green, 255; blue, 255 }  ,fill opacity=1 ][line width=0.75]  (310.86,198.27) .. controls (310.86,196.77) and (312.07,195.56) .. (313.57,195.56) .. controls (315.06,195.56) and (316.27,196.77) .. (316.27,198.27) .. controls (316.27,199.76) and (315.06,200.98) .. (313.57,200.98) .. controls (312.07,200.98) and (310.86,199.76) .. (310.86,198.27) -- cycle ;

% Text Node
\draw (81,175) node [anchor=north west][inner sep=0.75pt]    {$\Lambda ( \Gamma )$};
% Text Node
\draw (141,120) node [anchor=north west][inner sep=0.75pt]  [color={rgb, 255:red, 155; green, 155; blue, 155 }  ,opacity=1 ]  {$U$};
% Text Node
\draw (193,188) node [anchor=north west][inner sep=0.75pt]  [color={rgb, 255:red, 208; green, 2; blue, 27 }  ,opacity=1 ]  {$x_{n}$};
% Text Node
\draw (152,198) node [anchor=north west][inner sep=0.75pt]  [color={rgb, 255:red, 0; green, 0; blue, 0 }  ,opacity=1 ]  {$y_{0}$};
% Text Node
\draw (246.86,193.67) node [anchor=north west][inner sep=0.75pt]  [color={rgb, 255:red, 0; green, 0; blue, 0 }  ,opacity=1 ]  {$y_{1}$};
% Text Node
\draw (190.86,115) node [anchor=north west][inner sep=0.75pt]  [color={rgb, 255:red, 208; green, 2; blue, 27 }  ,opacity=1 ]  {$e_{0}$};
% Text Node
\draw (245,105) node [anchor=north west][inner sep=0.75pt]  [color={rgb, 255:red, 208; green, 2; blue, 27 }  ,opacity=1 ]  {$e_{1}$};
% Text Node
\draw (292.86,206) node [anchor=north west][inner sep=0.75pt]  [color={rgb, 255:red, 0; green, 0; blue, 0 }  ,opacity=1 ]  {$z$};
% Text Node
\draw (267,220.6) node [anchor=north west][inner sep=0.75pt]  [color={rgb, 255:red, 155; green, 155; blue, 155 }  ,opacity=1 ]  {$V$};
% Text Node
\draw (264,195) node [anchor=north west][inner sep=0.75pt]  [color={rgb, 255:red, 0; green, 0; blue, 0 }  ,opacity=1 ]  {$z_{0}$};
% Text Node
\draw (317.86,197.67) node [anchor=north west][inner sep=0.75pt]  [color={rgb, 255:red, 0; green, 0; blue, 0 }  ,opacity=1 ]  {$z_{1}$};
\end{tikzpicture}
  \caption{Local configuration near $x_n$}
    \label{Fig:nearX}
\end{figure}

\begin{lem}[One step extension by outside sequences]\label{Lem:outsideExt}
In the setting of \refconst{lamiByManyLeaves}, assume that \(\cJ(g)^+\) is left (resp.\ right) inaccessible and that \(\{x_k\}_{k=1}^{n}\) satisfies the bouncing condition if $n>1$.
If there is an outside sequence \(\{\lambda_m\}_{m\in\NN}\) in \(\cL(\Gamma,\{x_k\}_{k=0}^n)\) which is \(\opi[\Lambda(\Gamma)]{x_n}{z}\)-side (resp.\ \(\opi[\Lambda(\Gamma)]{z}{x_n}\)-side) for some \(z\in \opi[\Lambda(\Gamma)]{x_n}{x_{n-1}}\) (resp.\ \(z\in \opi[\Lambda(\Gamma)]{x_{n-1}}{x_n}\)), then there exists a point \(x_{n+1}\in \cJ(g)^+\) such that \(\{x_k\}_{k=1}^{n+1}\) satisfies the bouncing condition and \(x_{n-1}<x_n<x_{n+1}\) (resp.\ \(x_{n+1}<x_n<x_{n-1}\)) in \(\Lambda(\Gamma)\).
\end{lem}

\begin{proof}

By symmetry, it suffices to detail the case where \(\cJ(g)^+\) is left inaccessible.
By \reflem{oneStepAtFix}, if \(x_n\) is fixed by a non-trivial element of \(\Gamma\), then there is nothing to prove.
Assume therefore that no non-trivial element of \(\Gamma\) fixes \(x_n\).
Write \(\fD=\cD^L(\gamma_n)\).
Then \(x_n<z<x_{n-1}\) in \(\partial\fD\).

Choose a flat neighborhood \(V\) of \(z\) relative to \(\opi[\partial\fD]{x_n}{x_{n-1}}\) such that
\(\closure{V}\cap \gamma_n=\varnothing\).
Choose also a flat neighborhood \(U\) of \(x_n\) relative to \(\cJ(g)^+\) such that
\[
\closure{U}\cap \cJ(g)\subset \Int \cJ(g)^+,\qquad
\closure{U}\cap \cldi[\partial\fD]{z}{x_{n-1}}=\varnothing,\qquad
\closure{U}\cap \closure{V}=\varnothing.
\]
Then:
\begin{itemize}
\item \(\partial U\cap \cJ(g)=\{e_0,e_1\}\) with \(e_0<x_n<e_1\) in \(\cJ(g)^+\);
\item \(\partial V\cap \Lambda(\Gamma)=\{z_0,z_1\}\) with \(x_n<z_0<z<z_1<x_{n-1}\) in \(\Lambda(\Gamma)\);
\item there exists an interval \(\cI=\opi[\Lambda(\Gamma)]{y_0}{y_1}\) containing \(x_n\) and crossing \(U\);
\item \(y_0<x_n<y_1<z_0<z<z_1<x_{n-1}\) in \(\Lambda(\Gamma)\), since \(\closure{U}\cap \cldi[\partial\fD]{z}{x_{n-1}}=\varnothing\).
\end{itemize}
See \reffig{nearX}.

Since the ray \([e_0,x_n)\subset Z^+\cap \cH\) lands at \(x_n\) from the left side of \(\cI\), we have \(y_1<e_0<y_0\) on \(\partial U\).

Passing to a subsequence of \(\{\lambda_m\}\), we may assume that
\begin{itemize}
\item \(\ropi[\partial\fD]{x_n}{y_1}\) and \(\lopi[\partial\fD]{z_0}{z}\) contain the endpoints \(\alpha_m\) and \(\beta_m\) of \(\lambda_m\), respectively;
\item \(\{\alpha_m\}\) and \(\{\beta_m\}\) are monotone in \(\cldi[\partial\fD]{x_n}{z}\).
\end{itemize}
If infinitely many \(\lambda_i\) share the same endpoint \(e\), then \(e\in\{x_n,z\}\).
If \(e=x_n\), then \refprop{oneEndFix} implies that \(e\) is fixed by a non-trivial element of \(\Gamma\), contrary to assumption.
If \(e=z\), then \refprop{oneEndFix} implies that \(\{x_n,z\}\) is isolated, contradicting that \(\{\lambda_m\}_{m\in\NN}\) is an \(\opi[\Lambda(\Gamma)]{x_n}{z}\)-side sequence.
Thus, after passing to a further subsequence if necessary, we may assume that \(\{\alpha_m\}\) and \(\{\beta_m\}\) are both strictly monotone.

Since \(\{\lambda_m\}\) is outside, \(\lambda_m\) is written \(h_m(\ell_{k_m})\) for some $h_m\in \Gamma$ and $0<k_m\le n$.
Let \(\cC=\cldi[\partial\fD]{e_0}{y_1}\cup \cldi[\partial U]{y_1}{e_0}\), equipped with the circular order such that $e_0<x_n<y_1$, and write \(\fB=\cD^L(\cC)\) (which is a first quadrant of \(U\)).
Since \(\partial U\) separates \(\alpha_m\) and \(\beta_m\), there exists a subline \(\tau_m\subset h_m(\gamma_{k_m})\) crossing \(U^L(\opi[\Lambda(\Gamma)]{y_0}{y_1})\) and joining $\alpha_m$ with a point \(\mu_m\in \opi[\partial U]{y_1}{y_0}\)

\begin{claim}
\(\tau_m\cap [e_0,x_n]=\varnothing\), and \(\tau_m\) crosses \(\fB\) with $\mu_m\in \opi[\partial U]{y_1}{e_0}$.
\end{claim}
\begin{proof}[Proof of the claim]
Assume that \(\tau_m\) meets \([e_0,x_n]\subset Z^+\).
Then either \(\tau_m=\tau_m^+\), or \(\tau_m\) is mixed and \(\tau_m^+\cap [e_0,x_n]\neq \varnothing\).

We first assume that  \(\alpha_m\in Z^+\).
Since \(\alpha_m\in \opi[\cC]{e_0}{y_1}\), there exists a unique subline \(\cA\subset \tau_m^+\) crossing \(\fB\) and joining \(\alpha_m\) to a point \(v\in [e_0,x_n)\).
By the choice of \(U\), the subline \(\cA\) lands at \(v\) from the left side of \(\cJ(g)^+\), contradicting the fact that \(\cJ(g)^+\) is left inaccessible.
This proves the claim in this case.

It remains to consider the case where \(\tau_m\) is mixed and \(\alpha_m\in Z^-\).
Let \(s\) be the synapse of \(\tau_m\).
If \(s\in (e_0,x_n)\), then, since \(Z^+\cap Z^-=\varnothing\) and \(\alpha_m\in \opi[\cC]{e_0}{y_1}\), the arc \(\tau_m^-\) lands at \(s\) from the left side of \((e_0,x_n)\), contradicting the left-inaccessibility of \((e_0,x_n)\).
Hence \(s\notin (e_0,x_n)\).
Since \(\tau_m^-\) is disjoint from \([e_0,x_n]\), it follows that \(s\in \fB\).

Now \(s\in \tau_m^+\) and \(\tau_m^+\cap [e_0,x_n]\neq \varnothing\), so there exists a subray of \(\tau_m^+\) starting at \(s\) and landing at a point of \([e_0,x_n)\) from the left side.
This again contradicts the left-inaccessibility of \(\cJ(g)^+\).
Therefore the first assertion holds.

Once the first assertion is established, the second follows immediately from the left-inaccessibility of \(\cJ(g)^+\).
Thus the proof is complete.
\end{proof}

\begin{claim}\label{Clm:endOrder}
\(y_1<\mu_m\le e_1<e_0<y_0\) in \(\partial U\), hence \(y_1<\mu_m\le e_1<e_0\) in \(\cC\).
\end{claim}

\begin{proof}[Proof of the claim]
Since \(\opi[\partial U]{e_1}{e_0}\subset \cD(p)\) and the sequence is outside, \(\mu_m\in \lopi[\partial U]{e_0}{e_1}\).
The inequalities follow from the above claim.
\end{proof}

\begin{claim}\label{Clm:localCrossing}
the subarc \([x_n,e_1]\subset \cJ(g)^+\) crosses \(\fB\).
\end{claim}

\begin{proof}[Proof of the claim]
Assume \([x_n,e_1]\) meets \(\opi[\cC]{x_n}{y_1}\).
Then a subarc \([\sfe,e_1]\) of \([x_n,e_1]\)  crosses \(\fB\)  since  \(y_1<\mu_m\le e_1<e_0\) in \(\cC\). .
Choose \(m\) with \(\alpha_m\in \opi[\cC]{x_n}{\sfe}\).
Then \(\sfd \tau_m\) is linked with \(\{\sfe,e_1\}\) in $\cC$, contradicting \refprop{unlinkedCrossingLine} unless \(\mu_m=e_1\).

If \(\mu_m=e_1\), then  applying \refprop{enclosing} to \(\tau_m^+(\ni e_1)\) yields a ray landing from the left side of \(\cJ(g)^+\), again a contradiction.
\end{proof}

By \refclm{localCrossing}, a truncation of the ray \((x_n,e_1]\subset \cJ(g)^+\) lands at \(x_n\) from the left side of \(\partial \cH^L(\gamma_{n})\).
This yields the required point \(x_{n+1}\), and completes the proof.
\end{proof}

\setcounter{claim}{0}

A Jordan segment \(\cS\) is said to be \emph{alternatively adapted} to \(Z^\pm\) with respect to \(\{s_i\}_{i=1}^{2k}\) if \(\{s_i\}_{i=1}^{2k}\) is an increasing finite sequence in \(\Int \cS\) for some linear order on \(\cS\), and each Jordan arc \([s_{2i-1},s_{2i}]\subset \Int \cS\), viewed as a real tree, is adapted to \(Z^\pm\).
\begin{lem}[Dragging by inside sequences]\label{Lem:directionalDragging}
In the setting of \refconst{lamiByManyLeaves}, let
\(\cL=\cL(\Gamma,\{x_i\}_{i=0}^n)\)
be the circle lamination constructed in \refconst{lamiByManyLeaves}. Assume that
\(x_0<x_1<\cdots<x_n\) in \(\Lambda(\Gamma)\), that \(\cJ(g)^+\) is left inaccessible, and that \(\{\cE_i\}_{i\in\ZZ}\) is the eclipse on \(\cD(p)\) constructed in \refconst{eclipseD}.

Let \(\{z_i\}_{i=0}^m\subset \Lambda(\Gamma)\), where \(m>0\), be a finite sequence such that
\begin{itemize}
    \item \(x_n=z_0<z_1<\cdots<z_m<x_0\) in \(\cldi[\Lambda(\Gamma)]{x_n}{x_0}\);
    \item for each \(i\in\{1,\dots,m\}\), the interval \(\opi[\Lambda(\Gamma)]{z_{i-1}}{z_{i}}\) is not isolated in \(\cL\);
    \item there exists an inside sequence \(\{\lambda_i\}_{i\in\NN}\) in \(\cL\) that is \(\opi[\Lambda(\Gamma)]{z_0}{z_1}\)-side.
\end{itemize}
Then \(z_i\in \cD(p)\cap \cE_\infty \) for all \(i\in\{1,\dots,m\}\).

The same conclusion holds in the symmetric case where
\(x_n<x_{n-1}<\cdots<x_0\) in \(\Lambda(\Gamma)\),
\(\cJ(g)^+\) is right inaccessible, and
\(\{z_i\}_{i=0}^m\) is a decreasing finite sequence in
\(\lopi[\Lambda(\Gamma)]{x_0}{x_n}\) with \(z_0=x_n\).
\end{lem}

\begin{proof}
By symmetry, it suffices to consider the case where \(\cJ(g)^+\) is left inaccessible.

Write \(\fD=\bigcap_{k=1}^{n}\cH^L(\gamma_{n-1})\). Then
\[
\partial \fD=\opi[\Lambda(\Gamma)]{x_n}{x_0}\cup \bigcup_{k=1}^{n}\gamma_k
\]
since \(\{x_k\}_{k=0}^n\) satisfies the bouncing condition and is increasing in \(\cldi[\Lambda(\Gamma)]{x_0}{x_n}\).

Since each \(\opi[\partial \fD]{z_{i-1}}{z_{i}}\) is not isolated in $\cL$, it follows from \refconst{lamiByManyLeaves} that there exist  \(\opi[\partial \fD]{z_{i-1}}{z_{i}}\)-side sequences \(\{\lambda_{i,j}\}_{j\in\NN}\) in \(\cL\) for each \(i\), such that \(\lambda_{1,j}=\lambda_j\), \(\lambda_{i,j}=h_{i,j}(\ell_{\sfk(i,j)})\) for some \(\sfk(i,j)\in\{1,\dots,n\}\), and \(h_{i,j}\in \Gamma\).
Write \(\lambda_{i,j}=\{\alpha_{i,j},\beta_{i,j}\}\) so that
\(z_{-1}\le \alpha_{i,j}<\beta_{i,j}\le z_{i}\) in
\(\cldi[\partial \fD]{z_{i-1}}{z_{i}}\).

Since \(\{\lambda_j\}_{j\in\NN}\) is an inside sequence and \(z_0=x_n\in \Int \cJ(g)^+\subset \cE_\infty\setminus \Fix(g)\), by \refprop{eclipseDProperty}, \reflem{dragging} implies that \(z_1\in \cE_\infty\). Note that \(z_1\) need not belong to \(\cD(p)\)(see \refprop{eclipseDProperty}).
\begin{claim}\label{Clm:evenInside}
For each \(i\) with \(1<i\le m\), if \(z_{i-1}\in \cD(p)\cap \cE_\infty\), then \(\{\lambda_{i,j}\}_{j\in\NN}\) is eventually inside. In particular, \(z_i\in \cE_\infty\).
\end{claim}
\begin{proof}[Proof of the claim]
Since \(z_{i-1}\in \cD(p)\) and \(\alpha_{i,j}\to z_{i-1}\) as \(j\to\infty\), the sequence \(\{\alpha_{i,j}\}_{j\in\NN}\) is eventually contained in \(\cD(p)\). Hence \reflem{fenceSystem} implies that \(h_{i,j}(\cJ(p))\subset \closure{\cD(p)}\) for all sufficiently large \(j\), and therefore \(\beta_{i,j}\in \cD(p)\) for all sufficiently large \(j\). Thus \(\{\lambda_{i,j}\}_{j\in\NN}\) is eventually inside. It then follows from \reflem{dragging} that \(z_i\in \cE_\infty\).
\end{proof}

We show that \(z_i\in \cD(p)\) for all $i\in \{1,\cdots, m\}$
Assume for contradiction that \(\sfm\) is the smallest index with \(z_\sfm \notin \cD(p)\).

By applying \refclm{evenInside} successively, we obtain, for every \(i\in\{1,\dots,\sfm\}\), that \(z_i\in \cE_\infty\) and that \(\{\lambda_{i,j}\}_{j\in\NN}\) is eventually inside.
Passing to subsequences, we may assume that \(\{\lambda_{i,j}\}_{j\in\NN}\) are inside.
Also, since $z_m\in \cE_\infty \setminus \cD(p)$,  \refprop{eclipseDProperty} implies $z_\sfm \in \cJ(g)^+$ with $z_0=x_n<z_\sfm$ in $\cJ(g)^+$.

Since $h_{i,j}\in \Gamma$ can not fix $p$, \reflem{fenceSystem} implies  $h_{i,j}(p)\in \cD(p)$, and by \reflem{separatingFence}, $h_{i,j}(\cJ(g))\subset \closure{\cE_{s(i,j)}\setminus \cE_{s(i,j)+3}}$ for some $s(i,j)\in \ZZ$.
Since $z_i\in \cE_\infty \setminus \Fix(g)$, the separation property in \refprop{eclipseDProperty} implies that $s(i,j)\to \infty$ as $j\to \infty$.
Hence, after taking a subsequence of $\{s(i,j)\}_{i,j\in\NN}$, we may assume that
\begin{itemize}
    \item $\{s(i,j)\}_{j\in\NN}$ is strictly increasing;
    \item $|s(i,j)-s(i',j')|>3$ whenever $(i,j)\neq (i',j')$;

\end{itemize}

Set $\tau_{i,j}=h_{i,j}(\gamma_{\sfk(i,j)})$. Then $\tau_{i,j}$ is an adapted arc; it is mixed if $\sfk(i,j)=1$, and pure if $\sfk(i,j)>1$.

For $(i,j)\neq (i',j')$, since
$\closure{\cE_{s(i,j)}\setminus \cE_{s(i,j)+3}}
\cap
\closure{\cE_{s(i',j')}\setminus \cE_{s(i',j')+3}}
=
\{p\}$,
we have $\tau_{i,j}\cap \tau_{i',j'}\subset \{p\}$. We claim that $\tau_{i,j}$ cannot meet both $\tau_{i,j-1}$ and $\tau_{i,j+1}$. Suppose otherwise. For each $j'\in \{j-1,j+1\}$, choose $e_{j'}\in \sfd \tau_{i,j'}\cap Z^+$. Then the arc $\cA=[e_{j-1},e_{j+1}]\subset Z^+$ is contained in $\tau_{i,j-1}^+\cup \tau_{i,j+1}^+$ and crosses $\cH$. Since $\tau_{i,j+1}^+\cap \tau_{i,j-1}^+=\{p\}\subset \Int \tau_{i,j}^+$, and $\sfd \cA$ is linked with $\sfd \tau_{i,j}$ in $\Lambda(\Gamma)$, it follows that $\cA$ is linked with $\tau_{i,j}^+$. This contradicts that $\tau_{i,j}^+$ is at least one-sided inaccessible.

Therefore, after passing to a subsequence, we may assume that $\tau_{i,j}\cap \tau_{i,j'}=\varnothing$ whenever $j\neq j'$.
Then, $\cldi[\Lambda(\Gamma)]{\alpha_{i,j}}{\beta_{i,j}}\subset \opi[\Lambda(\Gamma)]{\alpha_{i,j+1}}{\beta_{i,j+1}}\subset \opi[\Lambda(\Gamma)]{z_{i+1}}{z_i}$ for all $i\in \{1,\cdots, \sfm\}$.
By the ordering of the endpoints of the $\{x_i\}_{i=0}^n$ and $\{z_i\}_{i=0}^\sfm$, we also have that $\tau_{i,j}\cap \tau_{i',j'}=\varnothing$ whenever $(i,j)\neq (i',j')$, and that $\tau_{i,j}\cap \gamma_k=\varnothing$ for all $k\in \{1,\dots,n\}$.
In particular,  $p\nin \tau_{i,j}$.

Take flat neighborhoods $U_0$  of $z_0$ and $U_\sfm$ of $z_\sfm$, relative to $\cJ(g)^+$, such that
\begin{itemize}
    \item $\closure{U_i}\cap \cJ(g)=\closure{U_i}\cap \Int \cJ(g)^+$ for $i\in\{0,\sfm\}$ and
$\closure{U_0}\cap\closure{U_\sfm}=\varnothing$;
\item $\closure{U_0}\cap\cldi[\partial \fD]{z_1}{x_{n-1}}=\varnothing$, and
$\closure{U_\sfm}\cap\cldi[\partial \fD]{x_0}{z_{\sfm-1}}=\varnothing$.
\end{itemize}
For each $i\in\{0,\sfm\}$, there is a unique subarc $\cldi[\Lambda(\Gamma)]{u_i}{v_i}$ in $\Lambda(\Gamma)$, that crosses $U_i$ and contains $z_i$. Note that
\[x_{n-1}<u_0<x_n=z_0<v_0<z_1<\cdots<z_{\sfm-1}<u_\sfm<z_\sfm<v_\sfm<x_0
\text{ in } \Lambda(\Gamma).\]

Since $z_i\in \cD(p)$ for each $i\notin\{0,\sfm\}$, there is $J\in\NN$ such that, for all $j\ge J$,
$\alpha_{0,j}\in \ropi[\partial \fD]{z_0}{v_0}$ and
$\beta_{\sfm,j}\in \lopi[\partial \fD]{u_\sfm}{z_\sfm}$, while
$\cldi[\partial \fD]{\beta_{i,J}}{\alpha_{i+1,J}}
\subset
\cD(p)\setminus (\closure{U_0}\cup \closure{U_\sfm})$
for every $i\neq \sfm$.
Discarding finitely many terms, we may assume that $J=0$.

Fix $j\in\NN$. Since $\alpha_{1,j}$ lies in the left open neighborhood $U_0^L$ of $z_0$ associated to $U_0$, and since
$z_0\in \partial U_0^L\cap \cJ(g)^+$ and
$\closure{U_0}\cap \cJ(g)=\closure{U_0}\cap \cJ(g)^+$,
the component of
$U_0^L\cap \cldi[\partial \fD]{z_0}{\alpha_{1,j}}$
containing $\alpha_{1,j}$ is of the form
$\lopi[\partial \fD]{\sfa_j}{\alpha_{1,j}}$
for some $\sfa_j\in U_0^L\cap \cJ(g)^+$.
Similarly, there exists $\sfb_j\in \cJ(g)^+\cap U_\sfm$ such that
$\cldi[\partial \fD]{\beta_{\sfm,j}}{\sfb_j}\cap\cJ(g)=\{\sfb_j\}$.
By construction, $\{\sfa_j\}_{j\in\NN}$ is decreasing, $\{\sfb_j\}_{j\in\NN}$ is increasing, both in $\cldi[\partial \fD]{z_0}{z_\sfm}$, and
$\sfa_j<\sfb_j$ in $\cJ(g)^+$ since $\sfa_j\in U_0$ and $\sfb_j\in U_\sfm$.

For each $j\in\NN$, define an arc $\cS_j$ by
\[\cS_j=
\begin{dcases}
\ \cldi[\partial \fD]{\sfa_j}{\alpha_{1,j}}\cup \tau_{1,j}\cup \cldi[\partial \fD]{\beta_{1,j}}{\sfb_j}, & \text{if }\sfm=1,\\
\ \cldi[\partial \fD]{\sfa_j}{\alpha_{1,j}}
\cup
\left(
\bigcup_{i=1}^{\sfm-1}\bigl(\tau_{i,j}\cup \cldi[\partial \fD]{\beta_{i,j}}{\alpha_{i+1,j}}\bigr)
\right)
\cup
\bigl(\tau_{\sfm,j}\cup \cldi[\partial \fD]{\beta_{\sfm,j}}{\sfb_j}\bigr), & \text{if }\sfm>1.
\end{dcases}
\]

Then $\cS_j$ is alternatively adapted to $Z^\pm$ with respect to $\{\alpha_{i,j},\beta_{i,j}:i\in\{1,\dots,\sfm\}\}$, and
$\cS_j\cap \cJ(g)^+=\{\sfa_j,\sfb_j\}=\sfd \cS_j$.
Hence $\hat\cS_j=[\sfa_j,\sfb_j]\cup \cS_j$ is a Jordan curve.
See \reffig{altArc}.
We orient $\hat\cS_j$ according to the circular order, compatible with the linear order on $\cJ(g)^+$, and orient $\cS_j$ accordingly.

\begin{claim}\label{Clm:unlinking}
The pairs \(\{\sfa_\sfj,\sfb_\sfj\}\) and \(\{\sfa_{\sfj'},\sfb_{\sfj'}\}\) are unlinked in \(\Int \cJ(g)^+\); that is, either \([\sfa_\sfj,\sfb_\sfj]\subset [\sfa_{\sfj'},\sfb_{\sfj'}]\) or \([\sfa_{\sfj'},\sfb_{\sfj'}]\subset [\sfa_\sfj,\sfb_\sfj]\). In particular, either \(\cD^L(\hat \cS_\sfj)\subset \cD^L(\hat \cS_{\sfj'})\) or \(\cD^L(\hat \cS_{\sfj'})\subset \cD^L(\hat \cS_\sfj)\).
\end{claim}

\begin{proof}[Proof of the claim]
It suffices to prove that either \(\cD^L(\hat \cS_\sfj)\subset \cD^L(\hat \cS_{\sfj'})\) or \(\cD^L(\hat \cS_{\sfj'})\subset \cD^L(\hat \cS_\sfj)\). Assume that \(\sfj'<\sfj\) and that there exist points \(\sfx,\sfy\in \Int \cS_{\sfj'}\) such that \(\sfx\in \cD^R(\hat \cS_\sfj)\) and \(\sfy\in \cD^L(\hat \cS_\sfj)\). Since \(\cJ(g)^+\cap \Int \cS_{\sfj'}=\varnothing\), there exists a Jordan subarc \(r\subset \Int \cS_{\sfj'}\) such that \(\sfd r=\{\sfx,\sfz\}\) for some \(\sfz\in \Int \cS_{\sfj}\cap \Int \cS_{\sfj'}\), and \(\Int r\subset \cD^R(\hat \cS_\sfj)\). Observe that for any right neighborhood \(\fU\) of \(\sfz\) relative to \(\cS_{\sfj}\), there exists a subarc \(r'\subset r\) such that \(\sfz\in r'\subset \fU\). Since \(\sfj'<\sfj\) and  $\cS_j\cap \cS_{j'}\subset \cS_j\cap \Lambda(\Gamma)$, no such subarc can exist. Hence \(\Int \cS_{\sfj'}\) is contained in the closure of a Jordan domain bounded by \(\hat \cS_\sfj\), proving the claim.
\end{proof}

\begin{claim}\label{Clm:ascending}
For all \(j\), one has \(\cD^L(\hat \cS_j)\subset \cD^L(\hat \cS_{j+1})\). In particular, \([\sfa_j,\sfb_j]\subset [\sfa_{j+1},\sfb_{j+1}]\) and $\cS_j$ crosses $\cD(p)$ for all \(j\in\NN\).
\end{claim}

\begin{proof}[Proof of the claim]
By \refclm{unlinking}, either \(\cD^L(\hat \cS_j)\subset \cD^L(\hat \cS_{j+1})\) or \(\cD^L(\hat \cS_{j+1})\subset \cD^L(\hat \cS_j)\).

First suppose that \(m>1\). Any left neighborhood of \(\beta_{1,j+1}\) relative to \(\cS_{j+1}\) contains a subarc \(\cldi[\partial \fD]{\beta}{\beta_{1,j+1}}\) of \(\cS_j\) for some \(\beta\in \opi[\partial \fD]{\beta_{1,j}}{\beta_{1,j+1}}\). Moreover, \(\cldi[\partial \fD]{\beta_{1,j}}{\beta_{1,j+1}}\cap \cS_{j+1}=\{\beta_{1,j+1}\}\). This forces \(\cD^L(\hat \cS_j)\subset \cD^L(\hat \cS_{j+1})\).

Now suppose that \(m=1\).
By the choice of $U_0$, $\closure{U_0}\cap \cJ(g)^+$ is written as $[e_0,e_1]$ with $e_0<e_1$ in $\cJ(g)^+$ and the ray $[e_0,z_0)$ is an end ray of $\Int \gamma_n$.
Hence, \([e_0,z_0)\) lands at \(z_0\) on the left side of \(\cldi[\Lambda(\Gamma)]{u_0}{v_0}\).
We have
\[
u_0<z_0<\alpha_{1,j+1}<\alpha_{1,j}<v_0
\quad\text{and}\quad
u_1<\beta_{1,j}<\beta_{1,j+1}<z_1
\]
in \(\Lambda(\Gamma)\).

Define a Jordan curve \(\cC\) as the union of \([e_0,z_0)\), \(\cldi[\Lambda(\Gamma)]{z_0}{v_0}\), and \(\cldi[\partial U_0]{v_0}{e_0}\). Choose the circular order on \(\cC\) compatible with the linear order on \(\cldi[\Lambda(\Gamma)]{z_0}{v_0}\). Since \(\partial U_0\) separates \(\alpha_{0,k}\) and \(\beta_{0,k}\) for \(k\in\{j,j+1\}\), there exist points \(\sfx_j\in \cldi[\partial U_0]{v_0}{e_0}\cap \Int \tau_{0,j}\) and \(\sfx_{j+1}\in \cldi[\partial U_0]{v_0}{e_0}\cap \Int \tau_{1,j+1}\) such that the subarc \(\sigma_k\) of \(\tau_{1,k}\) joining \(\sfx_k\) to \(\alpha_{0,k}\) crosses \(\cC\) through \(\cD^L(\cC)\) for each \(k\in\{j,j+1\}\). Since \(\sigma_j\) and \(\sigma_{j+1}\) are disjoint,  \(\sfd \sigma_j\) and \(\sfd \sigma_{j+1}\) are disjoint and unlinked in \(\cC\). Therefore, since \(z_0<\alpha_{1,j+1}<\alpha_{1,j}<v_0\), we obtain \(v_0<\sfx_j<\sfx_{j+1}<e_0\) on \(\partial U_0\). Since \(\sfx_j,\sfx_{j+1}\in \cD(p)\), we also have \(e_1<\sfx_j<\sfx_{j+1}<e_0\) in \(\cldi[\partial U_0]{e_1}{e_0}\subset \closure{\cD(p)}\).
This implies \(\cD^L(\hat \cS_j)\subset \cD^L(\hat \cS_{j+1})\).

Finally, we show that $\cS_j$ crosses $\cD(p)$. It suffices to prove that $\Int \cS_j\subset \cD(p)$. By the choice of $J$, we have $[\beta_{i,j},\alpha_{i+1,j}]\subset \cD(p)$, and by the choice of $\sfa_j,\sfb_j$, $\lopi{\sfa_j}{\alpha_{1,j}}\cup \ropi{\beta_{\sfm,j}}{\sfb_j}\subset \cD(p)$. Also, since $\cldi[\partial \fD]{\alpha_{i,j}}{\beta_{i,j}}\subset \opi[\partial \fD]{\alpha_{i,j+1}}{\beta_{i,j+1}}$ and  $\tau_{i,j}$ do not meet $\hat\cS_{j+1}$,  by $\cD^L(\hat \cS_j)\subset \cD^L(\hat \cS_{j+1})$, we obtain $\tau_{i,j} \subset \cD^L(\hat \cS_{j+1})\subset \cD(p)$. Therefore $\Int \cS_j\subset \cD(p)$, so $\cS_j$ crosses $\cD(p)$.
\end{proof}

We now complete the proof. Fix \(j\in\NN\).
Since $\tau_{i,j}\subset \closure{\cE_{s(i,j)}\setminus\cE_{s(i,j)+3}}\setminus\{p\}$ and $q\in S_\infty^2\setminus \closure{\cH}$, by the separation property in \refprop{eclipseDProperty}, we can choose $\sfn\in \ZZ$ that \(\cE_{\sfn}\cap \tau_{i,j}=\varnothing\) for all \(i\)  and $(\partial \cE_\sfn)^-\subset S_\infty^2\setminus \closure{\cH}$.
Write $\sigma$ for the connector arc associated to \(\cE_{\sfn}\), crossing \(\cD(p)\).
Recall that \(\sigma\) carries a canonical linear order that increases from the $(-)$-segment to the $(+)$-segment.

Since \(\partial(\cD(p)^R(\sigma))\) separates \(\alpha_{1,j}\) and \(\sfa_j\), the arcs \(\sigma\) and \(\cS_j\) intersect.
Since  $\sigma^- \subset (\partial \cE_\sfn)^-\subset S_\infty^2\setminus \closure{\cH}$ and $\cS_j\subset \closure{\cH}$, we have that \(\sigma^-\cap \cS_j=\varnothing\) and so we may find a point $\sfe\in \sigma^+ \cap \cS_j$ such that $[p,\sfe]\subset \sigma_+$ intersects $\cS_j$ only at $\sfe$, namely, $[p,\sfe]$ lands at $\sfe$ on the right side of $\cS_j$ .
Since \(\cE_{\sfn}\cap \tau_{i,j}=\varnothing\) for all \(i\), it follows that \(\sfe\in \cS_j\cap \Lambda(\Gamma)\). Moreover, since \(\sfe\in \sigma^+\setminus \{p\}\), we  have that  \(\sfe\notin \cE_\infty\) and hence \(\sfe\neq z_i\) for all \(i\).

\begin{claim}
\(\sfe\in \opi[\partial \fD]{\sfa_j}{\alpha_{1,j}}\cup \opi[\partial \fD]{\beta_{\sfm,j}}{\sfb_j}\).
\end{claim}
\begin{proof}[Proof of the claim]
Assume not. Then \(\sfm>1\), and
\(\sfe\in \opi[\partial \fD]{\beta_{\sfi,j}}{z_{\sfi}}\cup \opi[\partial \fD]{z_{\sfi}}{\alpha_{\sfi+1,j}}\)
for some \(\sfi\). Since the sequences \(\{\alpha_{i,j}\}_{j\in\NN}\) and
\(\{\beta_{i,j}\}_{j\in\NN}\) are strictly monotone in
\(\opi[\partial \fD]{z_i}{z_{i+1}}\), \refclm{ascending} allows us to choose
\(j'\) sufficiently large so that \(\sfe\in \cD^L(\hat \cS_{j'})\),
\(\tau_{i,j'}\subset \cE_{\sfn+3}\) for all \(i\), and
\(\lopi[\partial \fD]{\sfa_{j'}}{\alpha_{1,j'}}\cup
\ropi[\partial \fD]{\beta_{\sfm,j'}}{\sfb_{j'}}\subset \Int \cE_{\sfn+3}\).
Since \(\hat \cS_{j'}\) separates \(\sfe\) from \(p\), the arc \([p,\sfe]\)
must intersect \(\cS_{j'}\).

Now \([p,\sfe]\cap \cE_{\sfn+3}\subset \sigma\cap \cE_{\sfn+3}=\{p\}\), whereas
\(\tau_{i,j'}\subset \cE_{\sfn+3}\setminus\{p\}\) for all \(i\). Hence
\([p,\sfe]\) does not intersect any \(\tau_{i,j'}\). Also,
\(\lopi[\partial \fD]{\sfa_{j'}}{\alpha_{1,j'}}\cup
\ropi[\partial \fD]{\beta_{\sfm,j'}}{\sfb_{j'}}\subset \Int \cE_{\sfn+3}\), so
\([p,\sfe]\) meets neither of these rays. Therefore \([p,\sfe]\) meets
\(\cS_{j'}\) only in
\(\bigcup_{i=1}^{\sfm-1}\opi[\partial \fD]{\beta_{i,j'}}{\alpha_{i+1,j'}}\),
which is contained in \(\cS_{j}\cap \Lambda(\Gamma)\).

On the other hand, since \([p,\sfe]\) lands on the right side of \(\cS_j\), we
have \([p,\sfe)\cap \cS_j=\varnothing\), and
\(\sfe\in \opi[\partial \fD]{\alpha_{\sfi,j}}{\alpha_{\sfi,j'}}\cup
\opi[\partial \fD]{\beta_{\sfi,j'}}{\beta_{\sfi,j}}\). Thus \([p,\sfe]\) cannot
meet \(\bigcup_{i=1}^{\sfm-1}\opi[\partial \fD]{\beta_{i,j'}}{\alpha_{i+1,j'}}\), a contradiction. This proves the claim.
\end{proof}

Suppose first that \(\sfe\in \opi[\partial \fD]{\sfa_j}{\alpha_{1,j}}\). Define a Jordan curve \(\cK\) by
\(\cK=[p,\sfa_j]\cup \opi[\partial \fD]{\sfa_j}{\sfe}\cup [\sfe,p]\).
Equip \(\cK\) with the circular order induced by the linear order on \([p,\sfa_j]\).
By \refclm{ascending}, we may choose \(j'\) sufficiently large so that
\(\cldi[\partial \fD]{\sfa_{j'}}{\alpha_{1,j'}}\subset \closure{\cD^L(\cK)}\),
\(\tau_{i,j'}\subset \cE_{\sfn+3}\) for all \(i\in\{1,\dots,\sfm\}\), and
\(\lopi[\partial \fD]{\sfa_{j'}}{\alpha_{1,j'}}\cup \ropi[\partial \fD]{\beta_{\sfm,j'}}{\sfb_{j'}}\subset \Int \cE_{\sfn+3}\).

If \(\sfa_j=\sfa_{j'}\), then \(\{\sfe,p\}\) and \(\sfd\tau_{1,j'}\) are linked in
\(\partial \cD(p)^R(\cS_j)\), so \(\tau_{1,j'}\) intersects \([p,\sfe]\).
If \(\sfa_j\ne \sfa_{j'}\), then \(\{\sfe,p\}\) and
\(\{\sfa_{j'},\beta_{1,j'}\}\) are linked in \(\partial \cD(p)^R(\cS_j)\).
Since \(\lopi[\partial \fD]{\sfa_{j'}}{\alpha_{1,j'}}\subset \cD^L(\cK)\),
we again conclude that \(\tau_{1,j'}\) intersects \((p,\sfe)\).
In either case, this contradicts
\(\tau_{1,j'}\subset \cE_{\sfn+3}\) and
\(\cE_{\sfn+3}\cap [p,\sfe]=\{p\}\).

The case \(\sfe\in \opi[\partial \fD]{\beta_{\sfm,j}}{\sfb_j}\) is analogous.
Define \(\cK=[p,\sfe]\cup \opi[\partial \fD]{\sfe}{\sfb_j}\cup [\sfb_j,p]\).
Then for all sufficiently large \(j'\), the arc \(\tau_{\sfm,j'}\) intersects \((p,\sfe)\),
again contradicting \refprop{eclipseDProperty}.
Therefore \(z_\sfm\notin \cJ(g)^+\), and the proof is complete.
\end{proof}

\setcounter{claim}{0}
\begin{figure}[htpb]
    \centering
\tikzset{every picture/.style={line width=0.75pt}} %set default line width to 0.75pt        

\begin{tikzpicture}[x=0.75pt,y=0.75pt,yscale=-1.5,xscale=1.5]

%Curve Lines [id:da007067054094063607] 
\draw [color={rgb, 255:red, 74; green, 144; blue, 226 }  ,draw opacity=1 ]   (184.4,85) .. controls (206.37,83.74) and (232.4,84.8) .. (245.4,100.4) .. controls (258.4,116) and (245.4,158.8) .. (246.4,181.4) .. controls (247.4,204) and (194.84,214.4) .. (170.4,215.4) ;
%Curve Lines [id:da7012291428823031] 
\draw [color={rgb, 255:red, 208; green, 2; blue, 27 }  ,draw opacity=1 ]   (179.88,85.56) .. controls (153.43,90.26) and (108.4,98.6) .. (109.4,124) .. controls (110.4,149.4) and (95.4,186) .. (114.4,195) .. controls (133.4,204) and (130.4,220) .. (168.4,216) ;
%Curve Lines [id:da9819131681456408] 
\draw [color={rgb, 255:red, 208; green, 2; blue, 27 }  ,draw opacity=1 ]   (177.71,137.29) .. controls (209.71,118.29) and (186.71,141.29) .. (191.71,155.29) .. controls (196.71,169.29) and (194.71,172.29) .. (182.71,182.29) ;
%Curve Lines [id:da4077370652674174] 
\draw [color={rgb, 255:red, 74; green, 144; blue, 226 }  ,draw opacity=1 ]   (130.14,194.57) .. controls (139.14,201.57) and (149.71,184.83) .. (153.71,194.26) .. controls (157.71,203.69) and (166.71,206.29) .. (178.71,185.29) ;
%Curve Lines [id:da6069967534048927] 
\draw [color={rgb, 255:red, 74; green, 144; blue, 226 }  ,draw opacity=1 ]   (160.43,112.98) .. controls (169.14,114.67) and (168.71,112.43) .. (170.71,117.43) .. controls (172.71,122.43) and (166.57,119.17) .. (162.71,125.43) ;
%Shape: Circle [id:dp41726939774179383] 
\draw  [color={rgb, 255:red, 189; green, 16; blue, 224 }  ,draw opacity=1 ][fill={rgb, 255:red, 255; green, 255; blue, 255 }  ,fill opacity=1 ][line width=0.75]  (166.86,215.27) .. controls (166.86,213.77) and (168.07,212.56) .. (169.57,212.56) .. controls (171.06,212.56) and (172.27,213.77) .. (172.27,215.27) .. controls (172.27,216.76) and (171.06,217.98) .. (169.57,217.98) .. controls (168.07,217.98) and (166.86,216.76) .. (166.86,215.27) -- cycle ;
%Shape: Circle [id:dp2946578607219864] 
\draw  [color={rgb, 255:red, 74; green, 144; blue, 226 }  ,draw opacity=1 ][fill={rgb, 255:red, 208; green, 2; blue, 27 }  ,fill opacity=1 ][line width=0.75]  (177.86,183.27) .. controls (177.86,181.77) and (179.07,180.56) .. (180.57,180.56) .. controls (182.06,180.56) and (183.27,181.77) .. (183.27,183.27) .. controls (183.27,184.76) and (182.06,185.98) .. (180.57,185.98) .. controls (179.07,185.98) and (177.86,184.76) .. (177.86,183.27) -- cycle ;
%Shape: Circle [id:dp750634114902227] 
\draw  [color={rgb, 255:red, 74; green, 144; blue, 226 }  ,draw opacity=1 ][fill={rgb, 255:red, 208; green, 2; blue, 27 }  ,fill opacity=1 ][line width=0.75]  (178.86,85.27) .. controls (178.86,83.77) and (180.07,82.56) .. (181.57,82.56) .. controls (183.06,82.56) and (184.27,83.77) .. (184.27,85.27) .. controls (184.27,86.76) and (183.06,87.98) .. (181.57,87.98) .. controls (180.07,87.98) and (178.86,86.76) .. (178.86,85.27) -- cycle ;
%Shape: Circle [id:dp6123523758616866] 
\draw  [color={rgb, 255:red, 208; green, 2; blue, 27 }  ,draw opacity=1 ][fill={rgb, 255:red, 208; green, 2; blue, 27 }  ,fill opacity=1 ][line width=0.75]  (113.86,108.27) .. controls (113.86,106.77) and (115.07,105.56) .. (116.57,105.56) .. controls (118.06,105.56) and (119.27,106.77) .. (119.27,108.27) .. controls (119.27,109.76) and (118.06,110.98) .. (116.57,110.98) .. controls (115.07,110.98) and (113.86,109.76) .. (113.86,108.27) -- cycle ;
%Shape: Polygon Curved [id:ds5487634547799186] 
\draw  [color={rgb, 255:red, 155; green, 155; blue, 155 }  ,draw opacity=1 ][dash pattern={on 4.5pt off 4.5pt}] (90.29,94.86) .. controls (102.29,75.86) and (130.4,86.4) .. (134.4,97.4) .. controls (138.4,108.4) and (142.51,119.94) .. (133.4,124.4) .. controls (124.29,128.86) and (119.29,142.86) .. (102.29,140.86) .. controls (85.29,138.86) and (78.29,113.86) .. (90.29,94.86) -- cycle ;
%Shape: Polygon Curved [id:ds1235697797554598] 
\draw  [color={rgb, 255:red, 155; green, 155; blue, 155 }  ,draw opacity=1 ][dash pattern={on 4.5pt off 4.5pt}] (105.4,184.4) .. controls (118.4,166.6) and (133.29,174.86) .. (141.4,184.4) .. controls (149.51,193.94) and (150.29,205.06) .. (144.29,215.86) .. controls (138.29,226.66) and (116.29,213.66) .. (107.29,213.86) .. controls (98.29,214.06) and (92.4,202.2) .. (105.4,184.4) -- cycle ;
%Shape: Circle [id:dp8199861510480305] 
\draw  [color={rgb, 255:red, 208; green, 2; blue, 27 }  ,draw opacity=1 ][fill={rgb, 255:red, 208; green, 2; blue, 27 }  ,fill opacity=1 ][line width=0.75]  (118.86,199.27) .. controls (118.86,197.77) and (120.07,196.56) .. (121.57,196.56) .. controls (123.06,196.56) and (124.27,197.77) .. (124.27,199.27) .. controls (124.27,200.76) and (123.06,201.98) .. (121.57,201.98) .. controls (120.07,201.98) and (118.86,200.76) .. (118.86,199.27) -- cycle ;
%Curve Lines [id:da16958583148058193] 
\draw    (110.71,128) .. controls (124.71,131) and (126.4,110) .. (139.4,115) ;
%Shape: Circle [id:dp1276332935226674] 
\draw  [color={rgb, 255:red, 0; green, 0; blue, 0 }  ,draw opacity=1 ][fill={rgb, 255:red, 0; green, 0; blue, 0 }  ,fill opacity=1 ][line width=0.75]  (106.86,127.27) .. controls (106.86,125.77) and (108.07,124.56) .. (109.57,124.56) .. controls (111.06,124.56) and (112.27,125.77) .. (112.27,127.27) .. controls (112.27,128.76) and (111.06,129.98) .. (109.57,129.98) .. controls (108.07,129.98) and (106.86,128.76) .. (106.86,127.27) -- cycle ;
%Curve Lines [id:da3894198570267903] 
\draw [color={rgb, 255:red, 208; green, 2; blue, 27 }  ,draw opacity=1 ]   (131.71,115) .. controls (130.71,105) and (138.12,107.82) .. (143.29,108.71) .. controls (148.45,109.61) and (154.57,105.87) .. (158.14,111.43) ;
%Shape: Circle [id:dp3579013019920423] 
\draw  [color={rgb, 255:red, 74; green, 144; blue, 226 }  ,draw opacity=1 ][fill={rgb, 255:red, 208; green, 2; blue, 27 }  ,fill opacity=1 ][line width=0.75]  (155.86,112.27) .. controls (155.86,110.77) and (157.07,109.56) .. (158.57,109.56) .. controls (160.06,109.56) and (161.27,110.77) .. (161.27,112.27) .. controls (161.27,113.76) and (160.06,114.98) .. (158.57,114.98) .. controls (157.07,114.98) and (155.86,113.76) .. (155.86,112.27) -- cycle ;
%Shape: Polygon Curved [id:ds688642844175417] 
\draw  [color={rgb, 255:red, 155; green, 155; blue, 155 }  ,draw opacity=1 ][dash pattern={on 4.5pt off 4.5pt}] (153.4,126.4) .. controls (166.4,108.6) and (182.57,110.43) .. (186.57,121.43) .. controls (190.57,132.43) and (188.57,132.63) .. (182.57,143.43) .. controls (176.57,154.23) and (170.57,146.23) .. (161.57,146.43) .. controls (152.57,146.63) and (140.4,144.2) .. (153.4,126.4) -- cycle ;
%Curve Lines [id:da6635076228957261] 
\draw    (157.57,122.43) .. controls (166.57,125.43) and (170.57,135.43) .. (183.57,140.43) ;
%Shape: Circle [id:dp8112719686132208] 
\draw  [color={rgb, 255:red, 0; green, 0; blue, 0 }  ,draw opacity=1 ][fill={rgb, 255:red, 0; green, 0; blue, 0 }  ,fill opacity=1 ][line width=0.75]  (167.86,132.27) .. controls (167.86,130.77) and (169.07,129.56) .. (170.57,129.56) .. controls (172.06,129.56) and (173.27,130.77) .. (173.27,132.27) .. controls (173.27,133.76) and (172.06,134.98) .. (170.57,134.98) .. controls (169.07,134.98) and (167.86,133.76) .. (167.86,132.27) -- cycle ;
%Curve Lines [id:da8928280024173862] 
\draw    (127.4,205) .. controls (136.71,197.29) and (122.14,198.57) .. (133.14,192.57) ;
%Shape: Circle [id:dp8723124787255135] 
\draw  [color={rgb, 255:red, 0; green, 0; blue, 0 }  ,draw opacity=1 ][fill={rgb, 255:red, 0; green, 0; blue, 0 }  ,fill opacity=1 ][line width=0.75]  (124.86,205.27) .. controls (124.86,203.77) and (126.07,202.56) .. (127.57,202.56) .. controls (129.06,202.56) and (130.27,203.77) .. (130.27,205.27) .. controls (130.27,206.76) and (129.06,207.98) .. (127.57,207.98) .. controls (126.07,207.98) and (124.86,206.76) .. (124.86,205.27) -- cycle ;
%Curve Lines [id:da0880326060724772] 
\draw [color={rgb, 255:red, 208; green, 2; blue, 27 }  ,draw opacity=1 ]   (121.71,116) .. controls (121.71,106.43) and (131.55,101.54) .. (136.71,102.43) .. controls (141.88,103.32) and (153.71,94.43) .. (161.71,100.43) ;
%Shape: Circle [id:dp3532520790566761] 
\draw  [color={rgb, 255:red, 74; green, 144; blue, 226 }  ,draw opacity=1 ][fill={rgb, 255:red, 208; green, 2; blue, 27 }  ,fill opacity=1 ][line width=0.75]  (159.86,101.27) .. controls (159.86,99.77) and (161.07,98.56) .. (162.57,98.56) .. controls (164.06,98.56) and (165.27,99.77) .. (165.27,101.27) .. controls (165.27,102.76) and (164.06,103.98) .. (162.57,103.98) .. controls (161.07,103.98) and (159.86,102.76) .. (159.86,101.27) -- cycle ;
%Curve Lines [id:da13703580622688938] 
\draw [color={rgb, 255:red, 74; green, 144; blue, 226 }  ,draw opacity=1 ]   (165.43,101.98) .. controls (174.14,103.67) and (176.71,105.43) .. (176.71,113.43) .. controls (176.71,121.43) and (170.57,122.17) .. (166.71,128.43) ;
%Curve Lines [id:da23198991451180284] 
\draw [color={rgb, 255:red, 74; green, 144; blue, 226 }  ,draw opacity=1 ]   (174.71,135.29) .. controls (201.71,112.57) and (199.71,118.57) .. (204.71,144.57) .. controls (209.71,170.57) and (230.71,155.57) .. (189.71,187.57) ;
%Curve Lines [id:da653523065737875] 
\draw [color={rgb, 255:red, 208; green, 2; blue, 27 }  ,draw opacity=1 ]   (131.14,200.57) .. controls (146.14,201.86) and (147.71,218.57) .. (160.71,209.57) .. controls (173.71,200.57) and (174.71,210.57) .. (186.71,189.57) ;
%Shape: Circle [id:dp7961085001350284] 
\draw  [color={rgb, 255:red, 74; green, 144; blue, 226 }  ,draw opacity=1 ][fill={rgb, 255:red, 208; green, 2; blue, 27 }  ,fill opacity=1 ][line width=0.75]  (184.86,189.27) .. controls (184.86,187.77) and (186.07,186.56) .. (187.57,186.56) .. controls (189.06,186.56) and (190.27,187.77) .. (190.27,189.27) .. controls (190.27,190.76) and (189.06,191.98) .. (187.57,191.98) .. controls (186.07,191.98) and (184.86,190.76) .. (184.86,189.27) -- cycle ;
%Curve Lines [id:da5667501475453902] 
\draw    (111.71,114) .. controls (124.14,108.57) and (112.14,116.57) .. (123.14,115.57) ;
%Shape: Circle [id:dp39679132788561533] 
\draw  [color={rgb, 255:red, 0; green, 0; blue, 0 }  ,draw opacity=1 ][fill={rgb, 255:red, 0; green, 0; blue, 0 }  ,fill opacity=1 ][line width=0.75]  (108.86,113.27) .. controls (108.86,111.77) and (110.07,110.56) .. (111.57,110.56) .. controls (113.06,110.56) and (114.27,111.77) .. (114.27,113.27) .. controls (114.27,114.76) and (113.06,115.98) .. (111.57,115.98) .. controls (110.07,115.98) and (108.86,114.76) .. (108.86,113.27) -- cycle ;
%Curve Lines [id:da7408761165301735] 
\draw    (122.14,115.57) .. controls (133.14,115.57) and (113.14,125.57) .. (105.14,119.57) .. controls (97.14,113.57) and (96.14,121.57) .. (109.14,126.57) ;

% Text Node
\draw (176,73) node [anchor=north west][inner sep=0.75pt]  [font=\normalsize,color={rgb, 255:red, 208; green, 2; blue, 27 }  ,opacity=1 ]  {$p$};
% Text Node
\draw (164,220) node [anchor=north west][inner sep=0.75pt]  [font=\normalsize,color={rgb, 255:red, 208; green, 2; blue, 27 }  ,opacity=1 ]  {$q$};
% Text Node
\draw (255,133.4) node [anchor=north west][inner sep=0.75pt]  [font=\normalsize,color={rgb, 255:red, 74; green, 144; blue, 226 }  ,opacity=1 ]  {$\ell $};
% Text Node
\draw (114,187) node [anchor=north west][inner sep=0.75pt]  [font=\normalsize,color={rgb, 255:red, 208; green, 2; blue, 27 }  ,opacity=1 ]  {$z_{2}$};
% Text Node
\draw (110,95) node [anchor=north west][inner sep=0.75pt]  [font=\normalsize,color={rgb, 255:red, 208; green, 2; blue, 27 }  ,opacity=1 ]  {$z_{0}$};
% Text Node
\draw (97,73) node [anchor=north west][inner sep=0.75pt]  [color={rgb, 255:red, 155; green, 155; blue, 155 }  ,opacity=1 ]  {$U_{0}$};
% Text Node
\draw (87,180) node [anchor=north west][inner sep=0.75pt]  [color={rgb, 255:red, 155; green, 155; blue, 155 }  ,opacity=1 ]  {$U_{2}$};
% Text Node
\draw (96,125) node [anchor=north west][inner sep=0.75pt]  [font=\normalsize,color={rgb, 255:red, 0; green, 0; blue, 0 }  ,opacity=1 ]  {$\sfa_{j}$};
% Text Node
\draw (155,130) node [anchor=north west][inner sep=0.75pt]  [font=\normalsize,color={rgb, 255:red, 0; green, 0; blue, 0 }  ,opacity=1 ]  {$z_{1}$};
% Text Node
\draw (140,145) node [anchor=north west][inner sep=0.75pt]  [color={rgb, 255:red, 155; green, 155; blue, 155 }  ,opacity=1 ]  {$U_{1}$};
% Text Node
\draw (86,205) node [anchor=north west][inner sep=0.75pt]  [font=\normalsize,color={rgb, 255:red, 0; green, 0; blue, 0 }, fill=white  ,opacity=1 ]  {$\sfb_{j} =\sfb_{j+1}$};
% Text Node
\draw (170,170) node [anchor=north west][inner sep=0.75pt]  [font=\normalsize,color={rgb, 255:red, 0; green, 0; blue, 0 }  ,opacity=1 ]  {$\tau_{2,j}$};
% Text Node
\draw (165, 92) node [anchor=north west][inner sep=0.75pt]  [font=\normalsize,color={rgb, 255:red, 0; green, 0; blue, 0 }  ,opacity=1 ]  {$\tau_{1,j+1}$};
% Text Node
\draw (143,112) node [anchor=north west][inner sep=0.75pt]  [font=\normalsize,color={rgb, 255:red, 0; green, 0; blue, 0 }  ,opacity=1 ]  {$\tau_{1,j}$};
% Text Node
\draw (207,174.67) node [anchor=north west][inner sep=0.75pt]  [font=\normalsize,color={rgb, 255:red, 0; green, 0; blue, 0 }  ,opacity=1 ]  {$\tau_{2,j+1}$};
% Text Node
\draw (90,108) node [anchor=north west][inner sep=0.75pt]  [font=\normalsize,color={rgb, 255:red, 0; green, 0; blue, 0 }  ,opacity=1 ]  {$\sfa_{j+1}$};

\end{tikzpicture}
    \caption{$\cS_j$ and $\cS_{j+1}$ when $(n,\sfm)=(1,2)$; the black arcs are subsegments of $\Lambda(\Gamma)$}
    \label{Fig:altArc}
\end{figure}
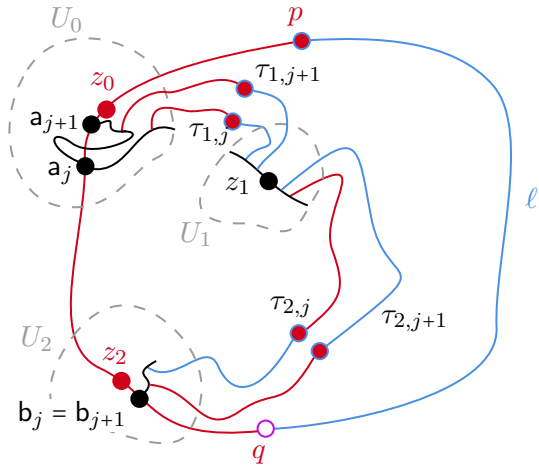

\subsection{Only one prong in $Z^\pm$}
We now prove the main theorem of this section, \refthm{oneProng}.
\begin{proof}[Proof of \refthm{oneProng}]
By \refthm{twoInOne}, it suffices to prove that  there is no element $g \in G$ such that $g$ fixes exactly one point of the zipper and admits a $g$-prong in $Z^+$.
Assume that such a \(g\) exists. By symmetry, it suffices to consider the case where \(\cJ(g)^+\) is left-inaccessible; see \refrmk{inaccessibleSide}. Assume the setting of \refconst{lamiByManyLeaves}. By \refconst{lamiByManyLeaves}, we obtain \(\cL_1=\cL(\Gamma,\{x_0,x_1\})\), generated by \(\gamma_1\), and by \refconst{eclipseD} an eclipse \(\{\cE_n\}_{n\in\NN}\) on \(\cD(p)\).

By \refprop{sideIsolated}, \(\opi[\Lambda(\Gamma)]{x_1}{x_0}\) is isolated in \(\cL_1\), and by \reflem{fanningArcs} no element of \(\Gamma\) preserves \(\sfd\gamma_1\) pointwise. Hence the leaf \(\wt{\sfg}_1=\sfg(x_0,x_1)\) corresponds to a boundary leaf \(\sfg_1\) of \(\cG_1\) and $\cG_1''\ne \varnothing$.

\begin{claim*}
Let \(n\ge 1\). Suppose \(\{x_i\}_{i=1}^n\subset \Lambda(\Gamma)\) satisfies the bouncing condition and \(x_0<x_1<\cdots<x_n\) in $\Lambda(\Gamma)$ if $n>1$. Then there exists \(x_{n+1}\in \opi[\Lambda(\Gamma)]{x_n}{x_0}\) such that \(\{x_i\}_{i=1}^{n+1}\) satisfies the bouncing condition.
\end{claim*}
\begin{proof}[Proof of the claim]
    Let \(\cL_n=\cL(\Gamma,\{x_i\}_{i=0}^n)\), and assume that no such \(x_{n+1}\) exists. If \(x_n\) were fixed by a non-trivial element \(h\in\Gamma\), then \(h\) could not fix \(x_{n-1}\) by \reflem{fanningArcs}. Thus \reflem{oneStepAtFix} would give a contradiction. Therefore \(x_n\) is not fixed by any non-trivial element of \(\Gamma\).

If \(n\ge 2\) and \(x_0<x_1<\cdots<x_n\) in \(\Lambda(\Gamma)\), then each \(\opi[\Lambda(\Gamma)]{x_i}{x_{i-1}}\) is isolated in \(\cL_n\). Otherwise, since \(\sfd\gamma_{i-1}\) and \(\sfd\gamma_i\) share the endpoint \(x_{i-1}\), there would be infinitely many leaves with endpoint \(x_{i-1}\) accumulating on \(\sfd\gamma_i\), and \refprop{oneEndFix} would imply that \(\sfd\gamma_i\) is isolated, a contradiction. Hence the leaves \(\wt{\sfg}_i=\sfg(x_{i-1},x_i)\) of \(\wt{\cG}_n=\wt{\cG}(\Gamma,\{x_i\}_{i=0}^{n})\) descend to boundary leaves \(\sfg_i\) of \(\cG_n=\cG(\Gamma,\{x_i\}_{i=0}^{n})\).

By \reflem{closureOfLeaf}, the closure \(\sG\) of \(\sfg_n\) is a geodesic lamination on \(\cH/\Gamma\). Let \(\sL\) denote the circle lamination on \(\Lambda(\Gamma)\) generated by \(\sfd\gamma_n\). Then the geometric realization of \(\sL\) is the lift of \(\sG\), and \reflem{closureOfManyLeaves} gives \(\sG\subset \cG_n\) and \(\sL\subset \cL_n\). If \(n=1\), then \(\sG=\cG_1\) and \(\sL=\cL_1\).

Since \(x_n\) is not fixed by any non-trivial element of \(\Gamma\), \reflem{closureOfLeaf} implies that there is a principal region \(U\) of \(\sG''\) such that either \(\sfg_n\) is a boundary leaf of \(U\), or the end of \(\sfg_n\) corresponding to \(x_n\) escapes toward an end of \(U\) with \(\sfg_n\subset U\). Since \(\opi[\Lambda(\Gamma)]{x_n}{x_{n-1}}\) is isolated in both \(\cL_n\) and \(\sL\), \refprop{gapShape} implies that one of the following occurs:
\begin{enumerate}
\item There exists a finite increasing sequence \(\{z_i\}_{i=0}^m\) in \(\cldi[\Lambda(\Gamma)]{x_n}{x_{n-1}}\), where \(m\ge 2\), such that \(z_0=x_n\), \(z_m=x_{n-1}\), and \(\opi[\Lambda(\Gamma)]{z_i}{z_{i+1}}\) is not isolated in \(\sL\) for each \(i\).
\item There exist \(h\in\Gamma\) and a bi-infinite increasing sequence \(\{z_i\}_{i\in\ZZ}\) in \(\opi[\Lambda(\Gamma)]{r(h)}{a(h)}\) such that \(z_0=x_n\), \(x_{n-1}\in \cldi[\Lambda(\Gamma)]{a(h)}{z_j}\) for some \(j<0\), \(h\) translates \(\{z_i\}_{i\in\ZZ}\) by \(k\in \NN\), that is, \(h(z_i)=z_{i+k}\) for all \(i\in\ZZ\), and \(\opi[\Lambda(\Gamma)]{z_i}{z_{i+1}}\) is not isolated in \(\sL\) for each \(i\).
\end{enumerate}
In either case, each \(\{z_i,z_{i+1}\}\) is a leaf of both \(\sL\) and \(\cL_n\). Moreover, since \(\opi[\Lambda(\Gamma)]{z_i}{z_{i+1}}\) is not isolated in either \(\sL\) or \(\cL_n\), each leaf \(\{z_i,z_{i+1}\}\) lies in \(\opi[\Lambda(\Gamma)]{x_n}{x_0}\) or in \(\opi[\Lambda(\Gamma)]{x_j}{x_{j+1}}\) for some \(j\). In particular, \(\{z_0,z_1\}\) lies in \(\opi[\Lambda(\Gamma)]{x_n}{x_0}\).

If there were a \(\opi[\Lambda(\Gamma)]{z_0}{z_1}\)-side outside sequence in both \(\sL\) and \(\cL_n\), this would contradict \reflem{outsideExt}. Hence there is a \(\opi[\Lambda(\Gamma)]{z_0}{z_1}\)-side inside sequence in both \(\sL\) and \(\cL_n\). If \(x_0=z_j\) for some \(j>0\), then \refprop{sideIsolated} implies that either \(n>1\) or \(j>1\), and \reflem{directionalDragging} gives \(z_i\in \cD(p)\) for all \(0<i\le j\), contradicting \(z_j=x_0\in \cJ(g)^-\).
Therefore,  \(z_i\in \opi[\Lambda(\Gamma)]{x_n}{x_0}\) for all $i\in\NN$ and thus \((1)\) cannot occur.

Hence we are in case \((2)\). Again by \reflem{directionalDragging}, \(z_i\in \cD(p)\cap \cE_\infty\) for all \(i\in\NN\). Since \(z_k=h(z_0)\in h(\cJ(g))\cap \cD(p)\), \reflem{fenceSystem} implies that \(h(p)\in \cD(p)\). Then \reflem{separatingFence} and \refprop{eclipseDProperty} imply that \(h(\cJ(g))\cap \cE_\infty\subset\{p\}\), contradicting \(z_k\in h(\cJ(g))\cap \cE_\infty\) and \(z_k\neq p\). This completes the proof.
\end{proof}

By induction, we can find a strictly increasing sequence $\{x_i\}_{i\in \NN}$ in $\cJ(g)^+$ such that, for every $n>2$, the finite sequence $\{x_i\}_{i=1}^n$ satisfies the bouncing condition, and $\{x_i\}_{i=1}^\infty$ is strictly increasing in $\cldi[\Lambda(\Gamma)]{x_1}{x_0}$.
Then the closure $\cL_\infty$ of the $\Gamma$-orbits of $\{\sfd \gamma_i:i\in \NN\}$ in $\cM$ is a $\Gamma$-invariant circle lamination by \refprop{unlinkedCrossingLine}, \refrmk{inaccessibleSide}, and \reflem{fenceSystem}.
Accordingly, the geometric realization $\wt{\cG}_\infty$ of $\cL_\infty$ induces a geodesic lamination on $\cH/\Gamma$ by \refrmk{realization}.

Since $\{x_i\}_{i=1}^\infty$ is strictly increasing in $\cldi[\Lambda(\Gamma)]{x_1}{x_0}$, for each $i\in \NN$, the interval $\opi[\Lambda(\Gamma)]{x_{i}}{x_{i-1}}$ is isolated in $\cL_\infty$ by \refprop{oneEndFix}, as above.
Hence, each leaf $\sfg_i$ of $\cG_\infty$ corresponding to $\sfd \gamma_i$ is a boundary leaf of $\cG_\infty$.
By \refprop{finiteBD}, there are only finitely many boundary leaves in $\cG_\infty$.
Therefore, there exist $j,j'\in \NN$ with $j<j'$ and a nontrivial element $f\in \Gamma$ such that
$f$ maps the subarc $[x_{j-1},x_{j}]$ of $\cJ(g)^+$ to the subarc $[x_{j'-1},x_{j'}]$ of $\cJ(g)^+$, preserving the order.
Observe that, by the left inaccessibility of $\cJ(g)^+$, the image $f([x_{j-1},x_{j'-1}])$ is also a subarc of $\cJ(g)^+$.
Hence,
\[
\ell_f=\bigcup_{\sfm\in\ZZ} f^m([x_{j-1},x_{j'-1}])
\]
is an $f$-invariant subline of $\cJ(g)^+$.
By \reflem{invRayLand}, $\sfd \ell_f=\Fix(f)$.
If $\sfd \ell_f$ intersects $\Fix(g)$, then $\Fix(g)=\Fix(f)$ by the discreteness of $G$.
However, this is a contradiction, since $\Fix(g)\cap \Lambda(\Gamma)=\varnothing$ and $\Fix(f)\subset \Lambda(\Gamma)$.
If, on the other hand, $\sfd \ell_f\subset \Int \cJ^+(g)$, this is also a contradiction by \refthm{twoInOne}.
This completes the proof.
    \qedhere
\end{proof}

\section{No Prong in $Z^+$ with a Two-Sided Branched Axis in $Z^-$}\label{Sec:noprongTwoside}
As explained at the beginning of \refsec{oneprong}, we now treat the case where $g$ admits no $g$-prong in $Z^+$ and the $g$-axis $\ell$ is two-sided branched.
To rule out this case, we adapt the strategy of \refthm{twoInOne} and prove \refthm{simplicialAxis}, namely, that $\ell$ is one-sided simplicial.

In the present setting, the argument of \refthm{twoInOne} does not apply directly, since there is no invariant fence.
We therefore begin by introducing a modified version of the invariant fence, which we call a \emph{fat fence}.
\subsection{Fat fences}

Let $Z^\pm$ be a minimal zipper for $M$.
Assume that $g\in G$ fixes a unique point $p\in Z^+$.
By \refthm{oneProng}, $g$ admits no $g$-prong in $Z^+$.
On the other hand, $g$ admits an axis $\ell$ in $Z^-$, and $p$ is the pivot of $\ell$.

Although there is no $g$-prong in $Z^+$ and no $g$-invariant fence, in the case where $\ell$ is two-sided branched we can define \emph{fat fences} as follows.
Recall that the canonical orientation of $\ell$ is the linear order increasing toward $p$.
By assumption, the pivot $p$ has both left and right branches.
Hence there exist  the left and right connectors of $\ell$ that are of type~2.
Thus, by \refconst{eclipseL}, we can construct an eclipse $\{\cE_n\}_{n\in\NN}$ along $\ell$ associated with a pair $(\delta_1,\delta_2)$, where $\delta_1$ and $\delta_2$ are respectively left and right relative to $\ell$.

\begin{figure}[htpb]
    \centering

\tikzset{every picture/.style={line width=0.75pt}} %set default line width to 0.75pt        

\begin{tikzpicture}[x=0.75pt,y=0.75pt,yscale=-1,xscale=1]
%uncomment if require: \path (0,349); %set diagram left start at 0, and has height of 349

%Shape: Circle [id:dp8483192385866364] 
\draw  [color={rgb, 255:red, 74; green, 144; blue, 226 }  ,draw opacity=1 ][fill={rgb, 255:red, 208; green, 2; blue, 27 }  ,fill opacity=1 ][line width=0.75]  (272.86,86.47) .. controls (272.86,84.97) and (274.07,83.76) .. (275.57,83.76) .. controls (277.06,83.76) and (278.27,84.97) .. (278.27,86.47) .. controls (278.27,87.96) and (277.06,89.18) .. (275.57,89.18) .. controls (274.07,89.18) and (272.86,87.96) .. (272.86,86.47) -- cycle ;
%Shape: Circle [id:dp2228347476853021] 
\draw  [color={rgb, 255:red, 74; green, 144; blue, 226 }  ,draw opacity=1 ][fill={rgb, 255:red, 255; green, 255; blue, 255 }  ,fill opacity=1 ][line width=0.75]  (269.86,211.47) .. controls (269.86,209.97) and (271.07,208.76) .. (272.57,208.76) .. controls (274.06,208.76) and (275.27,209.97) .. (275.27,211.47) .. controls (275.27,212.96) and (274.06,214.18) .. (272.57,214.18) .. controls (271.07,214.18) and (269.86,212.96) .. (269.86,211.47) -- cycle ;
%Curve Lines [id:da18088361018156862] 
\draw [color={rgb, 255:red, 74; green, 144; blue, 226 }  ,draw opacity=1 ]   (273.4,214) .. controls (288.4,245) and (343.4,238) .. (365.4,215) ;
%Shape: Circle [id:dp17942078516761317] 
\draw  [color={rgb, 255:red, 74; green, 144; blue, 226 }  ,draw opacity=1 ][fill={rgb, 255:red, 208; green, 2; blue, 27 }  ,fill opacity=1 ][line width=0.75]  (343.86,149.47) .. controls (343.86,147.97) and (345.07,146.76) .. (346.57,146.76) .. controls (348.06,146.76) and (349.27,147.97) .. (349.27,149.47) .. controls (349.27,150.96) and (348.06,152.18) .. (346.57,152.18) .. controls (345.07,152.18) and (343.86,150.96) .. (343.86,149.47) -- cycle ;
%Curve Lines [id:da21983249203204414] 
\draw [color={rgb, 255:red, 74; green, 144; blue, 226 }  ,draw opacity=1 ]   (350.4,149.4) .. controls (383.8,154.4) and (336.4,200) .. (344.4,228) ;
%Shape: Circle [id:dp29762501412948195] 
\draw   (203,151.7) .. controls (203,99.95) and (244.95,58) .. (296.7,58) .. controls (348.45,58) and (390.4,99.95) .. (390.4,151.7) .. controls (390.4,203.45) and (348.45,245.4) .. (296.7,245.4) .. controls (244.95,245.4) and (203,203.45) .. (203,151.7) -- cycle ;
%Shape: Ellipse [id:dp9990319621255483] 
\draw  [dash pattern={on 0.84pt off 2.51pt}] (203.4,155) .. controls (203.4,143.95) and (245.04,135) .. (296.4,135) .. controls (347.76,135) and (389.4,143.95) .. (389.4,155) .. controls (389.4,166.05) and (347.76,175) .. (296.4,175) .. controls (245.04,175) and (203.4,166.05) .. (203.4,155) -- cycle ;
%Curve Lines [id:da598462727137555] 
\draw [color={rgb, 255:red, 74; green, 144; blue, 226 }  ,draw opacity=1 ]   (277.4,84) .. controls (290.4,61) and (325.4,62) .. (340.4,69) ;
%Curve Lines [id:da9095420159766927] 
\draw [color={rgb, 255:red, 74; green, 144; blue, 226 }  ,draw opacity=1 ] [dash pattern={on 4.5pt off 4.5pt}]  (340,69) .. controls (369.4,85) and (377.4,125) .. (364.4,162) .. controls (351.4,199) and (393.4,178) .. (367.4,213) ;
%Curve Lines [id:da6849808765334008] 
\draw [color={rgb, 255:red, 208; green, 2; blue, 27 }  ,draw opacity=1 ]   (278,88) .. controls (299.4,95) and (300.4,145.4) .. (343.4,148.4) ;
%Curve Lines [id:da479511576765675] 
\draw [color={rgb, 255:red, 208; green, 2; blue, 27 }  ,draw opacity=1 ]   (277.4,89) .. controls (291.4,100) and (274.4,190) .. (310.4,193) ;
%Shape: Circle [id:dp8639424440799469] 
\draw  [color={rgb, 255:red, 74; green, 144; blue, 226 }  ,draw opacity=1 ][fill={rgb, 255:red, 208; green, 2; blue, 27 }  ,fill opacity=1 ][line width=0.75]  (309.86,193.47) .. controls (309.86,191.97) and (311.07,190.76) .. (312.57,190.76) .. controls (314.06,190.76) and (315.27,191.97) .. (315.27,193.47) .. controls (315.27,194.96) and (314.06,196.18) .. (312.57,196.18) .. controls (311.07,196.18) and (309.86,194.96) .. (309.86,193.47) -- cycle ;
%Curve Lines [id:da5751497696281425] 
\draw [color={rgb, 255:red, 74; green, 144; blue, 226 }  ,draw opacity=1 ]   (315.4,194.4) .. controls (348.8,199.4) and (307.4,212.6) .. (305.4,234.6) ;
%Curve Lines [id:da4992668122887517] 
\draw [color={rgb, 255:red, 208; green, 2; blue, 27 }  ,draw opacity=1 ]   (238.4,157.2) .. controls (232.4,131.2) and (226.4,119.8) .. (243.4,107.8) .. controls (260.4,95.8) and (236.4,86.8) .. (272.4,86.8) ;
%Shape: Circle [id:dp3765937891193817] 
\draw  [color={rgb, 255:red, 74; green, 144; blue, 226 }  ,draw opacity=1 ][fill={rgb, 255:red, 208; green, 2; blue, 27 }  ,fill opacity=1 ][line width=0.75]  (235.86,156.47) .. controls (235.86,154.97) and (237.07,153.76) .. (238.57,153.76) .. controls (240.06,153.76) and (241.27,154.97) .. (241.27,156.47) .. controls (241.27,157.96) and (240.06,159.18) .. (238.57,159.18) .. controls (237.07,159.18) and (235.86,157.96) .. (235.86,156.47) -- cycle ;
%Curve Lines [id:da17294908444304313] 
\draw [color={rgb, 255:red, 74; green, 144; blue, 226 }  ,draw opacity=1 ]   (240.4,159.2) .. controls (249.4,185.2) and (237.4,211.2) .. (249.4,227.2) .. controls (261.4,243.2) and (312.4,248.2) .. (318.4,234.2) ;
%Curve Lines [id:da38034245376038756] 
\draw [color={rgb, 255:red, 208; green, 2; blue, 27 }  ,draw opacity=1 ]   (274.4,89) .. controls (266.4,117.6) and (240.4,118) .. (253.4,150) .. controls (266.4,182) and (254.4,188) .. (255.4,211) ;
%Shape: Circle [id:dp285330198550213] 
\draw  [color={rgb, 255:red, 74; green, 144; blue, 226 }  ,draw opacity=1 ][fill={rgb, 255:red, 208; green, 2; blue, 27 }  ,fill opacity=1 ][line width=0.75]  (253.86,213.47) .. controls (253.86,211.97) and (255.07,210.76) .. (256.57,210.76) .. controls (258.06,210.76) and (259.27,211.97) .. (259.27,213.47) .. controls (259.27,214.96) and (258.06,216.18) .. (256.57,216.18) .. controls (255.07,216.18) and (253.86,214.96) .. (253.86,213.47) -- cycle ;
%Curve Lines [id:da4782169460492418] 
\draw [color={rgb, 255:red, 74; green, 144; blue, 226 }  ,draw opacity=1 ]   (257.4,215.6) .. controls (258.4,232.6) and (280.4,232.6) .. (285.4,227.6) ;
%Straight Lines [id:da5967795418338653] 
\draw [color={rgb, 255:red, 208; green, 2; blue, 27 }  ,draw opacity=1 ] [dash pattern={on 0.84pt off 2.51pt}]  (262,148) -- (277.43,148.12) ;
%Straight Lines [id:da9955695325322751] 
\draw [color={rgb, 255:red, 208; green, 2; blue, 27 }  ,draw opacity=1 ] [dash pattern={on 0.84pt off 2.51pt}]  (282.4,206.6) -- (297.4,199.6) ;
%Curve Lines [id:da46619819635654525] 
\draw [color={rgb, 255:red, 208; green, 2; blue, 27 }  ,draw opacity=1 ] [dash pattern={on 0.84pt off 2.51pt}]  (257.4,208.2) .. controls (259.8,201.4) and (263.4,203.2) .. (270.4,209.2) ;

% Text Node
\draw (266,65.2) node [anchor=north west][inner sep=0.75pt]  [font=\normalsize,color={rgb, 255:red, 208; green, 2; blue, 27 }  ,opacity=1 ]  {$p$};
% Text Node
\draw (264,192) node [anchor=north west][inner sep=0.75pt]  [font=\normalsize,color={rgb, 255:red, 0; green, 0; blue, 0 }  ,opacity=1 ]  {$q$};
% Text Node
\draw (343,53) node [anchor=north west][inner sep=0.75pt]  [font=\normalsize,color={rgb, 255:red, 74; green, 144; blue, 226 }  ,opacity=1 ]  {$\ell $};
% Text Node
\draw (175,135) node [anchor=north west][inner sep=0.75pt]  [font=\normalsize,color={rgb, 255:red, 0; green, 0; blue, 0 }  ,opacity=1 ]  {$S_{\infty }^{2}$};
% Text Node
\shade[inner color=gray!40, outer color=white]
  (326,100) ellipse [x radius=24pt, y radius=15pt];

\draw (320,93) node [anchor=north west][inner sep=0.75pt]
  [font=\normalsize,color=black,opacity=1]
  {$\cE_n^\leftslice$};

\end{tikzpicture}
   \caption{A ``fat" fence}
    \label{Fig:fatFence}
\end{figure}
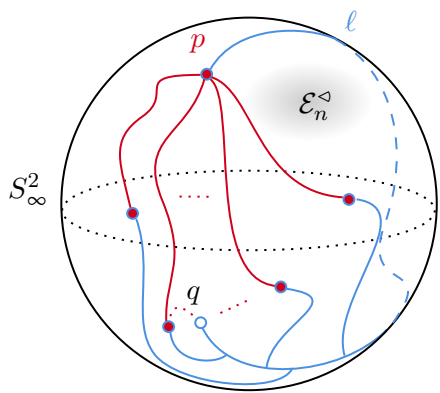

A \emph{fat fence} $\cF(g)$ of $g$ is defined to be the pair
\[
\cF(g)=(\ell,\{\cE_n\}_{n\in\ZZ}).
\]
Note that $S_\infty^2\setminus (\ell\cup \cE_n)$ is a disjoint union of two Jordan domains bounded by fences that intersect only along $\ell$ by \refprop{classifyConnector}, and that $\cE_\infty$ has empty interior, but is not necessarily a segment by \refprop{eclipseLProperty}.
See \reffig{fatFence}.
This explains the prefix ``fat.''

The fat fence $\cF(g)$ is $g$-invariant, namely,
\[
g(\cF(g))
=
\bigl(g(\ell),\{g(\cE_n)\}_{n\in\ZZ}\bigr)
=
(\ell,\{\cE_n\}_{n\in\ZZ})
=
\cF(g).
\]

A Jordan domain of $S_\infty^2\setminus (\ell\cup \cE_n)$ is called the \emph{left} (resp.\ \emph{right}) \emph{$n^{th}$ shadow} of $\cF(g)$, denoted by $\cE_n^\leftslice$ (resp.\ $\cE_n^\rightslice$), if it contains a left (resp.\ right) open neighborhood of a point of $\ell$ relative to $\ell$.
A Jordan domain $\fS$ is called an \emph{$n^{th}$ shadow} of $\cF(g)$, or simply a \emph{shadow}, if $\fS$ is either the left or the right $n^{th}$ shadow.

Note that $\partial \cE_n$, $\partial \cE_n^\leftslice$, and $\partial \cE_n^\rightslice$ are fences whose nodes lie in $Z^+$ by \refconst{eclipseL}.
A node of $\partial \cE_n$ is called \emph{left} (resp.\ \emph{right}) if it is contained in the $(n+1)^{th}$ left (resp.\ right) shadow.
Equip $(\partial \cE_n)^+$ with the linear order from the left node to the right node.
This determines a unique circular order on $\partial \cE_n$.
Under this order, $\Int \cE_n=\cD^L(\partial \cE_n)$.

\subsection{Unlinkedness of fat fences}
In order to construct a circle lamination on a quasi-Fuchsian limit circle with a fat fence, we first establish the unlinkedness of fat fences.
\begin{prop}[Unlinkedness between fences and fat fences]\label{Prop:fatNonFat}
Let $Z^\pm$ be a minimal zipper for $M$.
Assume that $\cF(g)=(\ell,\{\cE_n\}_{n\in\ZZ})$ is a fat fence for some $g\in G$, and that $\cJ$ is a fence of $Z^\pm$ whose nodes lie in $Z^+$.
Write $p$ for the pivot of $\ell$.
If $p \notin \cJ$, then $\Fix(g)$ lies in a Jordan domain $\cD$ of $\cJ$.
In particular, $\ell\subset \closure{\cD}$ and so $\ell$ and $\cJ^-$ are unlinked.
Moreover, $\cE_n\subset \cD$ and $\cJ$ is contained in the closure of either the left or right $n^{th }$ shadow of $\{\cE_n\}_{n\in \NN}$ for all sufficiently large $n$.
\end{prop}
\begin{proof}
Write $\Fix(g)=\{p,q\}$.
Since $q$ is neither a synapse or a point of $Z^\pm$ (\refprop{eclipseLProperty}), we have $q\notin \cJ$, so $q$ lies in a Jordan domain $\cD$ of $\cJ$.
Choose the circular order on $\cJ$ so that $\cD=\cD^L(\cJ)$.

We show that $p\in \cD$.
Since $q\in \cD$, there exists $N$ such that $\closure{\cE_n^-} \subset \cD$ for all $n>N$.
Fix $n>N$, and let $\{\sfn_1,\sfn_2\}$ be the set of nodes of $\partial \cE_n$.
Suppose that $p\in \cD^R(\cJ)$.
Then $\cJ$ separates each $\sfn_i$ from $p$, so by the disjointness of $Z^+$, each arc $[p,\sfn_i]\subset Z^+$ meets $\cJ^+$.
Hence $\cE_n^+\cup \cJ^+$ contains a Jordan curve in $Z^+$, contradicting the simple connectedness of $Z^+$.
Thus $p\notin \cD^R(\cJ)$, and since $p\notin \cJ$, we have $p\in \cD$.
Therefore $\Fix(g)\subset \cD$.

We now show that $\ell$ and $\cJ^-$ are unlinked.
If $\ell\cap \cJ^-=\emptyset$, then the disjointness of $Z^\pm$ implies $\ell\subset \cD$.
If $\ell\cap \cJ^-\neq\emptyset$, let $s\in \ell\cap \cJ^-$.
Applying \refprop{enclosing} to the two distinct end rays of $\ell$ starting at $s$, we obtain $\ell\subset \cD\cup \cJ^-$.
This proves the claim.

Finally, we prove the last statement.
For each $n>N$, since $\closure{\cE_n^-}\subset \cD$ and $p\in \cD$, another application of \refprop{enclosing} to each arc $[p,\sfn_i]$ yields $[p,\sfn_i]\subset \cD\cup \cJ^+$.
Hence either $\cE_n\subset \cD$ or $\cD^R(\cJ)\subset \cE_n$.
By the thinness of $\cE_\infty$ (see \refprop{eclipseLProperty}), the latter is impossible for large $n$, so $\cE_n\subset \cD$ for all sufficiently large $n$.
This completes the proof.
\end{proof}

A pair of fat fences, $\cF(g_i)=(\ell^i,\{\cE_n^i\}_{n\in\ZZ}), \ i \in \ZZ/2\ZZ$ , is said to be \emph{unlinked} if either $\ell^1=\ell^2$ or there are two shadows $\fS_1$ and $\fS_2$ of $\cF(g_1)$ and of $\cF(g_2)$, respectively, such that $(\fS_2)^c \subset \closure{\fS_1}$.

\begin{rmk}\label{Rmk:sameAxis}
Note that in the former case, $\Fix(g_1)=\Fix(g_2)$.
Since $G$ is discrete in $\PSL(\CC)$, there exists $h\in G$ that generates $\Stab{G}{\Fix(g_1)}$, and hence $\{g_1,g_2\}\subset \langle h\rangle$.
Therefore, $\ell^1$ is the axis in $Z^-$ of every nontrivial power $h^k$.
Moreover, if
\[
\cF(h^{n_1})=(\ell^1,\{\cE_n\}_{n\in \NN})
\quad\text{and}\quad
\cF(h^{n_2})=(\ell^1,\{\cI_n\}_{n\in \NN})
\]
are fat fences for some $n_i\neq 0$, then $\cI_\infty=\cE_\infty$ (\refprop{eclipseLProperty}).
Indeed, any fat fence
\(
\cF(h)=(\ell^1,\{\cE_n\}_{n\in \ZZ})
\)
of $h$ is preserved by every element of $\langle h\rangle$.
\qedhere

\end{rmk}
\begin{prop}\label{Prop:fatFenceSystem}(Invariant fat fence system)
Let $Z^\pm$ be a minimal zipper for $M$.
Assume that $\cF(g)=(\ell,\{\cE_n\}_{n\in\ZZ})$ is a fat fence of $g$.
Then, for any $h\in G$, the fat fences
$\cF(g)$ and $h(\cF(g))=\cF(hgh^{-1})=(h(\ell),\{h(\cE_n)\}_{n\in\ZZ})$
are unlinked.
Moreover, if $h(p)\neq p$, then there exists $m\in\ZZ$ such that either
$\closure{\cE_m^\leftslice\setminus \cE_{m-2}^\leftslice}$
or
$\closure{\cE_m^\rightslice\setminus \cE_{m-2}^\rightslice}$
contains the complement of a shadow of $h(\cF(g))$.
In particular, $\ell$ and $h(\ell)$ are unlinked.
\end{prop}

\begin{proof}
If $h=\id$, there is nothing to prove.
Assume therefore that $h$ is non-trivial.

If $h(p)=p$, then both $g$ and $hgh^{-1}$ fix $p$.
Since $G$ is discrete in $\PSL(\CC)$, it follows that $\Fix(g)=\Fix(hgh^{-1})=h(\Fix(g))$.
Hence, by \refrmk{sameAxis}, the fat fences $\cF(g)$ and $h(\cF(g))$ are unlinked.

It remains to consider the case $h(p)\neq p$.
Since $h(p)\neq p$ and $h(p)\in Z^+$, \refprop{eclipseLProperty} implies that $h(p)$ is contained in a shadow of $\cF(g)$.
By symmetry, we may assume that $h(p)\in \cE_n^\leftslice\setminus \cE_{n-1}^\leftslice$ for some $n\in\ZZ$.
Then $h(p)\in \cE_n^\leftslice\setminus \closure{\cE_{n-2}^\leftslice}=: \fD$.
Since $\partial \fD$ is a fence whose nodes lie in $Z^+$, applying \refprop{fatNonFat} to $h(\cF(g))$ and $\partial \fD$ completes the proof.
\end{proof}

\subsection{Circle laminations from fat fence systems}

We are now ready to construct the circle lamination induced by a fat fence.

\begin{const}\label{Const:fatFenceLami}
    Let $Z^\pm$ be a minimal zipper for $M$.
    Assume that $g\in G$ fixes exactly one point $p$ in $Z^+$ and acts freely on $Z^-$, with the $g$-axis $\ell$ two-sided branched.
    Let $\cF(g)=(\ell,\{\cE_n\}_{n\in\ZZ})$ be a fat fence of $g$.
    Write $\Fix(g)=\{p,q\}$.

    Take a closed quasi-Fuchsian surface group $\Gamma$ whose limit set $\Lambda(\Gamma)$ separates the fixed points $p,q$ of $g$ (\refthm{separatingQF}).
    Write $\cH$ for the Jordan domain bounded by $\Lambda(\Gamma)$ and containing $p$.
    Fix the circular order on $\Lambda(\Gamma)$ so that $\cH=\cD^L(\Lambda(\Gamma))$.
    There is a unique point $e_-\in \ell$ such that the end ray $\sfr_-$ of $\ell$ starting at $e_-$ and landing at $p$ intersects $\Lambda(\Gamma)$ exactly in $\{e_-\}$.

    Write $v_n^L$ and $v_n^R$ for the left and right nodes of $\partial\cE_n$, respectively.
    Since $q\in \cD^R(\Lambda(\Gamma))$, there exists $N\in\NN$ such that $\closure{(\partial\cE_n)^-}\subset \cD^R(\Lambda(\Gamma))$ for all $n\ge N$.
    For each $n\ge N$, the lines $(p,v_n^L)$ and $(p,v_n^R)$ in $(\partial\cE_n)^+$ intersect $\Lambda(\Gamma)$, and hence there exist points $x_n^L\in (p,v_n^L)$ and $x_n^R\in (p,v_n^R)$ such that $\opi[\partial\cE_n]{x_n^L}{x_n^R}$ crosses $\cH$.
    Observe that $\sfr_-$ lands at $p$ on the right side of $\opi[\partial\cE_n]{x_n^L}{x_n^R}$, and that \[x_m^L<x_n^L<e_-<x_n^R<x_m^R\] in $\Lambda(\Gamma)$ for all $n<m$.
    Therefore, $\{x_n^R\}_{n=N}^\infty$ and $\{x_n^L\}_{n=N}^\infty$ are strictly increasing and strictly decreasing sequences in $\Lambda(\Gamma)$ converging to $e_+^R$ and $e_+^L$, respectively.
    Note that $e_-<e_+^R\le e_+^L$, possibly with equality.

    Let $\sfa_n^L$ and $\sfa_n^R$ denote the connector arcs $\sfr_-\cup [p,x_n^L]$ and $\sfr_-\cup [p,x_n^R]$, respectively.
    Set $\ell^L=\{e_+^L,e_-\}$ and $\ell^R=\{e_+^R,e_-\}$.
    Define $\cL(\Gamma)$ to be the closure in $\cM$ of the $\Gamma$-orbit of $\ell^L$.
    \qedhere
\end{const}

\begin{rmk}[Positive endpoints are in $\cE_\infty$]\label{Rmk:positiveEndInInfty}
Note that $\{x_n^L,x_n^R\}\subset \Int \cE_{n-1}$ for all $n>N$.
Since $e_+^R$ and $e_+^L$ are the limit points of $\{x_n^R\}_{n=N}^\infty$ and $\{x_n^L\}_{n=N}^\infty$, respectively, we have $\{e_+^R,e_+^L\}\subset \cE_\infty$.
\end{rmk}

\begin{prop}[Well-defined circle laminations from fat fences]\label{Prop:lamiFromFatFence}
In \refconst{fatFenceLami}, $\ell^L$ is unlinked with $h(\ell^L)$ for all $h\in G$.
In particular, $\cL(\Gamma)$ is a $\Gamma$-invariant circle lamination of $\Lambda(\Gamma)$.
\end{prop}
\begin{proof}
Assume that $\ell^L$ is unlinked with $h(\ell^L)$ for some $h\in G$.
By symmetry, we may assume that $e_+^L<h(e_+^L)<e_-<h(e_-)$ in $\Lambda(\Gamma)$.
Then, for all sufficiently large $n\in\ZZ$, we have
$e_+^L<x_n^L<h(e_+^L)<h(x_n^L)<e_-$.
Hence there exists $K\in\NN$ such that, for all $n,m\ge K$, $\sfd \sfa_n^L$ and $\sfd h(\sfa_m^L)$ are linked.

We claim that $h(p)\in \cH^R(\sfa_n^L)$ for all $n\ge K$.
Suppose first that $h(p)\in \cH^L(\sfa_n^L)$ for some $n\ge K$.
Then, by the disjointness of $Z^\pm$, the rays $\sfr_-$ and $h(\sfr_-)$ are linked, contradicting \refprop{fatFenceSystem}.
Moreover, $h(p)\neq p$ since $\Fix(h)\subset \Lambda(\Gamma)$.
If $h(p)\in (p,x_n^L)$ for some $n\ge K$, then $h(\sfr_-)$ intersects $(p,x_{n+1}^L)$, again contradicting the disjointness of $Z^\pm$.
Hence $h(p)\notin \closure{\cH^L(\sfa_n^L)}$ for all $n\ge K$, and therefore $h(p)\in \cH^R(\sfa_n^L)$ for all $n\ge K$.

On the other hand, $p\notin h(\sfa_n^L)$ for all $n\ge K$.
Indeed, suppose that $p\in h(\sfa_n^L)$ for some $n\ge K$.
Then the union of $\sfa_n^L$, $h(\sfa_n^L)$, and $h(\sfa_{n+1}^L)$ contains a Jordan curve, contradicting the simple connectedness of $Z^+$.

Since $p\notin h(\sfa_n^L)$ and $h(p)\in \cH^R(\sfa_n^L)$ for all $n\ge K$, the union of $h(\sfa_K^L)$, $h(\sfa_{K+1}^L)$, and $\sfa_K^L$ contains a Jordan curve, again contradicting the simple connectedness of $Z^+$.
This completes the proof.
\end{proof}

Before proving the isolatedness of $\ell^L$, we prove the following auxiliary proposition.

\begin{prop}[Linked fat leaves]\label{Prop:fatLeafLink}
Let $Z^\pm$ be a zipper for $M$, and let $\cJ$ be an oriented Jordan curve in $S_\infty^2$.
Assume that $\sfa_1,\sfa_2,\sfb_1,\sfb_2$ are connector arcs satisfying the following:
\begin{itemize}
    \item each crosses $\cD^L(\cJ)$, and its synapse lies in $Z^+$;
    \item $\sfa_1^-=\sfa_2^-$ and $\sfb_1^-=\sfb_2^-$;
    \item $\sfa_1^+\cap \sfa_2^+=\{p\}$ and $\sfb_1^+\cap \sfb_2^+=\{q\}$; that is, the $(+)$-segments intersect only at the synapses $p$ and $q$.
\end{itemize}
Write $\sfa_1^-\cap \cJ=\{e^-\}$, $\sfb_1^-\cap \cJ=\{d^-\}$, $\sfa_i^+\cap \cJ=\{e_i^+\}$, and $\sfb_i^+\cap \cJ=\{d_i^+\}$.
If $d^-<e_1^+<e_2^+<d_1^+<d_2^+<e^-$ in $\cJ$, then either $p=q$ or $\sfa_1^-$ and $\sfb_1^-$ are linked.
\end{prop}
\begin{proof}
    Write $\fD=\cD^L(\cJ)$.
    It suffices to consider the case where $p\neq q$.

    We first claim that $p\notin \closure{\fD^L(\sfb_2)}$.
    Assume that $p\in \closure{\fD^L(\sfb_2)}$.
    By the disjointness of $Z^\pm$, we then have $p\in \fD^L(\sfb_2)\cup \sfb_2^+$.
    Since $p\neq q$ and $p\in \fD^L(\sfb_1)$, it follows that $\partial \fD^L(\sfb_1)$ separates $p$ from $e_i^+$ for all $i\in\{1,2\}$.
    By the disjointness of $Z^\pm$, both $\sfa_i^+$ intersect $\sfb_1^+$.
    Hence $\sfb_1^+$ contains an arc $\gamma$ joining a point of $\sfa_1^+\cap \sfb_1^+$ to a point of $\sfa_2^+\cap \sfb_1^+$.
    Since $\sfa_1^+\cap \sfa_2^+=\{p\}$ and $\sfa_1^+\cup \sfa_2^+$ is an arc in $Z^+$ crossing $\fD$, the arc $\gamma$ contains $p$.
    Thus $\sfb_1^+$ contains $p$, a contradiction since $p\in \fD^L(\sfb_1)$.

    Therefore, $p\in \fD^R(\sfb_2)$.
    This implies that $\sfa_1^-$ and $\sfb_1^-$ are linked, since these are rays in $Z^-$.
    Thus the proof is complete.
\end{proof}

\begin{lem}[The generating leaf is isolated]\label{Lem:isolatedLeftLeaf}
    In \refconst{fatFenceLami}, $\ell^L$ is isolated in $\cL(\Gamma)$ and no non-trivial element in $\Gamma$ fixes  $e_+^L$.
\end{lem}
\begin{proof}
    First observe that there is no element $h\in \Gamma$ such that $h(e_+^L)=e_+^L$ and $h(e_-)=e_-$.
    Indeed, suppose that such an $h$ exists.
    Without loss of generality, we may assume that $r(h)=e_-$ and $a(h)=e_+^L$.
    For each $n\ge N$, the set $\cH^L(\opi[\partial \cE_n]{x_n^L}{x_n^R})$ is an open left neighborhood of $e_+^L$ relative to $\opi[\Lambda(\Gamma)]{x_n^R}{x_n^L}$.
    Since $a(h)=e_+^L$ and $h(\cH)=\cH$, for each $n\ge N$ there exists a sufficiently large $k$ such that $h^k(p)\in \cH^L(\opi[\partial \cE_n]{x_n^L}{x_n^R})$.
    Hence $h^k(\sfr_-)$ intersects $\opi[\partial \cE_n]{x_n^L}{x_n^R}$, contradicting the disjointness of $Z^\pm$.
    Therefore, there is no element of $\Gamma$ fixing both endpoints of $\ell^L$.

    We next claim that $e_+^L$ is not fixed by any non-trivial element of $\Gamma$.
    Suppose that there exists $h\in \Gamma$ such that $h(e_+^L)=e_+^L$.
    Replacing $h$ by $h^{-1}$ if necessary, we may assume that $h(e_-)\in \opi[\Lambda(\Gamma)]{e_+^L}{e_-}$.
    Then there exist $n,m\in\NN$ with $m>2$ such that
    $e_+^L=h(e_+^L)<h(x_{m+1}^L)<h(x_m^L)<x_{n+1}^L<x_n^L<h(e_-)<e_-$ in $\Lambda(\Gamma)$.
    Since $\Fix(h)\subset \Lambda(\Gamma)$ and so $h(p)\ne p$,   \refprop{fatLeafLink} implies that $\sfr_-$ and $h(\sfr_-)$ are linked.
    This contradicts \refprop{fatFenceSystem}.
    Thus no non-trivial element of $\Gamma$ fixes $e_+^L$.

    We now show that $\ell^L$ is isolated in $\cL(\Gamma)$.
    If some non-trivial element of $\Gamma$ fixes $e_-$, then we are done by \refprop{oneEndFix}.
    Hence we may assume that no non-trivial element of $\Gamma$ fixes $e_-$.

    We first show that $\opi[\Lambda(\Gamma)]{e_+^L}{e_-}$ is isolated in $\cL(\Gamma)$.
    Suppose that there exists an $\opi[\Lambda(\Gamma)]{e_+^L}{e_-}$-side sequence $\{\ell_i\}_{i\in\NN}$ of leaves in $\cL(\Gamma)$ such that $\ell_i=h_i(\ell^L)$ for some $h_k\in \Gamma$.
    If infinitely many $\ell_i$ have the same endpoints $\sfe$, then $\sfe\in \cL(\Gamma)$.
    Since $\ell_i\to \ell^L$, it follows that $\sfe$ shares an endpoint with $\ell^L$.
    This contradicts \refprop{oneEndFix}, because no non-trivial element of $\Gamma$ fixes an endpoint of $\ell^L$.
    Passing to a subsequence, we may therefore assume that $\{\ell_i\}_{i\in\NN}$ is pairwise disjoint and contained in $\opi[\Lambda(\Gamma)]{e_+^L}{e_-}$.

   We can choose $\sfn,\sfk,\sfm$ such that either
\[
h_{\sfk}(e_-)<x_{\sfn+1}^L<x_\sfn^L<h_{\sfk}(x_{\sfm+1}^L)<h_{\sfk}(x_{\sfm}^L)<e_-
\quad \text{or}\quad
e_-<h_{\sfk}(x_{\sfm+1}^L)<h_{\sfk}(x_{\sfm}^L)<x_{\sfn+1}^L<x_\sfn^L<h_{\sfk}(e_-)
\]
in $\Lambda(\Gamma)$.
In either case, since $h_{\sfk}(p)\neq p$, \refprop{fatLeafLink} implies that $h_{\sfk}(\sfr_-)$ and $\sfr_-$ are linked, contradicting \refprop{fatFenceSystem}.
Therefore, $\opi[\Lambda(\Gamma)]{e_+^L}{e_-}$ is isolated in $\cL(\Gamma)$.

If $e_+^R=e_+^L$, then by a symmetric argument, $\opi[\Lambda(\Gamma)]{e_-}{e_+^L}$ is also isolated in $\cL(\Gamma)$, completing the proof.

Assume now that $e_+^R\neq e_+^L$ and that $\opi[\Lambda(\Gamma)]{e_-}{e_+^L}$ is not isolated in $\cL(\Gamma)$.
As above, there exists an $\opi[\Lambda(\Gamma)]{e_-}{e_+^L}$-side sequence $\{\lambda_i\}_{i\in\NN}$ such that $\lambda_i=h_i(\ell^L)$ and $\lambda_i\subset \opi[\Lambda(\Gamma)]{e_-}{e_+^L}$ for all $i\in\NN$.
As above, we can choose $\sfn,\sfk,\sfm$ such that either
\[
h_{\sfk}(e_-)<x_\sfn^R<x_{\sfn+1}^R<h_{\sfk}(x_{\sfm+1}^L)<h_{\sfk}(x_{\sfm}^L)<e_-
\quad \text{or}\quad
e_-<h_{\sfk}(x_{\sfm+1}^L)<h_{\sfk}(x_{\sfm}^L)<x_\sfn^R<x_{\sfn+1}^R<h_{\sfk}(e_-)
\]
in $\Lambda(\Gamma)$.
Again this contradicts \refprop{fatLeafLink} and \refprop{fatFenceSystem}.
Therefore, $\opi[\Lambda(\Gamma)]{e_-}{e_+^L}$ is isolated in $\cL(\Gamma)$, and hence $\ell^L$ is isolated in $\cL(\Gamma)$.
\end{proof}

Now, we modify \reflem{dragging} for fat fences:
\begin{lem}[Dragging limit points by fat fences]\label{Lem:draggingByFatFence}
In the setting of \refconst{fatFenceLami}, let $\{g_i\}_{i\in\NN}$ be a sequence in $\Gamma$, and let $x_i,y_i$ be distinct points of $g_i(\cE_\infty\cup \ell)$ that are not fixed points of $g$.
Assume that $g_i(p)\neq p$ for all $i\in\NN$, and that $\{x_i\}_{i\in\NN}$ is a sequence in $\Int \cE_j$ converging to $x\in \cE_\infty\setminus \Fix(g)$.
If $\{y_i\}_{i\in\NN}$ converges to a point $y$, then $y\in \cE_\infty$.
\end{lem}
\begin{proof}
For each $i\in\NN$, by \refprop{fatFenceSystem}, there is $k(i)\in \ZZ$ such that $\closure{\cE_{k(i)}\setminus \cE_{k(i)-2}}$ contains the complement of a shadow of $h(\cF(g))$.
Hence, $\{x_i,y_i\}\subset \closure{\cE_{k(i)}\setminus \cE_{k(i)-2}}$ for all $i\in\NN$.
Since by \refprop{eclipseLProperty}, $\closure{\cE_{k(i)}\setminus \cE_{k(i)-2}}\cap \cE_\infty=\{p\}$ and $x_i\to x\in \cE_\infty$ as $i\to\infty$, we have $k(i)\to+\infty$ as $i\to\infty$.
This implies that $y\in \cE_\infty$ since $\{x_i,y_i\}\subset \closure{\cE_{k(i)}\setminus \cE_{k(i)-2}}$ for all $i\in\NN$. \qedhere
\end{proof}

We then follow the strategy of the proof of \refthm{twoInOne}, together with \refconst{fatFenceLami}, to prove the main theorem of this section:
\begin{thm}[One-sided simpliciality of the axis]\label{Thm:simplicialAxis}
Let $Z^\pm$ be a minimal zipper for $M$. Assume that $g\in G$ fixes a unique point $p$ in $Z^+$ and acts freely on $Z^-$.
Then the $g$-axis $\ell$ in $Z^-$ is one-sided inaccessible.
\end{thm}

\begin{proof}
    By \refrmk{noSeg} and \reflem{sideSimplicial}, it suffices to show that $\ell$ is not two-sided branched.
    Assume for contradiction that $\ell$ is two-sided branched.
    By \refconst{fatFenceLami} and \refprop{lamiFromFatFence}, applied to a fat fence $\cF(g)=(\ell,\{\cE_n\}_{n\in \ZZ})$ of $g$, we obtain a $\Gamma$-invariant circle lamination $\cL(\Gamma)$ on the limit set $\Lambda(\Gamma)$ of a quasi-Fuchsian closed surface group $\Gamma$.
    For convenience, we use the notation of \refconst{fatFenceLami}.
    By \refrmk{realization}, there exists a unique geodesic lamination $\cG(\Gamma)$ on $\cH/\Gamma$ whose lift $\widetilde{\cG}(\Gamma)$ to $\cH$ is the geometric realization of $\cL(\Gamma)$.

    By \reflem{isolatedLeftLeaf}, $\ell^L$ is isolated in $\cL(\Gamma)$, and there is no element of $\Gamma$ that fixes $e_+^L$.
    Applying \reflem{closureOfLeaf} to $\cG(\Gamma)$, we see that $\cG''(\Gamma)\neq\varnothing$ and that there is a unique principal region $U$ of $\cG''(\Gamma)$ such that either the bi-infinite geodesic leaf $\sfg$ of $\cG(\Gamma)$ associated with $\ell^L$ is a boundary leaf of $U$, or $\sfg$ is contained in $U$, while the end associated with $e_+^L$ escapes toward an end of $U$.
Hence, by \refprop{gapShape}, we have two possible cases:
\begin{enumerate}
    \item\label{Itm:idealPolygon2}($U$ is a finite ideal polygon) there are circularly ordered finitely many points in $\Lambda(\Gamma)$,
    \[
    u_{-\sfm}<u_{-(\sfm-1)}<\cdots<u_0<\cdots<u_{\sfn-1}<u_\sfn
    \]
    for some $\sfm,\sfn\geq 2$, such that $\{u_i,u_{i+1}\}\in \cL(\Gamma)$ are boundary leaves but are not isolated, $u_0=e_+^L$, and $u_{-\sfm}=u_\sfn=e_-$;

    \item\label{Itm:crown2}($U$ admits a core) there are an element $h\in \Gamma$, $k\in \NN$, and a circularly ordered bi-infinite sequence $\{u_i\}_{i\in \ZZ}$ in $\Lambda(\Gamma)$ such that
    \[
    r(h)<\cdots<u_{-i}<\cdots<u_0<\cdots<u_i<\cdots<a(h)\text{ in }\Lambda(\Gamma),
    \]
    $\{u_i,u_{i+1}\}\in \cL(\Gamma)$ are boundary leaves but are not isolated, $u_0=e_+^L$, and $h(u_i)=u_{i+k}$ for all $i\in \ZZ$.
\end{enumerate}

We first consider \refitm{idealPolygon2}. By the construction of $\cL(\Gamma)$, there is a sequence $\{g_i\}_{i\in \NN}$ of elements of $\Gamma$ such that $\{g_i(\ell^L)\}_{i\in\NN}$ is a $\opi[\Lambda(\Gamma)]{u_0}{u_1}$-side sequence in $\cL(\Gamma)$. Note that  $g_i(p)\neq p$ since $\Fix(g_i)\subset \Lambda(\Gamma)$. Since by \refrmk{positiveEndInInfty}, $e_+^L=u_0\in \cE_\infty\setminus \Fix(g)$ and $e_-\in \ell$, applying \refprop{fatFenceSystem} and \reflem{draggingByFatFence} to $\{g_i(\ell^L)\}_{i\in\NN}$ gives $u_1\in \cE_\infty$. Moreover, $u_1\in \cE_\infty\setminus \Fix(g)$ since $\Fix(g)\cap \Lambda(\Gamma)=\varnothing$.

Repeating the same argument inductively for $\{u_{i-1},u_{i}\}$, we obtain $u_i\in \cE_\infty\setminus \Fix(g)$ for all $i\in\{2,\ldots,\sfn\}$, and hence $e_-\in \cE_\infty$. This contradicts \refprop{eclipseLProperty}, since $\cE_\infty\cap \ell=\varnothing$. Thus \refitm{idealPolygon2} cannot occur.

We next consider \refitm{crown2}. There are two possibilities for the position of $e_-$: either $e_-=u_i$ for some $i\notin \{-1,0,1\}$, or $e_-\in \cldi[\Lambda(\Gamma)]{a(h)}{r(h)}$.

Suppose first that $e_-=u_i$ for some $i\notin \{-1,0,1\}$. Then $u_0=e_+^L$ and $u_i=e_-$ are connected by a finite chain of boundary leaves of the form $\{u_{j-1},u_{j}\}$. Since each interval $\opi[\Lambda(\Gamma)]{u_{j-1}}{u_{j}}$ is not isolated in $\cL(\Gamma)$, we may apply \reflem{draggingByFatFence} successively along this chain, exactly as above. This rules out the possibility that $e_-=u_i$ for some $i\notin \{-1,0,1\}$.

It remains to consider the case where $e_-\in \cldi[\Lambda(\Gamma)]{a(h)}{r(h)}$. Observe that $h(\ell^L)$ is also a leaf of $\cL(\Gamma)$ and $h(e_+^L)=h(u_0)=u_k$. Applying the same argument to the consecutive boundary leaves $\{u_{j-1},u_{j}\}$ for $j=1,\dots,k$, we obtain $u_k\in \cE_\infty\setminus \Fix(g)$. On the other hand, by \refprop{fatFenceSystem}, the complement of a shadow of $h(\cF(g))$ is contained in $\closure{\cE_j\setminus \cE_{j+2}}$ for some $j\in \ZZ$, and hence $u_k\in \closure{\cE_j\setminus \cE_{j+2}}$. By \refprop{eclipseLProperty}, this contradicts $u_k\in \cE_\infty\setminus \Fix(g)$. Thus \refitm{crown2} also cannot occur.
This completes the proof.
\end{proof}

\section{No Prong in $Z^+$ with a One-Sided Simplicial Axis in $Z^-$}\label{Sec:noprongOneside}
In this final section, we complete the proof of \refthm{oneInOne} by treating the third case described at the beginning of \refsec{oneprong}. In \refsec{oneprong}, because of the failure of two-sided inaccessibility of the invariant $g$-prong, we had to analyze the bouncing direction carefully in \refsec{uniformBouncing}. In \refsec{noprongTwoside}, because there is no invariant $g$-prong, and hence no invariant fence, we introduced fat fences. In the present case, both difficulties arise: there is no $g$-prong by \refthm{oneProng}, and the $g$-axis $\ell$ is one-sided simplicial by \refthm{simplicialAxis}. We therefore combine the modifications developed in \refsec{oneprong} and \refsec{noprongTwoside}.

We begin by introducing a modified invariant fence, called a \emph{squashed fence}.

\subsection{Squashed Fences}
Let $Z^\pm$ be a minimal zipper for $M$. Assume that $g\in G$ fixes exactly one point $p$ in $Z^+$ and acts freely on $Z^-$. By \refthm{simplicialAxis}, the $g$-axis $\ell$ is one-sided inaccessible. Hence, there are no left or right branches at $p$. Recall that the canonical orientation on $\ell$ is the linear order increasing toward $p$.

A \emph{squashed fence} $\fF(g)$ of $g$ is defined to be a pair $(\ell,\{\cE_n\}_{n\in \ZZ})$ such that $\{\cE_n\}_{n\in \ZZ}$ is an eclipse along $\ell$, constructed in \refconst{eclipseL}. The squashed fence $\fF(g)$ is $g$-invariant, that is,
\[
g(\fF(g))=(g(\ell),\{g(\cE_n)\}_{n\in\ZZ})=(\ell,\{\cE_n\}_{n\in\ZZ})=\fF(g).
\]
One should think of $\closure{\ell}$ as a degenerate 2-fold fence bounding a type~I Jordan domain with pivot $p$.
See \reffig{squashedFence}.
This explains the prefix ``squashed''.

\begin{figure}[htpb]
    \centering
\tikzset{every picture/.style={line width=0.75pt}} %set default line width to 0.75pt        

\begin{tikzpicture}[x=0.75pt,y=0.75pt,yscale=-1,xscale=1]
%uncomment if require: \path (0,349); %set diagram left start at 0, and has height of 349

%Shape: Circle [id:dp37561613816206396] 
\draw  [color={rgb, 255:red, 74; green, 144; blue, 226 }  ,draw opacity=1 ][fill={rgb, 255:red, 208; green, 2; blue, 27 }  ,fill opacity=1 ][line width=0.75]  (238.86,100.47) .. controls (238.86,98.97) and (240.07,97.76) .. (241.57,97.76) .. controls (243.06,97.76) and (244.27,98.97) .. (244.27,100.47) .. controls (244.27,101.96) and (243.06,103.18) .. (241.57,103.18) .. controls (240.07,103.18) and (238.86,101.96) .. (238.86,100.47) -- cycle ;
%Shape: Circle [id:dp47413110758019605] 
\draw  [color={rgb, 255:red, 74; green, 144; blue, 226 }  ,draw opacity=1 ][fill={rgb, 255:red, 255; green, 255; blue, 255 }  ,fill opacity=1 ][line width=0.75]  (222.86,230.47) .. controls (222.86,228.97) and (224.07,227.76) .. (225.57,227.76) .. controls (227.06,227.76) and (228.27,228.97) .. (228.27,230.47) .. controls (228.27,231.96) and (227.06,233.18) .. (225.57,233.18) .. controls (224.07,233.18) and (222.86,231.96) .. (222.86,230.47) -- cycle ;
%Shape: Circle [id:dp13030945615819356] 
\draw  [color={rgb, 255:red, 74; green, 144; blue, 226 }  ,draw opacity=1 ][fill={rgb, 255:red, 208; green, 2; blue, 27 }  ,fill opacity=1 ][line width=0.75]  (211.86,232.47) .. controls (211.86,230.97) and (213.07,229.76) .. (214.57,229.76) .. controls (216.06,229.76) and (217.27,230.97) .. (217.27,232.47) .. controls (217.27,233.96) and (216.06,235.18) .. (214.57,235.18) .. controls (213.07,235.18) and (211.86,233.96) .. (211.86,232.47) -- cycle ;
%Curve Lines [id:da6120957118068008] 
\draw [color={rgb, 255:red, 74; green, 144; blue, 226 }  ,draw opacity=1 ]   (240.4,102.6) .. controls (219.4,139.6) and (240.4,171.6) .. (226.4,228.6) ;
%Shape: Circle [id:dp9153658817180073] 
\draw   (161,173.7) .. controls (161,121.95) and (202.95,80) .. (254.7,80) .. controls (306.45,80) and (348.4,121.95) .. (348.4,173.7) .. controls (348.4,225.45) and (306.45,267.4) .. (254.7,267.4) .. controls (202.95,267.4) and (161,225.45) .. (161,173.7) -- cycle ;
%Shape: Ellipse [id:dp5800413884253695] 
\draw  [dash pattern={on 0.84pt off 2.51pt}] (161.4,177) .. controls (161.4,165.95) and (203.04,157) .. (254.4,157) .. controls (305.76,157) and (347.4,165.95) .. (347.4,177) .. controls (347.4,188.05) and (305.76,197) .. (254.4,197) .. controls (203.04,197) and (161.4,188.05) .. (161.4,177) -- cycle ;
%Curve Lines [id:da9362137771419227] 
\draw [color={rgb, 255:red, 208; green, 2; blue, 27 }  ,draw opacity=1 ] [dash pattern={on 4.5pt off 4.5pt}]  (276.4,82.6) .. controls (309.8,93.2) and (314.4,138) .. (301.4,175) .. controls (288.4,212) and (347.4,203.6) .. (321.4,238.6) ;
%Curve Lines [id:da8632431240286986] 
\draw [color={rgb, 255:red, 208; green, 2; blue, 27 }  ,draw opacity=1 ]   (239.4,99.6) .. controls (205.4,95.6) and (203.4,112.6) .. (188.4,128.6) .. controls (173.4,144.6) and (160.4,175.6) .. (183.4,183.6) .. controls (206.4,191.6) and (190.4,219.6) .. (207.4,215.6) ;
%Curve Lines [id:da27020448490298454] 
\draw [color={rgb, 255:red, 208; green, 2; blue, 27 }  ,draw opacity=1 ]   (205.4,171.2) .. controls (199.4,145.2) and (193.4,133.8) .. (210.4,121.8) .. controls (227.4,109.8) and (203.4,100.8) .. (239.4,100.8) ;
%Shape: Circle [id:dp8230622608568857] 
\draw  [color={rgb, 255:red, 74; green, 144; blue, 226 }  ,draw opacity=1 ][fill={rgb, 255:red, 208; green, 2; blue, 27 }  ,fill opacity=1 ][line width=0.75]  (203.86,171.47) .. controls (203.86,169.97) and (205.07,168.76) .. (206.57,168.76) .. controls (208.06,168.76) and (209.27,169.97) .. (209.27,171.47) .. controls (209.27,172.96) and (208.06,174.18) .. (206.57,174.18) .. controls (205.07,174.18) and (203.86,172.96) .. (203.86,171.47) -- cycle ;
%Curve Lines [id:da05235088198418669] 
\draw [color={rgb, 255:red, 208; green, 2; blue, 27 }  ,draw opacity=1 ]   (213.4,233.6) .. controls (205.4,262.2) and (229.4,250.6) .. (250.4,254.6) .. controls (271.4,258.6) and (303.4,255.6) .. (317.4,243.6) ;
%Shape: Circle [id:dp0387474551879331] 
\draw  [color={rgb, 255:red, 74; green, 144; blue, 226 }  ,draw opacity=1 ][fill={rgb, 255:red, 208; green, 2; blue, 27 }  ,fill opacity=1 ][line width=0.75]  (205.86,216.47) .. controls (205.86,214.97) and (207.07,213.76) .. (208.57,213.76) .. controls (210.06,213.76) and (211.27,214.97) .. (211.27,216.47) .. controls (211.27,217.96) and (210.06,219.18) .. (208.57,219.18) .. controls (207.07,219.18) and (205.86,217.96) .. (205.86,216.47) -- cycle ;
%Curve Lines [id:da08268499151499409] 
\draw [color={rgb, 255:red, 74; green, 144; blue, 226 }  ,draw opacity=1 ]   (208.4,173.6) .. controls (218.4,175.6) and (214.4,188.6) .. (231.4,185.6) ;
%Curve Lines [id:da10724838593899855] 
\draw [color={rgb, 255:red, 74; green, 144; blue, 226 }  ,draw opacity=1 ]   (211.4,215.6) .. controls (221.4,217.6) and (214.4,201.6) .. (231.4,198.6) ;
%Curve Lines [id:da2352955400928708] 
\draw [color={rgb, 255:red, 208; green, 2; blue, 27 }  ,draw opacity=1 ]   (244.4,99.6) .. controls (261.4,92.6) and (257.4,80.6) .. (274.4,82.6) ;
%Curve Lines [id:da2534607927255955] 
\draw [color={rgb, 255:red, 74; green, 144; blue, 226 }  ,draw opacity=1 ]   (216.4,231.6) .. controls (223.4,221.6) and (217.4,216.6) .. (229.4,214.6) ;
%Curve Lines [id:da7230066601358913] 
\draw [color={rgb, 255:red, 74; green, 144; blue, 226 }  ,draw opacity=1 ]   (221.4,237.6) .. controls (217.4,232.6) and (221.4,223.6) .. (227.4,222.6) ;
%Shape: Circle [id:dp6120974390183166] 
\draw  [color={rgb, 255:red, 74; green, 144; blue, 226 }  ,draw opacity=1 ][fill={rgb, 255:red, 208; green, 2; blue, 27 }  ,fill opacity=1 ][line width=0.75]  (219.86,238.47) .. controls (219.86,236.97) and (221.07,235.76) .. (222.57,235.76) .. controls (224.06,235.76) and (225.27,236.97) .. (225.27,238.47) .. controls (225.27,239.96) and (224.06,241.18) .. (222.57,241.18) .. controls (221.07,241.18) and (219.86,239.96) .. (219.86,238.47) -- cycle ;
%Curve Lines [id:da7505567525356747] 
\draw [color={rgb, 255:red, 208; green, 2; blue, 27 }  ,draw opacity=1 ]   (225.4,239.6) .. controls (251.4,243.6) and (262.4,191.6) .. (249.4,171.6) .. controls (236.4,151.6) and (258.4,123.4) .. (243.4,102.6) ;

% Text Node
\draw (237,82) node [anchor=north west][inner sep=0.75pt]  [font=\normalsize,color={rgb, 255:red, 208; green, 2; blue, 27 }  ,opacity=1 ]  {$p$};
% Text Node
\draw (231,220) node [anchor=north west][inner sep=0.75pt]  [font=\normalsize,color={rgb, 255:red, 0; green, 0; blue, 0 }  ,opacity=1 ]  {$q$};
% Text Node
\draw (221,125) node [anchor=north west][inner sep=0.75pt]  [font=\normalsize,color={rgb, 255:red, 74; green, 144; blue, 226 }  ,opacity=1 ]  {$\ell $};
% Text Node
\draw (133,157) node [anchor=north west][inner sep=0.75pt]  [font=\normalsize,color={rgb, 255:red, 0; green, 0; blue, 0 }  ,opacity=1 ]  {$S_{\infty }^{2}$};

\end{tikzpicture}
    \caption{A ``squashed fence" with $\ell$ right inaccessible}
    \label{Fig:squashedFence}
\end{figure}
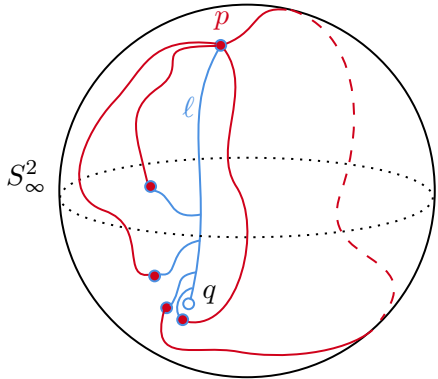

From now on, we study the unlinkedness of squashed fences.
\begin{prop}(Unlinkedness between fences and squashed fences)\label{Prop:squashedNonSquashed}
Let $Z^\pm$ be a minimal zipper for $M$. Assume that $\fF(g)=(\ell,\{\cE_n\}_{n\in \ZZ})$ is a squashed fence of $g\in G$ and $\cJ$ is a fence of the zipper, whose nodes are in $Z^+$.
Write $p$ for the pivot of $\ell$ and $\Fix(g)=\{p,q\}$.
Then, a Jordan domain $\cD$ of $\cJ$, contains $q$, and  $\closure{\cD}$ contains $\cE_n$ for all sufficiently large $n$.
\end{prop}
\begin{proof}
We first observe that $\ell$ is contained in the closure of a Jordan domain of $\cJ$. Assume not. Choose a circular order on $\cJ$. Then there exist $x_L,x_R\in \ell$ such that $x_\alpha\in \cD^\alpha(\cJ)$ for each $\alpha\in\{L,R\}$. Since $\cJ$ separates $x_L,x_R$, the arc $[x_L,x_R]\subset Z^-$ intersects $\cJ$. By the disjointness of $Z^\pm$, it does not intersect $\cJ^+$. Hence $[x_L,x_R]$ intersects $\cJ^-$ and is linked with $\cJ^-$. This contradicts the one-sided inaccessibility of $\ell$ (\refthm{simplicialAxis}). Thus there exists a Jordan domain $\cD$ of $\cJ$ whose closure contains $\ell$.

For each $n\in\ZZ$, let $\sigma_n$ be the connector arc in $\partial \cE_n$ crossing $\ell$. Since $q\notin Z^\pm$, we have $q\notin \cJ$ and $q\in \cD$. Hence there exists $N\in\NN$ such that $\sigma_n^-\subset \cD$ for all $n\ge N$. Fix $n\ge N$, and write $\partial\sigma_n^+=\{p,q_n\}$.

Now we claim that $\partial \cE_n\subset \overline{\cD}$ for every $n\ge N$. Hence either $\cE_n\subset \cD$ or $\cD\subset \cE_n$. By the thinness of $\cE_\infty$ in \refprop{eclipseLProperty}, we have $\cE_n\subset \cD$ for all sufficiently large $n$. This completes the proof.

If $p\in \cJ^+$, then \refprop{enclosing} gives $\sigma_n^+\subset \overline{\cD}$, and so $\partial \cE_n\subset \overline{\cD}$. Assume that $p\notin \cJ^+$, so $p\in \cD$. If $\sigma_n^+\cap \cJ^+=\varnothing$, then $\partial \cE_n\subset \overline{\cD}$. Otherwise, choose $s\in \sigma_n^+\cap \cJ^+$. Applying \refprop{enclosing} to the subarcs $[s,p]$ and $[s,q_n]$ of $\sigma_n^+$, we obtain $\partial \cE_n\subset \overline{\cD}$. Thus the claim follows.
\end{proof}

A pair of squashed fences, $\fF(g_i)=(\ell^i,\{\cE_n^i\}_{n\in\ZZ})$, $i\in \ZZ/2\ZZ$, is said to be \emph{unlinked} if either $\ell^1=\ell^2$, or there exist $j,k\in \ZZ$ such that $(\cE_j^2)^c$ contains $\Int \cE_k^1$.

\begin{rmk}\label{Rmk:sameAxis2}
In the first case, $\Fix(g_1)=\Fix(g_2)$.
Since $G$ is discrete in $\PSL(\CC)$, there exists a loxodromic element $h\in G$ that generates $\Stab{G}{\Fix(g_1)}$, and hence $\{g_1,g_2\}\subset \langle h\rangle$.
Therefore, $\ell^1$ is the axis of $h^k$ for every nonzero $h^k\in Z^-$.
Moreover, for any squashed fences
$\fF(h^{n_1})=(\ell^1,\{\cE_n\}_{n\in \ZZ})$ and
$\fF(h^{n_2})=(\ell^1,\{\cI_n\}_{n\in \ZZ})$,
we have $\cI_\infty=\cE_\infty$.
Indeed, any squashed fence $\fF(h)=(\ell^1,\{\cE_n\}_{n\in \ZZ})$ of $h$ is preserved by every element of $\langle h\rangle$.
\end{rmk}

\begin{prop}[Invariant squashed fence system]\label{Prop:squashedFenceSystem}
Let $Z^\pm$ be a minimal zipper for $M$.
Assume that $\fF(g)=(\ell,\{\cE_n\}_{n\in\ZZ})$ is a squashed fence of $g\in G$.
Write $p$ for the pivot of $\ell$, and $\Fix(g)=\{p,q\}$.
Then for any $h\in G$, the two squashed fences
\(
\fF(g)
\) and \(
h(\fF(g))=\fF(hgh^{-1})=(h(\ell),\{h(\cE_n)\}_{n\in\ZZ})
\)
are unlinked.
In particular, if $p\neq h(p)$, then there exist $i,j\in \ZZ$ such that
\(
h(\cE_j)\subset \closure{\cE_i\setminus \cE_{i+1}}.
\)
Hence
\(
h(\cE_\infty)\subset \closure{\cE_i\setminus \cE_{i+1}}.
\)
\end{prop}
\begin{proof}
Since  $q\nin Z^\pm$, $h(q)\nin \ell\cup \{p\}$.
If $h(q)=q$, then by the discreteness of $G$ in $\PSL(\CC)$, $\Fix(h)=\Fix(g)$, and by \refrmk{sameAxis2}, $\ell=h(\ell)$.
Hence $\fF(g)$ is unlinked with $h(\fF(g))$.

Assume $q\neq h(q)$.
Since $\cE_{-\infty}=\cS_\infty^2$, $h(q)\in \Int \cE_i$ for some $i\in \ZZ$, and $h(q)\in \Int \cE_j$ for all $i\ge j$.
Observe that the set $\fI=\{ i\in \ZZ: h(q)\in \Int \cE_i\}$ has a minimum.
If not, then $h(q)\in \Int \cE_i$ for all $i\in \ZZ$.
Hence, \refprop{squashedNonSquashed} implies that $h(\ell)\subset \cE_i$ for all $i\in \ZZ$.
Note that by \refrmk{sameAxis2}, $h(p)\neq p$.
Hence $h(p)\in \Int \cE_i$ for all $i\in \ZZ$, since $(\partial \cE_{i+1})^+\setminus \{p\}\subset \Int \cE_i$ for all $i\in \ZZ$.
Thus $h(p)\in \bigcap_{i\in \ZZ}\Int \cE_i$, which contradicts \refprop{eclipseLProperty}, since $h(p)$ is a synapse of type $2$.

Hence $\fI$ has a minimum $\sfi$, and so $h(q)\nin \cE_{\sfi+1}$.
Since the boundary of $\cD=\Int(\cE_{\sfi}\setminus \cE_{\sfi+1})$ is a fence with nodes in $Z^+$ and $h(q)\in \cD$, the result follows from \refprop{squashedNonSquashed}.
\end{proof}

\subsection{Circle laminations from one-sided simplicial axes}

Before constructing a circle lamination from a squashed fence, we recall the notion of conical limit points.

\begin{rmk}\label{Rmk:conicalLimit}
Let $\Gamma$ be a quasi-Fuchsian closed surface subgroup of $G$.
Recall that, for the action of $\Gamma$ on $\Lambda(\Gamma)$, every point $x\in \Lambda(\Gamma)$ is a \emph{conical limit point}: there exist distinct points $a,r\in \Lambda(\Gamma)$ and a sequence $\{g_i\}_{i\in\NN}$ in $\Gamma$ such that $g_i(x)\to r$ and $g_i(y)\to a$ for all $y\in \Lambda(\Gamma)\setminus \{x\}$, and hence also for all $y\in S_\infty^2\setminus \{x\}$ (\cite{Bowditch98}, \cite{Tukia98}).
We call the pair $\{a,r\}\in \cM$ a \emph{conical limit pair} for $x$, and the sequence $\{g_i\}_{i\in\NN}$ a \emph{conical limit sequence} for $x$ with respect to $\{a,r\}$.

The set of conical limit pairs for $x$ is closed in $\cM$.
Indeed, let $\{\ell_i\}_{i\in \NN}$ be a sequence of conical limit pairs for $x$ such that $\ell_i\to \ell_\infty$ in $\cM$.
For each $i\in\NN$, choose a conical limit sequence $\{g_{i,j}\}_{j\in\NN}$ for $x$ with respect to $\ell_i$.
Writing $\ell_i=\{a_i,r_i\}$, we have $g_{i,j}(x)\to r_i$ and $g_{i,j}(y)\to a_i$ for all $y\in \Lambda(\Gamma)\setminus \{x\}$.
After passing to a subsequence in $i$, we may assume that $a_i\to a_\infty$ and $r_i\to r_\infty$, where $\ell_\infty=\{a_\infty,r_\infty\}$.
A diagonal subsequence of $\{g_{i,j}\}_{(i,j)\in \NN\times \NN}$ is then a conical limit sequence for $x$ with respect to $\ell_\infty$.
\end{rmk}

We now begin the construction of a circle lamination from squashed fences.

\begin{const}[First-step construction from a squashed fence]\label{Const:firstStepConst}
Let $Z^\pm$ be a minimal zipper for $M$.
Assume that $g\in G$ acts freely on $Z^-$ and fixes a unique fixed point $p$ in $Z^+$.
Write $\Fix(g)=\{p,q\}$.
Let $\fF(g)=(\ell,\{\cE_n\}_{n\in \ZZ})$ be a squashed fence of $g$.
Write $\sigma_n$ for the connector arc in $\partial \cE_n$ that crosses $\ell$, and write $\sfd \sigma_n=\{s_n,p\}$.
Note that $s_n<s_{n-1}$ in $\ell$.

By \refthm{separatingQF}, take a quasi-Fuchsian closed surface subgroup $\Gamma$ of $G$ such that the limit set $\Lambda(\Gamma)$ of $\Gamma$ separates the fixed points of $g$.
Write $\cH$ for the Jordan domain of $\Lambda(\Gamma)$ that contains $p$.
Fix a circular order on $\Lambda(\Gamma)$ such that $\cH=\cD^L(\Lambda(\Gamma))$.

There exists $\cN\in \ZZ$ such that $\Lambda(\Gamma)$ separates $p$ and $\sigma_n^-$ for all $n\geq \cN$.
For each $n\geq \cN$, the arc $\sigma_n^+$ intersects $\Lambda(\Gamma)$, and there exists a unique connector arc $\sfa_n\subset \partial \cE_n$ crossing $\cH$ with synapse $p$.
Equip $\sfa_n$ with its canonical order, which is compatible with the linear order on $\ell$.
Write $\sfa_n^+\cap \Lambda(\Gamma)=x_n$ and $\sfd \sfa_n=\{x_n,z_1\}$.
Then,  in $\Lambda(\Gamma)$, $x_{n+1}<x_n<z_1$ if $\ell$ is right simplicial, and $z_1<x_n<x_{n+1}$ if $\ell$ is left simplicial.
Hence $\{x_n\}_{n=\cN}^\infty$ converges to a point $z_0$ such that $z_1\leq z_0<x_n$  in the first case and $x_n<z_0\leq z_1$ in the second.

Let $\cL(\Gamma,z_1)$ be the collection of all conical limit pairs for $z_1$ under the action of $\Gamma$ on $\Lambda(\Gamma)$.
By \refrmk{conicalLimit}, $\cL(\Gamma,z_1)$ is a non-empty closed subset of $\cM$, and by definition $\cL(\Gamma,z_1)$ is $\Gamma$-invariant.
\qedhere

\end{const}

\begin{prop}[Unlinkedness conical limit pairs]\label{Prop:unlinkedConicalLimitPairs}
In \refconst{firstStepConst}, any two elements of $\cL(\Gamma,z_1)$ are unlinked in $\Lambda(\Gamma)$.
In particular, $\cL(\Gamma,z_1)$ is a $\Gamma$-invariant circle lamination on $\Lambda(\Gamma)$.
\end{prop}
\begin{proof}
Choose distinct conical limit pairs $\ell_1,\ell_2$ for $z_1$.
Then, for each $i\in \{1,2\}$, there exists a conical limit sequence $\{g_{i,j}\}_{j\in\NN}$ for $z_1$ with respect to $\ell_i=\{a_i,r_i\}$ such that $g_{i,j}(z_1)\to r_i$ and $g_{i,j}(y)\to a_i$ for all $y\in \Lambda(\Gamma)\setminus \{r_i\}$.
Assume that $\ell_1$ and $\ell_2$ are linked.
Let $U_1,U_2,V_1,V_2$ be flat neighborhoods of $a_1,a_2,r_1,r_2$, respectively, relative to sufficiently small open neighborhoods of these points in $\Lambda(\Gamma)$, such that the closures of $U_1,U_2,V_1,V_2$ are pairwise disjoint.
Set $I_i=U_i \cap \Lambda(\Gamma)$ and $J_i=V_i\cap \Lambda(\Gamma)$, so that these are pairwise disjoint open intervals in $\Lambda(\Gamma)$.

Fix $n\geq \cN$.
Then there exists $M\in \ZZ$ such that $g_{i,j}(z_1)\in J_i$, $g_{i,j}(x_n)\in I_i$, and $g_{i,j}(\sfa_n^+)\subset U_i$ for all $j\geq M$.
By the choice of $U_i$ and $V_i$, it follows that $g_{1,j}(\sfa_n)$ and $g_{2,j}(\sfa_n)$ intersect for all $j\geq M$.
On the other hand, since $U_1\cap U_2=\varnothing$, we have $g_{1,j}(\sfa_n)^+\cap g_{2,j}(\sfa_n)^+=\varnothing$.
By the disjointness of $Z^\pm$, it follows that $g_{2,j}(\sfa_n)^+ \subset \closure{\cH^\alpha(g_{1,j}(\sfa_n))}\setminus g_{1,j}(\sfa_n)$ for some $\alpha\in \{R,L\}$.
Hence $g_{2,j}(p)\in \cH^\alpha(g_{1,j}(\sfa_n))$.
Since $\partial \cH^\alpha(g_{1,j}(\sfa_n))$ separates $g_{2,j}(p)$ from $g_{2,j}(x_n)$, and since $g_{i,j}(\sfa_n)$ has synapse in $Z^+$ and $g_{1,j}(\sfa_n)^+\cap g_{2,j}(\sfa_n)^+=\varnothing$, it follows that $g_{1,j}(\sfa_n)^-$ and $g_{2,j}(\sfa_n)^-$ are linked.
This contradicts the one-sided inaccessibility of $\ell$ (\refthm{simplicialAxis}).
\end{proof}

As mentioned above, we study the extension and bouncing directions as in \refsec{oneprong}.
We now describe how the extension and bouncing conditions are defined in the present setting.

\begin{const}[Extension of the circle lamination in \refconst{firstStepConst}]\label{Const:extConst}
Assume the hypotheses of \refconst{firstStepConst}.
Assume also that a finite strictly decreasing sequence $\{z_i\}_{i=1}^m$, with $m\ge 2$, in $\ell$ satisfies the \emph{bouncing condition}: for each $i\in \{2,\dots,m\}$, the pure arc $\gamma_i=[z_{i-1},z_i]$ on $\ell$ crosses $\cH$.
Define $\cL(\Gamma,\{z_i\}_{i=1}^m)$ to be the closure in $\cM$ of
\[
\{g(\sfd\gamma_i): g\in \Gamma \text{ and } 2\le i\le m\}\cup \cL(\Gamma,z_1).
\]
We write $\ell_i=\sfd\gamma_i$.
\end{const}

\begin{prop}\label{Prop:welldefinedExt}
    In \refconst{extConst}, $\cL(\Gamma,\{z_i\}_{i=1}^m)$ is a $\Gamma$-invariant circle lamination in $\Lambda(\Gamma)$.
\end{prop}
\begin{proof}
By \refprop{unlinkedConicalLimitPairs} and \refthm{simplicialAxis}, it suffices to show that for $m\geq 2$ and $i\ge 2$, $\sfd\gamma_i$ is unlinked with every element of $\cL(\Gamma,z_1)$.

Assume otherwise. Then $\sfd\gamma_i$ is linked with some $\lambda\in\cL(\Gamma,z_1)$. Since $\lambda\in\cL(\Gamma,z_1)$, there exists a sequence $\{g_i\}_{i\in\NN}\subset\Gamma$ such that $\lambda=\{u_2,u_4\}$, $g_i(z_1)\to u_2$, and $g_i(y)\to u_4$ for all $y\in \Lambda(\Gamma)\setminus\{u_2\}$. Write $\sfd\gamma_i=\{u_1,u_3\}$ so that $u_1<u_2<u_3<u_4$ in  $\Lambda(\Gamma)$.

Orient $\gamma_i$ so that it increases toward $u_3$, and fix $n\geq \cN$. Then there exists $j\in\NN$ such that $g_i(\sfa_n)^+\subset \overline{\cH^L(\gamma_i)}$ for all $i\geq j$. Since $Z^+\cap Z^-=\emptyset$ and $g_i(p)\in \cH^L(\gamma_i)$, it follows that $g_i(\sfa_n^-)$ is linked with $\gamma_i$ for all sufficiently large $i$, contradicting \refthm{simplicialAxis}.
\end{proof}

In fact, once the lamination $\cL(\Gamma,z_1)$ admits an extension, every element of $\cL(\Gamma,z_1)$ is realized as a limit of $\ell_2$. In this case, the set of generating leaves can be reduced.

\begin{prop}[Finite generating leaves]\label{Prop:finiteGenLeaves}
In \refconst{extConst}, $\cL(\Gamma,\{z_i\}_{i=1}^{m})$ agrees with the closure in $\cM$ of the $\Gamma$-orbits of $\{\ell_i\}_{i=2}^{m}$.
\end{prop}

\begin{proof}
It is enough to show that $\cL(\Gamma,z_1)$ is contained in the closure of the $\Gamma$-orbit of $\sfd\gamma_1$ in $\cM$.

Choose $\lambda\in\cL(\Gamma,z_1)$. Since $\lambda$ is a conical limit pair for $z_1$, there exists a sequence $\{g_i\}_{i\in\NN}\subset\Gamma$ such that $\lambda=\{r,a\}$, $g_i(z_1)\to r$, and $g_i(y)\to a$ for all $y\in \Lambda(\Gamma)\setminus\{z_1\}$. Since $\sfd\gamma_1=\{z_1,z_2\}$ and $z_2\in \Lambda(\Gamma)\setminus\{z_1\}$, it follows that $g_i(z_1)\to r$ and $g_i(z_2)\to a$. Hence $g_i(\sfd\gamma_1)$ converges to $\lambda$ in $\cM$.

Therefore $\cL(\Gamma,z_1)$ is contained in the closure of the $\Gamma$-orbit of $\sfd\gamma_1$, and the result follows.
\end{proof}

After \refprop{finiteGenLeaves}, when $m\geq 2$, we call $\{\ell_i\}_{i=2}^{m}$ the \emph{generating leaves} of $\cL(\Gamma,\{z_i\}_{i=1}^m)$ in \refconst{extConst}. We also say that $\cL(\Gamma,\{z_k\}_{k=1}^m)$ is \emph{generated} by $\{\ell_i\}_{i=2}^{m}$.

\subsection{Auxiliary leaves of $\cL(\Gamma,\{z_i\}_{i=1}^m)$}
We call $h(\sfd\sfa_n)$, with $n>\cN$ and $h\in\Gamma$, an \emph{auxiliary leaf} for $\cL(\Gamma,\{z_i\}_{i=1}^m)$. In general, for $n,n'>\cN$, the leaf $\sfd\sfa_n$ need not be unlinked with $h(\sfd\sfa_{n'})$ for $h\in\Gamma$. Nevertheless, every auxiliary leaf is unlinked with every leaf of $\cL(\Gamma,\{z_i\}_{i=1}^m)$:

\begin{prop}[Auxiliary leaves]\label{Prop:auxiliaryLeaf}
Assume the hypotheses of \refconst{firstStepConst}. Let $\cL(\Gamma,\{z_i\}_{i=1}^m)$ be the circle lamination constructed in \refconst{firstStepConst} or \refconst{extConst}. Then, for each $n>\cN$, the leaf $\sfd\sfa_n$ is unlinked with every leaf of $\cL(\Gamma,\{z_i\}_{i=1}^m)$, and no sequence of leaves of $\cL(\Gamma,\{z_i\}_{i=1}^m)$ converges to $\sfd\sfa_n$.
\end{prop}
\begin{proof}

Fix $n>\cN$.
Assume first that $\sfd \sfa_n$ is linked with $h(\gamma_i)$ for some $h\in \Gamma$ and some $1<i\le m$.
Since the synapse $p$ of $\sfa_n$ lies in $Z^+$ and $Z^+\cap Z^-=\varnothing$, it follows that $h(\gamma_i)\subset Z^-$ is linked with $\sfa_n^-$, contradicting the at least one-sided inaccessibility of $\sfa_n^-$ by \refthm{simplicialAxis}.

We next consider the case where $\sfd \sfa_n$ is linked with some leaf $\lambda\in \cL(\Gamma,z_1)$.
Then there exists a conical limit sequence $\{h_i\}_{i\in \NN}$ for $z_1$ with respect to $\lambda=\{a,r\}$ such that $h_i(z_1)\to r$ and $h_i(y)\to a$ for all $y\in \Lambda(\Gamma)\setminus\{z_1\}$, hence for all $y\in S^2_\infty\setminus\{z_1\}$.
Then either $z_1<r<x_n<a$ or $z_1<a<x_n<r$.

Suppose first that $z_1<r<x_n<a$. Then for all sufficiently large $i$, we have $h_i(\sfa_n)^+\subset \closure{\cH^L(\sfa_n)}\setminus \sfa_n$, so that $h_i(p)\in \cH^L(\sfa_n)$ and $h_i(z_1)\in \opi[\Lambda(\Gamma)]{z_1}{x_n}$. Since $\partial \cH^L(\sfa_n)$ separates $h_i(p)$ from $h_i(z_1)$ and $Z^+\cap Z^-=\varnothing$, it follows that $h_i(\sfa_n)^-$ is linked with $\sfa_n^-$, contradicting the at least one-sided inaccessibility of $\sfa_n^-$ by \refthm{simplicialAxis}.

The other case is treated in the same way and yields the same contradiction.
This proves the first assertion.

To see the second assertion, assume  for contradiction that there exists a sequence of leaves of $\cL(\Gamma,\{z_i\}_{i=1}^m)$ converging to $\sfd \sfa_n$. Hence, we may find a leaf $\lambda$ of $\cL(\Gamma,\{z_i\}_{i=1}^m)$ such that $\lambda$ is linked with $\{x_{\cN},x_{n+1}\}$ and either $\lambda\in \cL(\Gamma,z_1)$ or $\lambda=h(\ell_i)$ for some $h\in \Gamma$ and $i\in \{2,\dots,m\}$.

In the second case, since $\lambda$ and $\{x_\cN,x_{n+1}\}$ are linked, $h(\gamma_i)$ intersects $[x_\cN,x_{n+1}]$. This is a contradiction since $h(\gamma_i)\subset Z^-$ and $[x_\cN,x_{n+1}]\subset Z^+$.

Thus $\lambda\in \cL(\Gamma,z_1)$.
Then there exists a conical limit sequence $\{g_i\}_{i\in\NN}$ for $z_1$ with respect to $\lambda=\{a,r\}$ such that $g_i(z_1)\to a$ and $g_i(y)\to r$ for all $y\in \Lambda(\Gamma)\setminus\{z_1\}$. Since $\{x_{n+1},x_\cN\}$ and $\lambda$ are linked, for all sufficiently large $i$ the sets $g_i(\sfa_n^-)\subset Z^-$ and $[x_{n+1},x_\cN]\subset Z^+$ intersect, contradicting the disjointness of $Z^\pm$. Thus we are done.
\end{proof}

Under a stronger unlinked condition, the connector arcs associated with the auxiliary leaves intersect only along the $(-)$-segment.
\begin{lem}[Unlinked auxiliary arcs]\label{Lem:auxArcUnlinked}
Assume the conditions in \refconst{firstStepConst}. Suppose that $\cL(\Gamma,z_1)$ is a circle lamination constructed as in \refconst{firstStepConst}$,$ and that $\ell$ is right- or left-inaccessible. Let $i,j>\cN$ and $h\in \Gamma$. Assume that $\sfd \sfa_i$ and $h(\sfd \sfa_j)$ are unlinked in $\Lambda(\Gamma)$ and, respectively,
\[
\left(\cldi[\Lambda(\Gamma)]{x_{i+1}}{x_{i-1}} \cup \{z_1\}\right)
\cap h\!\left(\cldi[\Lambda(\Gamma)]{x_{j+1}}{x_{j-1}} \cup \{z_1\}\right)
= \varnothing,
\]
or
\[
\left(\cldi[\Lambda(\Gamma)]{x_{i-1}}{x_{i+1}} \cup \{z_1\}\right)
\cap h\!\left(\cldi[\Lambda(\Gamma)]{x_{j-1}}{x_{j+1}} \cup \{z_1\}\right)
= \varnothing.
\]
Then there exists $\alpha\in \{R,L\}$ such that $h(\sfa_j)\subset \overline{\cH^\alpha(\sfa_i)}$. Moreover, $\sfa_i \cap h(\sfa_j)=\sfa_i^- \cap h(\sfa_j^-)$.
\end{lem}

\begin{proof}
By symmetry, it suffices to consider the case in which $\ell$ is right-inaccessible. By assumption,
$h(\cldi[\Lambda(\Gamma)]{x_{j+1}}{x_{j-1}} \cup \{z_1\}) \subset \opi[\Lambda(\Gamma)]{z_1}{x_{i+1}}$ or $\opi[\Lambda(\Gamma)]{x_{i-1}}{z_1}$.
Without loss of generality, assume
$h(\cldi[\Lambda(\Gamma)]{x_{j+1}}{x_{j-1}} \cup \{z_1\}) \subset \opi[\Lambda(\Gamma)]{x_{i-1}}{z_1}$.
We claim $h(\sfa_j)\subset \closure{\cH^L(\sfa_i)}$.

First we show $h(p)\notin \cH^R(\sfa_i)$. If $h(p)\in \cH^R(\sfa_i)$, then $\partial \cH^R(\sfa_i)$ separates $h(p)$ and $h(z_1)$, so the pure ray $h(\sfa_j^-)$ intersects $\sfa_i$. Since $Z^+\cap Z^-=\varnothing$ and the synapses of $\sfa_i$ and $h(\sfa_j)$ lie in $Z^+$, it follows that $h(\sfa_j^-)$ is linked with $\sfa_i^-$, contradicting the one-sided simpliciality of $\ell$.

Next we show $h(\sfa_j^+)\cap \sfa_i=\varnothing$. Set $\delta_0=[x_{i+1},x_{i-1}]$ and $\delta_1=h([x_{j+1},x_{j-1}])$. Then $\delta_0,\delta_1\subset Z^+$ are pure arcs. Assume $x_{i+1}<x_{i-1}$ and $h(x_{j+1})<h(x_{j-1})$. Since $\sfd \delta_0$ and $\sfd \delta_1$ are unlinked, $\delta_0$ and $\delta_1$ are unlinked. Moreover $\delta_1\subset \closure{\cH^L(\delta_0)}$, $\delta_0\subset \closure{\cH^L(\delta_1)}$, and $\cH^R(\delta_0)\cap \cH^R(\delta_1)=\varnothing$. Since $\sfa_i^+$ crosses $\cH^R(\delta_0)$ and $\sfa_j^+$ crosses $\cH^R(\delta_1)$, we obtain $\Int \sfa_i \cap \Int \sfa_j=\varnothing$.

If $p=h(p)$, then $h$ is non-trivial and fixes $p\in \cH$, contradicting $\Fix(h)\subset \Lambda(\Gamma)$. Hence $\sfa_i^+\cap \sfa_j^+=\varnothing$, and therefore $\sfa_i^+\cap h(\sfa_j)=\varnothing$. It follows that $h(p)\in \cH^L(\sfa_i)$. Therefore,  since $\sfa_i^-$ is at least one-sided inaccessible, \refprop{enclosing}, applied to $h(\sfa_j^-)$ with respect to $\cH^L(\sfa_i)$, yields $h(\sfa_j)\subset \closure{\cH^L(\sfa_i)}$.
\end{proof}

\subsection{Geometric realization of $\cL(\Gamma,\{z_i\}_{i=1}^m)$}

In the setting of \refconst{firstStepConst} and \refconst{extConst}, \refrmk{realization} yields a unique geodesic lamination $\cG(\Gamma,\{z_i\}_{i=1}^m)$ on $\cH/\Gamma$ whose lift $\widetilde{\cG}(\Gamma,\{z_i\}_{i=1}^m)$ to $\cH$ is the geometric realization of $\cL(\Gamma,\{z_i\}_{i=1}^m)$.

We now discuss the geometry of auxiliary and generating leaves in $\cG(\Gamma,\{z_i\}_{i=1}^m)$, with \reflem{closureOfLeaf} and \reflem{closureOfManyLeaves} in view.

\begin{prop}[Fanning arcs]\label{Prop:fanning2}
Assume the conditions in \refconst{firstStepConst}.
Let $\cL(\Gamma,\{z_i\}_{i=1}^m)$ be the circle lamination constructed in \refconst{firstStepConst} or \refconst{extConst}.
Then each generating leaf $\ell_i$ and each auxiliary leaf $\sfd \sfa_n$ with $n>\cN$ has trivial stabilizer in $\Gamma$.
In particular, if $h\in \Gamma$ fixes $z_i$, then:
\begin{itemize}
    \item if $i=1$, then for each $n>\cN$, $h(\sfa_n)$ crosses either $\cH^L(\sfa_n)$ or $\cH^R(\sfa_n)$, and $h(x_n)\neq x_n$;
    \item if $i>1$, then $h(\gamma_i)$ crosses either $\cH^L(\gamma_i)$ or $\cH^R(\gamma_i)$, and $h(z_{i-1})\neq z_{i-1}$.
\end{itemize}
In the case $i=1$, we also have $z_0\neq z_1$.
\end{prop}

\begin{proof}
We first show that each generating leaf and each auxiliary leaf has trivial stabilizer.
For the generating leaves, this is exactly \refthm{twoInOne}.
It therefore remains to prove that, for each $n>\cN$, the auxiliary leaf $\sfd \sfa_n$ has trivial stabilizer.

Suppose that $\sfd \sfa_n$ is preserved by some $h\in \Gamma$.
Then $\Fix(h)=\sfd \sfa_n$ and $a(h)=x_n$.
Since $\{x_\cN,x_{n+1}\}$ is linked with $\sfd \sfa_n$, the ray $h^i(\sfa_n^-)\subset Z^-$ meets $[x_\cN,x_{n+1}]\subset Z^+$ for all sufficiently large $i$, which is impossible.
Hence the stabilizer of $\sfd \sfa_n$ is trivial.
This proves the first assertion.

We then prove the second assertion.
By symmetry, it suffices to consider the case where $\ell$ is right inaccessible.
Assume that $\ell$ is right inaccessible and that $h\in \Gamma$ fixes $z_i$.
Both $\sfa_n,n>\cN$ and $\gamma_i,i>1$  carries a natural linear order compatible with $\ell$.

Assume first that $i>1$.
By the first assertion, $h(z_{i-1})\neq z_{i-1}$.
Hence either $h(w_i)\in \opi[\Lambda(\Gamma)]{z_i}{z_{i-1}}$ or $h(w_i)\in \opi[\Lambda(\Gamma)]{z_{i-1}}{z_i}$.
Replacing $h$ by $h^{-1}$ in the latter case, we may assume that
\[
h(z_{i-1})\in \opi[\Lambda(\Gamma)]{z_i}{z_{i-1}}.
\]
Applying \refprop{enclosing} to $h(\gamma_i)$ with respect to $\cH^R(\gamma_i)$, the right inaccessibility of $\gamma_i$ implies that $h(\gamma_i)$ crosses $\cH^R(\gamma_i)$.
Thus the conclusion holds in this case.

Assume \(i=1\).
We first claim that \(h(p)\notin \sfa_n\).
Suppose otherwise.
Since \(\Fix(h)\subset \Lambda(\Gamma)\), we have \(h(p)\neq p\), hence \(h(p)\in \Int \sfa_n^+\) as \(Z^+\cap Z^-=\varnothing\).
Since \(\sfa_n^-\) and \(h(\sfa_n)^-\) are right inaccessible, \refprop{enclosing} implies that \(h(\sfa_n)^-\) crosses \(\cH^L(\sfa_n)\) or \(\cH^R(\sfa_n)\) and lands at \(h(p)\).
As \(n>\cN\), in either case \(h(\sfa_n)^-\) meets \([x_{n-1},x_{n+1}]\subset Z^+\), contradicting \(Z^+\cap Z^-=\varnothing\).

Hence \(h(p)\in \cH^L(\sfa_n)\) or \(\cH^R(\sfa_n)\).
We show that, respectively, \(h(\sfa_n)\) crosses \(\cH^L(\sfa_n)\) or \(\cH^R(\sfa_n)\).
Applying \refprop{enclosing} to \(h(\sfa_n)^-\), the right inaccessibility of \(h(\sfa_n)^-\) or \(\sfa_n^-\), respectively, gives
\[
\Int h(\sfa_n)^- \subset \cH^L(\sfa_n)
\quad\text{or}\quad
\Int h(\sfa_n)^- \subset \cH^R(\sfa_n).
\]
Thus it suffices to show \(h(\sfa_n)^+\cap \sfa_n^+=\emptyset\), which yields \(\Int h(\sfa_n)\subset \cH^L(\sfa_n)\) or \(\cH^R(\sfa_n)\), respectively

Applying the above claim to \(h^{-1}\), we have \(p\notin h(\sfa_n)\).
Hence it suffices to show \(h(\sfa_n)^+\cap \Int \sfa_n^+=\emptyset\).

Let \(\cA=[x_{n-1},x_{n+1}]\subset Z^+\), ordered from \(x_{n+1}\) to \(x_{n-1}\).
Then \(z_1<x_{n-1}<x_n<x_{n+1}\) in \(\Lambda(\Gamma)\).
We claim \(h(p)\notin \cH^L(\cA)\).
Otherwise \(\partial \cH^L(\cA)\) separates \(h(p)\) from \(z_1\), and since \(h(\sfa_n)^-\) lands at both, we get \(h(\sfa_n)^-\cap \cA\neq\emptyset\), a contradiction.

Since \(h(p)\neq p\), either
\[
h(p)\in \Int \cA\setminus\{p\}=\Int \sfa_{n+1}^+\cup \Int \sfa_{n-1}^+
\]
or \(h(p)\in \cH^R(\cA)\).
In either case, if \(v\in h(\sfa_n)^+\cap \Int \sfa_n^+\), then \([h(p),v]\ni p\), hence \(p\in h(\sfa_n)^+\), a contradiction.
This completes the proof. The last assertion follows from the previous one.
\qedhere

\end{proof}

\begin{lem}[Perfectness of the first-step geodesic lamination]\label{Lem:closureOfAuxiliaryLeaf}
Assume the conditions of \refconst{firstStepConst}. Let $\cL=\cL(\Gamma,\{z_i\}_{i=1}^m)$ be the circle lamination constructed in \refconst{firstStepConst} or \refconst{extConst}. Then, for each $n>\cN$, the leaf $\sfg(\sfd \sfa_n)$ corresponds to a bi-infinite geodesic $\sfg_n$ in $\cH/\Gamma$, possibly non-simple, which is disjoint from $\cG=\cG(\Gamma,\{z_i\}_{i=1}^m)$.

Moreover, if $z_1$ is not fixed by any non-trivial element of $\Gamma$, then $\cG(\Gamma,z_1)=\cG(\Gamma,z_1)'$, $\cG''\neq\varnothing$, and there exists a unique principal region $U$ of $\cG''$ such that $\sfg_n\subset U$ and each half-ray of $\sfg_n$ limiting to $z_1$ escapes the same end of $U$ and spirals toward the same component of $\cG''$.
\end{lem}

\begin{proof}
The first statement follows from \refprop{auxiliaryLeaf} and \refprop{fanning2}.

For the second statement, assume that $z_1$ is not fixed by any non-trivial element of $\Gamma$. We first observe that $\cG_1=\cG(\Gamma,z_1)$ has no isolated leaf, i.e. $\cG_1=\cG_1'$. This will imply that $\cG''\ne \varnothing$. Assume that $\lambda$ is an isolated leaf of $\cL(\Gamma,z_1)$ and let $\sfg$ be the corresponding leaf in $\cG_1$.

\begin{claim*}
The stabilizer of any leaf of $\cL(\Gamma,z_1)$ is trivial. In particular, $\sfg$ is bi-infinite.
\end{claim*}

\begin{proof}[Proof of the claim]
Assume not. Then there exists a leaf $\mu$ of $\cL(\Gamma,z_1)$ with $\Stab{\Gamma}{\mu}=\langle h\rangle$ for some $h\in \Gamma$. Since $\mu$ is a conical limit pair for $z_1$, there is a conical limit sequence $\{g_i\}$ for $z_1$ with respect to $\mu=\{a,r\}$ such that $g_i(z_1)\to r$, and $g_i(y)\to a$ for $y\in \Lambda(\Gamma)\setminus\{z_1\}$. We may assume $a=a(h)$ and $r=r(h)$.

Fix $n>\cN$ and write $\lambda_n=\sfd \sfa_n=\{z_1,x_n\}$. Then $g_i(\lambda_n)\to \mu$ in $\cM$. By \refprop{fanning2}, $\mu\neq g_i(\lambda_n)$ for all but finitely many $i$, so by \refprop{auxiliaryLeaf}, after passing to a subsequence, $\{g_i(\lambda_n)\}$ lies entirely on one side, either $\opi[\Lambda(\Gamma)]{r}{a}$ or $\opi[\Lambda(\Gamma)]{a}{r}$.

Assume first that it is $\opi[\Lambda(\Gamma)]{r}{a}$-side. After passing to a subsequence, one of the following holds:
\begin{enumerate}
\item $r=g_{i+1}(z_1)=g_i(z_1)<g_i(x_n)<g_{i+1}(x_n)<a$;
\item $r<g_{i+1}(z_1)<g_i(z_1)<g_i(x_n)=g_{i+1}(x_n)=a$;
\item $r<g_{i+1}(z_1)<g_i(z_1)<g_i(x_n)<g_{i+1}(x_n)<a$.
\end{enumerate}

The first case implies that some non-trivial element fixes $z_1$, contradiction by the assumption. In the second case, $g_i\circ g_1^{-1}$ fixes $a$ and hence equals $h^{k(i)}$. Since $g_1(z_1)\in \opi[\Lambda(\Gamma)]{r}{a}$ and  $h^{k(i)} g_1(z_1)=g_i(z_1)\to r=r(h)$, we have that $k(i)\to -\infty$, so $g_i(p)\to r$, contradicting the conical limit property.

Now consider the third case. For all sufficiently large $i$, the leaves $g_i(\sfd \sfa_n)$ and $hg_i(\sfd \sfa_n)$ are linked, i.e.
$r<g_i(z_1)<hg_i(z_1)<g_i(x_n)<hg_i(x_n)<a$.
Fix such an $i$. Choose a Jordan domain $B$ in $S_\infty^2$ containing $a$ such that $h(\closure{B})\subset B$ and
$B\cap (g_i(\sfa_n)\cup hg_i(\sfa_n))=\varnothing$.
Then choose a flat neighborhood $U$ of $a$ relative to $\opi[\Lambda(\Gamma)]{hg_i(x_n)}{g_i(z_1)}$ such that $U\subset B$ and $\cI_U:=U\cap \Lambda(\Gamma)$ is an open interval containing $a$.

Now take $j>i$ sufficiently large so that $g_j(\sfa_n^+)\subset U$ and $g_j(z_1)\in \opi[\Lambda(\Gamma)]{r}{g_i(z_1)}$.
For each $k\in \NN$, we have
$h^k(g_j(\sfa_n^+))\subset h^k(U)\subset h^k(B)\subsetneq h(B)\subsetneq B$,
so $h^k(U)$ is again a flat neighborhood of $a$ relative to $\opi[\Lambda(\Gamma)]{hg_i(x_n)}{g_i(z_1)}$, disjoint from $g_i(\sfa_n)\cup hg_i(\sfa_n)$.
Moreover,
$h^k(U)\cap \Lambda(\Gamma)=h^k(\cI_U)\subset \closure{h^k(\cI_U)}\subset \cI_U$.

Since $g_i(\sfd \sfa_n)$ and $hg_i(\sfd \sfa_n)$ are linked and $a=a(h)$, there exists $\sfk\in \NN$ such that $h^{\sfk}(g_j(\sfd \sfa_n))$ and $g_i(\sfd \sfa_n)$ are linked.
Because $h^\sfk(U)\cap g_i(\sfa_n)=\varnothing$, $h^\sfk(g_j(\sfa_n^+))\subset h^\sfk(U)$, and
$h^\sfk g_j(z_1)\in \opi[\Lambda(\Gamma)]{g_i(z_1)}{g_i(x_n)}$,
the disjointness of $Z^\pm$ implies that $h^\sfk(g_j(\sfa_n^-))$ is linked with $g_i(\sfa_n^-)$.
This contradicts the one-sided inaccessibility of $\ell$.

The case where $\{g_i(\lambda_n)\}_{i\in \NN}$ lies in $\opi[\Lambda(\Gamma)]{a}{r}$ can be handled by a symmetric argument. 
Thus the claim follows.
\end{proof}

After the claim, $\sfg$ is a bi-infinite simple isolated leaf of $\cG_1$, and $\cG_1$ has no closed leaf. Choose a conical limit sequence $\{h_i\}$ for $z_1$ with respect to $\lambda=\{\alpha,\rho\}$ such that $h_i(z_1)\to \rho$ and $h_i(y)\to \alpha$ for $y\in \Lambda(\Gamma)\setminus\{z_1\}$.

Let $\cL_0$ be the closure of the $\Gamma$-orbit of $\lambda$ in $\cM$, and let $\cG_0$ be the corresponding sublamination of $\cG_1$, namely the closure of $\sfg$ in $\cH/\Gamma$.
By assumption, $\lambda$ is isolated in $\cL_0$, and $\sfg$ is the unique isolated leaf of $\cG_0$. Thus $\cG_0'=\cG_0\setminus \sfg$. Since, by the above claim, $\cG_0$ has no closed leaf, we also have $\cG_0''=\cG_0'$ (see \cite[Corollary~4.7.1]{Casson88}). Moreover, each component of $\cG_0'$ is minimal. By construction, a half-ray of $\sfg$ limiting to $\rho$ spirals toward a component of $\cG_0'$. Hence there is a unique principal region $U$ of $\cG_0$ whose end is approached by this half-ray. By \refprop{gapShape}, there exist boundary leaves $\ell_L,\ell_R$ of $\cL_0$ with $\ell_L=\{x_L,\rho\}$ and $\ell_R=\{x_R,\rho\}$ for some $x_L<\rho<x_R<\alpha$ in $\Lambda(\Gamma)$, such that $\opi[\Lambda(\Gamma)]{x_L}{\rho}$ and $\opi[\Lambda(\Gamma)]{\rho}{x_R}$ are not isolated in $\cL_0$.

On the other hand, for each $n>\cN$, $h_i(\sfd \sfa_n)\to \lambda$ as $h_i(z_1)\to \rho$ and $h_i(x_n)\to \alpha$. By \refprop{fanning2}, after passing to a subsequence, we may assume that the auxiliary leaves $h_i(\sfd \sfa_n)$ are pairwise distinct. By \refprop{auxiliaryLeaf}, each of them is unlinked with both $\ell_L$ and $\ell_R$. It follows that there exist $i\neq j$ such that $h_i(\sfd \sfa_n)\cap h_j(\sfd \sfa_n)=\{\rho\}$ and $\rho=h_i(z_1)=h_j(z_1)$. Hence $h_j^{-1}h_i$ is a non-trivial element of $\Gamma$ fixing $z_1$, contradicting the assumption. Therefore $\cL(\Gamma,z_1)$ has no isolated leaf, so $\cG_1=\cG_1'$. Since $\cG_1\subset \cG''$, it follows that $\cG''\neq \varnothing$.

Finally, fix $n>\cN$. Let $\sfr$ be the half-ray of $\sfg_n$ limiting to $z_1$. By \refprop{auxiliaryLeaf}, $\sfg_n$ is not a leaf of $\cG$ and is disjoint from $\cG$. Hence $\sfr$ is contained in a principal region $U$ of $\cG''$.
Recall \refprop{gapShape}.
If $U$ is a finite ideal polygon, there is nothing to prove. Otherwise, $U$ has a unique core $U_0$, which is either a simple closed curve or a compact subsurface of $\cH/\Gamma$, and $U\setminus U_0$ is a disjoint union of open crowns. Since $\cG''$ contains $\cG_0$, the ray $\sfr$ accumulates on a component of $\cG''$ and therefore escapes an end of $U$.

Since $z_1$ is not fixed by any non-trivial element of $\Gamma$, \refprop{oneEndFix} implies that there are only finitely many boundary leaves of $\cL_0$ containing $z_1$. It follows from the preceding discussion that, for all $n>\cN$, each half-ray of $\sfg_n$ limiting to $z_1$ escapes the same end of $U$ and spirals toward the same component of $\cG''$. This proves the second statement.
\end{proof}

\subsection{Uniform bouncing direction at a fixed endpoint}

In this and the following sections, we study the bouncing direction as in \refsec{uniformBouncing}. We begin with the following lemma.
With \refprop{fanning2} at hand, we adapt the proof of \reflem{oneStepAtFix} to the present setting.
\begin{lem}[One-step extension at a fixed end]\label{Lem:oneStepAtFix2}
Assume the hypotheses of \refconst{firstStepConst}.
Let $\cL(\Gamma,\{z_i\}_{i=1}^m)$, where $m\ge 1$, be a circle lamination constructed in \refconst{firstStepConst} or \refconst{extConst}.
Suppose that $\ell$ is right (resp.\ left) inaccessible, and that $z_m$ is fixed by a non-trivial element $h\in \Gamma$.
Then there exists $z_{m+1}\in \ell$ such that $\{z_i\}_{i=1}^{m+1}$ satisfies the bouncing condition of \refconst{extConst}, and
$z_{m-1}<z_m<z_{m+1}$ (resp. $z_{m+1}<z_m<z_{m-1}$) in $\Lambda(\Gamma)$.
\end{lem}

\begin{proof}
The statement is symmetric according to whether $\ell$ is right or left inaccessible. Thus, without loss of generality, we may assume that $\ell$ is right inaccessible.

We first treat the case $m>1$. By \refprop{fanning2} and the assumption, there exists an element $h\in \Gamma$ such that
\(h(z_m)=z_m\), \(h(z_{m-1})\in \opi[\Lambda(\Gamma)]{z_m}{z_{m-1}}\),
and $h(\gamma_m)$ crosses $\cH^R(\gamma_m)$.

Set $\cA:=\gamma_m\cup h(\gamma_m)$, so that $\cA$ is the arc $[z_{m-1},h(z_{m-1})]$ in $Z^-$. Equip $\cA$ with the linear order compatible with $\gamma_m$, and hence with $\ell$. Let $\cR$ be the end ray of $\ell$ at $z_{m-1}$ landing at $q$. Then $\gamma_m$ is a common subarc of $\cA$ and $\cR$, containing an endpoint of each. Since $Z^-$ is a real tree, their intersection $\cI:=\cA\cap \cR$ is an arc containing $\gamma_m$, and hence $\cI=[z_{m-1},w]$ for some $w\in h(\gamma_m)$.

If $w=h(z_{m-1})$, then $\cA\subset \cR\subset \ell$, so setting $z_{m+1}:=w$ completes the proof. Assume now that $w\neq h(z_{m-1})$. Then the terminal ray $\cR_w:=\cR\setminus \cI$ of $\cR$ is contained in a right branch of $\cA$ at $w$.

Consider the Jordan curve $\cJ:=\cA\cup \opi[\Lambda(\Gamma)]{h(z_{m-1})}{z_{m-1}}$, equipped with the circular order compatible with $\cA$. One end of $\cR_w$ lands at $w\in \cD^R(\cJ)$, and the other lands at $q\in \cD^L(\cJ)$. Since $\cA\cap \cR_w=\varnothing$, the line $\cR_w$ must meet the complementary boundary arc $\opi[\Lambda(\Gamma)]{h(z_{m-1})}{z_{m-1}}$. Hence some end ray of $\cR_w$ landing at $w$ crosses $\cD^R(\cJ)$. Let $z_{m+1}$ be the initial point of this end ray. Then $z_{m+1}\in \ell$, and $z_{m-1}<z_m<z_{m+1}$ in $\Lambda(\Gamma)$. This proves the case $m>1$.

We now consider the case $m=1$. Fix $n>\cN$. By \refprop{fanning2} and the assumption, we have $z_0\neq z_1$, and we may choose $h\in \Stab{\Gamma}{z_1}$ such that $h(\sfa_n)$ crosses $\cH^R(\sfa_n)$ and $h(x_n)\in \opi[\Lambda(\Gamma)]{z_1}{x_n}$. Hence $\sfa_n^-\cap h(\sfa_n)^-=\{z_1\}$, and $L:=\sfa_n^-\cup h(\sfa_n)^-$ is a line in $Z^-$ landing at $p$ and $h(p)$. Equip $L$ with the linear order compatible with $\sfa_n^-$, and hence with $\ell$. Then $\sfa_n^-$ is a common end ray of $L$ and $\ell$. Since $Z^-$ is a real tree, their intersection $\cR:=L\cap \ell$ is an end ray of both, containing $\sfa_n^-$. Let $u=\partial \cR$. Then the complementary line $\cR':=\ell\setminus \cR$ is contained in the right branch of $L$ at $u$.

Now consider the Jordan curve
\(
\cJ:=\sfa_n\cup h(\sfa_n)\cup \opi[\Lambda(\Gamma)]{x_n}{h(x_n)},
\)
equipped with the circular order compatible with $L$. One end of $\cR'$ lands at $u\in \cD^R(\cJ)$ and the other lands at $q\in \cD^L(\cJ)$. Since $L\cap \cR'=\varnothing$ and $Z^+\cap Z^-=\varnothing$, the line $\cR'$ must meet the complementary boundary arc $\opi[\Lambda(\Gamma)]{h(x_n)}{x_n}$. Hence we may choose an end ray of $\cR'$ that lands at $u$ and crosses $\cD^R(\cJ)$. Let $v$ be the initial point of this end ray. Then $v\in \opi[\Lambda(\Gamma)]{h(x_n)}{x_n}$, and $[z_1,v]$ crosses $\cH^R(\sfa_n)$.

To show that $v$ is the desired $z_2$, it suffices to verify that $z_1<v<z_0$.
For each $n'>n$, since $\sfa_{n'}^+$ crosses $\cH^R(\sfa_n)$ and $[z_1,v]\subset Z^-$ cannot meet $\sfa_{n'}^+$, we have that $[z_1,v]$ crosses $\cH^R(\sfa_{n'})$. Hence $v\in \lopi[\Lambda(\Gamma)]{z_1}{z_0}$.

To see that $v\neq z_0$, assume for contradiction that $v=z_0$. If $v$ is fixed by a non-trivial element $f\in \Gamma$, then by \refthm{twoInOne}, $f$ cannot fix $z_1$, and by the previous case there exists $z_3\in \ell$ such that $v>z_3$ in $\ell$, $z_1<v<z_3$ in $\Lambda(\Gamma)$, and $[v,z_3]$ crosses $\cH^R([z_1,v])$. However, since $\{\sfd \sfa_n\}_{n>\cN}$ converges to $\{z_0,z_1\}$ in $\cldi[\Lambda(\Gamma)]{z_0}{z_1}$, the leaf $\{v,z_3\}$ of $\cL(\Gamma,\{z_1,v,z_3\})$ (given by \refconst{extConst}) is linked with $\sfd \sfa_n$ for all sufficiently large $n$, contradicting \refprop{auxiliaryLeaf}. Hence no non-trivial element of $\Gamma$ fixes $v$.

On the other hand, since $z_1$ is fixed by a non-trivial element of $G$, it follows from \refprop{oneEndFix} that $\{z_1,v\}$ is isolated in $\cL(\Gamma,\{z_1,v\})$. Let $\sfg$ be the leaf in $\cG(\Gamma,\{z_1,v\})$ associated with $\{z_1,v\}$, which is a bi-infinite isolated leaf of $\cG(\Gamma,\{z_1,v\})$. Since $v$ is not fixed by any non-trivial element of $\Gamma$, by \reflem{closureOfLeaf} and \reflem{closureOfManyLeaves}, we have $\cG(\Gamma,\{z_1,v\})''\neq \varnothing$, and there exists a principal region $U$ of $\cG(\Gamma,\{z_1,v\})''$ such that $U$ contains $\sfg$ and every subray of $\sfg$ escapes toward an end of $U$. Therefore, since $\sfg$ is isolated, we can find a boundary leaf $\{v,w\}$ of $\cL(\Gamma,\{z_1,v\})$, corresponding to a boundary leaf of $U$, such that $z_1<v<w$ in $\Lambda(\Gamma)$. However, $\{v,w\}$ is linked with $\sfd \sfa_n$ for all sufficiently large $n$, again contradicting \refprop{auxiliaryLeaf}. Thus $v\neq z_0$, and $v$ is the desired $z_2$.
\end{proof}

The following corollary shows that a bouncing sequence in a uniform direction does not jump over $z_0$.
\begin{cor}[Well ordered bouncing points]\label{Cor:limitAuxiliaryLeaf}
Assume the conditions in \refconst{firstStepConst}.
Let $\cL=\cL(\Gamma,\{z_i\}_{i=1}^m)$ be the circle lamination constructed in \refconst{firstStepConst} or \refconst{extConst}.
Then $z_0\neq z_1$, and no leaf of $\cL$ is linked with $\{z_0,z_1\}$.

In particular, the following hold:
\begin{itemize}
    \item If $m=2$, then $z_2\ne z_0$.
    \item If $m>2$ and $\ell$ is right (resp.\ left) inaccessible, and if
    \[
    z_0<z_1<\cdots<z_{m-1}
    \quad\text{and}\quad
    z_{m-2}<z_{m-1}<z_m
    \quad\text{in }\Lambda(\Gamma)
    \]
    (resp.\
    \[
    z_0>z_1>\cdots>z_{m-1}
    \quad\text{and}\quad
    z_{m-2}>z_{m-1}>z_m
    \quad\text{in }\Lambda(\Gamma),
    \]
    then
    \[
    z_{m-1}<z_m<z_0
    \qquad
    (\text{resp.\ } z_{m-1}>z_m>z_0)\quad \text{in }\Lambda(\Gamma).
    \]
\end{itemize}
\end{cor}

\begin{proof}
The statement is symmetric according to whether $\ell$ is right or left inaccessible. Thus, it suffices to treat the case where $\ell$ is right inaccessible.

To show that $z_0\neq z_1$, consider $\cG_1=\cG(\Gamma,z_1)$. If $z_1$ is fixed by a non-trivial element of $\Gamma$, then \reflem{oneStepAtFix2} implies that $z_0\neq z_1$. Otherwise, the second statement of \reflem{closureOfAuxiliaryLeaf} implies that $z_0\neq z_1$. Moreover, in either case, since $\sfd \sfa_n\to \{z_0,z_1\}$ in $\cM$ as $n\to\infty$, \refprop{auxiliaryLeaf} implies that no leaf of $\cL$ is linked with $\{z_0,z_1\}$. This proves the first statement.

We next consider the second statement in the case $m=2$. Assume for contradiction that $z_2=z_0$, that is, $\ell_2=\{z_0,z_1\}$.

If $z_1$ is fixed by a non-trivial element of $\Gamma$, then the desired conclusion follows from \reflem{oneStepAtFix2}. Thus, we may assume that no non-trivial element of $\Gamma$ fixes $z_1$.

In this case, by \reflem{closureOfAuxiliaryLeaf} and \refprop{gapShape}, there are two leaves $\{u_0,z_1\}$ and $\{z_1,u_1\}$ of $\cL(\Gamma,\ell_2)$ such that
\[
x_{\cN+1}<u_0<z_1<u_1\le z_0
\]
in $\Lambda(\Gamma)$. We claim that $\opi[\Lambda(\Gamma)]{z_0}{z_1}$ is isolated in $\cL(\Gamma,\ell_2)$. Otherwise, there exists a $\opi[\Lambda(\Gamma)]{z_0}{z_1}$-side sequence $\{\lambda_i\}_{i\in \NN}$ in $\cL(\Gamma,\ell_2)$. Since each $\lambda_i$ is unlinked with $\{u_0,z_1\}$, it follows that $\lambda_i$ has $z_1$ as an endpoint for all sufficiently large $i$. Then \refprop{oneEndFix} implies that $z_1$ is fixed by a non-trivial element of $\Gamma$, contradicting the assumption.

On the other hand, no non-trivial element of $\Gamma$ fixes $z_2$. Indeed, if some non-trivial element fixed $z_2$, then \reflem{oneStepAtFix2} would imply that there exists
\(
z_3\in \opi[\Lambda(\Gamma)]{z_2}{z_1}=\opi[\Lambda(\Gamma)]{z_0}{z_1}
\)
such that $\{z_i\}_{i=1}^3$ satisfies the bouncing condition in \refconst{extConst}. However, since $\sfd \sfa_n \to \ell_2$ as $n\to\infty$ and $\sfd \sfa_n\subset \cldi[\Lambda(\Gamma)]{z_0}{z_1}$, the leaf $\ell_3$ is linked with $\sfd \sfa_n$ for all sufficiently large $n$. This contradicts \refprop{auxiliaryLeaf}, applied to $\cL(\Gamma,\{z_i\}_{i=1}^3)$. Hence no non-trivial element of $\Gamma$ fixes $z_2$.

Since $\cL(\Gamma,\ell_2)$ is generated by $\ell_2$, no non-trivial element of $\Gamma$ fixes either endpoint of $\ell_2$, and $\opi[\Lambda(\Gamma)]{z_0}{z_1}$ is isolated in $\cL(\Gamma,\ell_2)$, it follows from \reflem{closureOfLeaf}, together with \refprop{gapShape}, that there exists a boundary leaf $\{z_0,w\}$ such that $w\in \opi[\Lambda(\Gamma)]{z_0}{z_1}$. Hence, as above, $\{z_0,w\}$ is linked with $\sfd \sfa_n$ for all sufficiently large $n$, contradicting \refprop{auxiliaryLeaf}. This proves the case $m=2$.

Assume now that $m>2$.
By \refprop{auxiliaryLeaf}, the leaf $\sfd\gamma_m$ is unlinked with $\sfd\gamma_k$ for $1<k<m$, and also with $\sfd\sfa_n$ for all $n>\cN$.
Hence
\[
z_m\in \{z_i:i=1,\dots,m-2\}\cup \lopi[\Lambda(\Gamma)]{z_{m-1}}{z_0}.
\]
Since $Z^-$ is simply connected, it follows that
\(
z_m\in \lopi[\Lambda(\Gamma)]{z_{m-1}}{z_0}.
\)

We now show that $z_m\neq z_0$.
Assume for contradiction that $z_m=z_0$.

Suppose first that $z_m$ is fixed by a non-trivial element of $\Gamma$.
By \reflem{oneStepAtFix2}, there exists a point
\(z_{m+1}\in \opi[\Lambda(\Gamma)]{z_m}{z_{m-1}}\)
such that $\{z_i\}_{i=1}^{m+1}$ satisfies the bouncing condition in \refconst{extConst}.
Since $Z^-$ is simply connected, we have
\(
z_{m+1}\in \opi[\Lambda(\Gamma)]{z_0}{z_1},
\)
and therefore $\sfd\gamma_{m+1}$ is linked with $\sfd\sfa_n$ for all sufficiently large $n$, contradicting \refprop{auxiliaryLeaf}.
Therefore no non-trivial element of $\Gamma$ fixes $z_m$.

We claim that there exists a leaf $\{z_m,w\}$ of $\cL$ such that
\begin{itemize}
    \item $w\in \opi[\Lambda(\Gamma)]{z_m}{z_{m-1}}$,
    \item no non-trivial element of $\Gamma$ fixes $w$, and
    \item $\opi[\Lambda(\Gamma)]{z_m}{w}$ is not isolated in $\cL$.
\end{itemize}

This completes the proof. Indeed, since $\opi[\Lambda(\Gamma)]{z_m}{w}$ is not isolated and $w\neq z_i$ for $i=2,\dots,m-1$, we have $w\in \lopi[\Lambda(\Gamma)]{z_0}{z_1}$.
By \refprop{oneEndFix}, there exists a $\opi[\Lambda(\Gamma)]{z_m}{w}$-side sequence $\{\lambda_i\}$ in $\cL$ with $\lambda_i\cap\{w,z_0\}=\varnothing$ for all $i$.
Since $z_1\in \sfd\sfa_n$ and $x_n\in \opi[\Lambda(\Gamma)]{z_0}{z_1}$ for $n>\cN$, the leaf $\sfd\sfa_n$ is linked with $\lambda_i$ for all large $i$, contradicting \refprop{auxiliaryLeaf}.

It remains to prove the claim.

Since $m>2$ and $\ell_{m-1}$ exists, $\opi[\Lambda(\Gamma)]{z_m}{z_{m-1}}$ is isolated, hence $\ell_m$ is a boundary leaf.
Let $\cL_0$ be the closure of the $\Gamma$-orbit of $\ell_m$ in $\cM$, and $\cG_0\subset \cG=\cG(\Gamma,\{z_i\}_{i=1}^m)$ the corresponding geodesic sublamination in $\cH/\Gamma$.
Let $\sfg$ be the leaf associated with $\ell_m$.

If $z_{m-1}$ is fixed by a non-trivial element, then $\ell_m$ (hence $\sfg$) is isolated by \refprop{oneEndFix}.
Since no non-trivial element fixes $z_m$, \reflem{closureOfLeaf} implies that $\sfg$ lies in a principal region of $\cG_0''$ and its end at $z_m$ spirals toward a component of $\cG_0''$.
\refprop{gapShape} yields the desired leaf.

If no non-trivial element fixes $z_{m-1}$, then $\opi[\Lambda(\Gamma)]{z_m}{z_{m-1}}$ is isolated, so $\sfg$ is a boundary leaf of $\cG_0$.
\reflem{closureOfLeaf} implies that both ends of $\sfg$ spiral toward components of $\cG_0''$, possibly the same.
\refprop{gapShape} again yields the desired leaf.\qedhere

\end{proof}

\subsection{Outside and inside sequences}
Unlike in \refsec{uniformBouncing}, there is no well-defined invariant fence in the present setting.
Nevertheless, it is still useful to think of $\ell$ as a two-fold fence.
From this point of view, approaching $\ell$ from its inaccessible side should be regarded as an inside approximation, whereas approaching $\ell$ from the other side should be regarded as an outside approximation.
In this heuristic picture, the former corresponds to approaching the $(+)$-side of a fence, and the latter to approaching the $(-)$-side, respectively.

Let $\cL=\cL(\Gamma,\{z_k\}_{k=1}^m)$ be a circle lamination constructed by either \refconst{firstStepConst} or \refconst{extConst}.
Given a leaf $\lambda$ of $\cL(\Gamma,z_1)$, choose a conical limit sequence $\{h_i\}_{i\in \NN}$ in $\Gamma$ with respect to $\lambda$.

We call $\{h_i\}_{i\in\NN}$ an \emph{outside approximation} to $\lambda$ if $\{h_i(p)\}_{i\in\NN}\subset (\Int \cE_\fN)^c$ for some $\fN\in\ZZ$.
We call it an \emph{inside approximation} if, for each $n\in\ZZ$, the sequence $\{h_i(p)\}_{i\in\NN}$ is eventually contained in $\Int \cE_n$.

By \refprop{eclipseLProperty}, every conical limit sequence admits a subsequence of one of these two types, since $h_i(p)\in Z^+$ and $h_i(p)\notin\{p,q\}$.

\begin{lem}[Canonical approximations for conical limit pairs]\label{Lem:canonicalApprox}
Assume the conditions in \refconst{firstStepConst}.
Let $\cL=\cL(\Gamma,\{z_k\}_{k=1}^m)$ be a circle lamination constructed by either \refconst{firstStepConst} or \refconst{extConst}.
Let $\lambda$ be a leaf of $\cL(\Gamma,z_1)$. Then
the following hold:
\begin{itemize}
    \item if $\lambda$ admits an outside approximation, then $\lambda \subset \closure{\cE_\fN \setminus \cE_{\fN+1}}$ for some $\fN \in \ZZ$;
    \item if $\lambda$ admits an inside approximation, then $\lambda\subset \cE_\infty$.
\end{itemize}
Moreover, the following hold:
\begin{enumerate}
    \item  $\lambda \not\subset \ell$:
    \begin{itemize}
        \item if $\lambda\subset \cE_\infty$, then $\lambda$ has no outside approximation;
        \item if $\lambda \not\subset  \cE_\infty$, then $\lambda$ has no inside approximation;
    \end{itemize}
    \item $\lambda \subset \ell$:
    \begin{itemize}

        \item if $\lambda$ admits an outside approximation, then  $\lambda\subset [s_n,s_{n+1}]$ for some $n\in \ZZ$.
    \end{itemize}
\end{enumerate}
\end{lem}
\begin{proof}
Without loss of generality, we may assume that $\ell$ is right inaccessible.
Let $\{h_i\}_{i\in\NN}\subset \Gamma$ be a conical limit sequence for $z_1$ with respect to $\lambda=\{a,r\}$ such that $h_i(z_1)\to r$ and $h_i(y)\to a$ for all $y\in \Lambda(\Gamma)\setminus\{z_1\}$.
By \refcor{limitAuxiliaryLeaf}, we have $z_0\neq z_1$, and hence $\cldi[\Lambda(\Gamma)]{z_0}{x_\cN}$ is a compact subset of $\Lambda(\Gamma)\setminus\{z_1\}$.
Therefore \(h_i(\cldi[\Lambda(\Gamma)]{z_0}{x_\cN})\to a\) as $i\to\infty$.

On the other hand, since $\Fix(h_i)\subset \Lambda(\Gamma)$ and $h_i(p)\neq p$, \refprop{squashedFenceSystem} implies that
\(h_i(\sfa_{n(i)})\subset \closure{\cE_{k(i)}\setminus \cE_{k(i)+1}}\)
for some $n(i)>\cN$ and $k(i)\in \ZZ$.
By the above observation, we have $\sfd h_i(\sfa_{n(i)})\to \lambda$ as $i\to\infty$.

If $\{h_i\}_{i\in\NN}$ is an outside approximation, then \refprop{eclipseLProperty} implies, after passing to a subsequence, that $k(i)$ is constant.
It follows that
\(
\lambda \subset \closure{\cE_{k(i)}\setminus \cE_{k(i)+1}},
\)
so this constant value is the desired $\fN$.
If $\{h_i\}_{i\in\NN}$ is an inside approximation, then \refprop{eclipseLProperty} implies that $k(i)\to\infty$, and hence $\lambda\subset \cE_\infty$.
This proves the first assertion.

The second statement follows by applying the first assertion together with \refprop{eclipseLProperty}.
This completes the proof.
\end{proof}

A sequence $\{\lambda_i\}_{i\in\NN}$ of leaves  of $\cL$ is said to be \emph{regular} if the following hold:
\begin{itemize}
    \item If $m=1$, then either every $\lambda_i$ admits an outside approximation, or every $\lambda_i$ admits an inside approximation.
    \item if $m>1$, then every $\lambda_i$ is in the $\Gamma$-orbit of $\{\ell_j\}_{j=2}^{m}$.
\end{itemize}

Note that any leaf in $\cL$ is approximated by a regular sequence of leaves by the construction of $\cL$ (\refprop{finiteGenLeaves}).
A regular sequence $\{\lambda_i\}_{i\in\NN}$ is said to be:
\begin{itemize}
    \item \emph{outside} if there is an integer $\fN\in \ZZ$ such that $\Int \cE_\fN\cap\lambda_i=\varnothing$ for all $i\in \NN$ and when  $m=1$, every $\lambda_i$ admits an outside approximation;
    \item \emph{inside}, otherwise.
\end{itemize}

\begin{prop}[Dragging the limit points by a squashed fence]\label{Prop:squashedFenceDragging}
    Let $\cL=\cL(\Gamma,\{z_i\}_{i=1}^m)$ be a circle lamination constructed by either \refconst{firstStepConst} or \refconst{extConst}.
    If an open interval $I$ in $\Lambda(\Gamma)$ admits a $I$-side sequence of leaves, then $\partial I=\lambda_\infty$ is a leaf of $\cL$ and there is a $I$-side regular sequence  $\{\lambda_i\}_{i\in \NN}$ of leaves such that $I_i\subsetneq I_{i+1}\subset I$ where $ I_i$ is a component of $\Lambda(\Gamma)\setminus \lambda_i$.
    In particular, the following hold:
\begin{itemize}
    \item if  $\{\lambda_i\}_{i\in\NN}$ is outside,  then  $\lambda_\infty \cap \Int \cE_\fN=\varnothing$ for some $\fN \in \ZZ$;
    \item if $\{\lambda_i\}_{i\in\NN}$ is inside, then $\lambda_\infty\subset \cE_\infty$.
\end{itemize}
In particular, if $\lambda_\infty \subset \cE_\infty$ and $\lambda_\infty \not \subset \ell\cup \Fix(g)$, then $\{\lambda_i\}_{i\in\NN}$ is inside.
\end{prop}
\begin{proof}
By assumption, there is an $I$-sided sequence $\{\eta_i\}_{i\in\NN}$ of leaves in $\cL$.
After passing to a subsequence, we may assume that $J_i\subsetneq J_{i+1}\subset I$, where $\partial J_i=\eta_i$.
By the construction of $\cL$ and \refprop{finiteGenLeaves}, the following hold:
\begin{itemize}
    \item if $m=1$, then every $\eta_i$ lies in $\cL(\Gamma,z_1)$ and hence admits either an outside or an inside approximation;
    \item if $m>1$, then every $\eta_i$ lies in the closure in $\cM$ of the $\Gamma$-orbit of $\{\ell_j\}_{j=2}^{m}$.
\end{itemize}
Hence we may choose an $I$-side regular sequence $\{\lambda_i\}_{i\in\NN}$ of leaves such that $I_i\subsetneq I_{i+1}\subset I$, where $I_i$ is a component of $\Lambda(\Gamma)\setminus\lambda_i$.

To prove the second statement, first assume that $\{\lambda_i\}_{i\in\NN}$ is outside.
Then there exists $\fN\in\NN$ such that $\Int \cE_{\fN}\cap \lambda_i=\varnothing$ for all $i\in\NN$.
Therefore $\Int \cE_{\fN}\cap \lambda_\infty=\varnothing$.

Now assume that $\{\lambda_i\}_{i\in\NN}$ is inside.
If $m>1$, then by \refprop{squashedFenceSystem}, after passing to a subsequence of $\{\lambda_i\}_{i\in\NN}$, there exists a strictly increasing sequence $\{n(i)\}_{i\in\NN}$ of integers such that $\lambda_i\subset \Int \cE_{n(i)}$.
This implies that $\lambda_\infty\subset \cE_\infty$.

If $m=1$, there are two cases:
\begin{enumerate}
    \item every $\lambda_i$ admits an inside approximation;
    \item every $\lambda_i$ admits an outside approximation and, after passing to a subsequence, there exists a strictly increasing sequence $\{n(i)\}_{i\in\NN}$ of integers such that $\lambda_i\cap \Int \cE_{n(i)}\neq\varnothing$.
\end{enumerate}
Thus the desired result follows from \reflem{canonicalApprox}. \qedhere
\end{proof}

\subsection{Uniform bouncing direction}
We now return to the analysis of the uniform bouncing direction induced by outside and inside sequences associated with the squashed fence.

\begin{lem}[One-step extension by outside sequences for a squashed fence]\label{Lem:outsideExt2}
Assume \refconst{firstStepConst}.
Suppose that $\cL=\cL(\Gamma,\{z_i\}_{i=1}^m)$ is a circle lamination constructed by \refconst{firstStepConst} or \refconst{extConst}, and that $\ell$ is right (resp.\ left) inaccessible.
Assume further that there exists an outside regular sequence $\{\lambda_i\}_{i\in\NN}$ in $\cL$ which is $\opi[\Lambda(\Gamma)]{z_m}{w}$-sided (resp.\ $\opi[\Lambda(\Gamma)]{w}{z_m}$-sided) for some $w\in \opi[\Lambda(\Gamma)]{z_m}{z_{m-1}}$ (resp.\ $w\in \opi[\Lambda(\Gamma)]{z_{m-1}}{z_m}$).
Then there exists a point $z_{m+1}$ in $\opi[\Lambda(\Gamma)]{z_m}{z_{m-1}}$ (resp.\ $\opi[\Lambda(\Gamma)]{z_{m-1}}{z_m}$) such that $\{z_i\}_{i=1}^{m+1}$ satisfies the bouncing condition in \refconst{extConst}.
\end{lem}

\begin{proof}
Assume that $z_m$ is not fixed by any non-trivial element of $\Gamma$; otherwise, the conclusion follows from \reflem{oneStepAtFix2}.
By symmetry, it suffices to treat the case where $\ell$ is right inaccessible.

Since $\{\lambda_i\}_{i\in \NN}$ is an outside regular sequence with respect to $\fF(g)$, we may choose $\sfn>\cN$ so that
$\Int \cE_{\sfn-2}\cap \lambda_i=\varnothing$ for all $i\in \NN$, $s_{\sfn-2}<z_m$ in $\ell$, and the end ray of $\ell$ from $s_{\sfn-2}$ to $q$ is contained in $\cD^R(\Lambda(\Gamma))$.
Moreover, since $\{\lambda_i\}_{i\in \NN}$ is $\opi[\Lambda(\Gamma)]{z_m}{w}$-sided, the leaf
$\lambda_\infty=\{z_m,w\}$ is also a leaf of $\cL$ and
$\lambda_\infty\cap \Int \cE_{\sfn-2}=\varnothing$.

Set $\gamma=\sfa_{\sfn}$ if $m=1$ and $\gamma=\gamma_m$ if $m>1$, and let
\[
\cJ=\sigma_{\sfn-2}\cup \sigma_{\sfn+1}\cup [s_{\sfn-2},s_{\sfn+1}].
\]
Then $\gamma$ is adapted to $Z^\pm$, and $\cJ$ is a fence whose nodes are the synapses of $\sigma_{\sfn-2}$ and $\sigma_{\sfn+1}$.
Write $\sfd\gamma=\{z_m,\overline{z_m}\}$ and $\fD=\cH^R(\gamma)$.
Then $z_m<w<\overline{z_m}$ in $\partial\fD$.

Take a flat neighborhood $U$ of $z_m$ relative to $\ell$ such that
\[
\closure{U}\cap \partial \cE_{\sfn+1}=\closure{U}\cap \ell,\quad
\closure{U}\cap \cJ=\varnothing,\quad
\closure{U}\cap \cldi[\Lambda(\Gamma)]{w}{\overline{z_m}}=\varnothing.
\]
Then take a flat neighborhood $V$ of $w$ relative to $\opi[\partial \fD]{z_m}{\overline{z_m}}$ such that
\[
\closure{V}\cap \closure{U}=\varnothing
\quad\text{and}\quad
\closure{V}\cap \partial \fD \subset \opi[\partial \fD]{z_m}{\overline{z_m}}.
\]

Observe that:
\begin{itemize}
    \item $\partial U\cap \ell=\{e_0,e_1\}$ with $e_0>z_m>e_1$ in $\ell$, and in fact $e_0>z_m>e_1>s_{\sfn-2}$ since $\closure{U}\cap \cJ=\varnothing$;
    \item there exists an open interval $\cI=\opi[\Lambda(\Gamma)]{u_0}{u_1}$ containing $z_m$ and crossing $U$, since $\partial U$ separates $z_m$ and $w$, and the ray $[e_0,z_m)$ lands at $z_m$ from the left side of $\Lambda(\Gamma)$;
    \item $\partial V\cap \Lambda(\Gamma)=\{u_2,u_3\}$ with $z_m<u_2<w<u_3<\overline{z_m}$ in $\Lambda(\Gamma)$;
    \item hence $u_0<z_m<u_1<u_2<w<u_3<\overline{z_m}$ in $\Lambda(\Gamma)$.
\end{itemize}

After passing to a subsequence, we may assume that
\[
\cldi[\Lambda(\Gamma)]{u_1}{u_2}\subset I_i\subsetneq I_{i+1}\subset \opi[\Lambda(\Gamma)]{z_m}{w},
\]
where $I_i$ is a component of $\Lambda(\Gamma)\setminus \lambda_i$.
If infinitely many $\lambda_i$ share an endpoint $\varepsilon\in\{z_m,w\}$, then \refprop{oneEndFix} implies that $\varepsilon$ is fixed by a non-trivial element $h\in \Gamma$.
Since, by assumption, no non-trivial element of $\Gamma$ fixes $z_m$, we must have $\varepsilon=w$.
But then \refprop{oneEndFix} implies that $\lambda_\infty$ is isolated, contradicting $\lambda_i\to \lambda_\infty$.
Thus, after passing to a subsequence, we may assume that
\[
\closure{I_i}\subset I_{i+1}\subset \opi[\Lambda(\Gamma)]{z_m}{w}.
\]

\begin{claim}
There exists a sequence $\{\sfb_i\}_{i\in\NN}$ of adapted arcs crossing $\cH$ such that:
\begin{itemize}
    \item if $m=1$, then $\sfb_i$ is a mixed arc whose synapse lies in $Z^+$, while if $m>1$, then $\sfb_i$ is a pure arc in $Z^-$;
    \item $\sfb_i\cap \sfb_j=\varnothing$ whenever $i\ne j$, and $\sfb_i\cap \Int \cE_{\sfn-1}=\varnothing$;
    \item
    $\cldi[\Lambda(\Gamma)]{u_1}{u_2}\subset \closure{J_i}\subset J_{i+1}\subset \opi[\Lambda(\Gamma)]{z_m}{w}$
    and $\bigcup_{i\in\NN} J_i=\opi[\Lambda(\Gamma)]{z_m}{w}$, where $J_i$ denotes a component of $\Lambda(\Gamma)\setminus \sfd \sfb_i$,.
\end{itemize}
\end{claim}
\begin{proof}
    Assume first that \(m>1\).
Since \(\{\lambda_i\}_{i\in\NN}\) is regular, for each \(i\in\NN\) we may write
\(\lambda_i=h_i(\ell_{j(i)})\) for some \(h_i\in \Gamma\) and \(j(i)\in \{2,\dots,m\}\).
By \refprop{squashedFenceSystem}, there exists \(n(i)\in\ZZ\) such that
\(h_i(\gamma_{j(i)})\subset h_i(\ell)\subset \closure{\cE_{n(i)}\setminus \cE_{n(i)+1}}\).
Since \(\lambda_i\subset h_i(\gamma_{j(i)})\) and \(\lambda_i\cap \Int \cE_{\sfn-2}=\varnothing\), \refprop{eclipseLProperty} implies \(n(i)\le \sfn-2\). Hence
\(h_i(\gamma_{j(i)})\cap \Int \cE_{\sfn-1}=\varnothing\).
Set \(\sfb_i=h_i(\gamma_{j(i)})\).

It remains to show that the arcs \(\sfb_i\) are pairwise disjoint.
We first claim that \(\sfb_i\) cannot meet both \(\sfb_{i-1}\) and \(\sfb_{i+1}\) simultaneously.
Indeed, otherwise, since \(\sfd \sfb_i=\lambda_i\) and \(\closure{I_i}\subset I_{i+1}\) for all \(i\in\NN\), the leaves \(\lambda_{i-1}\) and \(\lambda_{i+1}\) lie in distinct components of \(\Lambda(\Gamma)\setminus \lambda_i\).
It follows that the arc joining an endpoint of \(\sfb_{i-1}\) to an endpoint of \(\sfb_{i+1}\) is a pure arc in \(Z^-\) linked with \(\sfb_i\), contradicting the right inaccessibility of \(\sfb_i\).

Therefore, after passing to a subsequence, we may assume that \(\sfb_i\cap \sfb_j=\varnothing\) whenever \(i\neq j\).
This gives the desired sequence.

We next consider the case \(m=1\).
Since \(\{\lambda_i\}_{i\in \NN}\) is an outside regular sequence, each \(\lambda_i\) admits an outside side approximation \(\{h_{i,j}\}_{j\in\NN}\).
By \refprop{squashedFenceSystem}, \(h_{i,j}(\cE_{m(i,j)}) \subset \closure{\cE_{n(i,j)}\setminus\cE_{n(i,j)+1}}\) for some \(m(i,j),n(i,j)\in \ZZ\).
Since \(\{h_{i,j}(z_1)\}_{j\in\NN}\) converges to an endpoint of \(\lambda_i\), \(h_{i,j}(z_1)\in h_{i,j}(\ell)\subset \closure{\cE_{n(i,j)}\setminus\cE_{n(i,j)+1}}\), and \(\lambda_i\cap \Int \cE_{\sfn-2}=\varnothing\), we see that for each \(i\), the sequence \(\{n(i,j)\}_{j\in\NN}\) is eventually less than or equal to \(\sfn-2\).
Therefore, after passing to a subsequence of \(\{h_{i,j}\}_{(i,j)\in \NN\times \NN}\), we have
\(h_{i,j}(\cE_{m(i,j)})\cap \Int \cE_{\sfn-1}=\varnothing\).

On the other hand, write \(I_i=\opi[\Lambda(\Gamma)]{a_i}{b_i}\).
Since \(\closure{I_i}\subset I_{i+1}\), we may choose sequences \(\{U_i\}_{i\in\NN}\) and \(\{V_i\}_{i\in\NN}\) of pairwise disjoint open intervals such that
\(\closure{U_1}\cap \closure{V_1}=\varnothing\), \(a_i\in U_i\), and \(b_i\in V_i\) for every \(i\in\NN\).

For each \(i\in\NN\), since \(\{h_{i,j}\}_{j\in \NN}\) is a conical limit sequence for \(z_1\) with respect to \(\lambda_i\), for each $i\in\NN$, there exists \(\sfj(i)\in\NN\) such that one of the following holds for all \(j\ge \sfj(i)\):
\begin{itemize}
    \item \(h_{i,j}(z_1)\in U_i\) and \(h_{i,j}(\cldi[\Lambda(\Gamma)]{z_0}{x_\cN})\subset V_i\); or
    \item \(h_{i,j}(z_1)\in V_i\) and \(h_{i,j}(\cldi[\Lambda(\Gamma)]{z_0}{x_\cN})\subset U_i\).
\end{itemize}

Now set \(\sfb_i=h_{i,\sfj(i)}(\sfa_{m(i,\sfj(i))})\).
Since \(h_{i,\sfj(i)}(\cE_{m(i,\sfj(i))})\cap \Int \cE_{\sfn-1}=\varnothing\), we have \(\sfb_i\cap \cE_{\sfn-1}=\varnothing\).
Also, by the choice of \(U_i\), \(V_i\), and \(\sfj(i)\), we have
\(\closure{J_i}\subset J_{i+1}\subset \opi[\Lambda(\Gamma)]{z_m}{w}\) and \(\bigcup_{i\in\NN} J_i=\opi[\Lambda(\Gamma)]{z_m}{w}\).

It remains to show that the arcs \(\sfb_i\) are pairwise disjoint.
As above, we claim that \(\sfb_i\) cannot meet both \(\sfb_{i-1}\) and \(\sfb_{i+1}\) simultaneously.
Otherwise, by \reflem{auxArcUnlinked}, the arcs \(\sfb_i,\sfb_{i-1},\sfb_{i+1}\) intersect only along their \((-)\)-segments.
Since \(\sfd \sfb_{i-1}\) and \(\sfd \sfb_{i+1}\) lie in distinct components of \(\Lambda(\Gamma)\setminus \sfd \sfb_i\) and the synapse of \(\sfb_i\) lies in \(Z^+\), the arc joining the endpoint of \(\sfb_{i-1}^-\) to the endpoint of \(\sfb_{i+1}^-\) is an arc in $Z^-$, linked with \(\sfb_i^-\), contradicting the right inaccessibility of \(\sfb_i^-\).

Therefore, after passing to a subsequence, we may assume that \(\sfb_i\cap \sfb_j=\varnothing\) whenever \(i\ne j\).
This gives the desired sequence.
\end{proof}

Let \(\{\sfb_i\}_{i\in \NN}\) be the sequence given by the above claim.
Define a Jordan curve
\[
\cC=[e_0,z_m]\cup \opi[\Lambda(\Gamma)]{z_m}{u_1}\cup \opi[\partial U]{u_1}{e_0},
\]
equipped with the circular order for which \(z_m<u_1<e_0\).
Write \(J_i=\opi[\Lambda(\Gamma)]{\sfx_i}{\sfy_i}\).

Since \(\partial U\) separates \(\sfx_i\) and \(\sfy_i\), there exists a subline \(\sft_i\subset \sfb_i\) such that
\(\sfd \sft_i=\{\sfx_i,\sfw_i\}\) for some \(\sfw_i\in \partial U\cap \sfb_i\), and \(\sft_i\) crosses \(\cD^L(\cC)\).

Since the arcs \(\{\sfb_i\}_{i\in\NN}\) are pairwise disjoint and the sequence \(\{\sfx_i\}_{i\in\NN}\) is decreasing, we have
\[
u_1<\sfw_i<\sfw_{i+1}<e_0 \quad \text{in} \quad \cC.
\]
Moreover, since \(\opi[\partial U]{e_1}{e_0}\subset \Int \cE_\sfn\) and \(\sfb_i\subset (\Int \cE_{\sfn-1})^c\), we have
\(\sfw_i\in \cldi[\partial U]{e_0}{e_1}\).
Hence \(u_1<\sfw_i\le e_1<e_0<u_0\) in \(\partial U\), and thus \(u_1<\sfw_i\le e_1<e_0\) in \(\cC\).

\begin{claim}
\([z_m,e_1]\) crosses \(\cD^L(\cC)\).
\end{claim}
\begin{proof}[Proof of the claim]
It suffices to show that \([z_m,e_1]\) does not meet \(\opi[\Lambda(\Gamma)]{z_m}{u_1}\).
Assume otherwise.
Since \(u_1<e_1<e_0\) in \(\cC\), the ray \([z_m,e_1)\) lands at \(e_1\) from the left side of \(\opi[\cC]{u_1}{e_0}\).
Since \((z_m,e_1)\) is disjoint from \([e_0,z_m)\), by assumption, there exists \(\sfe \in \opi[\Lambda(\Gamma)]{z_m}{u_1}\cap (z_m,e_1)\) such that \(\cA=(\sfe,e_1)\) crosses \(\cC\) through \(\cD^L(\cC)\).

Since \(\{\sfx_i\}_{i\in \NN}\) converges to \(z_m\) on \(\opi[\Lambda(\Gamma)]{z_m}{u_1}\), we may choose \(\sfk\in \NN\) such that \(\sfx_\sfk\in \opi[\cC]{z_m}{\sfe}\).
On the other hand, \(\sfw_\sfk\in \lopi[\cC]{u_1}{e_1}\).

Since \(\sfb_i\cap \sfb_j=\varnothing\) for all \(i\neq j\) and \(\sfw_i<\sfw_{i+1}\le e_1\) in \(\cC\) for all \(i\in \NN\), we have \(\sfw_i\neq e_1\).
Hence \(\sfd \sft_j\) and \(\sfd \cA\) are linked in \(\cC\).
Since \(\cA\subset Z^-\), the disjointness of \(Z^\pm\) implies that \(\sft_j\) is not pure in \(Z^+\).
Therefore \(\sft_\sfk^-\) and \(\cA\) are linked, contradicting the one-sided simpliciality of \(\ell\).
\end{proof}
By the above claim, a truncation of the ray \((z_m,e_1]\subset \ell\) lands at \(z_m\) on $\fD$.
Extending this truncation, we obtain an arc in \(Z^-\) crossing \(\cH^R(\gamma)\).
If \(m>1\), the proof is complete.
If \(m=1\), \refcor{limitAuxiliaryLeaf} yields the conclusion.
\qedhere
\end{proof}

\setcounter{claim}{0}

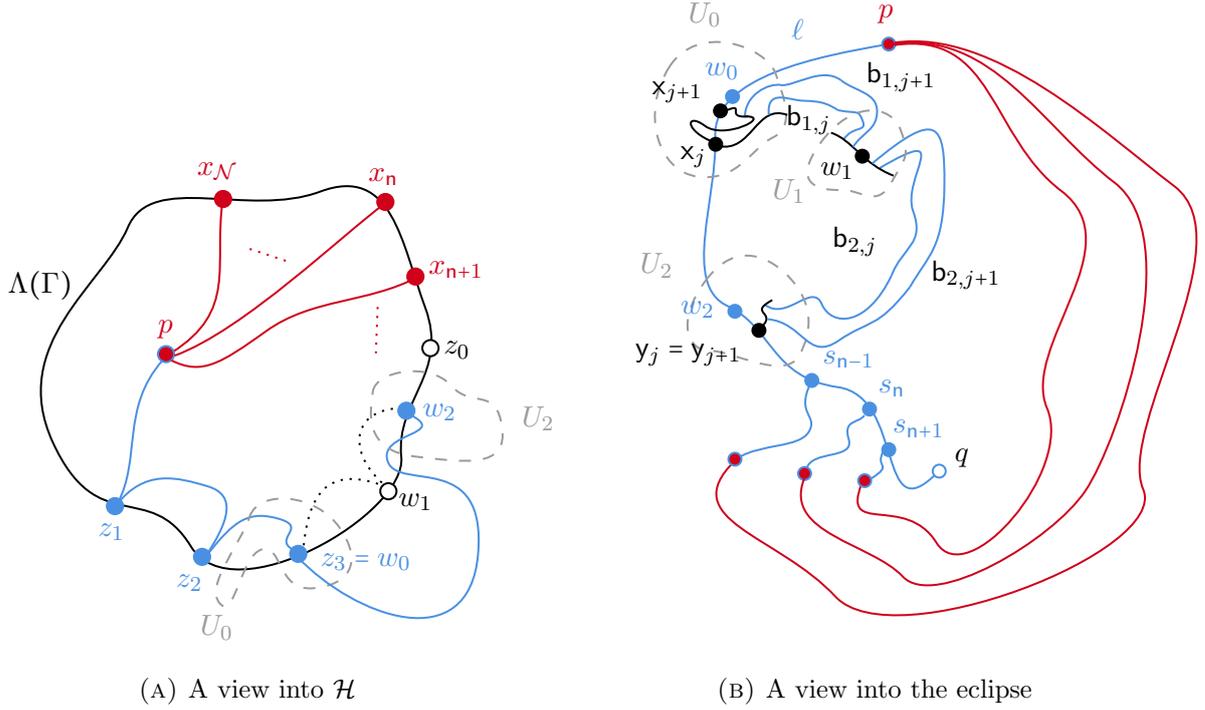
\begin{figure}
    \centering

\begin{subfigure}[t]{0.4\textwidth}

\tikzset{every picture/.style={line width=0.75pt}} %set default line width to 0.75pt        

\begin{tikzpicture}[x=0.75pt,y=0.75pt,yscale=-1.5,xscale=1.5]
%uncomment if require: \path (0,349); %set diagram left start at 0, and has height of 349

%Shape: Polygon Curved [id:ds05540374141696458] 
\draw   (224.86,92.79) .. controls (248.86,83.79) and (251.86,109.79) .. (259.86,131.79) .. controls (267.86,153.79) and (249.86,162.79) .. (251.86,181.79) .. controls (253.86,200.79) and (195.86,234.79) .. (181.86,211.79) .. controls (167.86,188.79) and (150.86,209.79) .. (135.86,176.79) .. controls (120.86,143.79) and (152.86,123.79) .. (160.86,104.79) .. controls (168.86,85.79) and (200.86,101.79) .. (224.86,92.79) -- cycle ;
%Curve Lines [id:da4462564185393767] 
\draw [color={rgb, 255:red, 208; green, 2; blue, 27 }  ,draw opacity=1 ]   (175.4,146.8) .. controls (197.4,132.8) and (189.4,117.8) .. (192.4,95.8) ;
%Shape: Circle [id:dp011581361592361339] 
\draw  [color={rgb, 255:red, 74; green, 144; blue, 226 }  ,draw opacity=1 ][fill={rgb, 255:red, 208; green, 2; blue, 27 }  ,fill opacity=1 ][line width=0.75]  (170.86,147.29) .. controls (170.86,145.8) and (172.07,144.58) .. (173.57,144.58) .. controls (175.06,144.58) and (176.27,145.8) .. (176.27,147.29) .. controls (176.27,148.79) and (175.06,150) .. (173.57,150) .. controls (172.07,150) and (170.86,148.79) .. (170.86,147.29) -- cycle ;
%Curve Lines [id:da733176881155499] 
\draw [color={rgb, 255:red, 74; green, 144; blue, 226 }  ,draw opacity=1 ]   (157.4,197.4) .. controls (165.4,182.4) and (158.4,165.4) .. (172.4,149.4) ;
%Curve Lines [id:da048287397004118104] 
\draw [color={rgb, 255:red, 208; green, 2; blue, 27 }  ,draw opacity=1 ]   (176.4,147.8) .. controls (199.4,139.4) and (233.4,106.2) .. (246.4,96.8) ;
%Curve Lines [id:da9937446761034313] 
\draw [color={rgb, 255:red, 208; green, 2; blue, 27 }  ,draw opacity=1 ]   (175.4,149.4) .. controls (185.4,154.4) and (191.44,150.71) .. (206.92,139.55) .. controls (222.4,128.4) and (240.4,130.4) .. (257.4,121.4) ;
%Straight Lines [id:da5675501007197415] 
\draw [color={rgb, 255:red, 208; green, 2; blue, 27 }  ,draw opacity=1 ] [dash pattern={on 0.84pt off 2.51pt}]  (244,131.48) -- (243.4,147) ;
%Shape: Circle [id:dp5283793105166963] 
\draw  [color={rgb, 255:red, 208; green, 2; blue, 27 }  ,draw opacity=1 ][fill={rgb, 255:red, 208; green, 2; blue, 27 }  ,fill opacity=1 ][line width=0.75]  (189.86,95.27) .. controls (189.86,93.77) and (191.07,92.56) .. (192.57,92.56) .. controls (194.06,92.56) and (195.27,93.77) .. (195.27,95.27) .. controls (195.27,96.76) and (194.06,97.98) .. (192.57,97.98) .. controls (191.07,97.98) and (189.86,96.76) .. (189.86,95.27) -- cycle ;
%Shape: Circle [id:dp5753513497216594] 
\draw  [color={rgb, 255:red, 208; green, 2; blue, 27 }  ,draw opacity=1 ][fill={rgb, 255:red, 208; green, 2; blue, 27 }  ,fill opacity=1 ][line width=0.75]  (243.86,96.27) .. controls (243.86,94.77) and (245.07,93.56) .. (246.57,93.56) .. controls (248.06,93.56) and (249.27,94.77) .. (249.27,96.27) .. controls (249.27,97.76) and (248.06,98.98) .. (246.57,98.98) .. controls (245.07,98.98) and (243.86,97.76) .. (243.86,96.27) -- cycle ;
%Shape: Circle [id:dp7107677083809859] 
\draw  [color={rgb, 255:red, 208; green, 2; blue, 27 }  ,draw opacity=1 ][fill={rgb, 255:red, 208; green, 2; blue, 27 }  ,fill opacity=1 ][line width=0.75]  (253.86,121.27) .. controls (253.86,119.77) and (255.07,118.56) .. (256.57,118.56) .. controls (258.06,118.56) and (259.27,119.77) .. (259.27,121.27) .. controls (259.27,122.76) and (258.06,123.98) .. (256.57,123.98) .. controls (255.07,123.98) and (253.86,122.76) .. (253.86,121.27) -- cycle ;
%Straight Lines [id:da7349562917293778] 
\draw [color={rgb, 255:red, 208; green, 2; blue, 27 }  ,draw opacity=1 ] [dash pattern={on 0.84pt off 2.51pt}]  (201.4,112) -- (213.4,116) ;
%Shape: Circle [id:dp5126954816248799] 
\draw  [color={rgb, 255:red, 0; green, 0; blue, 0 }  ,draw opacity=1 ][fill={rgb, 255:red, 255; green, 255; blue, 255 }  ,fill opacity=1 ][line width=0.75]  (258.86,145.27) .. controls (258.86,143.77) and (260.07,142.56) .. (261.57,142.56) .. controls (263.06,142.56) and (264.27,143.77) .. (264.27,145.27) .. controls (264.27,146.76) and (263.06,147.98) .. (261.57,147.98) .. controls (260.07,147.98) and (258.86,146.76) .. (258.86,145.27) -- cycle ;
%Shape: Circle [id:dp5623198978559015] 
\draw  [color={rgb, 255:red, 74; green, 144; blue, 226 }  ,draw opacity=1 ][fill={rgb, 255:red, 74; green, 144; blue, 226 }  ,fill opacity=1 ][line width=0.75]  (153.86,198.27) .. controls (153.86,196.77) and (155.07,195.56) .. (156.57,195.56) .. controls (158.06,195.56) and (159.27,196.77) .. (159.27,198.27) .. controls (159.27,199.76) and (158.06,200.98) .. (156.57,200.98) .. controls (155.07,200.98) and (153.86,199.76) .. (153.86,198.27) -- cycle ;
%Curve Lines [id:da1940972361564719] 
\draw [color={rgb, 255:red, 74; green, 144; blue, 226 }  ,draw opacity=1 ]   (157,198) .. controls (166.4,185.6) and (177.4,188.6) .. (187.4,191.6) .. controls (197.4,194.6) and (194.4,203.6) .. (186.4,215.6) ;
%Shape: Circle [id:dp1926953431119346] 
\draw  [color={rgb, 255:red, 74; green, 144; blue, 226 }  ,draw opacity=1 ][fill={rgb, 255:red, 74; green, 144; blue, 226 }  ,fill opacity=1 ][line width=0.75]  (182.86,215.27) .. controls (182.86,213.77) and (184.07,212.56) .. (185.57,212.56) .. controls (187.06,212.56) and (188.27,213.77) .. (188.27,215.27) .. controls (188.27,216.76) and (187.06,217.98) .. (185.57,217.98) .. controls (184.07,217.98) and (182.86,216.76) .. (182.86,215.27) -- cycle ;
%Shape: Circle [id:dp38026521250894774] 
\draw  [color={rgb, 255:red, 74; green, 144; blue, 226 }  ,draw opacity=1 ][fill={rgb, 255:red, 74; green, 144; blue, 226 }  ,fill opacity=1 ][line width=0.75]  (214.86,214.27) .. controls (214.86,212.77) and (216.07,211.56) .. (217.57,211.56) .. controls (219.06,211.56) and (220.27,212.77) .. (220.27,214.27) .. controls (220.27,215.76) and (219.06,216.98) .. (217.57,216.98) .. controls (216.07,216.98) and (214.86,215.76) .. (214.86,214.27) -- cycle ;
%Curve Lines [id:da8843946433202704] 
\draw [color={rgb, 255:red, 74; green, 144; blue, 226 }  ,draw opacity=1 ]   (187,215) .. controls (196.4,202.6) and (201.4,199.6) .. (211.4,202.6) .. controls (221.4,205.6) and (207.4,209.6) .. (217.4,213.6) ;
%Curve Lines [id:da265547329278496] 
\draw  [dash pattern={on 0.84pt off 2.51pt}]  (219.4,211.4) .. controls (221.4,191.4) and (229.4,184.4) .. (248.4,192.4) ;
%Shape: Circle [id:dp7372754944325217] 
\draw  [color={rgb, 255:red, 0; green, 0; blue, 0 }  ,draw opacity=1 ][fill={rgb, 255:red, 255; green, 255; blue, 255 }  ,fill opacity=1 ][line width=0.75]  (244.86,193.27) .. controls (244.86,191.77) and (246.07,190.56) .. (247.57,190.56) .. controls (249.06,190.56) and (250.27,191.77) .. (250.27,193.27) .. controls (250.27,194.76) and (249.06,195.98) .. (247.57,195.98) .. controls (246.07,195.98) and (244.86,194.76) .. (244.86,193.27) -- cycle ;
%Curve Lines [id:da704378340621452] 
\draw  [dash pattern={on 0.84pt off 2.51pt}]  (245.4,190.4) .. controls (233.4,179.4) and (235.4,164.4) .. (254.4,165.4) ;
%Shape: Circle [id:dp19034171415149415] 
\draw  [color={rgb, 255:red, 74; green, 144; blue, 226 }  ,draw opacity=1 ][fill={rgb, 255:red, 74; green, 144; blue, 226 }  ,fill opacity=1 ][line width=0.75]  (250.86,166.27) .. controls (250.86,164.77) and (252.07,163.56) .. (253.57,163.56) .. controls (255.06,163.56) and (256.27,164.77) .. (256.27,166.27) .. controls (256.27,167.76) and (255.06,168.98) .. (253.57,168.98) .. controls (252.07,168.98) and (250.86,167.76) .. (250.86,166.27) -- cycle ;
%Shape: Polygon Curved [id:ds6469359180943016] 
\draw  [color={rgb, 255:red, 155; green, 155; blue, 155 }  ,draw opacity=1 ][dash pattern={on 4.5pt off 4.5pt}] (201.4,201.8) .. controls (204.4,192.6) and (223.4,193.8) .. (231.4,205.6) .. controls (239.4,217.4) and (234.4,224.8) .. (222.4,225.8) .. controls (210.4,226.8) and (212.4,210.8) .. (205.4,212.8) .. controls (198.4,214.8) and (196.4,236.8) .. (191.4,230.8) .. controls (186.4,224.8) and (198.4,211) .. (201.4,201.8) -- cycle ;
%Curve Lines [id:da579831816067839] 
\draw [color={rgb, 255:red, 74; green, 144; blue, 226 }  ,draw opacity=1 ]   (217.4,215.4) .. controls (236.4,235.4) and (274.4,249.4) .. (277.4,214.4) .. controls (280.4,179.4) and (251.15,193.41) .. (247.4,183.4) .. controls (243.65,173.39) and (268.52,174.46) .. (254.4,167.4) ;
%Shape: Polygon Curved [id:ds3212489293048526] 
\draw  [color={rgb, 255:red, 155; green, 155; blue, 155 }  ,draw opacity=1 ][dash pattern={on 4.5pt off 4.5pt}] (242.4,159.8) .. controls (243.4,152.8) and (261.4,149.8) .. (268.4,157.8) .. controls (275.4,165.8) and (283.4,157.8) .. (285.4,171.8) .. controls (287.4,185.8) and (273.29,183.6) .. (256.29,181.6) .. controls (239.29,179.6) and (241.4,166.8) .. (242.4,159.8) -- cycle ;

% Text Node
\draw (120,118.4) node [anchor=north west][inner sep=0.75pt]    {$\Lambda ( \Gamma )$};
% Text Node
\draw (170,135) node [anchor=north west][inner sep=0.75pt]  [font=\normalsize,color={rgb, 255:red, 208; green, 2; blue, 27 }  ,opacity=1 ]  {$p$};
% Text Node
\draw (240,84) node [anchor=north west][inner sep=0.75pt]  [font=\normalsize,color={rgb, 255:red, 208; green, 2; blue, 27 }  ,opacity=1 ]  {$x_{\sfn}$};
% Text Node
\draw (260,115) node [anchor=north west][inner sep=0.75pt]  [font=\normalsize,color={rgb, 255:red, 208; green, 2; blue, 27 }  ,opacity=1 ]  {$x_{\sfn+1}$};
% Text Node
\draw (183,82) node [anchor=north west][inner sep=0.75pt]  [font=\normalsize,color={rgb, 255:red, 208; green, 2; blue, 27 }  ,opacity=1 ]  {$x_{\cN}$};
% Text Node
\draw (265,142) node [anchor=north west][inner sep=0.75pt]  [font=\normalsize,color={rgb, 255:red, 0; green, 0; blue, 0 }  ,opacity=1 ]  {$z_{0}$};
% Text Node
\draw (150,203) node [anchor=north west][inner sep=0.75pt]  [font=\normalsize,color={rgb, 255:red, 74; green, 144; blue, 226 }  ,opacity=1 ]  {$z_{1}$};
% Text Node
\draw (176,220) node [anchor=north west][inner sep=0.75pt]  [font=\normalsize,color={rgb, 255:red, 74; green, 144; blue, 226 }  ,opacity=1 ]  {$z_{2}$};
% Text Node
\draw (224,213) node [anchor=north west][inner sep=0.75pt]  [font=\normalsize,color={rgb, 255:red, 74; green, 144; blue, 226 }  ,opacity=1 ]  {$z_{3} =w_{0}$};
% Text Node
\draw (250,193) node [anchor=north west][inner sep=0.75pt]  [font=\normalsize,color={rgb, 255:red, 0; green, 0; blue, 0 }  ,opacity=1 ]  {$w_{1}$};
% Text Node
\draw (258,163) node [anchor=north west][inner sep=0.75pt]  [font=\normalsize,color={rgb, 255:red, 74; green, 144; blue, 226 }  ,opacity=1 ]  {$w_{2}$};
% Text Node
\draw (184,233.8) node [anchor=north west][inner sep=0.75pt]  [color={rgb, 255:red, 155; green, 155; blue, 155 }  ,opacity=1 ]  {$U_{0}$};
% Text Node
\draw (291,163.8) node [anchor=north west][inner sep=0.75pt]  [color={rgb, 255:red, 155; green, 155; blue, 155 }  ,opacity=1 ]  {$U_{2}$};
\end{tikzpicture}
\subcaption{A view into $\cH$}
\end{subfigure}
\hspace{0.1\textwidth}
\begin{subfigure}[t]{0.4\textwidth}

\tikzset{every picture/.style={line width=0.75pt}} %set default line width to 0.75pt        

\begin{tikzpicture}[x=0.75pt,y=0.75pt,yscale=-1.2,xscale=1.2]
%uncomment if require: \path (0,349); %set diagram left start at 0, and has height of 349

%Curve Lines [id:da49664971489976817] 
\draw [color={rgb, 255:red, 74; green, 144; blue, 226 }  ,draw opacity=1 ]   (297.88,80.89) .. controls (271.43,85.59) and (226.4,93.93) .. (227.4,119.33) .. controls (228.4,144.73) and (213.4,181.33) .. (232.4,190.33) .. controls (251.4,199.33) and (250.4,220.4) .. (278.4,224.4) .. controls (306.4,228.4) and (294.4,285.6) .. (319.4,260.6) ;
%Curve Lines [id:da037224961871722195] 
\draw [color={rgb, 255:red, 74; green, 144; blue, 226 }  ,draw opacity=1 ]   (295.71,132.62) .. controls (327.71,113.62) and (304.71,136.62) .. (309.71,150.62) .. controls (314.71,164.62) and (305.4,165.4) .. (300.71,177.62) ;
%Curve Lines [id:da1602623902071365] 
\draw [color={rgb, 255:red, 74; green, 144; blue, 226 }  ,draw opacity=1 ]   (248.14,189.9) .. controls (257.14,196.9) and (267.71,180.17) .. (271.71,189.59) .. controls (275.71,199.02) and (290.4,199.4) .. (300.71,177.62) ;
%Curve Lines [id:da1859451159650124] 
\draw [color={rgb, 255:red, 74; green, 144; blue, 226 }  ,draw opacity=1 ]   (278.43,108.31) .. controls (287.14,110) and (286.71,107.76) .. (288.71,112.76) .. controls (290.71,117.76) and (284.57,114.5) .. (280.71,120.76) ;
%Shape: Circle [id:dp14552939039724488] 
\draw  [color={rgb, 255:red, 74; green, 144; blue, 226 }  ,draw opacity=1 ][fill={rgb, 255:red, 255; green, 255; blue, 255 }  ,fill opacity=1 ][line width=0.75]  (317.86,259.6) .. controls (317.86,258.1) and (319.07,256.89) .. (320.57,256.89) .. controls (322.06,256.89) and (323.27,258.1) .. (323.27,259.6) .. controls (323.27,261.1) and (322.06,262.31) .. (320.57,262.31) .. controls (319.07,262.31) and (317.86,261.1) .. (317.86,259.6) -- cycle ;
%Shape: Circle [id:dp3307410174618869] 
\draw  [color={rgb, 255:red, 74; green, 144; blue, 226 }  ,draw opacity=1 ][fill={rgb, 255:red, 208; green, 2; blue, 27 }  ,fill opacity=1 ][line width=0.75]  (296.86,80.6) .. controls (296.86,79.1) and (298.07,77.89) .. (299.57,77.89) .. controls (301.06,77.89) and (302.27,79.1) .. (302.27,80.6) .. controls (302.27,82.1) and (301.06,83.31) .. (299.57,83.31) .. controls (298.07,83.31) and (296.86,82.1) .. (296.86,80.6) -- cycle ;
%Shape: Circle [id:dp6398620601025935] 
\draw  [color={rgb, 255:red, 74; green, 144; blue, 226 }  ,draw opacity=1 ][fill={rgb, 255:red, 74; green, 144; blue, 226 }  ,fill opacity=1 ][line width=0.75]  (231.86,102.6) .. controls (231.86,101.1) and (233.07,99.89) .. (234.57,99.89) .. controls (236.06,99.89) and (237.27,101.1) .. (237.27,102.6) .. controls (237.27,104.1) and (236.06,105.31) .. (234.57,105.31) .. controls (233.07,105.31) and (231.86,104.1) .. (231.86,102.6) -- cycle ;
%Shape: Polygon Curved [id:ds9904742317568836] 
\draw  [color={rgb, 255:red, 155; green, 155; blue, 155 }  ,draw opacity=1 ][dash pattern={on 4.5pt off 4.5pt}] (208.29,90.19) .. controls (220.29,71.19) and (248.4,81.73) .. (252.4,92.73) .. controls (256.4,103.73) and (260.51,115.27) .. (251.4,119.73) .. controls (242.29,124.19) and (237.29,138.19) .. (220.29,136.19) .. controls (203.29,134.19) and (196.29,109.19) .. (208.29,90.19) -- cycle ;
%Shape: Polygon Curved [id:ds32559968235640446] 
\draw  [color={rgb, 255:red, 155; green, 155; blue, 155 }  ,draw opacity=1 ][dash pattern={on 4.5pt off 4.5pt}] (223.4,179.73) .. controls (236.4,161.93) and (251.29,170.19) .. (259.4,179.73) .. controls (267.51,189.27) and (268.29,200.39) .. (262.29,211.19) .. controls (256.29,221.99) and (234.29,208.99) .. (225.29,209.19) .. controls (216.29,209.39) and (210.4,197.53) .. (223.4,179.73) -- cycle ;
%Curve Lines [id:da10596837732015407] 
\draw    (228.71,123.33) .. controls (242.71,126.33) and (244.4,105.33) .. (257.4,110.33) ;
%Shape: Circle [id:dp43567048450442325] 
\draw  [color={rgb, 255:red, 0; green, 0; blue, 0 }  ,draw opacity=1 ][fill={rgb, 255:red, 0; green, 0; blue, 0 }  ,fill opacity=1 ][line width=0.75]  (224.86,122.6) .. controls (224.86,121.1) and (226.07,119.89) .. (227.57,119.89) .. controls (229.06,119.89) and (230.27,121.1) .. (230.27,122.6) .. controls (230.27,124.1) and (229.06,125.31) .. (227.57,125.31) .. controls (226.07,125.31) and (224.86,124.1) .. (224.86,122.6) -- cycle ;
%Curve Lines [id:da9651804409534774] 
\draw [color={rgb, 255:red, 74; green, 144; blue, 226 }  ,draw opacity=1 ]   (249.71,110.33) .. controls (248.71,100.33) and (256.12,103.15) .. (261.29,104.05) .. controls (266.45,104.94) and (271.4,103.4) .. (278.43,108.31) ;
%Shape: Polygon Curved [id:ds13670878570302025] 
\draw  [color={rgb, 255:red, 155; green, 155; blue, 155 }  ,draw opacity=1 ][dash pattern={on 4.5pt off 4.5pt}] (271.4,121.73) .. controls (284.4,103.93) and (300.57,105.76) .. (304.57,116.76) .. controls (308.57,127.76) and (306.57,127.96) .. (300.57,138.76) .. controls (294.57,149.56) and (288.57,141.56) .. (279.57,141.76) .. controls (270.57,141.96) and (258.4,139.53) .. (271.4,121.73) -- cycle ;
%Curve Lines [id:da040459252093603126] 
\draw    (275.57,117.76) .. controls (284.57,120.76) and (288.57,130.76) .. (301.57,135.76) ;
%Shape: Circle [id:dp9060386468955414] 
\draw  [color={rgb, 255:red, 0; green, 0; blue, 0 }  ,draw opacity=1 ][fill={rgb, 255:red, 0; green, 0; blue, 0 }  ,fill opacity=1 ][line width=0.75]  (285.86,127.6) .. controls (285.86,126.1) and (287.07,124.89) .. (288.57,124.89) .. controls (290.06,124.89) and (291.27,126.1) .. (291.27,127.6) .. controls (291.27,129.1) and (290.06,130.31) .. (288.57,130.31) .. controls (287.07,130.31) and (285.86,129.1) .. (285.86,127.6) -- cycle ;
%Curve Lines [id:da017394383025079674] 
\draw    (245.4,200.33) .. controls (254.71,192.62) and (240.14,193.9) .. (251.14,187.9) ;
%Shape: Circle [id:dp7184575924420116] 
\draw  [color={rgb, 255:red, 0; green, 0; blue, 0 }  ,draw opacity=1 ][fill={rgb, 255:red, 0; green, 0; blue, 0 }  ,fill opacity=1 ][line width=0.75]  (242.86,200.6) .. controls (242.86,199.1) and (244.07,197.89) .. (245.57,197.89) .. controls (247.06,197.89) and (248.27,199.1) .. (248.27,200.6) .. controls (248.27,202.1) and (247.06,203.31) .. (245.57,203.31) .. controls (244.07,203.31) and (242.86,202.1) .. (242.86,200.6) -- cycle ;
%Curve Lines [id:da800959694500028] 
\draw [color={rgb, 255:red, 74; green, 144; blue, 226 }  ,draw opacity=1 ]   (239.71,111.33) .. controls (239.71,101.76) and (249.55,96.87) .. (254.71,97.76) .. controls (259.88,98.65) and (271.71,89.76) .. (279.71,95.76) ;
%Curve Lines [id:da8575905686141982] 
\draw [color={rgb, 255:red, 74; green, 144; blue, 226 }  ,draw opacity=1 ]   (279.71,95.76) .. controls (285.4,100.4) and (294.71,100.76) .. (294.71,108.76) .. controls (294.71,116.76) and (288.57,117.5) .. (284.71,123.76) ;
%Curve Lines [id:da5031823879220163] 
\draw [color={rgb, 255:red, 74; green, 144; blue, 226 }  ,draw opacity=1 ]   (292.71,130.62) .. controls (319.71,107.9) and (323.03,113.4) .. (322.71,139.9) .. controls (322.4,166.4) and (316.09,168.9) .. (308.4,182.4) ;
%Curve Lines [id:da9583590849311909] 
\draw [color={rgb, 255:red, 74; green, 144; blue, 226 }  ,draw opacity=1 ]   (249.14,195.9) .. controls (264.14,197.19) and (265.71,213.9) .. (278.71,204.9) .. controls (291.71,195.9) and (294.4,200.4) .. (308.4,182.4) ;
%Curve Lines [id:da14455647634072488] 
\draw    (229.71,109.33) .. controls (242.14,103.9) and (230.14,111.9) .. (241.14,110.9) ;
%Shape: Circle [id:dp029692140830384073] 
\draw  [color={rgb, 255:red, 0; green, 0; blue, 0 }  ,draw opacity=1 ][fill={rgb, 255:red, 0; green, 0; blue, 0 }  ,fill opacity=1 ][line width=0.75]  (226.86,108.6) .. controls (226.86,107.1) and (228.07,105.89) .. (229.57,105.89) .. controls (231.06,105.89) and (232.27,107.1) .. (232.27,108.6) .. controls (232.27,110.1) and (231.06,111.31) .. (229.57,111.31) .. controls (228.07,111.31) and (226.86,110.1) .. (226.86,108.6) -- cycle ;
%Curve Lines [id:da17603251633378747] 
\draw    (240.14,110.9) .. controls (251.14,110.9) and (231.14,120.9) .. (223.14,114.9) .. controls (215.14,108.9) and (214.14,116.9) .. (227.14,121.9) ;
%Curve Lines [id:da37352408472556997] 
\draw [color={rgb, 255:red, 208; green, 2; blue, 27 }  ,draw opacity=1 ]   (302,80) .. controls (343.4,74) and (388.4,117.2) .. (418.4,140.2) .. controls (448.4,163.2) and (394.4,241.2) .. (405.4,267.2) .. controls (416.4,293.2) and (301.4,338.2) .. (277.4,312.2) .. controls (253.4,286.2) and (210.4,280.4) .. (235.4,255.4) ;
%Curve Lines [id:da5276437058784719] 
\draw [color={rgb, 255:red, 74; green, 144; blue, 226 }  ,draw opacity=1 ]   (237,254) .. controls (246.4,248) and (262.4,250) .. (266.4,243) .. controls (270.4,236) and (260.4,236) .. (267.4,222) ;
%Shape: Circle [id:dp6183164069342888] 
\draw  [color={rgb, 255:red, 74; green, 144; blue, 226 }  ,draw opacity=1 ][fill={rgb, 255:red, 208; green, 2; blue, 27 }  ,fill opacity=1 ][line width=0.75]  (232.86,254.6) .. controls (232.86,253.1) and (234.07,251.89) .. (235.57,251.89) .. controls (237.06,251.89) and (238.27,253.1) .. (238.27,254.6) .. controls (238.27,256.1) and (237.06,257.31) .. (235.57,257.31) .. controls (234.07,257.31) and (232.86,256.1) .. (232.86,254.6) -- cycle ;
%Shape: Circle [id:dp6112773096976698] 
\draw  [color={rgb, 255:red, 74; green, 144; blue, 226 }  ,draw opacity=1 ][fill={rgb, 255:red, 74; green, 144; blue, 226 }  ,fill opacity=1 ][line width=0.75]  (232.86,192.6) .. controls (232.86,191.1) and (234.07,189.89) .. (235.57,189.89) .. controls (237.06,189.89) and (238.27,191.1) .. (238.27,192.6) .. controls (238.27,194.1) and (237.06,195.31) .. (235.57,195.31) .. controls (234.07,195.31) and (232.86,194.1) .. (232.86,192.6) -- cycle ;
%Shape: Circle [id:dp4493817229849214] 
\draw  [color={rgb, 255:red, 74; green, 144; blue, 226 }  ,draw opacity=1 ][fill={rgb, 255:red, 74; green, 144; blue, 226 }  ,fill opacity=1 ][line width=0.75]  (264.86,221.6) .. controls (264.86,220.1) and (266.07,218.89) .. (267.57,218.89) .. controls (269.06,218.89) and (270.27,220.1) .. (270.27,221.6) .. controls (270.27,223.1) and (269.06,224.31) .. (267.57,224.31) .. controls (266.07,224.31) and (264.86,223.1) .. (264.86,221.6) -- cycle ;
%Curve Lines [id:da9357977148403693] 
\draw [color={rgb, 255:red, 208; green, 2; blue, 27 }  ,draw opacity=1 ]   (302,81) .. controls (343.4,75) and (341.4,119.2) .. (371.4,142.2) .. controls (401.4,165.2) and (354.4,207.2) .. (365.4,233.2) .. controls (376.4,259.2) and (332.4,310.6) .. (319.4,289.6) .. controls (306.4,268.6) and (279.4,287.6) .. (288.4,264.6) ;
%Curve Lines [id:da5998391360305778] 
\draw [color={rgb, 255:red, 74; green, 144; blue, 226 }  ,draw opacity=1 ]   (292.4,262.6) .. controls (301.8,256.6) and (293.4,258.2) .. (297.4,251.2) ;
%Shape: Circle [id:dp38756147448517453] 
\draw  [color={rgb, 255:red, 74; green, 144; blue, 226 }  ,draw opacity=1 ][fill={rgb, 255:red, 74; green, 144; blue, 226 }  ,fill opacity=1 ][line width=0.75]  (296.86,250.6) .. controls (296.86,249.1) and (298.07,247.89) .. (299.57,247.89) .. controls (301.06,247.89) and (302.27,249.1) .. (302.27,250.6) .. controls (302.27,252.1) and (301.06,253.31) .. (299.57,253.31) .. controls (298.07,253.31) and (296.86,252.1) .. (296.86,250.6) -- cycle ;
%Shape: Circle [id:dp7761250736253853] 
\draw  [color={rgb, 255:red, 74; green, 144; blue, 226 }  ,draw opacity=1 ][fill={rgb, 255:red, 208; green, 2; blue, 27 }  ,fill opacity=1 ][line width=0.75]  (286.86,263.6) .. controls (286.86,262.1) and (288.07,260.89) .. (289.57,260.89) .. controls (291.06,260.89) and (292.27,262.1) .. (292.27,263.6) .. controls (292.27,265.1) and (291.06,266.31) .. (289.57,266.31) .. controls (288.07,266.31) and (286.86,265.1) .. (286.86,263.6) -- cycle ;
%Shape: Circle [id:dp16215820043552287] 
\draw  [color={rgb, 255:red, 74; green, 144; blue, 226 }  ,draw opacity=1 ][fill={rgb, 255:red, 74; green, 144; blue, 226 }  ,fill opacity=1 ][line width=0.75]  (288.86,233.6) .. controls (288.86,232.1) and (290.07,230.89) .. (291.57,230.89) .. controls (293.06,230.89) and (294.27,232.1) .. (294.27,233.6) .. controls (294.27,235.1) and (293.06,236.31) .. (291.57,236.31) .. controls (290.07,236.31) and (288.86,235.1) .. (288.86,233.6) -- cycle ;
%Curve Lines [id:da6987535947683101] 
\draw [color={rgb, 255:red, 74; green, 144; blue, 226 }  ,draw opacity=1 ]   (265.4,258.6) .. controls (267.4,250.6) and (285.45,254.95) .. (282.92,246.27) .. controls (280.4,237.6) and (287.4,239.7) .. (289.4,236.2) ;
%Shape: Circle [id:dp17726686750387366] 
\draw  [color={rgb, 255:red, 74; green, 144; blue, 226 }  ,draw opacity=1 ][fill={rgb, 255:red, 208; green, 2; blue, 27 }  ,fill opacity=1 ][line width=0.75]  (261.86,260.6) .. controls (261.86,259.1) and (263.07,257.89) .. (264.57,257.89) .. controls (266.06,257.89) and (267.27,259.1) .. (267.27,260.6) .. controls (267.27,262.1) and (266.06,263.31) .. (264.57,263.31) .. controls (263.07,263.31) and (261.86,262.1) .. (261.86,260.6) -- cycle ;
%Curve Lines [id:da8308395356977258] 
\draw [color={rgb, 255:red, 208; green, 2; blue, 27 }  ,draw opacity=1 ]   (301.4,80.6) .. controls (342.8,74.6) and (363.4,115.6) .. (393.4,138.6) .. controls (423.4,161.6) and (395.4,246.8) .. (381.4,267.8) .. controls (367.4,288.8) and (318.4,317.8) .. (305.4,303.8) .. controls (292.4,289.8) and (255.4,285.8) .. (264.4,262.8) ;

% Text Node
\draw (294,63) node [anchor=north west][inner sep=0.75pt]  [font=\normalsize,color={rgb, 255:red, 208; green, 2; blue, 27 }  ,opacity=1 ]  {$p$};
% Text Node
\draw (325.86,249) node [anchor=north west][inner sep=0.75pt]  [font=\normalsize,color={rgb, 255:red, 0; green, 0; blue, 0 }  ,opacity=1 ]  {$q$};
% Text Node
\draw (258,69) node [anchor=north west][inner sep=0.75pt]  [font=\normalsize,color={rgb, 255:red, 74; green, 144; blue, 226 }  ,opacity=1 ]  {$\ell$};
% Text Node
\draw (211.86,187) node [anchor=north west][inner sep=0.75pt]  [font=\normalsize,color={rgb, 255:red, 74; green, 144; blue, 226 }  ,opacity=1 ]  {$w_{2}$};
% Text Node
\draw (222,87) node [anchor=north west][inner sep=0.75pt]  [font=\normalsize,color={rgb, 255:red, 74; green, 144; blue, 226 }  ,opacity=1 ]  {$w_{0}$};
% Text Node
\draw (215,61.93) node [anchor=north west][inner sep=0.75pt]  [color={rgb, 255:red, 155; green, 155; blue, 155 }  ,opacity=1 ]  {$U_{0}$};
% Text Node
\draw (195,167) node [anchor=north west][inner sep=0.75pt]  [color={rgb, 255:red, 155; green, 155; blue, 155 }  ,opacity=1 ]  {$U_{2}$};
% Text Node
\draw (211.86,122) node [anchor=north west][inner sep=0.75pt]  [font=\normalsize,color={rgb, 255:red, 0; green, 0; blue, 0 }  ,opacity=1 ]  {$\sfx_{j}$};
% Text Node
\draw (269.86,129) node [anchor=north west][inner sep=0.75pt]  [font=\normalsize,color={rgb, 255:red, 0; green, 0; blue, 0 }  ,opacity=1 ]  {$w_{1}$};
% Text Node
\draw (250,135.93) node [anchor=north west][inner sep=0.75pt]  [color={rgb, 255:red, 155; green, 155; blue, 155 }  ,opacity=1 ]  {$U_{1}$};
% Text Node
\draw (193,205) node [anchor=north west][inner sep=0.75pt]  [font=\normalsize,color={rgb, 255:red, 0; green, 0; blue, 0 }  ,opacity=1 ]  {$\sfy_{j} =\sfy_{j+1}$};
% Text Node
\draw (275,156) node [anchor=north west][inner sep=0.75pt]  [font=\normalsize,color={rgb, 255:red, 0; green, 0; blue, 0 }  ,opacity=1 ]  {$\sfb_{2,j}$};
% Text Node
\draw (289,87) node [anchor=north west][inner sep=0.75pt]  [font=\normalsize,color={rgb, 255:red, 0; green, 0; blue, 0 }  ,opacity=1 ]  {$\sfb_{1,j+1}$};
% Text Node
\draw (255.71,105) node [anchor=north west][inner sep=0.75pt]  [font=\normalsize,color={rgb, 255:red, 0; green, 0; blue, 0 }  ,opacity=1 ]  {$\sfb_{1,j}$};
% Text Node
\draw (316,170) node [anchor=north west][inner sep=0.75pt]  [font=\normalsize,color={rgb, 255:red, 0; green, 0; blue, 0 }  ,opacity=1 ]  {$\sfb_{2,j+1}$};
% Text Node
\draw (200,95) node [anchor=north west][inner sep=0.75pt]  [font=\normalsize,color={rgb, 255:red, 0; green, 0; blue, 0 }  ,opacity=1 ]  {$\sfx_{j+1}$};
% Text Node
\draw (270.86,208) node [anchor=north west][inner sep=0.75pt]  [font=\normalsize,color={rgb, 255:red, 74; green, 144; blue, 226 }  ,opacity=1 ]  {$s_{\sfn-1}$};
% Text Node
\draw (300,237) node [anchor=north west][inner sep=0.75pt]  [font=\normalsize,color={rgb, 255:red, 74; green, 144; blue, 226 }  ,opacity=1 ]  {$s_{\sfn+1}$};
% Text Node
\draw (294,220) node [anchor=north west][inner sep=0.75pt]  [font=\normalsize,color={rgb, 255:red, 74; green, 144; blue, 226 }  ,opacity=1 ]  {$s_{\sfn}$};

\end{tikzpicture}
\subcaption{A view into the eclipse}
\end{subfigure}
\caption{The case $(m,\sft)=(3,2)$ with $\ell$ right-inaccessible}
\label{Fig:UV}
\end{figure}

\begin{lem}[Dragging by inside sequences of a squashed fence]\label{Lem:insideExt2}
   Assume the conditions in \refconst{firstStepConst}.
Suppose that $\cL=\cL(\Gamma, \{z_i\}_{i=1}^m)$ is a circle lamination, constructed by \refconst{firstStepConst} or by \refconst{extConst}.
Assume that $\ell$ is right inaccessible,  $z_0<z_1<z_2<\cdots<z_m$ in $\Lambda(\Gamma)$ and $z_m$ is not fixed by a non-trivial element of $\Gamma$.
If  $\{w_i\}_{i=0}^t, t\ge 1$ is a finite sequence in $\Lambda(\Gamma)$, satisfying the following conditions:
\begin{itemize}
    \item $z_m=w_0<w_1<\cdots< w_t<z_0$ in $\cldi[\Lambda(\Gamma)]{z_m}{z_0}$;
    \item all $\opi[\Lambda(\Gamma)]{w_{i-1}}{w_{i}}$ are not isolated in $\cL$;
    \item there exists an inside sequence $\{\lambda_j\}_{j\in \NN}$, that is $\opi[\Lambda(\Gamma)]{w_0}{w_1}$-side in $\cL$.
\end{itemize}
Then, $w_i \in \bigcap_{n\in \ZZ} \Int \cE_n\setminus q=\cE_\infty\setminus (\ell\cup \Fix(g))$ for all $i \in \{1,2,\dots,t\}$.

Moreover, the same statement holds for the case where
$\ell$ is left inaccessible and
\[
z_m<z_{m-1}<\cdots<z_0 \text{ in }\Lambda(\Gamma),
\]
and $\{w_i\}_{i=0}^t$ is a decreasing finite sequence in $\lopi[\Lambda(\Gamma)]{z_0}{z_m}$ with $w_0=z_m$.
In particular, $w_i\in \cE_\infty\setminus (\ell\cup \Fix(g))$ for all $i \in \{1,2,\dots,t\}$.
\end{lem}
\begin{proof}
By symmetry, it is enough to consider the case where \(\ell\) is right inaccessible.
Since each interval \(\opi[\Lambda(\Gamma)]{w_{i-1}}{w_i}\) is not isolated in \(\cL\), Proposition~\ref{Prop:squashedFenceDragging} yields, for each \(i\), an \(\opi[\Lambda(\Gamma)]{w_{i-1}}{w_i}\)-side regular sequence \(\{\lambda_{i,j}\}_{j\in\NN}\) of leaves of \(\cL\) such that \(\lambda_{1,j}=\lambda_j\) and \(I_{i,j}\subsetneq I_{i,j+1}\subset \opi[\Lambda(\Gamma)]{w_{i-1}}{w_i}\), where \(I_{i,j}\) is a component of \(\Lambda(\Gamma)\setminus \lambda_{i,j}\).

We claim that, after passing to subsequences, one may arrange that \(\closure{I_{i,j}}\subset I_{i,j+1}\) for every \(i\), and that no \(w_i\) is fixed by a nontrivial element of \(\Gamma\).

Suppose that infinitely many \(\lambda_{1,j}=\lambda_j\) share an endpoint \(\varepsilon\).
Then \(\varepsilon\in\{w_0,w_1\}\).
By \refprop{oneEndFix}, \(\varepsilon\) is fixed by a nontrivial element \(h\in\Gamma\).
Since no nontrivial element fixes \(w_0=z_m\), it follows that \(\varepsilon=w_1\) and \(h(w_0)\neq w_0\).
Applying \refprop{oneEndFix} again, we obtain that \(\{w_0,w_1\}\) is isolated in \(\cL\), contradicting \(\lambda_j\to\{w_0,w_1\}\).

Thus, after passing to a subsequence, we may assume that \(\closure{I_{1,j}}\subset I_{1,j+1}\), and \refprop{oneEndFix} implies that \(w_1\) is not fixed by any nontrivial element of \(\Gamma\).
Repeating the argument inductively in \(i\) and passing to further subsequences, we obtain, for each \(i\), that \(\closure{I_{i,j}}\subset I_{i,j+1}\) and that \(w_i\) is not fixed by any nontrivial element of \(\Gamma\).
Write \(I_{i,j}=\opi[\Lambda(\Gamma)]{\alpha_{i,j}}{\beta_{i,j}}\).

Since \(\{\lambda_{1,j}\}_{j\in\NN}\) is inside, by \refprop{squashedFenceDragging}, \(\{w_0,w_1\}\subset \cE_\infty\).

\begin{claim}\label{Clm:transferInside}
For each \(i\) with \(1<i\le t\), if \(w_{i-1}\in \cE_\infty\setminus \closure{\ell}\), then \(w_i\in \cE_\infty\setminus \Fix(g)\) and \(\{\lambda_{i,j}\}_{j\in\NN}\) is inside.
\end{claim}

\begin{proof}
This follows from \refprop{squashedFenceDragging}.
\end{proof}

We now show that \(w_i\notin \ell\) for all \(i\in\{1,\dots,t\}\).
Suppose for contradiction that \(\sft\) is the smallest index such that \(w_\sft\in \ell\).
Applying \refclm{transferInside} successively, we obtain, for every \(i\in\{1,\dots,\sft\}\), that \(w_i\in \cE_\infty\) and \(\{\lambda_{i,j}\}_{j\in\NN}\) is inside.
In particular, \(w_\sft<w_0=z_m\) in \(\ell\).

Choose \(\sfn\in\NN\) such that \(\sfn>\cN\), \(s_{\sfn-1}<w_\sft<z_m\) in \(\ell\), and the end ray of \(\ell\) from \(s_{\sfn-1}\) to \(q\) is contained in \(\cD^R(\Lambda(\Gamma))\).
Set \(\gamma_1=\sfa_\sfn\) and \(\fD=\bigcap_{i=1}^m \cH^R(\gamma_i)\), and define \(\cC=\sigma_{\sfn-1}\cup [s_{\sfn-1},s_{\sfn+1}]\cup \sigma_{\sfn+1}\).
Equip \(\partial \fD\) with the circular order compatible with the linear order on \(\opi[\Lambda(\Gamma)]{z_m}{z_0}\).

Take flat neighborhoods \(U_0\) of \(w_0\) and \(U_\sft\) of \(w_\sft\) relative to \(\ell\), that satisfy the following:
\begin{itemize}
    \item \(\closure{U_i}\cap \cC =\varnothing\) and  \(\closure{U_i}\cap \ell= \closure{U_i}\cap \partial \cE_{\sfn+1}\) for all \(i\in \{0,\sft\}\),
    \item \(\closure{U_0}\cap \closure{U_\sft}=\varnothing\) and \(\closure{U_0}\cap \cldi[\partial \fD]{w_1}{z_{m-1}}=\varnothing\);
    \item \(\closure{U_\sft}\cap \cldi[\partial \fD]{z_0}{w_{\sft-1}}=\varnothing\).
\end{itemize}
See \reffig{UV} (left).

Since \(w_i\in \Int \cE_{\sfn+1}\) for all \(i\notin \{0,\sft\}\), we can take flat neighborhoods \(U_i\) of \(w_i\), \(i\notin \{0,\sft\}\), relative to \(\opi[\Lambda(\Gamma)]{w_0}{w_\sft}\) satisfying the following:
\begin{itemize}
    \item \(\closure{U_i}\subset \Int \cE_{\sfn+1}\) for all \(i\notin \{0,\sft\}\);
    \item \(\closure{U_i}\cap \partial \fD= \closure{U_i}\cap \opi[\Lambda(\Gamma)]{w_0}{w_\sft}\) for all \(i\notin \{0,\sft\}\);
    \item \(\closure{U_i}\cap \closure{U_j}=\varnothing\) if \(i\neq j\), possibly \(i,j\in \{0,\sft\}\).
\end{itemize}
See \reffig{UV} (right).

For each \(i\in \{0,1,\cdots, \sft\}\), there is a unique arc \(\cldi[\Lambda(\Gamma)]{u_{2i}}{u_{2i+1}}\) that crosses \(U_i\) and contains \(w_i\).
Note that
\[
z_{m-1}<u_0<z_m=w_0<u_1<w_1<\cdots <w_{\sft-1}<u_{2\sft}<w_\sft<u_{2\sft+1}< x_0 \text{ in } \Lambda(\Gamma).
\]

After taking a subsequence of \(\{\lambda_{i,j}\}_{j\in\NN}\), we may assume that \(\cldi[\Lambda(\Gamma)]{u_{2i-1}}{u_{2i}}\subset I_{i,j}\) for all \(i \neq \sft\).
Note that \(\cldi[\partial \fD]{\beta_{i,j}}{\alpha_{i+1,j}} \subset \cldi[\Lambda(\Gamma)]{u_{2i}}{u_{2i+1}} \subset \closure{U_{i}}\subset \Int \cE_{\sfn+1} \setminus (\closure{U_0}\cup \closure{U_\sft})\) for all \(i \nin \{0,\sft\}\).
\begin{claim}\label{Clm:adaptedArcs}
For each \(i\in\{1,\dots,\sft\}\), there exists a sequence \(\{\sfb_{i,j}\}_{j\in\NN}\) of adapted arcs crossing \(\cH\) such that:
\begin{itemize}
    \item \(\sfb_{i,j}\) is mixed with the synapse in \(Z^+\) if \(m=1\), and pure in \(Z^-\) if \(m>1\);
    \item \(\sfb_{i,j}\cap \sfb_{i',j'}=\varnothing\) whenever \((i,j)\neq (i',j')\);
    \item for each \(i\), \[\cldi[\Lambda(\Gamma)]{u_{2i-1}}{u_{2i}}\subset K_{i,j}\subset \closure{K_{i,j}}\subset K_{i,j+1}\subset \opi[\Lambda(\Gamma)]{w_{i-1}}{w_i},\]
    where \(K_{i,j}\) is a component of \(\Lambda(\Gamma)\setminus \sfd\sfb_{i,j}\);
    \item for each \(i\), \(\bigcup_{j\in\NN} K_{i,j}=\opi[\Lambda(\Gamma)]{w_{i-1}}{w_i}\);
    \item \(\sfb_{i,j}\subset \closure{\cE_{\sfn(i,j)}\setminus \cE_{\sfn(i,j)+1}}\) for some \(\sfn(i,j)\in\ZZ\) with \(\sfn(i,j)>\sfn\);
    \item for each \(i\), the sequence \(\{\sfn(i,j)\}_{j\in\NN}\) is strictly increasing.
\end{itemize}
\end{claim}

\begin{proof}[Proof of the claim]
Fix \(i\in\{1,\dots,\sft\}\).
First consider the case \(m>1\).
Since \(\{\lambda_{i,j}\}_{j\in\NN}\) is regular and \(m>1\), we have \(\lambda_{i,j}=h_{i,j}(\ell_{\sfk(i,j)})\) for some \(h_{i,j}\in\Gamma\) and \(\sfk(i,j)\in\{2,\dots,m\}\).
By \refprop{squashedFenceSystem}, \(h_{i,j}(\gamma_{\sfk(i,j)})\subset h_{i,j}(\cE_\infty)\subset \closure{\cE_{\sfn(i,j)}\setminus \cE_{\sfn(i,j)+1}}\) for some \(\sfn(i,j)\in\ZZ\).
Note that \(\closure{\cE_n\setminus \cE_{n+1}}\cap \Int \cE_{n'}=\varnothing\) for all \(n'>n\).
Hence, since \(\{\lambda_{i,j}\}_{j\in\NN}\) is inside, we have \(\sfn(i,j)\to\infty\) as \(j\to\infty\).
After passing to a subsequence in \(j\), we may assume that \(\sfn(i,j)>\sfn\) for all \(j\) and that \(\{\sfn(i,j)\}_{j\in\NN}\) is strictly increasing.
Set \(\sfb_{i,j}=h_{i,j}(\gamma_{\sfk(i,j)})\).

To ensure that the arcs \(\sfb_{i,j}\) are pairwise disjoint, we claim that, for each \(j>1\), the arc \(\sfb_{i,j}\) cannot intersect both \(\sfb_{i,j+1}\) and \(\sfb_{i,j-1}\).
Indeed, otherwise \([\alpha_{i,j+1},\alpha_{i,j-1}]\) in \(Z^-\) would be linked with \(\sfb_{i,j}\), since \(\alpha_{i,j+1}\) and \(\alpha_{i,j-1}\) lie in distinct components of \(\Lambda(\Gamma)\setminus \lambda_{i,j}\).
This is impossible because \(\sfb_{i,j}\) is right inaccessible.
Therefore, after passing to a further subsequence in \(j\), we may also assume that \(\{\sfb_{i,j}\}_{j\in\NN}\) is pairwise disjoint.
By the choice of \(\{\lambda_{i,j}\}_{j\in\NN}\), the arcs \(\sfb_{i,j}\) and the integers \(\sfn(i,j)\) satisfy the required properties.

It remains to consider the case \(m=1\).
Since \(\{\lambda_{i,j}\}_{j\in\NN}\) is regular, exactly one of the following holds:
\begin{enumerate}
    \item every \(\lambda_{i,j}\) admits an outside approximation;
    \item every \(\lambda_{i,j}\) admits an inside approximation.
\end{enumerate}
Hence, for each \((i,j)\), we choose an outside approximation sequence in case \((1)\), and an inside approximation sequence in case \((2)\), denoted by \(\{h_{i,j,k}\}_{k\in\NN}\), to \(\lambda_{i,j}\).
By \reflem{canonicalApprox}, \(h_{i,j,k}(\cE_{m(i,j,k)})\subset \closure{\cE_{n(i,j,k)}\setminus \cE_{n(i,j,k)+1}}\)
for some \(n(i,j,k), m(i,j,k)\in\ZZ\).

On the other hand, since \(\closure{I_{i,j}}\subset I_{i,j+1}\), we may choose sequences \(\{U_{i,j}\}_{j\in\NN}\) and \(\{V_{i,j}\}_{j\in\NN}\) of pairwise disjoint open intervals such that \(\closure{U_{i,1}}\cap \closure{V_{i,1}}=\varnothing\), \(\alpha_{i,j}\in U_{i,j}\), and \(\beta_{i,j}\in V_{i,j}\) for all \(j\).

Since \(\{h_{i,j,k}\}_{k\in\NN}\) is a conical limit sequence for \(z_1\) with respect to \(\lambda_{i,j}\), and \(\cldi[\Lambda(\Gamma)]{z_0}{x_\cN}\) is a compact subset of \(\Lambda(\Gamma)\setminus\{z_1\}\) by \refcor{limitAuxiliaryLeaf}, for each \((i,j)\) there exists \(\sfk(i,j)>\cN\) such that, for all \(k\ge \sfk(i,j)\), one of the following holds:
\begin{itemize}
    \item \(h_{i,j,k}(z_1)\in U_{i,j}\) and \(h_{i,j,k}(\cldi[\Lambda(\Gamma)]{z_0}{x_\cN})\subset V_{i,j}\);
    \item \(h_{i,j,k}(z_1)\in V_{i,j}\) and \(h_{i,j,k}(\cldi[\Lambda(\Gamma)]{z_0}{x_\cN})\subset U_{i,j}\).
\end{itemize}

In case \((1)\), for each \((i,j)\), by \reflem{canonicalApprox}, there exists \(\fN(i,j)\in \ZZ\) such that \(|\fN(i,j)-n(i,j,k)|\le 1\) for all sufficiently large \(k\).
Moreover, since \(\{\lambda_{i,j}\}_{j\in\NN}\) is inside, \(\fN(i,j)\to\infty\) as \(j\to\infty\).
Hence, for each \((i,j)\), we choose \(\kappa(i,j)\in \NN\) so that \(|\fN(i,j)-n(i,j,\kappa(i,j))|\le 1\) and \(\kappa(i,j)>\sfk(i,j)\).
Then we set \[\sfb_{i,j}=h_{i,j,\kappa(i,j)}(\sfa_{m(i,j,\kappa(i,j))}) \quad \text{and}\quad \sfn(i,j)=n(i,j,\kappa(i,j)).\]
Note that \(\sfb_{i,j}\subset \closure{\cE_{\sfn(i,j)}\setminus \cE_{\sfn(i,j)+1}}\) and \(\sfn(i,j)\to\infty\) as \(j\to\infty\).
Hence, after taking subsequences in \(j\), we can assume that \(\{\sfn(i,j)\}_{j\in \NN}\) is strictly increasing with \(\sfn(i,j)>\sfn\).

In case \((2)\), for each \((i,j)\), \(n(i,j,k)\to\infty\) as \(k\to\infty\).
Hence, for each \(i\), we can choose a subsequence \(\{(a(i,j),b(i,j))\}_{j\in\NN}\) of \(\NN\times \NN\) such that \(b(i,j)>\sfk(i,j)\) for all \(j\), and \(\{n(i,a(i,j),b(i,j))\}_{j\in\NN}\) is strictly increasing with \(n(i,a(i,j),b(i,j))>\sfn\).
Hence, in this case, we let \[\sfb_{i,j}=h_{i,a(i,j),b(i,j)}(\sfa_{m(i,a(i,j),b(i,j))})\quad \text{and}\quad \sfn(i,j)=n(i,a(i,j),b(i,j)).\]
Again we have that \(\sfb_{i,j}\subset \closure{\cE_{\sfn(i,j)}\setminus \cE_{\sfn(i,j)+1}}\).

To ensure that \(\sfb_{i,j}\) are pairwise disjoint, as above, we claim that \(\sfb_{i,j}\) cannot meet both \(\sfb_{i,j-1}\) and \(\sfb_{i,j+1}\) simultaneously.
Indeed, otherwise, by \reflem{auxArcUnlinked}, \(\sfb_{i,j}\) meets \(\sfb_{i,j-1}\) and \(\sfb_{i,j+1}\) along their \((-)\)-segments.
By the choice of \(U_{i,j}\) and \(V_{i,j}\), \(\sfd\sfb_{i,j-1}\) and \(\sfd\sfb_{i,j+1}\) lie in distinct components of \(\Lambda(\Gamma)\setminus \sfd\sfb_{i,j}\). Since the synapse of \(\sfb_{i,j}\) lies in \(Z^+\), the arc joining the endpoint of \(\sfb_{i,j-1}\) in its \((-)\)-segment to the endpoint of \(\sfb_{i,j+1}\) in its \((-)\)-segment is an arc in \(Z^-\) linked with \(\sfb_{i,j}^-\). This contradicts the right inaccessibility of \(\sfb_{i,j}^-\).
Therefore, after passing to a subsequence in \(j\), we may assume that \(\{\sfb_{i,j}\}_{j\in\NN}\) is pairwise disjoint and satisfies the third item.
\end{proof}

Let \(\{\sfb_{i,j}\}_{i,j}\) and \(\{\sfn(i,j)\}_{i,j}\) be given by \refclm{adaptedArcs}.
Since $\sfb_{i,j}$ are pairwise disjoint and $\closure{K_{i,j}}\subset K_{i,j+1}\subset \opi[\Lambda(\Gamma)]{w_{i-1}}{w_{i}}$, we can see that $p\nin \sfb_{i,j}$.
Since \(\sfb_{i,j}\subset \closure{\cE_{\sfn(i,j)}\setminus \cE_{\sfn(i,j)+1}}\) with \(\sfn(i,j)>\sfn\) and \(\closure{\cE_{\sfn(i,j)}\setminus \cE_{\sfn(i,j)+1}}\setminus\{p\}\subset \Int\cE_\sfn\), we have \[\sfd\sfb_{i,j}\subset \sfb_{i,j}\subset \Int\cE_\sfn.\]

Write \(K_{i,j}=\opi[\Lambda(\Gamma)]{\mu_{i,j}}{\nu_{i,j}}\), and let \(U_0^R\) and \(U_\sft^R\) be the right open neighborhoods of \(w_0\) and \(w_\sft\) associated with \(U_0\) and \(U_\sft\), respectively. Since \[\mu_{1,j}\in \opi[\Lambda(\Gamma)]{w_0}{u_1}\subset U_0, \quad \nu_{\sft,j}\in \opi[\Lambda(\Gamma)]{u_{2\sft}}{w_\sft}\subset U_\sft, \quad  U_i\cap \Int\cE_\sfn=U_i^R \text{ for } i\in\{0,\sft\},\] it follows that \(\mu_{1,j}\in U_0^R\) and \(\nu_{\sft,j}\in U_\sft^R\).
By assumption, \(w_i\in\ell\) and \(U_i^R\cap \ell=\varnothing\) for \(i\in\{0,\sft\}\), and hence, for each \(j\in\NN\) there exist points \(\sfx_j\in \ropi[\partial\fD]{w_0}{\mu_{1,j}}\) and \(\sfy_j\in \lopi[\partial\fD]{\nu_{\sft,j}}{w_\sft}\) such that \(\cldi[\partial\fD]{\sfx_j}{\mu_{1,j}}\cap \ell=\{\sfx_j\}\) and \(\cldi[\partial\fD]{\nu_{\sft,j}}{\sfy_j}\cap \ell=\{\sfy_j\}\). Thus \(\sfx_j\) is the first point of \(\ell\) encountered when moving from \(\mu_{1,j}\) toward \(w_0\) along \(\partial\fD\), and \(\sfy_j\) is the first point of \(\ell\) encountered when moving from \(\nu_{\sft,j}\) toward \(w_\sft\). By construction, \(\{\sfx_j\}_{j\in\NN}\) is decreasing and \(\{\sfy_j\}_{j\in\NN}\) is increasing in \(\cldi[\partial\fD]{w_0}{w_\sft}\), and \(s_\sfn<s_{\sfn-1}<\sfy_j<\sfx_j\le z_m\) on \(\ell\).
See \reffig{UV} (right).

For each \(j\in\NN\), define an arc \(\cS_j\) by
\[
\cS_j=
\begin{dcases}
\ \cldi[\partial\fD]{\sfx_j}{\mu_{1,j}}\cup \sfb_{1,j}\cup \cldi[\partial\fD]{\nu_{1,j}}{\sfy_j} & \text{if }\sft=1,\\
\ \cldi[\partial\fD]{\sfx_j}{\mu_{1,j}}
\cup \left(\bigcup_{i=1}^{\sft-1}\bigl(\sfb_{i,j}\cup \cldi[\partial\fD]{\nu_{i,j}}{\mu_{i+1,j}}\bigr)\right)
\cup \bigl(\sfb_{\sft,j}\cup \cldi[\partial\fD]{\nu_{\sft,j}}{\sfy_j}\bigr) & \text{if }\sft>1.
\end{dcases}
\]
Then \(\cS_j\) is alternatively adapted to \(Z^\pm\) with respect to \(\{\mu_{i,j},\nu_{i,j}:i=1,\dots,\sft\}\), and \(\cS_j\subset \closure{\cH}\).

\begin{claim}\label{Clm:altCrossing}
    $\cS_j$ crosses $\Int   \cE_\sfn$.
    In particular, $\Int \cS_j\cap \closure{\ell}=\varnothing$ and $\sfd \cS_j=\{\sfx_j,\sfy_j\}\subset (s_{\sfn-1}, z_m] \subset  \ell$.
\end{claim}
\begin{proof}
First, note that by the choice of $\sfx_j,\sfy_j$, we have that  $\sfd \cS_j=\{\sfx_j,\sfy_j\}\subset (s_{\sfn-1}, z_m] \subset  \ell$.
By the choice of \(U_i\), we have
\[
\Lambda(\Gamma)\cap \Int \cS_j\subset \left( U_0^R\cup \bigcup_{i=1}^{\sft-1}U_i\cup U_\sft^R \right)\subset \Int\cE_{\sfn+1},
\]
where the middle union is understood to be empty when \(\sft=1\).
On the other hand, \(\cS_j\setminus \Lambda(\Gamma)\subset \Int\cE_\sfn\). Therefore \(\Int \cS_j\subset \Int \cE_\sfn\) and so $\cS_j$ crosses $\Int \cE_\sfn$.

We then show that $\Int \cS_j\cap \closure{\ell}=\varnothing$.
Since $\sfd \ell=\{p,q\}$ is disjoint from $\sfb_{i,j}$ and $\Lambda(\Gamma)$, it suffices to prove $\Int \cS_j\cap \ell=\varnothing$.
Since \(\Int \cS_j\subset \Int \cE_\sfn\),  $\Int \cS_j$ could meet $\ell$ only along the end ray of $\ell$ from $s_\sfn$ to $q$.
By the choice of $\sfn$, the end ray of $\ell$ from $s_{\sfn}$ to $q$ lies in $\cH^R(\Lambda(\Gamma))$, whereas $\cS_j\subset \closure{\cH}$.
Thus $\Int \cS_j\cap \ell=\varnothing$, hence $\Int \cS_j\cap \closure{\ell}=\varnothing$.
\end{proof}

After \refclm{altCrossing},   $\hat\cS_j=[\sfx_j,\sfy_j]\cup \cS_j$ is a well-defined Jordan curve.
Note that $\hat \cS_j\cap \closure{\ell}=[\sfx_j,\sfy_j]$.
We orient $\hat\cS_j$ by the circular order compatible with the linear order on $\ell$, and orient $\cS_j$ by the induced linear order.

\begin{claim}\label{Clm:unlinking2}
For any $\sfj,\sfj'$, the pairs $\{\sfx_\sfj,\sfy_\sfj\}$ and $\{\sfx_{\sfj'},\sfy_{\sfj'}\}$ are unlinked in $[s_{\sfn-1},z_m]$; equivalently, either $[\sfx_\sfj,\sfy_\sfj]\subset [\sfx_{\sfj'},\sfy_{\sfj'}]$ or $[\sfx_{\sfj'},\sfy_{\sfj'}]\subset [\sfx_\sfj,\sfy_\sfj]$. In particular, either $\cD^R(\hat\cS_\sfj)\subset \cD^R(\hat\cS_{\sfj'})$ or $\cD^R(\hat\cS_{\sfj'})\subset \cD^R(\hat\cS_\sfj)$.
\end{claim}

\begin{proof}[Proof of the claim]
It is enough to prove that either $\cD^R(\hat\cS_\sfj)\subset \cD^R(\hat\cS_{\sfj'})$ or $\cD^R(\hat\cS_{\sfj'})\subset \cD^R(\hat\cS_\sfj)$. Suppose $\sfj'<\sfj$ and there exist points $\sfp,\sfq\in \Int\cS_{\sfj'}$ with $\sfp\in \cD^L(\hat\cS_\sfj)$ and $\sfq\in \cD^R(\hat\cS_\sfj)$. Since $[s_\sfn,z_m]\cap \Int \cS_{\sfj'}=\varnothing$ by \refclm{altCrossing}, there is a Jordan subarc $\sfc\subset \Int\cS_{\sfj'}$ with $\sfd\sfc=\{\sfp,\sfr\}$ for some $\sfr\in \Int\cS_\sfj\cap \Int\cS_{\sfj'}$ and $\Int\sfc\subset \cD^L(\hat\cS_\sfj)$. For any left neighborhood $\fU$ of $\sfr$ relative to $\cS_\sfj$, there exists a subarc $\sfc'$ of $\sfc$ such that $\sfr\in \sfc'\subset \fU$. Since $\sfj'<\sfj$ and $\cS_j\cap \cS_{j'}\subset \cS_j\cap \Lambda(\Gamma)$, there is no such a subarc in $\Int\cS_{\sfj'}$. Hence $\Int\cS_{\sfj'}$ is contained in the closure of a Jordan domain bounded by $\hat\cS_\sfj$, which proves the claim. \qedhere
\end{proof}

\begin{claim}\label{Clm:ascending2}
We have $\cD^R(\hat\cS_j)\subset \cD^R(\hat\cS_{j+1})$ for all $j$. In particular, $[\sfx_j,\sfy_j]\subset [\sfx_{j+1},\sfy_{j+1}]$ for all $j$.
\end{claim}

\begin{proof}[Proof of the claim]
By \refclm{unlinking2}, either $\cD^R(\hat\cS_j)\subset \cD^R(\hat\cS_{j+1})$ or $\cD^R(\hat\cS_{j+1})\subset \cD^R(\hat\cS_j)$.

Assume first that $\sft>1$. Any right neighborhood of $\nu_{1,j+1}$ relative to $\cS_{j+1}$ contains a subarc $\cldi[\partial\fD]{\nu}{\nu_{1,j+1}}$ of $\cS_j$ for some $\nu\in \opi[\partial\fD]{\nu_{1,j}}{\nu_{1,j+1}}$, while $\cldi[\partial\fD]{\nu_{1,j}}{\nu_{1,j+1}}\cap \cS_{j+1}=\{\nu_{1,j+1}\}$. Hence $\cD^R(\hat\cS_j)\subset \cD^R(\hat\cS_{j+1})$.

Now assume that $\sft=1$. Since $U_0$ is flat relative to $\ell$, we have $\partial U_0\cap \partial\cE_\sfn=\{e_0,d_0\}$ with $e_0>d_0$ on $\ell$. Recall that $\cldi[\Lambda(\Gamma)]{u_0}{u_1}$ crosses $U_0$, that $[e_0,w_0)$ lands at $w_0$ on the left side of $\cldi[\Lambda(\Gamma)]{u_0}{u_1}$, and that $u_0<w_0<\mu_{1,j+1}<\mu_{1,j}<u_1$ and $u_2<\nu_{1,j}<\nu_{1,j+1}<w_1$ in $\Lambda(\Gamma)$. Let $\fC$ be the Jordan curve given by the union of $[e_0,w_0)$, $\cldi[\Lambda(\Gamma)]{w_0}{u_1}$, and $\cldi[\partial U_0]{u_1}{e_0}$, equipped with the circular order compatible with the linear order on $\cldi[\Lambda(\Gamma)]{w_0}{u_1}$.

Since $\partial U_0$ separates $\mu_{1,k}$ and $\nu_{1,k}$ for $k\in\{j,j+1\}$, there exist points $p_j\in \cldi[\partial U_0]{u_1}{e_0}\cap \Int\sfb_{1,j}$ and $p_{j+1}\in \cldi[\partial U_0]{u_1}{e_0}\cap \Int\sfb_{1,j+1}$ such that, for each $k\in\{j,j+1\}$, the subarc $\sfc_k$ of $\sfb_{1,k}$ joining $p_k$ to $\mu_{1,k}$ crosses $\fC$ through $\cD^L(\fC)$. Since $\sfc_j$ and $\sfc_{j+1}$ are disjoint, the pairs $\sfd\sfc_j$ and $\sfd\sfc_{j+1}$ are disjoint and unlinked in $\fC$. It follows from $w_0<\mu_{1,j+1}<\mu_{1,j}<u_1$ that $u_1<p_j<p_{j+1}<e_0$ on $\partial U_0$. Since $p_j,p_{j+1}\in \Int\cE_\sfn$, we also have $d_0<p_j<p_{j+1}<e_0$ in $\cldi[\partial U_0]{d_0}{e_0}$. Therefore $\cD^R(\hat\cS_j)\subset \cD^R(\hat\cS_{j+1})$. \qedhere
\end{proof}

We now complete the proof.
By \refclm{altCrossing}, we can take  an increasing sequence $\{\sfm(j)\}_{j\in\NN}$ such that $\max_{1\le i\le \sft}\sfn(i,j)+2<\sfm(j)$ and $\sigma_{\sfm(j)}^-\cap \cS_j=\varnothing$.
Then, since $p\nin \sfb_{i,j}$ as observed above, $\cE_{\sfm(j)}\cap \sfb_{i,j}=\varnothing$ and \[s_\sfn\ge s_{\sfn(i,j)-1}>s_{\sfn(i,j)+2}>s_{\sfm(j)} \text{ in } \ell\] for all $i\in\{1,\dots,\sft\}$.

Let $\fK_j$ be the Jordan curve given by the union of $\sigma_{\sfm(j)}$, $[s_{\sfm(j)},s_\sfn]$, and $\sigma_\sfn$, equipped with the circular order satisfying $p<s_{\sfm(j)}<s_\sfn$.
Since $\cE_{\sfm(j)}\cap \sfb_{i,j}=\varnothing$ and $\sfb_{i,j}\subset \Int \cE_\sfn$, we can see that $\sfb_{i,j}\subset \cD^L(\fK_j)$ and so $\sfd \sfb_{i,j}=\{\mu_{i,j},\nu_{i,j}\}\subset \cD^L(\fK_j)$.

On the other hand, $\closure{\cD^L(\fK_j)}$ does not meet the end ray $\sfr$ of $\ell$ starting at $s_{\sfn-1}$ and landing at $p$, while $\sfx_j\in \sfr$, so $\sfx_j\in \cD^R(\fK_j)$. Hence $\fK_j$ separates $\mu_{1,j}$ and $\sfx_j$. Since $\lopi[\Lambda(\Gamma)]{\sfx_j}{\mu_{1,j}}\subset U_0^R\subset \Int\cE_\sfn$ and $U_0^R\cap \ell=\varnothing$, the connector arc $\sigma_{\sfm(j)}$ meets $\lopi[\Lambda(\Gamma)]{\sfx_j}{\mu_{1,j}}$, hence also $\Int\cS_j$.

Fix $j\in\NN$. Since $\cS_j$ crosses $\Int\cE_\sfn$ and $\sigma_{\sfm(j)}^-\cap \cS_j=\varnothing$, we may choose $\sfe\in \sigma_{\sfm(j)}^+\cap \cS_j$ so that the subarc $[p,\sfe]$ of $\sigma_{\sfm(j)}^+$ crosses the left Jordan domain $(\Int\cE_\sfn)^L(\cS_j)$ of $\cS_j$. Since $\cE_{\sfm(j)}\cap \sfb_{i,j}=\varnothing$ for all $i\in\{1,\dots,\sft\}$, we have $\sfe\in \cS_j\cap \Lambda(\Gamma)$. As $\sfe\in \sigma_{\sfm(j)}^+\subset Z^+$, \refprop{eclipseLProperty} implies that $\sfe\notin \cE_\infty$ and $\sfe\neq w_i$ for all $i\in\{0,1,\dots,\sft\}$.

We first claim that $\sfe\in \opi[\partial\fD]{\sfx_j}{\mu_{1,j}}\cup \opi[\partial\fD]{\nu_{\sft,j}}{\sfy_j}$. Suppose not. Then $\sft>1$ and $\sfe\in \opi[\partial\fD]{\nu_{i,j}}{w_i}\cup \opi[\partial\fD]{w_i}{\mu_{i+1,j}}$ for some $i\in\{1,\dots,\sft-1\}$. By the choice of $\sfe$, we have $[p,\sfe]\cap \cldi[\partial\fD]{\nu_{i',j}}{\mu_{i'+1,j}}=\varnothing$ for all $i'\neq i$, and $[p,\sfe)$ lands at $\sfe$ on the left side of $\cS_j$. Since, for each $i\in\{1,\dots,\sft\}$, the sequences $\{\mu_{i,j}\}_j$ and $\{\nu_{i,j}\}_j$ are strictly monotone in $\opi[\partial\fD]{w_{i-1}}{w_i}$, \refclm{ascending2} allows us to choose $j'$ large enough that $\sfe\in \cD^R(\hat\cS_{j'})$, $\sfb_{i,j'}\subset \cE_{\sfm(j)+3}$ for all $i\in\{1,\dots,\sft\}$, and $\lopi[\partial\fD]{\sfx_{j'}}{\mu_{1,j'}}\cup \ropi[\partial\fD]{\nu_{\sft,j'}}{\sfy_{j'}}\subset \Int\cE_{\sfm(j)+3}$. Since $\hat\cS_{j'}$ separates $\sfe$ from $p$, the arc $[p,\sfe]$ meets $\cS_{j'}$. On the other hand, because $\sfe\in \cD^R(\hat\cS_{j'})$, we have $[p,\sfe]\cap \cldi[\partial\fD]{\nu_{i,j'}}{\mu_{i+1,j'}}=\varnothing$ for all $i\in\{1,\dots,\sft-1\}$, and \refprop{eclipseLProperty} shows that $[p,\sfe]$ cannot meet $\cS_{j'}\cap \Lambda(\Gamma)$. Hence $[p,\sfe]$ meets $\sfb_{i',j'}$ for some $i'$, contradicting $\sfb_{i',j'}\subset \cE_{\sfm(j)+3}\setminus\{p\}$ and $[p,\sfe]\cap \cE_{\sfm(j)+3}=\{p\}$. This proves the claim.

Assume first that $\sfe\in \opi[\partial\fD]{\sfx_j}{\mu_{1,j}}$. Let $\fO$ be the Jordan curve $\cldi[\partial\cE_\sfn]{p}{\sfx_j}\cup \opi[\partial\fD]{\sfx_j}{\sfe}\cup [\sfe,p]$, equipped with the circular order compatible with the linear order on $\lopi[\partial\cE_\sfn]{p}{\sfx_j}$ induced from $\ell$. By \refclm{ascending2}, we may choose $j'$ so large that $\cldi[\partial\fD]{\sfx_{j'}}{\mu_{1,j'}}\subset \closure{\cD^R(\fO)}$, $\sfb_{i,j'}\subset \cE_{\sfm(j)+3}$ for all $i\in\{1,\dots,\sft\}$, and $\lopi[\partial\fO]{\sfx_{j'}}{\mu_{1,j'}}\cup \ropi[\partial\fD]{\nu_{\sft,j'}}{\sfy_{j'}}\subset \Int\cE_{\sfm(j)+3}$.

If $\sfx_j=\sfx_{j'}$, then $\{\sfe,p\}$ and $\sfd\sfb_{1,j'}$ are linked in $\partial(\Int\cE_\sfn)^L(\cS_j)$, so $\sfb_{1,j'}$ meets $[p,\sfe]$. If $\sfx_j\neq \sfx_{j'}$, then $\{\sfe,p\}$ and $\{\sfx_{j'},\nu_{1,j'}\}$ are linked in $\partial(\Int\cE_\sfn)^L(\cS_j)$, and the inclusion $\cldi[\partial\fD]{\sfx_{j'}}{\mu_{1,j'}}\subset \closure{\cD^R(\fO)}$ again implies that $\sfb_{1,j'}$ meets $[p,\sfe]$. In either case this contradicts $\sfb_{1,j'}\subset \cE_{\sfm(j)+3}\setminus\{p\}$ and $\cE_{\sfm(j)+3}\cap [p,\sfe]=\{p\}$.

Finally, suppose $\sfe\in \opi[\partial\fD]{\nu_{\sft,j}}{\sfy_j}$. Let $\fP$ be the Jordan curve $[p,\sfe]\cup \opi[\partial\fD]{\sfe}{\sfy_j}\cup \cldi[\partial\cE_\sfn]{\sfy_j}{p}$. Then $\sfb_{\sft,j'}$ meets $[p,\sfe]$ for all sufficiently large $j'$, contradicting \refprop{eclipseLProperty}. This completes the proof.\qedhere

\end{proof}

\subsection{No prong in $Z^+$}

\refthm{tameFix} follows from \refthm{twoInOne} and \refthm{oneInOne}.
We now complete the proof by proving \refthm{oneInOne}.
\begin{proof}[Proof of \refthm{oneInOne}]

By symmetry, it suffices to rule out the existence of an element $g \in G$ that fixes a unique point $p \in Z^+$ and acts freely on $Z^-$, assuming that $\ell$ is right-inaccessible. By \refthm{simplicialAxis}, $\ell$ is one-sided inaccessible. Applying \refconst{firstStepConst}, we obtain a squashed fence $\fF(g)=(\ell,\{\cE_n\}_{n\in \ZZ})$ for $g$ and a circle lamination $\cL(\Gamma,z_1)$ in the limit circle $\Lambda(\Gamma)$ of a quasi-Fuchsian closed surface subgroup $\Gamma \le G$. We henceforth use the notation and assumptions from \refconst{firstStepConst}.

We claim that there exists a point $z_2 \in \ell$ such that $z_0<z_1<z_2$ in $\Lambda(\Gamma)$ and $\{z_1,z_2\}$ satisfies the bouncing condition in \refconst{extConst}.

Assume otherwise. If $z_1$ is fixed by some non-trivial element of $\Gamma$, then \reflem{oneStepAtFix2} yields a contradiction. Hence $z_1$ is not fixed by any nontrivial element of $\Gamma$. By \reflem{closureOfAuxiliaryLeaf}, the lamination $\cG(\Gamma,z_1)$ is perfect, i.e.\ $\cG(\Gamma,z_1)=\cG(\Gamma,z_1)'$. Moreover, there exists a unique principal region $U$ of $\cG(\Gamma,z_1)$ such that, for every $n>\cN$, the half-ray of $\sfg_n$ toward $z_1$ exits through the same end $\{e_i\}_{i\in\NN}$ of $U$ and spirals toward the same component of $\cG(\Gamma,z_1)$, where $\sfg_n$ is the bi-infinite geodesic in $\cH/\Gamma$ corresponding to $\sfg(\sfd\sfa_n)$.

Since $\{\sfd \sfa_n\}_{n>\cN}$ converges to $\{z_0,z_1\}$, lying in $\opi[\Lambda(\Gamma)]{z_0}{z_1}$, by \refprop{auxiliaryLeaf}, $U$ can not be an ideal polygon.
Moreover, by \refprop{gapShape}, $U$ has a crown that contains  the end $\{e_i\}_{i\in\NN}$ of $U$.
Hence, there exists an element $h\in \Gamma$ and a bi-infinite sequence $\{w_i\}_{i\in \ZZ}$ that satisfies:
\begin{itemize}
    \item $r(h)<\cdots<w_{-i}<\cdots<w_0<\cdots<w_i<\cdots<a(h)$ in $\Lambda(\Gamma)$;
    \item $\opi[\Lambda(\Gamma)]{w_{i-1}}{w_{i}}$ are not isolated in $\cL_1$;
    \item there is a $k\in \NN$ such that $h(w_i)=w_{i+k}$ for all $i\in \ZZ$;
    \item  $z_0\in \ropi{a(h)}{r(h)}$ (since $\{\sfd \sfa_n\}_{n>\cN}$ converges to $\{z_0,z_1\}$ in $\cldi[\Lambda(\Gamma)]{z_0}{z_1}$).
\end{itemize}
By \refprop{squashedFenceDragging}, for each $i\in\NN$, there is a $\opi[\Lambda(\Gamma)]{w_{i-1}}{w_{i}}$-side regular sequence $\{\lambda_{i,j}\}_{j\in \NN}$.
If  $\{\lambda_{i,j}\}_{j\in \NN}$ is outside, then by \reflem{outsideExt2}, this is a contradiction by the assumption.
Otherwise, by \reflem{insideExt2}, $w_i\in \cE_\infty\setminus (\ell \cup \Fix(g))$ for all $i\in \NN$ and so $h(w_0)=w_k\in \cE_\infty\setminus (\ell \cup \Fix(g))$.
However, since $h(p)\neq p$,  by \refprop{squashedFenceSystem}, $h(w_k)\in \closure{\cE_m\setminus \cE_{m+1}}$ for some $m\in \ZZ$.
This is a contradiction since  $\cE_\infty  \cap \closure{\cE_m\setminus \cE_{m+1}} \subset \{p\}\cup \ell$ by \refprop{eclipseLProperty}.
Therefore, we are done.

Now, we assume that a finite sequence of points in $\ell$ ,  $\{z_i\}_{i=1}^m, m\ge2$, satisfies the bouncing condition in \refconst{extConst} with $z_0<z_1<z_2<\cdots<z_m$.
We then claim that there is a point $z_{m+1}$ in  $\opi[\Lambda(\Gamma)]{z_m}{z_0}$ such that $\{z_i\}_{i=1}^{m+1}$ satisfies the bouncing condition.
Assume not.
If $z_m$ is fixed by a non-trivial element in $\Gamma$, then by \reflem{oneStepAtFix2} and \refcor{limitAuxiliaryLeaf}, this is a contradiction.
Hence, $z_m$ is not fixed by a non-trivial element in $\Gamma$.

Recall that  when $m\ge 2$, \refprop{finiteGenLeaves} implies that $\cL(\Gamma,\{z_i\}_{i=1}^m)$ is generated by $\{\sfd \gamma_i\}_{i=2}^{m}$.
Namely, if $\sfg_i$ are the corresponding leaf of $\sfd \gamma_i$ of $\cG=\cG(\Gamma, \{z_i\}_{i=1}^m))$ in $\cH/\Gamma$.
Since $z_m$ is not fixed by a non-trivial element of $\Gamma$, the end of $\sfg_{m}$, that corresponds to $z_m$, spirals toward a component $\cG_0$ of $\cG''$.
Moreover, by \reflem{closureOfLeaf} and \reflem{closureOfManyLeaves}, there is a principal region $U$ of $\cG''$ where $\sfg_{m}$ is either contained in $U$ or a boundary leaf of $U$.

If $U$ is a finite sided ideal polygon, then the end of $\sfg_{m}$, that corresponds to $z_{m-1}$, spirals toward a $\cG_0$ as $\sfg_{m}$ is either a boundary leaf of $U$ or contained in $U$.
Then, it follows from \reflem{closureOfManyLeaves} that  $\sfg_{m-1}$ also is either a boundary leaf of $U$ or contained in $U$.
By repeating applying \reflem{closureOfManyLeaves}, we can see that  $\sfg_i$ also is either a boundary leaf of $U$ or contained in $U$ for all $k$ with $0<i<m$.
However, this is again a contradiction by \reflem{closureOfAuxiliaryLeaf}.
Hence, $U$ can not be a finite sided ideal polygon.

By \refprop{gapShape}, $U$ admits the core.
Then, there is a crown $\fC$ in $U$ such that $\fC$ contains the end of $\sfg_m$, associated with $z_m$, or $\sfg_m$ is a boundary leaf of $U$, that is in the metric completion of $\fC$.
Therefore, we can find an element $f\in \Gamma$ and a bi-infinite sequence $\{w_i\}_{i\in \ZZ}$ in $\Lambda(\Gamma)$, satisfying the following:
\begin{itemize}
    \item $r(f)<\cdots<w_{-i}<\cdots<w_0=z_m<\cdots<w_i<\cdots<a(f)$ in $\Lambda(\Gamma)$;
    \item $z_{m-1}\in \cldi[\Lambda(\Gamma)]{w_1}{w_{-1}}$;
    \item $\opi[\Lambda(\Gamma)]{w_{i-1}}{w_{i}}$ are not isolated in $\cL(\Gamma,\{z_i\}_{i=1}^m)$;
    \item there is a $\sfk\in \NN$ such that $f(w_i)=w_{i+\sfk}$ for all $i\in \ZZ$;
\end{itemize}
Note that $\sfg(a(f), r(f))/\langle f \rangle$ is the base of $\fC$.

Observe that $\opi[\Lambda(\Gamma)]{z_i}{z_{i-1}}$ is isolated in $\cL(\Gamma,\{z_i\}_{i=1}^m)$ for all $1<i\le m$. Suppose for contradiction that there exists an $\opi[\Lambda(\Gamma)]{z_i}{z_{i-1}}$-side sequence $\{\kappa_j\}_{j\in \NN}$ in $\cL(\Gamma,\{z_i\}_{i=1}^m)$. Since each $\kappa_j$ is unlinked with $\{z_{i-2},z_{i-1}\}$ (by \refprop{auxiliaryLeaf} if $i=2$), it follows that, for all sufficiently large $j$, the leaf $\kappa_j$ contains $z_{i-1}$. By \refprop{oneEndFix} and \refprop{fanning2}, this implies that some nontrivial element of $\Gamma$ fixes $z_{i-1}$ but not $z_i$. Consequently, \refprop{oneEndFix} implies that $\{z_i,z_{i-1}\}$ is isolated, a contradiction. Hence $\sfd\gamma_i$, for $1<i\le m$, are boundary leaves of $\cL(\Gamma,\{z_i\}_{i=1}^m)$.

In particular, $z_{m-1}\neq w_1$. Indeed, $\opi[\Lambda(\Gamma)]{z_m}{z_{m-1}}$ is isolated, whereas $\opi[\Lambda(\Gamma)]{w_0}{w_1}$ is not.

\begin{claim*}
  $z_0\nin \opi[\Lambda(\Gamma)]{w_0}{a(f)}$.
\end{claim*}
\begin{proof}[Proof of the claim]
It suffices to show that  $z_{i}\nin \opi[\Lambda(\Gamma)]{w_0}{a(f)}$ for all $i\ge1$.
Then, the desired statement follows from  \refprop{auxiliaryLeaf} since $\{\sfd \sfa_n\}_{n>\cN}$ converges to $\{z_{0},z_{1}\}$ in $\cldi[\Lambda(\Gamma)]{z_0}{z_1}$ and $z_0\in \opi[\Lambda(\Gamma)]{z_m}{z_1}$.

Assume that $z_\sfm=w_{j(0)}$ for some $\sfm$ with $0<\sfm<m$ and for some $0<j(0)$.
Since $\sfg(\sfd \gamma_{\sfm})$ and $\sfg(\sfd \gamma_{\sfm-1})$ are asymptotic at $z_{\sfm}$, it follows from \reflem{closureOfManyLeaves} and the circular ordering of $\{z_i\}_{i=1}^m$ that
$z_{\sfm-1}=w_{j(1)}$ for some $0<j(1)<j(0)$.
After repeating applying \reflem{closureOfManyLeaves}, we can find $j(s)$ such that $0<j(s+1)<j(s)$ for all $s\in \{0,1,\cdots, \sfm-1\}$ and $z_{\sfm-s}=w_{j(s)}$ for all $s\in \{1,\cdots, \sfm-1\}$.
By \reflem{closureOfAuxiliaryLeaf}, each $\sfd \sfa_n,n>\cN$ is linked with $\sfd \gamma_r$ for some $\sfm\le r$, which is a contradiction by \refprop{auxiliaryLeaf}.
\end{proof}

For each $i\in \NN\cup\{0\}$, by \refprop{squashedFenceDragging}, we can find a $\opi[\Lambda(\Gamma)]{w_{i-1}}{w_{i}}$-side regular sequence $\{\lambda_{i,j}\}_{j\in\NN}$.
If $\{\lambda_{1,j}\}_{j\in\NN}$ is outside, then by \reflem{outsideExt2} and \refcor{limitAuxiliaryLeaf}, we can extend $\{z_i\}_{i=1}^m$, keeping the bouncing condition and the circular ordering condition.
Hence, $\{\lambda_{1,j}\}_{j\in\NN}$ is inside.
Therefore, by \reflem{insideExt2}, we have that  $w_\sfk\in \cE_\infty \setminus (\ell\cup \Fix(g))$.
However, since $h(w_0)=w_\sfk$ and  $h(p)\neq p$,  by \refprop{squashedFenceSystem}, $h(w_\sfk)\in \closure{\cE_m\setminus \cE_{m+1}}$ for some $m\in \ZZ$.
This is a contradiction since  $\cE_\infty  \cap \closure{\cE_m\setminus \cE_{m+1}} \subset \{p\}\cup \ell$ by \refprop{eclipseLProperty}.
Therefore, there is a point $z_{m+1}$ in $\opi[\Lambda(\Gamma)]{z_m}{z_0}$ such that $\{z_i\}_{i=1}^{m+1}$ satisfies the bouncing condition.

Inductively, we can find a strictly decreasing sequence $\{z_i\}_{i=1}^\infty$ in $\ell$ such that $\{z_i\}_{i=1}^m$ satisfies the bouncing condition for any $m>2$ and $\{z_i\}_{i=1}^\infty$ is strictly increasing in $\cldi[\Lambda(\Gamma)]{z_1}{z_0}$.
    Then, the closure $\cL_\infty$ of $\Gamma$-orbits of $\{\{z_{i-1},z_{i}\}:i\in \NN\}$ in $\cM$ is  a $\Gamma$-invariant circle lamination  by \refprop{finiteGenLeaves}.
    Then, by \refrmk{realization}, the geometric realization $\wt{\cG}_\infty$ of $\cL_\infty$ induces a geodesic lamination on $\cH/\Gamma$.
    Since  $\{z_i\}_{i=1}^\infty$ is strictly increasing in $\cldi[\Lambda(\Gamma)]{z_1}{z_0}$, for each $i\in \NN$,
    $\opi[\Lambda(\Gamma)]{z_i}{z_{i-1}}$ is isolated in $\cL_\infty$ by the right simpliciality of $\ell$.
    Hence, each leaf $\sfg_i$ of $\cG_\infty$, that corresponds to $\{z_{i-1},z_i\}$, is a boundary leaf in $\cG_\infty$.
    By \refprop{finiteBD}, there are only finitely many boundary leaves in $\cG_\infty$.
    This implies that there are $\sft, \sft'$  in $\NN$ such that $1<\sft <  \sft'$ and $(\sff(z_{\sft-1}),\sff(z_{\sft}))=(z_{\sft'-1},z_{\sft'})$ for some non-trivial element $\sff$ in $\Gamma$.
    Then, $\sff$ maps the subarc $[z_{\sft-1},z_{\sft}]$ of $\ell$ to the subarc $[z_{\sft'-1},z_{\sft'}]$ of $\ell$, preserving the orders.
    Observe that by the right inaccessibility of $\ell$,  $\sff([z_{\sft-1}, z_{\sft'-1}])$ is also a subarc of $\ell$.
    Hence, $\ell_\sff=\bigcup_{\sfm\in\ZZ} \sff([z_{\sft-1}, z_{\sft'-1}])$ is a $\sff$-invariant subline of $\ell$.
    By \reflem{invRayLand}, $\sfd \ell_\sff=\Fix(\sff)$.
    If $\sfd \ell_\sff$ intersects $\Fix(g)$, then $\Fix(g)=\Fix(\sff)$ by the discreteness of $G$.
    However, this is a contraction since $\Fix(g)\cap \Lambda(\Gamma)=\varnothing$ and $\Fix(\sff)\subset \Lambda(\Gamma)$.
    When  $\sfd \ell_\sff\subset \ell$, this is also a contracition by \refthm{twoInOne}.
    This completes the proof.
    \qedhere

\end{proof}

\appendix
\section{Geodesic laminations in hyperbolic surfaces}\label{App:geodLami}
In this appendix, we briefly summarize facts and propositions about geodesic laminations on surfaces, based on  \cite{Casson88}.
We follow the terminology of \cite{Casson88}, and for further details, we refer the reader to \cite{Casson88}.

Let $S_g$ be a closed hyperbolic surface and $\{\ell_i\}_{i=1}^n$ a disjoint collection of bi-infinite simple geodesics in $S_g$.
Note that the closure $\cG$ of  $\bigcup_{i=1}^n \ell_i$ is a geodesic lamination on $S_g$ (\cite[Lemma~3.2]{Casson88}).
The following summarizes \cite[Lemma~4.2, Lemma~4.3]{Casson88}:
\begin{prop}[Finitenss of boundary leaves]\label{Prop:finiteBD}
      Let $S_g$ be a closed hyperbolic surface and  $\cG$ a geodesic lamination on $S_g$.
      Then, $\cG$ has only finitely many principal regions, each with only finitely many boundary leaves.
      Moreover, the union of all boundary leaves in $\cG$ is dense in $\cG$.
\end{prop}

Recall that a \emph{crown} is a complete hyperbolic surface $\cC$ with finite area and geodesic boundary, constructed as follows:
let $g\in \Isomp(\HH^2)$ be a hyperbolic isometry.
Then, there is a $g$-axis $\ell$ in $\HH^2$.
Say that $a(g),r(g)$ are attracting and repelling fixed points of $g$, respectively.
Take distinct points $\{t_0,t_1,\cdots,t_n\}$ in $\partial \HH^2$ such that $r(g)<t_0<t_1<\cdots<t_n<a(g)$ on $\partial \HH^2$ and $g(t_0)=t_n$.
There is the unique closed hyperbolically convex region $\widetilde{\cC}$ in $\HH^2$ such that the closure of $\widetilde{\cC}$ in $\HH^2\cup \partial \HH^2$ is the union of $\widetilde{\cC}$ and $\{g^k(t_i):k\in \ZZ \text{ and } i=1,2,\cdots,n\}\cup \Fix(g)$.
Then, $\cC$ is defined as the quotient surface $\widetilde{\cC}/\langle g \rangle$.

Note that each component of the boundary of $\cC$ is either  the simple closed geodesic $\ell/\langle g \rangle$ or a bi-infinite geodesic whose lifting in $\widetilde{\cC}$ is the geodesic, joining $t_i,t_{i+1}$ for some $i$.
We call $\ell/\langle g \rangle$ the \emph{base} of $\cC$.
\begin{prop}\label{Prop:gapShape}
In \refprop{finiteBD}, each principal region $U$ is either a finite sided ideal polygon or there is a unique compact set $U_0$ such that each component of  $U\setminus U_0$ is isometric to the  interior of a crown.
\end{prop}
The compact set $U_0$ is called the  \emph{core} of $U$ and each component of $U\setminus U_0$ is called the \emph{open crown} of $U$.
Note that $U_0$ is either a simple closed geodesic or a compact hyperbolic surface with geodesic boundary.

We define an \emph{end} of $U$ as the restiction of a topological end of its metric completion, which is a descending sequence of open subsets.
We say that an end of a bi-infinite geodesic or of a geodesic ray \(\ell\) in \(S_g\)
\emph{escapes} toward an end \(\{e_i\}_{i\in\mathbb{N}}\) of \(U\) if each \(e_i\)
contains a proper subray of \(\ell\) corresponding to that end of $\ell$.

\begin{prop}\label{Prop:oneEndFix}
    Let $S_g$ be a closed hyperbolic surface and  $\cG$ a geodesic lamination on $S_g$.
    Write $\widetilde{\cG}$ for the lifting of $\cG$ to the universal cover $\HH^2$.
    If $\lambda$ is a leaf of $\widetilde{\cG}$ and there exists an element $h\in \pi_1(S_g)\leq \PSL(\RR)$ that fixes exactly one end point $e$ of $\lambda$ in $\partial \HH^2$, then $\lambda$ is isolated in $\cG$.
    In particular, if there is a point $p$ that is an end point of infinitely many leaves of $\widetilde{\cG}$, then $p$ is fixed by a non-trivial element of $\pi_1(S_g)$.
\end{prop}
\begin{proof}
    Note that there is a natural circular order on the Gromov boundary $\partial \HH^2$, induced from the orientation of $\HH^2$.
    Write $e,e'$ for the end points of $\lambda$ in $\partial \HH^2$.
    Recall that there is a unique maximal abelian subgroup $H$ of $\pi_1(S_g)$ that contains $h$, and since $\pi_1(S_g)$ is hyperbolic, $H\cong \ZZ$.
    Say that $k$ is a generator of $H$.
    Also, we may assume that $k(e')\in \opi[\partial \HH^2]{e}{e'}=R$ and so $k^{-1}(e')\in \opi[\partial \HH^2]{e'}{e}=L$.
    Assume that $\lambda$ is not isolated.
    By the density of the boundary leaves (\refprop{finiteBD}), we can take a non-trivial sequence of boundary leaves in $\cG$, such that it converges to $\lambda$ and the end points are contained in $\closure{R}$ or in $\closure{L}$.
    Without loss of the generality, we may assume that the end points of such a sequence in $\closure{R}$.
    Hence, we can find infinitely many distinct boundary leaves $\{\lambda_n\}_{n\in \NN}$  with one endpoint at $e$ and the other endpoints $e_n$ in $\opi[\partial \HH^2]{k(e')}{e'}$.
    However, since by \refprop{finiteBD}, there are only finitely many orbit classes of boundary leaves in $\widetilde{\cG}$, we can find two boundary leaves $\lambda_{i}, \lambda_{j}$ and $f\in \pi_1(M)$ such that  $f(e)=e$ and $\lambda_i=f(\lambda_j)$.
    Since $\{e_i,e_j\}\subset \opi[\partial \HH^2]{h(e')}{e'}$ and $e_j\neq e_i=f(e_j)$, by the choice of $k$,  we have that $f\nin H$.
    This implies that $h$ and $f$ are hyperbolic isometries, that share only one fixed point.
    This is a contradiction by the discreteness of $\pi_1(S_g)$ in $\PSL(\RR)$.
    Therefore, the first statement has been proved.
    The second statement follows from the proof of \cite[Lemma~4.5]{Casson88}.
    Thus, we are done.
\end{proof}

Recall that the \emph{derived lamination} $\cG'$ of a geodesic lamination $\cG$ in $S_g$ is defined as the union of non-isolated leaves of $\cG$.
We say that $\cG$ is \emph{perfect} if $\cG'=\cG$.
Note that  $\cG$ is perfect and connected if and only if every leaf of $\cG$ is dense in $\cG$ (\cite[Corollary4.7.2]{Casson88}).
Since for any geodesic lamination $\cG$ in $S_g$,  $\cG'''=\cG''$ (\cite[Corollary4.7.1]{Casson88}).
Hence, in each component $\cL$ of $\cG''$, every leaf of $\cL$ are dense in $\cL$.

The following is a combination of  \cite[Lemma~3.2, Lemma~3.6, Corollary~4.7.1, Corollary~4.7.2, Lemma~4.4, Lemma~4.2, Corollary4.7.2]{Casson88}, in \cite{Casson88}:
\begin{lem}\label{Lem:closureOfLeaf}
Let $\ell$ be a bi-infinite simple geodesic in a closed hyperbolic surface $S_g$. Then, the closure $\cG$ of $\ell$ in $S$ is a geodesic lamination.
In particular, the following hold:
\begin{itemize}
    \item If $\ell$ is isolated (from both side) in $\cG$, then  one of the following hold:
\begin{enumerate}
    \item if $\cG''=\varnothing$ , then $\cG'$ consists of one or two disjoint simple closed geodesics,  to which the ends of $\ell$ spiral.
    \item if $\cG''\neq \varnothing$, then there is a unique principal region $U$ of $\cG''$, that  contains $\ell$.

    In particular, $\cG''$ has at most two connected components and for each component $\cL$ of $\cG''$,  at least one end of $\ell$ escapes toward an end of $U$, spiraling toward $\cL$.
\end{enumerate}
\item If $\ell$ is not isolated in $\cG $ and is a boundary leaf of a principal region $U$ of $\cG$, then $\cG$ has neither isolated leaves nor closed leaves. Consequently, $\cG'=\cG$ and $\cG$ is connected.
\end{itemize}
\end{lem}

For a pair of disjoint geodesics $\ell_0,\ell_1$ in $S_g$, we say that $\ell_0$ and $\ell_1$ are \emph{asymptotic} if there exist lifts $\widetilde{\ell}_0$ of $\ell_0$ and $\widetilde{\ell}_1$ of $\ell_1$ in $\mathbb{H}^2$ that share one endpoint in $\partial \mathbb{H}^2$.
Then, it follows from \reflem{closureOfLeaf}, \refprop{oneEndFix} and \cite[Lemma~4.5]{Casson88}:
\begin{lem}\label{Lem:closureOfManyLeaves}
    Let $\{\ell_i\}_{i=1}^n$ be a finite collection of pairwise disjoint bi-infinite simple geodesics in a closed hyperbolic surface $S_g$. Then, the closure $\cG$ of $\bigcup_{i=1}^n\ell_i$ in $S$ is a geodesic lamination and it is the union of the geodesic laminations $\cG_i$, defined as the closure of $\ell_i$.

    In particular, if $\ell_i$ are boundary leaves of $\cG$, each $\ell_i$ is contained in the metric completion of a principal region of $\cG''$.
    Moreover, if $\ell_i$ and $\ell_j$ with $i\neq j$ are asymptotic and the asymptotic ends of those spiral toward a  component of $\cG''$, then there exists a principal region $U$ of $\cG''$ such that each of $\ell_i$ and $\ell_j$ is either contained in $U$ or a boundary leaf of $U$.

\end{lem}

\section*{Acknowledgement}

The author would like to thank Harry Hyungryul Baik, Michele Triestino, Chi Cheuk Tsang, Minju Lee, and Sam Taylor for several helpful conversations. The author is especially grateful to Danny Calegari and Ino Loukidou for many stimulating email exchanges, valuable comments, and helpful discussions. Special thanks are due to Danny Calegari for his encouragement to write this paper.

The author was supported by the National Research Foundation of Korea (NRF) grants funded by the Korean government (RS-2022-NR072395, RS-2025-02293115), by the Mid-Career Researcher Program through the National Research Foundation of Korea funded by the Korean government (RS-2023-00278510), and by the KIAS–KAIST joint research grant.

\bibliographystyle{alpha}
\bibliography{biblio}

\end{document}

%% file: header_newcommand.tex
\newcommand{\CC}{\mathbb{C}}

\newcommand{\HH}{\mathbb{H}}

\newcommand{\NN}{\mathbb{N}}

\newcommand{\RR}{\mathbb{R}}

\newcommand{\ZZ}{\mathbb{Z}}

\newcommand{\cA}{\mathcal{A}}

\newcommand{\cC}{\mathcal{C}}
\newcommand{\cD}{\mathcal{D}}
\newcommand{\cE}{\mathcal{E}}
\newcommand{\cF}{\mathcal{F}}
\newcommand{\cG}{\mathcal{G}}
\newcommand{\cH}{\mathcal{H}}
\newcommand{\cI}{\mathcal{I}}
\newcommand{\cJ}{\mathcal{J}}
\newcommand{\cK}{\mathcal{K}}
\newcommand{\cL}{\mathcal{L}}
\newcommand{\cM}{\mathcal{M}}
\newcommand{\cN}{\mathcal{N}}

\newcommand{\cR}{\mathcal{R}}
\newcommand{\cS}{\mathcal{S}}

\newcommand{\sG}{\mathscr{G}}

\newcommand{\sL}{\mathscr{L}}

\newcommand{\fA}{\mathsf{A}}
\newcommand{\fB}{\mathsf{B}}
\newcommand{\fC}{\mathsf{C}}
\newcommand{\fD}{\mathsf{D}}

\newcommand{\fF}{\mathsf{F}}

\newcommand{\fI}{\mathsf{I}}

\newcommand{\fK}{\mathsf{K}}

\newcommand{\fN}{\mathsf{N}}
\newcommand{\fO}{\mathsf{O}}
\newcommand{\fP}{\mathsf{P}}

\newcommand{\fS}{\mathsf{S}}

\newcommand{\fU}{\mathsf{U}}

\newcommand{\sfa}{\mathsf{a}}
\newcommand{\sfb}{\mathsf{b}}
\newcommand{\sfc}{\mathsf{c}}
\newcommand{\sfd}{\mathsf{d}}
\newcommand{\sfe}{\mathsf{e}}
\newcommand{\sff}{\mathsf{f}}
\newcommand{\sfg}{\mathsf{g}}

\newcommand{\sfi}{\mathsf{i}}
\newcommand{\sfj}{\mathsf{j}}
\newcommand{\sfk}{\mathsf{k}}

\newcommand{\sfm}{\mathsf{m}}
\newcommand{\sfn}{\mathsf{n}}

\newcommand{\sfp}{\mathsf{p}}
\newcommand{\sfq}{\mathsf{q}}
\newcommand{\sfr}{\mathsf{r}}
\newcommand{\sfs}{\mathsf{s}}
\newcommand{\sft}{\mathsf{t}}

\newcommand{\sfw}{\mathsf{w}}
\newcommand{\sfx}{\mathsf{x}}
\newcommand{\sfy}{\mathsf{y}}
\newcommand{\sfz}{\mathsf{z}}

\newcommand{\Isomp}{\operatorname{Isom_+}}

\newcommand{\Fix}{\operatorname{Fix}}

\newcommand{\Stab}[2]{\operatorname{Stab}_{#1}(#2)}

\newcommand{\PSL}{\operatorname{PSL}_2}

\newcommand{\id}{\mathrm{id}}

\newcommand{\opi}[3][S^1]{(#2, #3)_{#1}}
\newcommand{\ropi}[3][S^1]{[#2, #3)_{#1}}
\newcommand{\lopi}[3][S^1]{(#2, #3]_{#1}}
\newcommand{\cldi}[3][S^1]{[#2, #3]_{#1}}

\newcommand{\closure}[1]{\overline{#1}}
\newcommand{\wt}[1]{\widetilde{#1}}

\renewcommand{\setminus}{\smallsetminus}

\DeclareMathOperator{\Int}{\mathsf{Int}}

%% file: main.bbl
\begin{thebibliography}{Bow98}

\bibitem[Bow98]{Bowditch98}
Brian~H. Bowditch.
\newblock A topological characterisation of hyperbolic groups.
\newblock {\em J. Amer. Math. Soc.}, 11(3):643--667, 1998.

\bibitem[Cal06]{Calegari06}
Danny Calegari.
\newblock Universal circles for quasigeodesic flows.
\newblock {\em Geom. Topol.}, 10:2271--2298, 2006.

\bibitem[Cal07]{Calegari07}
Danny Calegari.
\newblock {\em Foliations and the geometry of 3-manifolds}.
\newblock Oxford Mathematical Monographs. Oxford University Press, Oxford,
  2007.

\bibitem[CB88]{Casson88}
Andrew~J. Casson and Steven~A. Bleiler.
\newblock {\em Automorphisms of surfaces after {N}ielsen and {T}hurston},
  volume~9 of {\em London Mathematical Society Student Texts}.
\newblock Cambridge University Press, Cambridge, 1988.

\bibitem[CD03]{CalegariDunfield03}
Danny Calegari and Nathan~M. Dunfield.
\newblock Laminations and groups of homeomorphisms of the circle.
\newblock {\em Invent. Math.}, 152(1):149--204, 2003.

\bibitem[CL24]{CalegariLoukidou24}
Danny Calegari and Ino Loukidou.
\newblock Zippers.
\newblock \url{https://arxiv.org/abs/2411.15610v1}, 2024.
\newblock arXiv:2411.15610v1.

\bibitem[CL26]{CalegariLoukidou26}
Danny Calegari and Ino Loukidou.
\newblock Zippers, 2026.
\newblock arXiv:2411.15610, to appear in Geometry \& Topology.

\bibitem[Fra13]{Frankel13}
Steven Frankel.
\newblock Quasigeodesic flows and {M}\"{o}bius-like groups.
\newblock {\em J. Differential Geom.}, 93(3):401--429, 2013.

\bibitem[Fra18]{Frankel18}
Steven Frankel.
\newblock Coarse hyperbolicity and closed orbits for quasigeodesic flows.
\newblock {\em Ann. of Math. (2)}, 188(1):1--48, 2018.

\bibitem[KM12]{KahnMarkovic12}
Jeremy Kahn and Vladimir Markovic.
\newblock Immersing almost geodesic surfaces in a closed hyperbolic three
  manifold.
\newblock {\em Ann. of Math. (2)}, 175(3):1127--1190, 2012.

\bibitem[Lev98]{Levitt98}
Gilbert Levitt.
\newblock Non-nesting actions on real trees.
\newblock {\em Bull. London Math. Soc.}, 30(1):46--54, 1998.

\bibitem[Mar13]{Markovic13}
Vladimir Markovic.
\newblock Criterion for {C}annon's conjecture.
\newblock {\em Geom. Funct. Anal.}, 23(3):1035--1061, 2013.

\bibitem[New51]{Newman51}
M.~H.~A. Newman.
\newblock {\em Elements of the topology of plane sets of points}.
\newblock Cambridge, at the University Press, 1951.
\newblock 2nd ed.

\bibitem[Thu22]{ThurstonUniI}
William~P. Thurston.
\newblock Three-manifolds, foliations and circles, {I} preliminary version.
\newblock In {\em Collected works of {W}illiam {P}. {T}hurston with commentary.
  {V}ol. {I}. {F}oliations, surfaces and differential geometry}, pages
  353--412. Amer. Math. Soc., Providence, RI, [2022] \copyright 2022.
\newblock December 1997 eprint.

\bibitem[Tuk98]{Tukia98}
Pekka Tukia.
\newblock Conical limit points and uniform convergence groups.
\newblock {\em J. Reine Angew. Math.}, 501:71--98, 1998.

\end{thebibliography}
